\documentclass[graybox]{svmult}

\usepackage{mathptmx}       
\usepackage{helvet}         
\usepackage{courier}        
\usepackage{type1cm}        
\usepackage{makeidx}         
\usepackage{graphicx}        
\usepackage{multicol}        
\usepackage[bottom]{footmisc}

\makeindex             

\usepackage[justification=centering]{caption}

\usepackage{amssymb}
\usepackage{amsmath}
\DeclareMathAlphabet{\mathcal}{OMS}{cmsy}{m}{n}

\DeclareMathOperator{\Div}{div}

\newcommand{\Th}{\mathcal{T}_{h}}
\newcommand{\R}{\mathbb{R}}
\renewcommand{\P}{\mathbb{P}}
\newcommand{\FF}{\mathcal{F}}
\newcommand{\triple}{|\!|\!|}
\begin{document}

\title*{Stability and optimal convergence of unfitted  extended finite element methods with Lagrange multipliers for the Stokes equations}
\titlerunning{Unfitted XFEM for Stokes equations}
\author{ Michel Fourni\'e and Alexei Lozinski }
\institute{
Michel Fourni\'e \at Name, Address of Institute \email{name@email.address}
\and 
Alexei Lozinski \at Laboratoire de Math\'ematiques de Besan\c{c}on, UMR CNRS 6623, Univ. Bourgogne Franche-Comt\'e, \email{alexei.lozinski@univ-fcomte.fr}
}
%
%
\maketitle

\abstract{We study a fictitious domain approach with Lagrange multipliers to discretize Stokes equations
on a mesh that does not fit the boundaries. A mixed finite element method is used for fluid flow.
Several stabilization terms are added to improve the approximation of the normal trace of the stress tensor and to avoid the inf-sup conditions between the spaces of the velocity and the Lagrange multipliers.
We generalize first an approach based on eXtended Finite Element Method due to Haslinger-Renard \cite{Haslinger09} involving a Barbosa-Hughes stabilization and a robust reconstruction on the badly cut elements.
Secondly, we adapt the approach due to Burman-Hansbo \cite{BurmanHansbo10} involving a stabilization  only on the Lagrange multiplier. 
Multiple choices for the finite elements for velocity, pressure and multiplier are considered.
Additional stabilization on pressure (Brezzi-Pitkäranta, Interior Penalty) is added, if needed. 
We prove the stability and the optimal convergence of several variants of these methods under appropriate assumptions.
Finally, we perform numerical tests  to illustrate the capabilities of the methods.
}

\section{Introduction}
\label{sec:intro}

Let $\mathcal{D}\subset\R^d$, $d=2$ or $3$, be a bounded polygonal (polyhedral) domain. We are interested in the Stokes equations in a setting motivated by the  fluid-structure interaction, especially by simulations of particulate flows. We thus assume that $\mathcal{D}$ is decomposed into the fluid domain $\FF$ and the solid one $\mathcal{S}$. The domains $\FF$ and $\mathcal{S}$ are separated by the interface $\Gamma$, cf. Fig. \ref{fig:domain}. We also denote  $\Gamma_{wall}=\partial\mathcal{D}$ and assume, for simplicity, that $\Gamma$ and $\Gamma_{wall}$ are disjoint. Consider the problem
\begin{align}
- 2\Div D(u)  + \nabla p & =  f \quad &&\text{in } \mathcal{F}, \label{system1} \\
\Div  u  & =  0 \quad &&\text{in }\mathcal{F}, \label{system2} \\
 u  & =  g \quad &&\text{on } \Gamma, \label{system5} \\
 u  & =  g_{wall} \quad &&\text{on } \Gamma_{wall}, \label{system3} 
\end{align}
for the velocity $u$ and the pressure $p$ of the fluid filling $\mathcal{F}$.
Here $D(u) = \frac{1}{2} \left(\nabla u + \nabla u^T \right)$  and the viscosity has been set to 1 for simplicity.
In applications we have in mind, i.e. simulations of the motion of rigid or elastic particles flowing in the fluid, the interface $\Gamma$  is moving in time while the outer boundary $\Gamma_{wall}$ is immobile. In this chapter, we shall study  Finite Element (FE) discretizations of the problem above on a mesh fixed on $\mathcal{D}$ which is thus fitted to $\Gamma_{wall}$ but is cut in an arbitrary manner by interface $\Gamma$. The interest of these methods in the context of fluid-structure interaction is that it allows one to avoid remeshing when the interface advances with time.  

\begin{figure}
\centering
\includegraphics[trim=0cm 0cm 0cm 0cm,clip,scale=0.3]{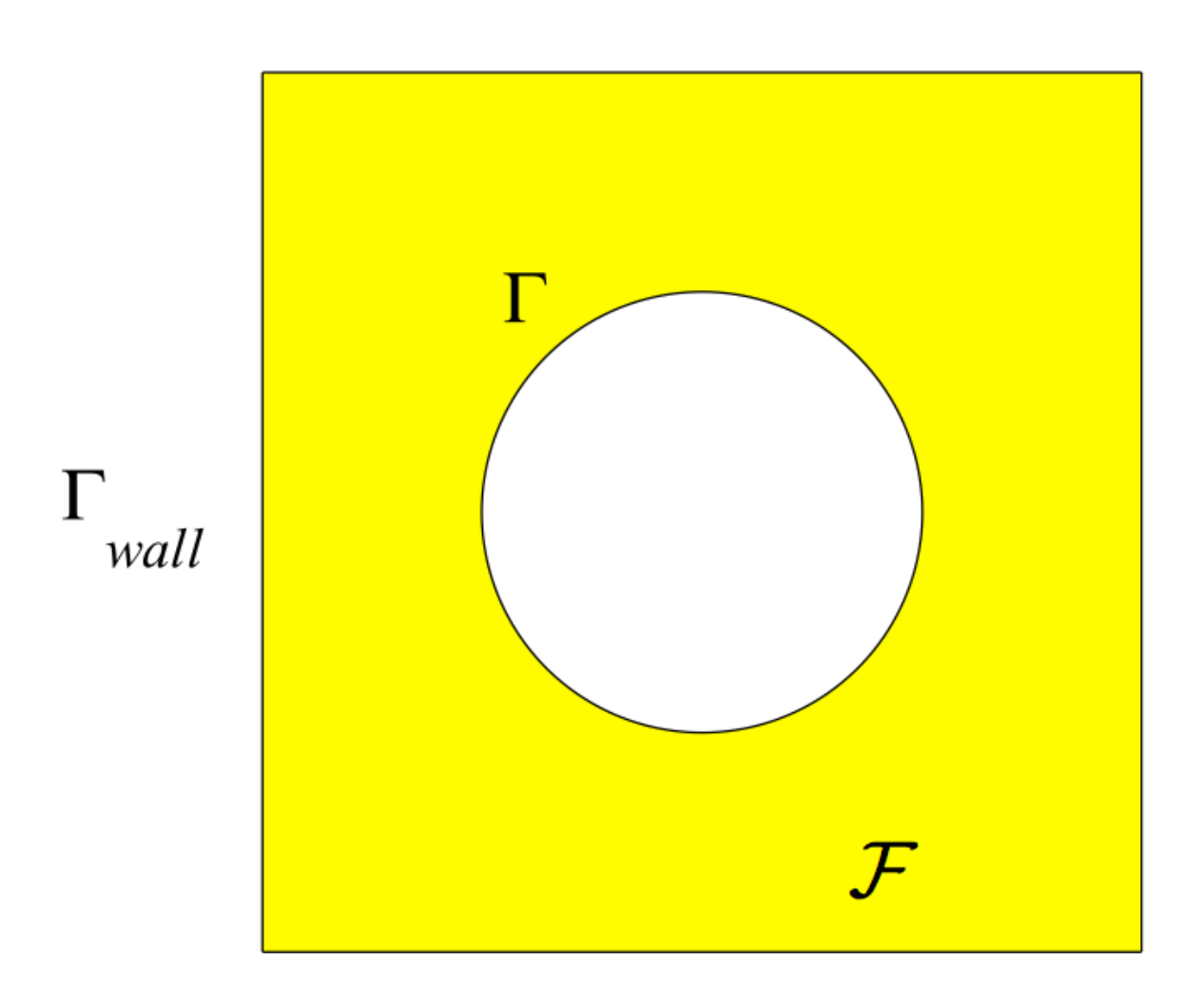}
\caption{The fluid domain $\FF$, the interface $\Gamma$ and the outer boundary $\Gamma_{wall}$.}
\label{fig:domain}
\end{figure}

Introducing the force exerted by the fluid on the solid at each point of $\Gamma$
\begin{eqnarray}\label{ForcLam}
\lambda =  -2 D(u)n +pn, \text{ on }\Gamma 
\end{eqnarray}
with $n$ the unit normal looking outside from $\mathcal{F}$, 
and interpreting $\lambda$ as the Lagrange multiplier associated with the Dirichlet conditions \eqref{system5}, we can write the weak formulation of \eqref{system1}--\eqref{system3} with $g_{wall}=0$ as
\begin{equation}\label{VarOrig}
\begin{array}{l}
 \text{Find }(u,p,\lambda)\in{H}^{1}_{wall}(\mathcal{F})^{d}\times L^2_0(\mathcal{F})\times{H}^{-\frac{1}{2}}(\Gamma)^d\text{ such that}\\
 \mathcal{A}(u,p,\lambda;v,q,\mu)=\mathcal{L}(v,\mu), \quad\forall(v,q,\mu)\in{H}^{1}_{wall}(\mathcal{F})^{d}\times L^2_0(\mathcal{F})\times{H}^{-\frac{1}{2}}(\Gamma)^d
\end{array}
\end{equation}
where
\begin{eqnarray*}
 \mathcal{A}(u,p,\lambda;v,q,\mu) & = & 2\int_{\mathcal{F}}D(u):D(v)-\int_{\mathcal{F}}(p\Div v+q\Div u)+\int_{\Gamma}(\lambda\cdot v+\mu\cdot u)
 \\
 \mathcal{L}(v,\mu) &=& \int_{\mathcal{F}}f\cdot v + \int_{\Gamma}g\cdot\mu
\end{eqnarray*}
and ${H}^{1}_{wall}(\mathcal{F})$ is the space of $H^1$ functions on $\mathcal{F}$ vanishing on $\Gamma_{wall}$ (we assume $g_{wall}=0$ in the theoretical analysis part of this paper to simplify the notations, the extension to $g_{wall}\not=0$ being trivial). The FE methods studied in this chapter will be based on the variational formulation (\ref{VarOrig}). They shall thus discretize the Lagrange multiplier $\lambda$, alongside $u$ and $p$, thus giving a natural approximation of the force exerted by the fluid on the solid.

As mentioned above, our FE methods will rely on a "background" fixed mesh $\Th$ that lives on the fluid-structure domain $\mathcal{D}\supset\mathcal{F}$ (the boundary of $D$ is $\Gamma_{wall}$ and is well fitted by $\Th$). In the actual computations, the elements of $\Th$ having no intersection with $\mathcal{F}$ will be discarded and the FE spaces for velocity and pressure will be defined on the mesh $\Th^e:=\Th^i\cup\Th^\Gamma$ where $\Th^\Gamma$ is the union of elements of $\mathcal{T}_h$ that are cut by $\Gamma$ and $\Th^i$ is the union of elements of $\mathcal{T}_h$ inside $\mathcal{F}$. The FE space for the Lagrange multiplier will live only on the cut elements $\Th^\Gamma$, cf. Fig. \ref{fig:mesh}.

\begin{figure}
\centering
\includegraphics[trim=0cm 0cm 0cm 0cm,clip,scale=0.3]{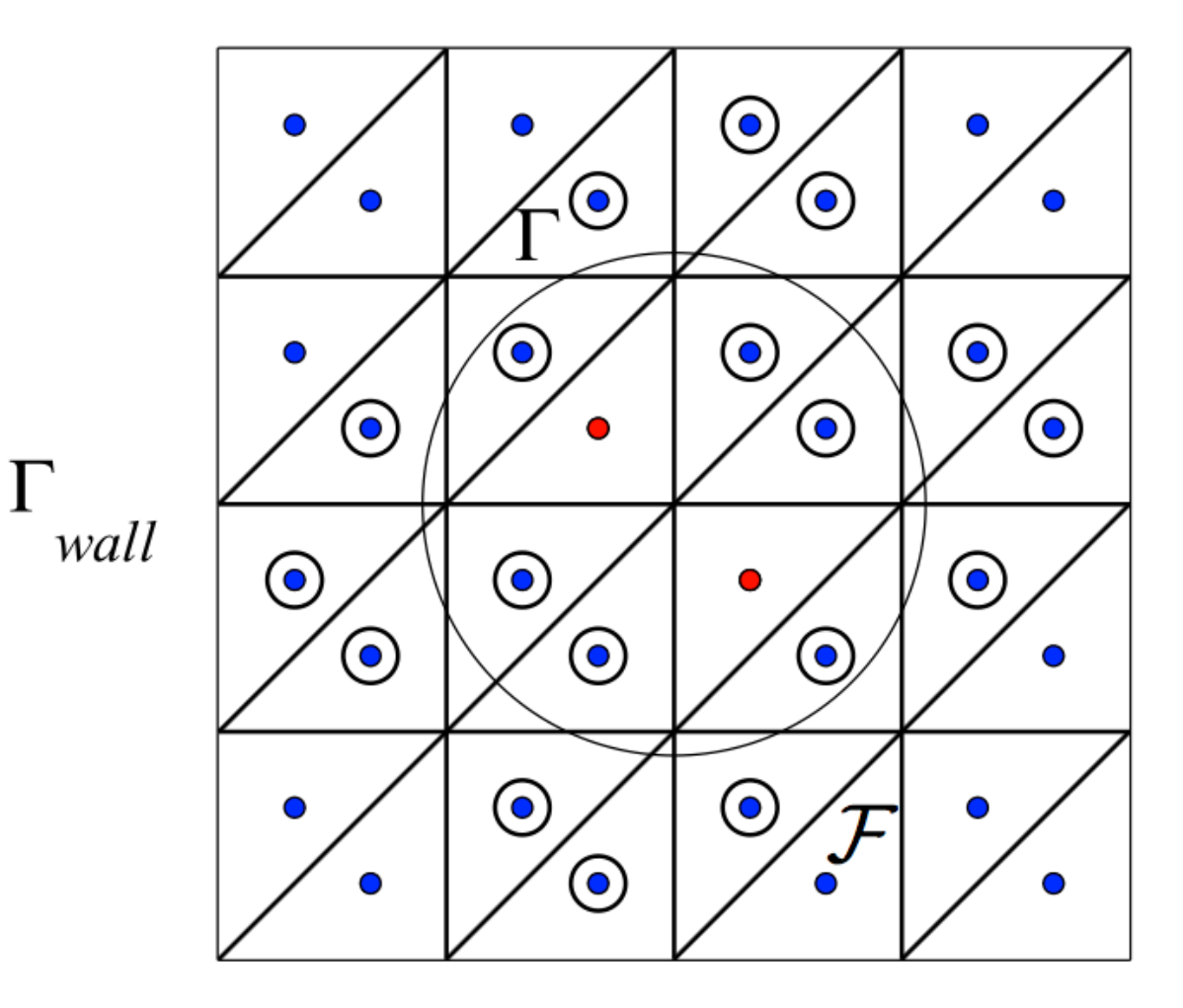}
\caption{The meshes $\Th^e=\Th^i \cup \Th^\Gamma$: the triangles of  $\Th^i$ are marked by \textcolor{blue}{$\bullet$} and those of $\Th^\Gamma$  are marked by  $\bigcirc$ \hspace*{-0.31cm}\textcolor{blue}{$\bullet$} ;  triangles marked by \textcolor{red}{$\bullet$} are not used.}
\label{fig:mesh}
\end{figure}

Denoting by $\mathcal{F}_h^e$ (resp. $\mathcal{F}_h^i$, $\mathcal{F}_h^\Gamma$) the domain covered by mesh $\Th^e$ (resp. $\Th^i$, $\Th^\Gamma$) we introduce three FE spaces 
\begin{equation}\label{fespaces}
V_{h}\subset{H}^{1}_{wall}(\mathcal{F}_h^e)^{d},\ 
{Q}_{h}\subset L^{2}(\mathcal{F}_h^e) \cap L_{0}^{2}(\mathcal{F}),\ {{W}}_{h}\subset{L}^{2}(\mathcal{F}_h^\Gamma)^{d}
\end{equation}
to approximate velocity, pressure and Lagrange multiplier respectively. Several choices of FE spaces $V_{h}$, $Q_h$, and $W_h$ will be considered, but we restrict ourselves in this chapter to triangular (tetrahedral) quasi-uniform meshes $\Th$ and to the standard continuous piecewise polynomial FE-spaces $\P_k$ ($k\ge 1$) or the piecewise constant space $\P_0$ on such a mesh.\footnote{The case of regular non-quasi-uniform meshes can also be easily treated at the expense of some technicalities. However, in applications, one will typically use a simplest possible mesh on $\mathbf{D}$ (for example, structured Cartesian) so that the quasi-uniformity restriction seems quite acceptable.}  Our FE spaces will be always based on meshes inherited from $\Th$: $\Th^e$ for $V_{h}$, $Q_h$, and $\Th^\Gamma$ for $W_h$. Note that velocity and pressure are approximated on a domain $\FF_h^e$ slightly larger than $\FF$ but all the integrals in the discretized problem will be calculated on $\FF$ or $\Gamma$. Note also that we choose the FE space for $\lambda$ on a domain $\mathcal{F}_h^\Gamma$ rather than on the surface $\Gamma$ to avoid the complicated issue of meshing a surface. 

A straightforward Galerkin approximation of (\ref{VarOrig}) is not stable in general (although it often works in practice, as will be seen in the numerical experiments at the end of this chapter). Several stabilization techniques were therefore proposed in the literature, using either Lagrange multipliers \cite{Haslinger09, BurmanHansbo10} or a Nitsche-like method \cite{BurmanHansbo14} to take into account the boundary conditions on $\Gamma$. We shall be concerned in this chapter only with the methods based on Lagrange multipliers. Firstly, we adapt the method of Haslinger-Renard (cf. \cite{Haslinger09} for the Poisson problem) to Stokes equations. The method is based on a Barbosa-Hughes stabilization \cite{Barbosa} on $\Gamma$ with additional local treatment on badly cut mesh elements. An extension to Stokes equations was already presented in \cite{Court14} but the analysis there relied on a number of hypotheses, difficult to verify. In this paper, we present a complete theoretical analysis in two cases :
\begin{enumerate}
\item LBB-unstable velocity-pressure FE pairs, namely, $\P_1-\P_1$ or $\P_1-\P_0$ elements. A stabilization is needed in this case even on a fitted mesh. We shall show, that adding the well known stabilization terms such as Brezzi-Pitk{\"a}ranta \cite{Brezzi84} for $\P_1-\P_1$ elements (or interior penalty for $\P_1-\P_0$ elements) to a Haslinger-Renard  fictitious domain method, as in \cite{Court14}, makes it stable and optimally convergent. 
\item LBB-stable velocity-pressure FE pairs, namely, $\P_k-\P_{k-1}$ Taylor-Hood elements. We show that a version of the method above (with and additional pressure stabilization on $\Gamma$ but avoiding stabilization over the whole domain $\FF$) is also stable and optimally convergent. Our proofs are presented here only in the 2D case and under some additional assumptions on the mesh. 
\end{enumerate}

We generalize moreover a method by  Burman-Hansbo \cite{BurmanHansbo10} to Stokes equations. This is also a fictitious domain method with Lagrange multipliers. Unlike the method by Haslinger-Renard (where the stabilization comes by enforcing (\ref{ForcLam}) on $\Gamma$ and thus involves all the variables $u$, $p$, $\lambda$), one stabilizes  here only the multiplier $\lambda$ by enforcing its continuity in some sense, so that the structure of resulting matrices is simpler. Fortunately, much of the theory outlined above can be reused for the analysis of this method. We are thus able to prove the stability and optimal convergence for the same choices of the FE spaces as above. 

The chapter is concluded by numerical experiments aiming at comparing different stabilizations and choices of of FE spaces.

\medskip

\noindent\textbf{Nomenclature.} 
\begin{description}
\item[Domains:] $\FF$ is the fluid domain where the problem (\ref{system1})--(\ref{system3}) is posed while $\FF_h^i$, $\FF_h^e$, $\FF_h^\Gamma$ are the domains occupied by the meshes  $\Th^i$, $\Th^e$, $\Th^\Gamma$ respectively. We have thus $\FF_h^i \subset \FF \subset \FF_h^e$ and  $\FF_h^\Gamma=\FF_h^e\setminus\FF_h^i$.\\
\item[Meshes:] $\Th^i$, $\Th^e$, $\Th^\Gamma$ are submeshes of a background mesh $\Th$ so that $\Th^i=\{T\in\Th : T\subset\FF\}$, $\Th^\Gamma=\{T\in\Th : T\cap\Gamma\not=\varnothing\}$ and $\Th^e:=\Th^i\cup\Th^\Gamma$. \\
$\mathcal{E}_h^e$ and $\mathcal{E}_h^\Gamma$ stand for the sets of interior edges of $\mathcal{T}_h^e$ and $\mathcal{T}_h^\Gamma$ respectively.\\
$\mathcal{F}_T$ (resp. $\Gamma_T$) denotes $T\cap\mathcal{F}$ (resp. $T\cap\Gamma$) for any cut element $T\in\Th^\Gamma$.\\
\item[Norms:] $\|\cdot\|_{k,\omega}$ stands for the norm in $H^k(\omega)$ where $\omega$ can be a domain in $\R^d$ or a $(d-1)$-dimensional manifold. We identify $H^0(\omega)$ with $L^2(\omega)$.\\ $|\cdot|_{k,\omega}$ stands for the semi-norm in $H^k(\omega)$, $k>0$.\\ $\|\cdot\|_{\infty,\omega}$ stands for the norm in $L^\infty(\omega)$. 
\end{description}

\section{Methods \`a la Haslinger-Renard}
The starting point for the construction of the Haslinger-Renard method (proposed in \cite{Haslinger09} for the Poisson equation) is to add to the variational formulation (\ref{VarOrig}) the Barbosa-Hughes stabilization \cite{Barbosa}, which enforces the relation $\lambda +2 D(u)n -pn =0$ on $\Gamma$. These terms take the form 
\begin{equation}\label{BarbosaStab}
-\gamma_{0}h\int_{\Gamma}(\lambda+2D({u}){n}-pn)\cdot\left(\mu+2D({v}){n}-qn\right)
\end{equation}
with a mesh-independent $\gamma_{0}>0$. This idea, at least in the context of the Poisson equation as in \cite{Haslinger09}, produces a stable and optimally convergent approximation provided the mesh elements are  cut by $\Gamma$ in a certain way so that $\mathcal{F}\cap T$ is a big enough portion of $T$ for any $T\in\mathcal{T}_h^\Gamma$. 
If, for some elements, this is not the case the method can be still cured by replacing the approximating polynomial in such ``bad elements'' by the polynomial extended from from adjacent ``good elements''. The relation between bad and good elements is made precise in the following \\[1mm]
\textbf{Assumption A.} We fix a threshold $\theta_{\min}\in(0,1]$ and declare any $T\in\mathcal{T}_h^\Gamma$ a good element (resp. bad element) if $\frac{|\mathcal{F}_T|}{|T|}\ge\theta_{\min}$ (resp. $\frac{|\mathcal{F}_T|}{|T|}<\theta_{\min}$). We assume that one can choose for any bad element $T$ a ``good neighbor'' $T'\in\mathcal{T}_h^e$, $\frac{|T'\cap\mathcal{F}|}{|T'|}\ge\theta_{\min}$, such that $T$ and $T'$ share at least one node, cf. Fig. \ref{fig:good_bad_T}.

\begin{remark} Typically, Assumption A will hold true even for $\theta_{\min}=1$ if the mesh is sufficiently refined. One could also relax the notion of a neighbor (at the expense of some complication of the forthcoming proofs) to the requirement $\operatorname{dist}(T,T')\le Ch$ with a mesh-independent $C>0$.
\end{remark}

\begin{figure}
\centering
\includegraphics[trim=0cm 0cm 0cm 0cm,clip,scale=0.9]{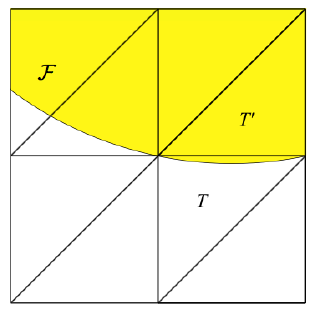}
\caption{Good element $T'$ and bad element $T$}
\label{fig:good_bad_T}
\end{figure}

We now define a ``robust reconstruction'' on $\mathcal{F}_h^\Gamma$ for the FE functions on $\mathcal{F}_h^e$
\begin{definition}\label{RobRec} For any $v_h\in V_h$ set $\widehat{v_h}$ on any $T\in\mathcal{T}_h^\Gamma$ as
\begin{itemize}
 \item $\widehat{v_h}=v_h$ on $T$ if $T$ is a good element,
 \item $(\widehat{v_h})_{|T}=(v_h)_{|T'}$ if $T$ is a bad element. Here $T'$ is the good neighbor of $T$ from Assumption A and the relation should be understood in the sense that $\widehat{v_h}$ on $T$ is taken as the same polynomial as the polynomial giving $v_h$ on $T'$.
\end{itemize}
For any $q_h\in Q_h$, one constructs $\widehat{q_h}$ in the same way.
\end{definition}

We shall show in the subsequent paragraphs that adding stabilization (\ref{BarbosaStab}) to (\ref{VarOrig}) and replacing $u,v$ in these terms (sometimes also $p,q$) by their robust reconstructions produces indeed a stable approximation to the Stokes equations. We end this general introduction  to the Haslinger-Renard method by a Proposition illustrating the usefulness of the selection criterion for good elements, showing that the $L^2$ norm on the cut portion of an element $T$ controls $L^\infty$ (and hence any other) norm on the whole element with an equivalence constant depending on the relative measure of the cut portion, followed by a list of interpolation error estimates that shall be needed in the forthcoming analysis.

\begin{proposition}\label{A0}
  Let $p$ be a polynomial of degree $\leq k$ and $\theta \in (0,
  1]$. Then for any $T \in \mathcal{T}_h$ and any measurable set $S \subset T$
  with $| S | \geq \theta | T |$ one has
  \begin{equation}\label{A0neq} 
    \|p\|_{\infty, T} \leq \frac{C}{h^{d / 2}} \| p \|_{0, S}
  \end{equation}
  with a constant $C > 0$ depending only on $\theta$, $k$ and mesh regularity.
\end{proposition}

\begin{proof}
  By scaling, it is sufficient to prove (\ref{A0neq}) on a reference
  element. We thus fix a simplex $T\in\R^d$ of diameter $h=1$ and consider for any $p
  \in \mathbb{P}_k$
  \[ N_{\theta} (p) = \inf_{S \subset T, | S | \geq \theta | T |} \| p \|_{0, S} 
  \]
  It is easy to see that $N_{\theta}$ is a continuous function on the
  finite-dimensional space $\mathbb{P}_k$. Consequently, it attains a minimum
  on the set $\Sigma_1 : = \{ p \in \mathbb{P}_k, \| p \|_{\infty, T} = 1
  \}$, i.e. $\exists \alpha \geq 0$ and $p_{\alpha} \in \Sigma_1$ such that
  $N_{\theta} (p) \geq N_{\theta} (p_{\alpha}) = \alpha$ for all $p\in\Sigma_1$. It remains to prove
  $\alpha > 0$. To this end, let $T_\delta = \{ x \in T : |  p_{\alpha} (x) | \leq \delta \}$,
  $m (\delta) = \operatorname{meas} \{ T_\delta \}$ for any $\delta \geq 0$. Since $m (\delta)$
  is decreasing down to 0 as $\delta \rightarrow 0$, one can find $\varepsilon
  > 0$ s.t. $m (\varepsilon) \leq \frac{\theta}{2} | T | $. We
  observe now $\| p_{\alpha} \|^2_{0, S} \geq \int_{S\setminus T_\varepsilon}p_{\alpha}^2 \geq \varepsilon^2 \left( | S | -
  \frac{\theta}{2} | T | \right)$ for any $S \subset T$, hence $\alpha^2 =
  N^2_{\theta} (p_{\alpha}) \geq \varepsilon^2 \frac{\theta}{2} | T |>0$. 
  By homogeneity, this also proves $N_{\theta} (p) \geq \alpha \|p\|_{\infty,T}$ for all $p\in\P_k$
  entailing (\ref{A0neq}) with $C = \frac{1}{\alpha}$ (we recall that the proof is done on the reference element with $h=1$).
\qed  
\end{proof}

We are going to establish interpolation estimates on the cut domain. To this end, we introduce\\[1mm]
\textbf{Assumption B.} $\Omega$ is a Lipschitz domain and there exist constants $c_\Gamma,C_\Gamma>0$ such that for any $T\in\mathcal{T}_h^\Gamma$
\begin{enumerate}
 \item $|\Gamma_T|\le C_\Gamma h^{d-1}$ with $\Gamma_T:=T\cap\Gamma$;
 \item there exists a unit vector $\chi_T\in\R^d$ such that $\chi_T\cdot n \ge c_\Gamma$ a.e. on $\Gamma_T$ where $n$ is the unit normal looking outward from $\mathcal{F}$.
\end{enumerate}

\begin{remark}\label{remark2}
  The bound on $|\Gamma_T|$ in the first part of the Assumption B is automatically satisfied on Lipschitz domain. We prefer however to write this bound explicitly in order to emphasize that some of the estimates below will depend on the constant $C_\Gamma$, so that $\Gamma$ should be supposed not too oscillating. The second part of the Assumption B is not too restrictive either. Typically, one can take $\chi_T$ as the normal $n$ at the middle point of $\Gamma_T$ if $\Gamma_T$ is smooth or as the average between the two normals if $\Gamma_T$ is the union of two segments (in the case when $\Omega$ is a 2D polygon). Such choices will suffice on a sufficiently refined mesh. 
\end{remark}

\begin{proposition}\label{InterpLem}
Let $V_h,Q_h,W_h$ be (respectively) $\P_{k_u},\P_{k_p},\P_{k_\lambda}$ FE spaces on meshes $\Th^e,\Th^e,\Th^\Gamma$ as in (\ref{fespaces}). Under Assumptions A and B, there exist interpolation operators $I_h^u: H^{1}_{wall}(\mathcal{F})^d \to V_h$, $I_h^p\in L^{2}_0(\mathcal{F}) \to Q_h$, $I_h^\lambda : H^{\frac{1}{2}}(\Gamma)^d \to W_h$ s.t. for any sufficiently smooth $u,p,\lambda$ 
\begin{align}
\frac 1h \|u-I_h^uu\|_{0,\mathcal{F}} + |u-I_h^uu|_{1,\mathcal{F}} + \frac{1}{\sqrt{h}} \|u-I_h^uu\|_{0,\Gamma} &\le C h^{s_u}|u|_{s_u+1,\mathcal{F}}
\label{IntEst}\\
& \quad\text{(for all integer }s_u : 0\le s_u \le k_u)
\notag\\
\left(\|\nabla u-\nabla{I_h^uu}\|_{0,\Gamma} + \|\nabla u-\nabla\widehat{I_h^uu}\|_{0,\Gamma} \right) &\le C h^{s_u -\frac 12}|u|_{s_u+1,\mathcal{F}}
\label{IntEstU}\\
& \quad\text{(for all integer }s_u : 1\le s_u \le k_u)
\notag\\
\frac 1h \|p-I_h^pp\|_{0,\mathcal{F}} + |p-I_h^pp|_{1,\mathcal{F}} \hspace{2cm}&\notag\\
 + \frac{1}{\sqrt{h}} \left( \|p-I_h^pp\|_{0,\Gamma} + \|p-\widehat{I_h^pp}\|_{0,\Gamma} \right) &\le C h^{s_p}|p|_{s_p+1,\mathcal{F}}
\label{IntEstP}\\
& \quad\text{(for all integer }s_p : 0\le s_p \le k_p)
\notag\\
\frac {1}{\sqrt h} \|\lambda-I_h^\lambda\lambda\|_{0,\Gamma} &\le C h^{s_\lambda}|\lambda|_{s_\lambda+\frac{1}{2},\Gamma}
\label{IntEstLam}\\
& \quad\text{(for all integer }s_\lambda : 0\le s_\lambda \le k_\lambda)
\notag
\end{align}
with $C>0$ depending only on the constants in Assumptions A, B, and on the mesh regularity, and $k_u\ge 1$ in the case of estimate (\ref{IntEstU}). 
Moreover, operator $I_h^\lambda$ can be extended to  $I_h^\lambda : H^{1}_{wall}(\mathcal{F})^d \to W_h$ s.t. for any $\tilde\lambda\in (H^{s_\lambda+1}(\mathcal{F})\cap H^1_{wall})^d $ and any integer $s_\lambda$, $0\le s_\lambda \le k_\lambda$
\begin{equation}\label{IntEstLamOm}
\frac 1h \|\tilde\lambda-I_h^\lambda\tilde\lambda\|_{0,\mathcal{F}_h^\Gamma} + |\tilde\lambda-I_h^\lambda\tilde\lambda|_{1,\mathcal{F}_h^\Gamma} +
\frac {1}{\sqrt h} \|\tilde\lambda-I_h^\lambda\tilde\lambda\|_{0,\Gamma} \le C h^{s_\lambda}|\tilde\lambda|_{s_\lambda+1,\mathcal{F}}
\end{equation}
\end{proposition}
\begin{proof} We start with the construction of $I_h^u$. Extension theorems for Sobolev spaces guarantee for any $u\in{H}^{s_u+1}(\mathcal{F})^{d}$ existence of  $\tilde u\in{H}^{s_u+1}(\mathcal{F}_h^e)^{d}$ with 
$\|\tilde{u} \|_{s_u+1, \mathcal{F}_h^e} \leq C \| u \|_{s_u+1, \mathcal{F}}$ and $\tilde{u} =
u$ on $\mathcal{F}$. Let $\tilde I_h:H^1_{wall}(\FF_h^e)^d\to V_h$ be a Cl{\'e}ment-type interpolation operator \cite{Ern} satisfying
\[ \frac{1}{h} \|  \tilde{u} -\tilde I_h \tilde u \|_{0, T} + |\tilde{u} -\tilde I_h \tilde u|_{1, T} 
   + \frac{1}{\sqrt{h}} \|  \tilde{u} -\tilde I_h \tilde u \|_{0, \partial T} 
   + {\sqrt{h}} \|  \nabla(\tilde{u} - \tilde I_h \tilde u) \|_{0, \partial T} 
   \le C h^{s_u} |\tilde{u} |_{s_u+1, \omega_T} \]
on any  $T\in\mathcal{T}_h^e$ with $\omega_T$ begin the patch of elements of $\Th^e$ touching $T$. 
Let $I_h^u u = \tilde I_h \tilde u|_\FF$. Summing the estimates above  over all
the mesh elements yields immediately the estimates in $L^2 (\mathcal{F})$ and $H^1 (\mathcal{F})$ in
(\ref{IntEst}). Now, on any element $T\in\mathcal{T}_h^{\Gamma}$
\[ c_\Gamma \| u - I_h^u u \|_{0, \Gamma_T}^2 \leq \int_{\Gamma_T} (\tilde u - \tilde I_h \tilde u )^2
   \chi_T \cdot n = \int_{\mathcal{F}_T} \Div ((\tilde u - \tilde I_h \tilde u )^2 \chi_T) -
   \int_{\mathcal{F} \cap \partial T} (\tilde u - \tilde I_h \tilde u )^2 \chi_T \cdot n \]
since $\partial \mathcal{F}_T =\Gamma_T \cup (\mathcal{F} \cap \partial T)$. Developing and
applying the interpolation estimates above gives
\begin{align*}
 c_\Gamma \| u - I_h^u u \|_{0, \Gamma_T}^2 
 &\leq \int_{\mathcal{F}_T} 2 (\tilde u - \tilde I_h \tilde u )
   \nabla (\tilde u - \tilde I_h \tilde u ) \cdot \chi_T 
   + \| \tilde u - \tilde I_h \tilde u  \|^2_{0, \mathcal{F} \cap \partial T} \\
 &\leq 2\| \tilde u - \tilde I_h \tilde u  \|_{0, T} | \tilde u - \tilde I_h \tilde u  |_{1, T} + \|
   \tilde u - \tilde I_h \tilde u  \|^2_{0, \partial T} \leq Ch^{2s_u+1} (| \tilde{u} |^2_{s_u+1,
   \omega_T}) 
\end{align*}
Summing this over all the elements in $\mathcal{F}_h^{\Gamma}$ yields the $L^2 (\Gamma)$-estimate
in (\ref{IntEst}).  

If $s_u \geq 1$, we have moreover on any $T \in \mathcal{T}_h^e$
\[ h | \tilde u - \tilde I_h \tilde u |_{2, T} + \sqrt{h}  \| \nabla (\tilde{u} - I_h \tilde u)\|_{0, \partial T} \le Ch^{s_u} | \tilde{u} |_{s_u + 1, \omega_T} \]
This, by the same argument as above, gives the $L^2 (\Gamma)$ estimate on
$\nabla (u - I_h^u u)$ in (\ref{IntEstU}). In order to extend this to $\nabla
(u - \widehat{I_h^u u})$ consider a bad element $T$ and its good neighbor
$T'$. Both $T$ and $T'$ belong to the patch $\omega_{T'}$ and examining the
derivation of interpolation estimates for the Cl{\'e}ment interpolator $\tilde I_h$
reveals that the polynomial $(\tilde I_h \tilde u) _{|T'}$ gives actually an optimal approximation of
$\tilde{u}$ on the whole $\omega_{T'}$, i.e.
\[ | u - \widehat{I^u_h u} |_{1, T} = | \tilde{u} - (\tilde I_h \tilde u )_{|T'}|_{1, T} 
\le| \tilde{u} - (\tilde I_h \tilde u )_{|T'}|_{1, \omega_{T'}} \leq {Ch}^{^{s_u}} | u |_{s_u + 1, \omega_{T'}} \]
Similarly, $\frac {1}{h} \| u - \widehat{I^u_h u} \|_{0, T} + \frac {1}{\sqrt{h}} \| u - \widehat{I^u_h u} \|_{0, \partial T}
\leq {Ch}^{^{s_u}} | u |_{s_u + 1, \omega_{T'}} $.  Thus, the same argument as above gives the $L^2 (\Gamma)$ estimate on
$\nabla (u - \widehat{I^u_h u})$ in (\ref{IntEstU}). 

The remaining estimates (\ref{IntEstP}), (\ref{IntEstLam}) and (\ref{IntEstLamOm}) are proved in a similar manner. We skip the details and make only the following remarks:
\begin{itemize}
 \item The operator $I_h^p$ should preserve the restriction that pressure is of zero mean on $\FF$. We thus define it as $I_h^p p=\tilde I_h \tilde p -i_h(p)$ where $\tilde I_h$ is the Cl{\'e}ment interpolation operator on $\Th^e$,  $\tilde p$ is an extension of $p$ to $\FF_h^e$, and $i_h(p) = \int_\FF\tilde I_h \tilde p$. The correction $i_h(p)$ can be bounded as 
 $$
  |i_h(p)| = \left|\int_\FF (\tilde I_h\tilde p -p)\right| 
  \le |\FF|^{\frac{1}{2}} \|\tilde p-\tilde I_h\tilde p\|_{0,\FF_h^e} \le C h^{s_p+1}|\tilde p|_{s_p+1,\FF_h^e}
 $$
 and thus it does not perturb the estimates (\ref{IntEstP}).
 \item Concerning the interpolation of $\lambda$, we note that (\ref{IntEstLam}) is in fact an easy corollary to (\ref{IntEstLamOm}). Indeed, for any $\lambda \in H^{k_\lambda+\frac{1}{2}}(\Gamma)^d$ there  exists (by the trace theorem) $\tilde\lambda\in H^{k_\lambda+1}(\mathcal{F}_h^\Gamma)^d$ satisfying $\tilde\lambda|_\Gamma=\lambda$ and $|\tilde\lambda|_{k_\lambda+1,\mathcal{F}_h^\Gamma} \le C|\lambda|_{k_\lambda+\frac{1}{2},\Gamma}$. We can thus define $I_h^\lambda\lambda:=I_h^\lambda\tilde\lambda$ and observe that (\ref{IntEstLamOm}) entails (\ref{IntEstLam}). 
\end{itemize}
\qed
\end{proof}

\subsection{$\P_1-\P_1$ velocity-pressure spaces with Brezzi-Pitk{\"a}ranta stabilization.}
\label{P1P1}

Let us choose $\P_1$ FE spaces for both ${{V}}_{h}$ and ${{Q}}_{h}$, add Brezzi-Pitk{\"a}ranta-like stabilization for the pressure and the Barbosa-Hughes-like stabilization on the interface as described above. We choose to introduce the robust reconstruction from Definition \ref{RobRec} in the last terms only for the velocity in this case (on both trial function $u_h$ and test function $v_h$). The method thus reads 
\begin{equation}\label{methBP}
\begin{array}{l}
 \text{Find }(u_{h},p_{h},\lambda_{h})\in{V}_{h}\times Q_{h}\times{W}_{h}\text{ such that}\\
 \mathcal{A}^{HR-BP}(u_{h},p_{h},\lambda_{h}; v_{h},q_{h},\mu_{h})=\mathcal{L}(v_{h},\mu_{h}),\hspace{1em}\forall(v_{h},q_{h},\mu_{h})\in{V}_{h}\times Q_{h}\times{W}_{h}
\end{array}
\end{equation}
where 
\begin{multline*}
\mathcal{A}^{HR-BP}(u,p,\lambda; v,q,\mu) = \mathcal{A}(u,p,\lambda; v,q,\mu)\\
  -\gamma_{0}h\int_{\Gamma}(\lambda+2D(\widehat{u}){n}-pn)\cdot\left(\mu+2D(\widehat{v}){n}-qn\right)-\theta h^{2}\int_{\mathcal{F}_h^e}\nabla p\cdot\nabla q
\end{multline*}
$V_h$, $Q_h$ are continuous $\P_1$ FE spaces on mesh $\Th^e$ and $W_h$ is $\P_1$ or $\P_0$ FE space on mesh $\Th^\Gamma$, cf. (\ref{fespaces}).

 We recall that the Brezzi-Pitk{\"a}ranta stabilization (the last term above) should be present on $\mathcal{F}$ to compensate the lack of the discrete inf-sup in P1-P1 velocity-pressure FE spaces. In addition, in our fictitious domain situation, it is extended to the larger domain $\mathcal{F}_h^e$ thus helping to ensure stability near $\Gamma$. 

In the following propositions, Assumptions A and B are implicitly implied and the constants $C$ may vary from line to line and depend on $c_\Gamma,C_\Gamma>0$ from Assumption B, $\theta_{\min}$ from Assumption A, and on the mesh regularity.
\begin{proposition}\label{A1}
 For all $v_{h}\in{V}_{h}$ one has 
\begin{equation}\label{A1neq}
h\|\nabla{\widehat{v_{h}}}\|_{0,\Gamma}^{2}  \leq  C|v_{h}|_{1,\mathcal{F}}^{2}
\end{equation}
\end{proposition}
\begin{proof}
Taking any $T\in\mathcal{T}_h^\Gamma$ and denoting its good neighbor by $T'$ we observe  
\[ \| \nabla{\widehat{v_{h}}} \|_{0, \Gamma_T} \le \sqrt{| \Gamma_T |} \| \nabla\widehat{v_{h}} \|_{L^{\infty} (T)} 
 \le C\sqrt{| \Gamma_T |} \| \nabla{v_{h}} \|_{L^{\infty} (T')} 
 \le C\frac{\sqrt{| \Gamma_T |}}{h^{d/2}}  \| \nabla{v_{h}} \|_{0,\mathcal{F}_{T'}} \]
  The last inequality above holds by Proposition \ref{A0} with a constant dependent on $\theta_{\min}$. The last but one inequality is easily proven by scaling given that $T$ and $T'$ are neighbors. 
  Using the bound $| \Gamma_T | \leq  C_{\Gamma}h^{d-1}$ and summing over all $T\in\mathcal{T}_h^\Gamma$ 
  yields  (\ref{A1neq}).
\qed  
\end{proof}

\begin{proposition}\label{A2}
 For all $q_{h}\in Q_{h}$ one has 
\begin{eqnarray*}
h\|q_{h}\|_{0,\Gamma}^{2} & \leq & C\left(\|q_{h}\|_{0,\mathcal{F}}^{2}+h^{2}|q_{h}|_{1,\mathcal{F}_h^e}^{2}\right)
\end{eqnarray*}
\end{proposition}
\begin{proof}
Using the notations $T,T'$ as in the preceding proof and assuming that these two elements share a node $x$, we observe  
\begin{align*}
 \| q_h \|_{0, \Gamma_T} &\le \sqrt{| \Gamma_T |} \| q_h \|_{L^{\infty} (T)} 
 \le \sqrt{| \Gamma_T |} (|q_h(x)| +  h\| \nabla{q_{h}} \|_{L^{\infty} (T)} )
 \\
 &\le \sqrt{| \Gamma_T |} (\|q_h\|_{L^{\infty} (T')} +  h\| \nabla{q_{h}} \|_{L^{\infty} (T)} )
 \le C\frac{\sqrt{| \Gamma_T |}}{h^{d/2}}   (\|q_h\|_{0,\mathcal{F}_{T'}} +  h\| \nabla{q_{h}} \|_{0,T} )
\end{align*}
We have used again Proposition \ref{A0} on the good element $T'$. 
  We conclude thanks to $| \Gamma_T | \leq  C_{\Gamma}h^{d-1}$ from Assumption B and the summation over all $T\in\mathcal{T}_h^\Gamma$. 
\qed  
\end{proof}

We shall also need a special interpolation operator adapted to functions
vanishing on $\Gamma$, the idea of which goes to \cite{Reusken16}.

\begin{proposition}\label{A3int}
  There exists an interpolation operator $I_h^0 : H^1_0(\mathcal{F})^d  \rightarrow V_h$ 
  such that 
  \[ \| v - I^0_h v \|_{0, \mathcal{F}} \le Ch |v|_{1, \mathcal{F}},  \quad |I^0_h v|_{1, \mathcal{F}} \le C |v|_{1, \mathcal{F}} \]
  and $I_h^0v=0$ on $\FF_h^\Gamma$ (and consequently $I_h^0v=0$ on $\Gamma$) for any $v \in H^1_0 (\mathcal{F})^d$ with a mesh-independent constant $C > 0$.
\end{proposition}

\begin{proof}
  The construction of $I_h^0$ will be based on the interpolator $I_h^u$ from Proposition \ref{InterpLem} with $k_u=1$. 
  For any $v \in H^1_0 (\mathcal{F})^d $,  let us put $I_h^0 v (x) = I_h^u (x)$ at all the interior nodes $x$ of $\Th^i$ (i.e. excepting the nodes lying on $\partial\Th^i$)
  and $I_h^0 v (x)= 0$ on all the nodes of  $\Th^{\Gamma}$. Since $I_h^0 v$ is the piecewise linear function on $\Th^e$, this uniquely defines it everywhere on $\FF_h^e$. Moreover, $I_h^0v=0$ on $\FF_h^\Gamma$.
  
  Let us denote, for a mesh edge $E$ lying on  $\partial \mathcal{F}_h^i$, the adjacent element from $\mathcal{T}_h^{\Gamma}$ by $T^{\Gamma}$ and the union of all the elements from $\mathcal{T}_h^i$ sharing at least a node with $E$ by $\omega_E^i$. By scaling
  \[ \| I_h^u v - I_h^0 v \|_{0, \omega_E^i} \leq C \sqrt{h} \| I_h^u v \|_{0, E}
     \leq C (\| I_h^u v \|_{0, T^\Gamma} + h | I_h^u v |_{1, T^\Gamma}) \]
  Summing over all such edges and introducing the extension $\tilde{v}$ to $\mathcal{F}_h^e$ as in the proof of Proposition \ref{InterpLem} yields
  \[ \| I_h^u v - I_h^0 v \|_{0, \mathcal{F}_h^i} 
  \leq C (\| I_h^u v \|_{0, \mathcal{F}_h^{\Gamma}} +  h | I_h^u v |_{1, \mathcal{F}_h^{\Gamma}})
  \leq C (\| \tilde{v} - I_h^u v \|_{0,
     \mathcal{F}_h^{\Gamma}} + \| \tilde{v} \|_{0, \mathcal{F}_h^{\Gamma}} + h | I_h^u v
     |_{1, \mathcal{F}_h^{\Gamma}}) \]
  Since $I_h^0 v = 0$ on $\mathcal{F}^{\Gamma}_h$ this entails
  \begin{align*}
    \| v - I_h^0 v \|_{0, \mathcal{F}} &\leq \| {v} \|_{0, \mathcal{F}} + \| v - I_h^u v \|_{0, \mathcal{F}_h^i} + \| I_h^u v
     - I_h^0 v \|_{0, \mathcal{F}_h^i} \\
     &\leq C (\| \tilde{v} \|_{0,\mathcal{F}_h^{\Gamma}} + \| \tilde{v} - I_h^u v \|_{0, \mathcal{F}_h^e} + h | I_h^u v
     |_{1, \mathcal{F}_h^{\Gamma}}) 
  \end{align*}
  We now employ the bound $\| \tilde{v} \|_{0,
  \mathcal{F}_h^{\Gamma}} \leq Ch | \tilde{v} |_{1, \mathcal{F}_h^{\Gamma}}$, which is
  valid since $\mathcal{F}_h^{\Gamma}$ is a band of thickness $h$ around $\Gamma$ and
  $\tilde{v} = 0$ on $\Gamma$. Moreover,
  $$      
     \frac 1h \| \tilde{v} - I_h^u v \|_{0, \mathcal{F}_h^e} + | I_h^u v |_{1, \mathcal{F}_h^e}
     \le C | \tilde{v} |_{1, \mathcal{F}_h^e}
  $$
  as follows from the proof of Proposition \ref{InterpLem}, cf. (\ref{IntEst}) with $s_u=0$. 
  Since $| \tilde{v} |_{1, \mathcal{F}_h^e} \le C | {v} |_{1, \mathcal{F}}$ by the extension theorem, this proves the announced estimate of  $\| v - I_h^0 v \|_{0, \mathcal{F}}$.
  
  The estimate for the $H^1$ norm of
  $I_h^0 v$ follows using the inverse inequality and the $L^2$ error estimates proved above:
  \[ | I_h^0 v |_{1, \mathcal{F}} = | I_h^0 v |_{1, \mathcal{F}_h^i} 
     \leq | I_h^0 v - I_h^u v |_{1, \mathcal{F}_h^i} + |
     I_h^u v |_{1, \mathcal{F}_h^i} \leq \frac{C}{h} \| I_h^0 v - I_h^u v \|_{0, \mathcal{F}_h^i} + | I_h^u v
     |_{1, \mathcal{F}_h^i} \leq C | v |_{1, \mathcal{F}} \]
\qed  
\end{proof}

\begin{lemma}\label{lemmainfsup1}
Under Assumption A and B, taking $\gamma_{0}>0$ small enough and any $\theta>0$, there
exists a mesh-independent constant $c>0$ such that 
\begin{eqnarray*}
\inf_{(u_{h},p_{h},\lambda_{h})\in{V}_{h}\times Q_{h}\times{W}_{h}}\sup_{(v_{h},q_{h},\mu_{h})\in{V}_{h}\times Q_{h}\times{W}_{h}}\frac{\mathcal{A}^{HR-BP}(u_{h},p_{h},\lambda_{h}; v_{h},q_{h},\mu_{h})}{\triple u_{h},p_{h},\lambda_{h}\triple \hspace{0.25em}\triple v_{h},q_{h},\mu_{h}\triple }\geq c
\end{eqnarray*}
where the triple norm is defined by 
\begin{eqnarray*}
\triple u,p,\lambda\triple =\left(|u|_{1,\mathcal{F}}^{2}+\|p\|_{0,\mathcal{F}}^{2}+h^2|p|_{1,\mathcal{F}_h^e}^{2}+h\|\lambda\|_{0,\Gamma}^{2}+\frac{1}{h}\|u\|_{0,\Gamma}^{2}\right)^{1/2}
\end{eqnarray*}
\end{lemma}

\begin{proof}
We observe, using Proposition \ref{A1},
\begin{multline*}
\mathcal{A}^{HR-BP}(u_{h},p_{h},\lambda_{h}; u_{h},-p_{h},-\lambda_{h})  \\
 =  2\|D(u_{h})\|_{0,\mathcal{F}}^{2}
 -4\gamma_{0}h\|D(\widehat{u_h})\|^{2}_{0,\Gamma}
 +\gamma_{0}h\|\lambda_{h}-p_{h}n\|^{2}_{0,\Gamma} 
 +\theta h^{2}|p_{h}|_{1,\mathcal{F}_h^e} \\
  \geq  2\|D(u_{h})\|_{0,\mathcal{F}}^{2} -C\gamma_0|u_{h}|_{1,\mathcal{F}}^{2}
   +\gamma_{0}h\|\lambda_{h}-p_{h}n\|_{0,\Gamma}^{2}+\theta h^{2}|p_{h}|_{1,\mathcal{F}_h^e} \\
  \geq  \frac 1K |u_{h}|_{1,\mathcal{F}}^{2}
   +\gamma_{0}h\|\lambda_{h}-p_{h}n\|_{0,\Gamma}^{2}+\theta h^{2}|p_{h}|_{1,\mathcal{F}_h^e}
\end{multline*}
We have used in the last line the assumption that $\gamma_{0}$ is sufficiently small and 
Korn inequality 
\begin{equation}\label{korn}
|v|_{1, \mathcal{F}}^2 \leq K \|D (v) \|_{0, \mathcal{F}}^2,
\quad\forall v\in H^1_{wall}(\FF)
\end{equation}
Note that the inequality is valid in this form because the functions from $H^1_{wall}(\FF)$ vanish on $\Gamma_{wall}$, i.e. on a part of the boundary $\partial\FF$ with non zero measure. 

The continuous inf-sup condition \cite{Girault} implies for all $p_{h}\in Q_{h}$ there exists $v_{p}\in H^1_{0}(\mathcal{F})^d$
such that 
\begin{equation}
-\int_{\mathcal{F}}p_{h}\Div v_{p}=\|p_{h}\|_{0,\mathcal{F}}^{2} \text{ and }  |v_{p}|_{1,\mathcal{F}}\leq C\|p_{h}\|_{0,\mathcal{F}}.\label{vp}
\end{equation}
Recalling that $v_{p}=I_{h}^0 v_{p}=0$ on $\Gamma$ we can write
\begin{eqnarray}
-\int_{\mathcal{F}}p_{h}\Div(I_h^0 v_{p}) & = & \|p_{h}\|_{0,\mathcal{F}}^{2}-\int_{\mathcal{F}}p_{h}\Div(I_{h}^0 v_{p}-v_{p})
 \notag\\
 &=& \|p_{h}\|_{0,\mathcal{F}}^{2}-\int_{\mathcal{F}}\nabla p_{h}\cdot(v_p-I_{h}^0v_{p})
  \geq  \|p_{h}\|_{0,\mathcal{F}}^{2} -Ch|p_{h}|_{1,\mathcal{F}_h^e}|v_{p}|_{1,\mathcal{F}}
 \notag\\
   &\geq& \|p_{h}\|_{0,\mathcal{F}}^{2} -Ch|p_{h}|_{1,\mathcal{F}_h^e}\|{p}_h\|_{0,\mathcal{F}}
 \label{pdivvDisc}  
\end{eqnarray}
where we have used the bounds from Proposition \ref{A3int} and (\ref{vp}). 
Combining this with Young inequality we obtain
\begin{eqnarray*}
\mathcal{A}^{HR-BP}(u_{h},p_{h},\lambda_{h}; I_{h}^0v_{p},0,0) 
 & \geq & -\|D(u_{h})\|_{0,\mathcal{F}} \|D(I_{h}^0 v_{p})\|_{0,\mathcal{F}} + \|p_{h}\|_{0,\mathcal{F}}^{2} -Ch|p_{h}|_{1,\mathcal{F}_h^e}\|p_{h}\|_{0,\mathcal{F}} \\
 & \ge & \frac{1}{2}\|p_{h}\|_{0,\mathcal{F}}^{2}-C|u_{h}|_{1,\mathcal{F}}^{2}-Ch^{2}|p_{h}|_{1,\mathcal{F}_h^e}^{2}
\end{eqnarray*}

Recall interpolation operator  $I_{h}^\lambda$ from Proposition \ref{InterpLem} and
observe, using Proposition \ref{A1} with Young inequality,
\begin{multline*}
\mathcal{A}^{HR-BP}(u_{h},p_{h},\lambda_{h}; 0,0,\frac{1}{h}I_{h}^\lambda u_{h})   =   \frac{1}{h}\int_\Gamma u_h \cdot I_{h}^\lambda u_{h}-\gamma_{0}\int_{\Gamma}(2D(\widehat{u}_h)n-p_{h}n+\lambda_h)\cdot I_{h}^\lambda u_{h}\\
   \geq   \frac{1}{2h}\|u_{h}\|_{0,\Gamma}^{2} - \frac{1}{2h}\|u_h-I_{h}^\lambda u_{h}\|_{0,\Gamma}^{2} 
    -\gamma_{0}\left(\frac{C}{\sqrt{h}}|{u}_{h}|_{1,\mathcal{F}}+\|\lambda_{h}-p_{h}n\|_{0,\Gamma}\right) \left( \|u_{h}\|_{0,\Gamma} + \|u_h-I_{h}^\lambda u_{h}\|_{0,\Gamma} \right) \\
 \geq  \frac{1}{4h}\|u_{h}\|_{0,\Gamma}^{2}-\frac Ch \|I_{h}^\lambda u_{h} - u_{h}\|_{0,\Gamma}^{2} -C|u_{h}|_{1,\mathcal{F}}^{2} -Ch\|\lambda_{h}-p_{h}n\|_{0,\Gamma}^{2}\\
 \geq  \frac{1}{4h}\|u_{h}\|_{0,\Gamma}^{2}-C|u_{h}|_{1,\mathcal{F}}^{2}-Ch\|\lambda_{h}-p_{h}n\|_{0,\Gamma}^{2}
\end{multline*}
In the last line, we have used the bound $ \|u_h-I_h^{\lambda}u_h\|_{0,\Gamma} \le C{\sqrt{h}}|u_h|_{1,\mathcal{F}}$, i.e. (\ref{IntEstLamOm}) with $s_\lambda=0$. 

Combining the above inequalities, we can obtain for any $\kappa,\eta>0$,  
\begin{multline}\label{ine10}
  \mathcal{A}^{HR-BP}(u_{h},p_{h},\lambda_{h}; u_{h}+\kappa I_{h}^0v_{p},-p_{h},-\lambda_{h}+\frac{\eta}{h}I_{h}^\lambda u_{h})
 \geq \frac 1K |u_{h}|_{1,\mathcal{F}}^{2}+\frac{\kappa}{2}\|p_{h}\|_{0,\mathcal{F}}^{2}+\frac{\eta}{4h}\|u_{h}\|_{0,\Gamma}^{2}\\
  +(\theta-C\kappa)h^{2}|p_{h}|_{1,\mathcal{F}_h^e}+(\gamma_{0}-C\eta)h\|\lambda_{h}-p_{h}n\|_{0,\Gamma}^{2}-C(\kappa+\eta)|u_{h}|_{1,\mathcal{F}}^{2}
\end{multline}
In order to split $p_h$ and $\lambda_h$ inside $\| \lambda_h -
p_h n\|_{0, \Gamma}$ we establish the following bounds with any $t > 0$ and use finally Proposition \ref{A2}
\begin{multline}\label{plamsplit}
  \|p_h n - \lambda_h \|_{0, \Gamma}^2  \geq \|p_h \|_{0, \Gamma}^2 + \|
  \lambda_h \|_{0, \Gamma}^2 - (t + 1) \|p_h \|_{0, \Gamma}^2 - \frac{1}{t +
  1} \| \lambda_h \|_{0, \Gamma}^2\\
  = \frac{t}{t + 1} \| \lambda_h \|_{0, \Gamma}^2 - t \|p_h \|_{0, \Gamma}^2 
  \ge \frac{t}{t + 1} \| \lambda_h \|_{0, \Gamma}^2 
   -\frac{Ct}{h} \left(\|p_{h}\|_{0,\mathcal{F}}^{2}+h^{2}|p_{h}|_{1,\mathcal{F}_h^e}^{2}\right)
\end{multline}
Substituting this into inequality (\ref{ine10}) and assuming $\gamma_0$, $\kappa$, $\eta$, $t$ sufficiently small, we obtain finally
\begin{align}  \label{al1}
  & \mathcal{A}^{HR-BP} (u_h, p_h, \lambda_h;u_h + \kappa I_h^0 v_p, - p_h, -
  \lambda_h + \frac{\eta}{h}I_{h}^\lambda u_{h})\\
  & \hspace{1em} \geq c \left( |u_h |_{1, \mathcal{F}}^2 +\|p_h \|_{0, \mathcal{F}}^2 +
  h^2 |p_h |_{1, \mathcal{F}_h^e}^2 + h\| \lambda_h \|_{0, \Gamma}^2 + \frac{1}{h}\|u_h \|_{0, \Gamma}^2 \right) 
    =  c \triple u_h, p_h, \lambda_h \triple ^2 . 
  \notag
\end{align}

On the other hand, the estimates of Propositions \ref{InterpLem} and \ref{A3int} give immediately
\begin{eqnarray}
  \triple u_h + \kappa I_h^0 v_p, - p_h, - \lambda_h + \frac{\eta}{h}I_{h}^\lambda u_{h} \triple  & \leq
  & C \triple u_h, p_h, \lambda_h \triple   \label{al2}
\end{eqnarray}
Dividing (\ref{al1}) by (\ref{al2}) yields the result of the Lemma.
\qed 
\end{proof}

\begin{theorem}\label{ThP1P1}
Under Assumptions A, B, $\gamma_{0}>0$ small enough, any $\theta>0$, and $(u,p,\lambda)\in H^2(\mathcal{F})^d \times L^2_0(\mathcal{F}) \times H^\frac{1}{2}(\Gamma)$,  
the following a priori error estimates hold  for method (\ref{methBP}):
\begin{equation}\label{ErrupP1P1}
 |u-u_{h}|_{1,\mathcal{F}}+\|p-p_{h}|_{0,\mathcal{F}}+\sqrt{h}\|\lambda-\lambda_{h}\|_{0,\Gamma}\le Ch(|u|_{2,\mathcal{F}}+|p|_{1,\mathcal{F}}+|\lambda|_{1/2,\Gamma})
\end{equation}
Moreover, assuming the usual elliptic regularity for the Stokes problem in $\mathcal{F}$, i.e. the bound (\ref{StokesElip}) for the solution to (\ref{Stokesvq}), one has $\forall\varphi\in H^{3/2}(\Gamma)$
\begin{equation}\label{ErrlamP1P1}
\left|\int_{\Gamma}(\lambda-\lambda_{h})\varphi\right|\le Ch^{2}(|u|_{2,\mathcal{F}}+|p|_{1,\mathcal{F}}+|\lambda|_{1/2,\Gamma})|\varphi|_{3/2,\Gamma}
\end{equation}
\end{theorem}

\begin{proof}
Use Galerkin orthogonality (taking $\widehat{u}=u$ for the exact solution $u$ and extending $p$ from $\FF$ to $\FF_h^e$) 
\begin{equation}\label{GalOrt}
     \mathcal{A}^{HR-BP} (u_h - u, p_h - p, \lambda_h - \lambda;v_h, q_h, \mu_h)
     =  \theta h^2 \int_{\mathcal{F}_h^e} \nabla p \cdot \nabla q_h
\end{equation}
to conclude
\begin{multline*}
     \mathcal{A}^{HR-BP} (u_h - I_h^uu, p_h - I_h^pp, \lambda_h - I_h^\lambda\lambda;v_h,
     q_h, \mu_h) = 2 \int_{\mathcal{F}} D (u - I_h^uu) : D (v_h) \\
     - \int_{\mathcal{F}} ((p - I_h^pp) \Div{v}_h + q_h \Div (u - I_h^uu)) 
     + \int_{\Gamma} ((\lambda - I_h^\lambda\lambda) \cdot v_h +\mu_h \cdot (u -
     I_h^uu))\\
      - \gamma_0 h \int_{\Gamma} (\lambda - I_h^\lambda\lambda + 2 D (u -
     \widehat{I_h^uu}) n - (p - I_h^pp) n) \cdot (\mu_h + 2 D
     (\widehat{v_h}) n - q_h n) \\
     + \theta h^2  \int_{\mathcal{F}_h^e} \nabla I_h^pp \cdot \nabla q_h
\end{multline*} 
All the terms in the right-hand side can be bounded thanks to Proposition \ref{InterpLem} with $s_u=1$, $s_p=s_\lambda=0$ so that
$$
\mathcal{A}^{HR-BP} (u_h - I_h^uu, p_h - I_h^pp, \lambda_h - I_h^\lambda\lambda;v_h,q_h, \mu_h)
\le Ch(|u|_{2,\mathcal{F}}+|p|_{1,\mathcal{F}}+|\lambda|_{1/2,\Gamma}) \triple v_h,q_h, \mu_h\triple 
$$
The inf-sup lemma \ref{lemmainfsup1} now gives (\ref{ErrupP1P1}).

To prove (\ref{ErrlamP1P1}), choose any $\varphi \in H^{3 / 2} (\Gamma)$ and take $v, q$ solution to
\begin{equation}\label{Stokesvq}
- 2 \Div{D} (v) + \nabla q = 0, \hspace{1em} \Div{v} = 0 \text{ on }
   \mathcal{F}, \hspace{1em} v = \varphi \text{ on } \Gamma 
\end{equation}
as well as $\mu= - (2 D (v) n - qn) |_{\Gamma}$. Integration by parts gives
\[ 2 \int_{\mathcal{F}} D (u - u_h) : D (v) - \int_{\mathcal{F}} q\Div (u - u_h) +
   \int_{\Gamma} (u - u_h) \mu= 0 \]
Subtracting this from Galerkin orthogonality relation (\ref{GalOrt}) gives
\begin{multline*}
 \int_{\Gamma} (\lambda - \lambda_h) \cdot \varphi = 2
   \int_{\mathcal{F}} D (u - u_h) : D (v - v_h) - \int_{\mathcal{F}} ((p - p_h)
   \Div (v - v_h) + (q - q_h) \Div (u - u_h)) \\
   + \int_{\Gamma}
   ((\lambda - \lambda_h) \cdot (v - v_h) + (\mu-\mu_h) \cdot (u -
   u_h)) \\
- \gamma_0 h \int_{\Gamma} (\lambda - \lambda_h + 2 D (u - \widehat{u_h}) n - (p
- p_h) n) \cdot (\mu-\mu_h + 2 D (v - \widehat{v_h}) n - (q - q_h) n) \\
- \theta h^2  \int_{\mathcal{F}_h^e} \nabla p_h \cdot \nabla q_h
\end{multline*}
Taking $v_h = I_h^u v$, $q_h = I_h^p q$, $\mu_h = I_h^\lambda \mu$, applying Proposition \ref{InterpLem} with $s_u=1$, $s_p=s_\lambda=0$ and recalling  that
\begin{equation}\label{StokesElip}
(|v|_{2,\mathcal{F}}+|q|_{1,\mathcal{F}}+|\mu|_{1/2,\Gamma})\le C|\varphi|_{3/2,\Gamma}
\end{equation}
thanks to the elliptic regularity of the Stokes problem, yields (\ref{ErrlamP1P1}).
\qed	
\end{proof}

\begin{remark}
The mesh elements with very small cuts may be present in method (\ref{methBP}) as well as in all its forthcoming variants. They can thus produce very ill conditioned matrices despite the stability guaranteed by Lemma \ref{lemmainfsup1} in the mesh dependent norms. The influence of this phenomenon on the accuracy of linear algebra solvers is yet to be investigated and remains  out of the scope of the present work. However, some partial results are available in \cite{Court14}. Note also that alternative methods based on the Ghost Penalty \cite{burmanghost} are free from this drawback, cf. \cite{BurmanHansbo14}. Indeed, the Ghost Penalty allows one to control velocity and pressure in the natural norms on the extended domain $\FF_h^e$ rather than on the fluid domain only, as in Lemma \ref{lemmainfsup1}.
\end{remark}

\subsection{$\P_1-\P_0$ velocity-pressure spaces with interior penalty stabilization.}
\label{P1P0}

Let us now choose $\P_1$ FE for ${{V}}_h$ and $\P_0$ for
${Q}_h$ and add interior penalty (IP) stabilization to the Haslinger-Renard method. The method becomes:
\begin{equation}\label{methIP}
\begin{array}{l}
 \text{Find }(u_{h},p_{h},\lambda_{h})\in{V}_{h}\times Q_{h}\times{W}_{h}\text{ such that}\\
 \mathcal{A}^{HR-IP}(u_{h},p_{h},\lambda_{h}; v_{h},q_{h},\mu_{h})=\mathcal{L}(v_{h},\mu_{h}),\hspace{1em}\forall(v_{h},q_{h},\mu_{h})\in{V}_{h}\times Q_{h}\times{W}_{h},
\end{array}
\end{equation}
where
\begin{multline*}
  \mathcal{A}^{HR-IP} (u, p, \lambda;v, q, \mu) = \mathcal{A} (u, p,
  \lambda;v, q, \mu)\\
  - \gamma_0 h \int_{\Gamma} (\lambda + 2 D (\widehat{u}) n - pn) \cdot
  (\mu+ 2 D (\widehat{v}) n - qn) - \theta h  \sum_{E \in \mathcal{E}_h^e}
  \int_{E} [p] [q]
\end{multline*}
$V_h$ is continuous $\P_1$ FE space on mesh $\Th^e$, $Q_h$ is $\P_0$ FE space on mesh $\Th^e$, and $W_h$ is $\P_0$ FE space on mesh $\Th^\Gamma$, cf. (\ref{fespaces}).

Note that the IP stabilization is applied to the pressure in the interior on $\mathcal{F}$ as well as on the cut elements. The analysis of this method is similar to that of the previous one (\ref{methBP}) and we give immediately the final result: 
\begin{theorem}\label{ThP1P0}
Under Assumptions A and B, $\gamma_{0}>0$ small enough, any $\theta>0$, and $(u,p,\lambda)\in H^2(\mathcal{F})^d \times L^2_0(\mathcal{F}) \times H^\frac{1}{2}(\Gamma)$,  
the a priori error estimates (\ref{ErrupP1P1}) and (\ref{ErrlamP1P1}) hold for method (\ref{methIP}).
\end{theorem}

\begin{proof} We shall not repeat all the technical details but only point out some important changes that should be made in Propositions \ref{A1}--\ref{A3int} and the inf-sup lemma from the preceding section in order to adapt them to the the analysis of method (\ref{methIP}):
\begin{itemize}
  \item The estimate of Proposition \ref{A2} should be changed to
  \[ \begin{array}{lll}
       h \|q_h \|_{0, \Gamma}^2 & \leq & C \left( \|q_h \|_{0, \mathcal{F}}^2 +
       h  \sum_{E \in \mathcal{E}_h^e} \int_{E} [q_h]^2 \right)
     \end{array} \]
  This can be proved observing on any bad element $T \in \mathcal{T}^{\Gamma}_h$ sharing an edge $E$ with its good neighbor $T'$
\begin{multline*}
       \|q_h \|_{0, \Gamma_T} = \sqrt{| \Gamma_T |} \, |(q_h) _{|T} | \le
     \sqrt{| \Gamma_T |}  (| [q_h]_E | + |(q_h)_{|T'} |)\\
     = \sqrt{| \Gamma_T |}  \left( \frac{1}{\sqrt{| E |}} \| [q_h] \|_{0,
     E} + \frac{1}{\sqrt{| T' |}} \| q_h \|_{0, T'} \right) \le C \left( \|
     [q_h] \|_{0, E} + \frac{1}{\sqrt{h}} \| q_h \|_{0, T'}) \right)
\end{multline*}
The case of a bad element that does not share an edge with its good neighbor can be treated similarly by introducing a chain of elements connecting $T$ to $T'$. The case when $T \in \mathcal{T}^{\Gamma}_h$ is ``good'' itself is trivial.

  \item The term $h^2 |p|_{1, \mathcal{F}_h^e}$ in the triple norm in Lemma \ref{lemmainfsup1} should be replaced by $h\sum_{E \in \mathcal{E}_h^e}
  \int_{E} [p]^2$
  
  \item The treatment (\ref{pdivvDisc}) of the velocity-pressure term inside the proof of Lemma
  \ref{lemmainfsup1} is now replaced by
  \begin{align*}
    - \int_{\mathcal{F}} p_h \Div I_h^0 v_p &= \|p_h \|^2_{0, \mathcal{F}} +
     \int_{\mathcal{F}} p_h \Div  (v_p - I_h^0 v_p) \\
     & = \|p_h \|^2_{0, \mathcal{F}} + \sum_{E \in \mathcal{E}_h^e} \int_{E
     \cap \mathcal{F}} [p_h] n \cdot (v_p - I_h^0 v_p) 
  \end{align*}
  and the bound $\sum_{E \in \mathcal{E}_h^e} \|v_p - I_h^0 v_p\|^2_{0,E} \leq {Ch} | v_p |_{1, \mathcal{F}}$ which is proved as in Proposition \ref{A3int}.
\end{itemize}
\qed 
\end{proof}

\subsection{Taylor-Hood spaces.}
\label{PkPk1}

We now choose $\P_k$ (resp. $\P_{k-1}$) FE space with $k\ge 2$ for ${V}_{h}$ (resp. $Q_{h}$). These are well known Taylor-Hood spaces which satisfy the discrete inf-sup conditions in the usual setting and thus no stabilization for pressure ``in the bulk'' is needed. Intuitively, some extra stabilization should be now added for the pressure on the cut triangles.  We thus propose the following modification of the Haslinger-Renard method for Taylor-Hood spaces:
\begin{equation}\label{methTH}
\begin{array}{l}
 \text{Find }(u_{h},p_{h},\lambda_{h})\in{V}_{h}\times Q_{h}\times{W}_{h}\text{ such that}\\
 \mathcal{A}^{HR-TH}(u_{h},p_{h},\lambda_{h}; v_{h},q_{h},\mu_{h})=\mathcal{L}(v_{h},\mu_{h}),\hspace{1em}\forall(v_{h},q_{h},\mu_{h})\in{V}_{h}\times Q_{h}\times{W}_{h},
\end{array}
\end{equation}
where 
\begin{eqnarray*}
\mathcal{A}^{HR-TH}(u,p,\lambda; v,q,\mu) & = & \mathcal{A}(u,p,\lambda; v,q,\mu)\\
 &  & -\gamma_{0}h\int_{\Gamma}(\lambda+D(\widehat{u}){n}-\widehat pn)\cdot\left(\mu+D(\widehat{v}){n}-\widehat qn\right)
\end{eqnarray*}
$V_h$ is continuous $\P_{k}$ FE space on mesh $\Th^e$, $Q_h$ (resp. $W_h$) is continuous $\P_{k-1}$ FE space on mesh $\Th^e$ (resp. $\Th^\Gamma$) for $k\ge 2$, cf. (\ref{fespaces}).
The notation  $\widehat{\cdot}$ stands here again for the ``robust reconstruction'' from Definition \ref{RobRec}. We emphasize that it is applied here not only to the velocity, but also to pressure, unlike versions of the method (\ref{methBP}) and (\ref{methIP}) studied above.

The analysis of this method will be done under more restrictive assumptions than that of the previous ones:\\[1mm]
\textbf{Assumption C.} The dimension is $d = 2$, $\mathcal{F}_h^i$ contains
at least 3 triangles, $\Gamma$ is a curve of class $C^2$, the mesh $\Th$ is composed of non-obtuse triangles and is sufficiently fine (with respect to the curvature of $\Gamma$).

\begin{remark}
  Assumption C covers Assumption B, cf Remark \ref{remark2}.
\end{remark}

We shall tacitly assume Assumption C in all the Propositions until the end of this Section. Proposition \ref{A1} will be reused in the analysis of the present case but Proposition \ref{A2} should be replaced with the following
\begin{proposition}\label{A2bis}
 For all $q_{h}\in Q_{h}$ one has 
\begin{eqnarray*}
h\|\widehat{q_h}\|_{0,\Gamma}^{2} & \leq & C \|q_{h}\|_{0,\mathcal{F}}^{2} .
\end{eqnarray*}
\end{proposition}
The proof is a straight-forward adaptation of Proposition \ref{A1} to the pressure space.

Another important ingredient in our analysis will be the discrete velocity-pressure inf-sup condition robust with respect to the cut triangles, cf. Proposition \ref{A4c} below. We recall first a well-known auxiliary result:
\begin{proposition}  \label{A4}
The exists a mesh independent constant $\beta > 0$ such that for
  any $q_h \in Q_h$
  \begin{equation}
    \beta h |q_h |_{1, \mathcal{F}_h^i} \le \sup_{v_h \in V_h^i} 
    \frac{\int_{\mathcal{F}} q_h \Div{v}_h}{|v_h |_{1, \mathcal{F}_h^i}}
  \end{equation}
  where $V_h^i = V_h \cap (H^1_0 (\mathcal{F}_h^i))^d$.
\end{proposition}

This result is customarily applied to the analysis of FE
discretization of the Stokes equations via the Verf{\"u}rth trick
{\cite{Ern}}. The proof in the 2D case under the assumption that the mesh contains at least 3 triangles can be found in \cite{Boffi09}. We note in passing that a 3D generalization in a similar context is presented in \cite{Guzman16}.

Let $B_h^{\Gamma}:=\mathcal{F}\setminus\mathcal{F}_h^i$  and note that the boundary of $B_h^{\Gamma}$ consists of $\partial\mathcal{F}_h^i$ and $\Gamma$. 
\begin{proposition}\label{A4aux1}
Let $p_h \in Q_h$ and $v \in H^1 (B_h^{\Gamma})$
  vanishing on $\Gamma$. Then
  \begin{equation}\label{A4aux1bound} 
    \int_{\partial \mathcal{F}_h^i} | p_h v | \leq C \|p_h
    \|_{0, B_h^{\Gamma}} |v|_{1, B^{\Gamma}_h}
  \end{equation}
\end{proposition}
\begin{proof}
  Take any triangle $T \in \mathcal{T}_h^{\Gamma}$ such that one of its
  sides $E$ is an  edge on $\partial \mathcal{F}_h^i$. Introduce the polar
  coordinates $(r, \varphi)$ centered at the vertex $O$ of $T$ opposite to
  side $E$ (thus $O$ lies outside $\mathcal{F}$). The part of $T$ inside $\mathcal{F}$
  can be represented in these coordinates as
  \[ \mathcal{F}_T = \{(r, \varphi) \text{ such that } \alpha < \varphi < \beta,
     \ r_{\Gamma} (\varphi) < r < r_i (\varphi)\} \]
  with $r_{\Gamma} (\varphi)$ and $r_i (\varphi)$ representing,
  respectively, $\Gamma$ and $E \subset \partial \mathcal{F}_h^i$. In view of
  Assumption C, $r_{\Gamma} (\varphi)$ is a $C^2$ function and there are positive numbers $r_{\min}$ and $r_{\max}$ such that $r_{\min} \leq
  r_{\Gamma} (\varphi) < r_i (\varphi) \leq r_{\max}$ for all $\varphi \in  [\alpha, \beta]$. There
  are 2 options: either $\Gamma_T$ is very close to edge $E$ so that  $\frac{r_{\max}}{r_{\min}} \leq \rho$, 
  or $\mathcal{F}_T$ covers a significant portion of $T$  so that $|\mathcal{F}_T | \geq \theta |T|$. 
  The positive numbers $\rho$ and $\theta$ here can be chosen in a mesh-independent manner.
  
  We start with the first option: $\frac{r_{\max}}{r_{\min}} \leq \rho$.
  Using the notations above and recalling $v = 0$ at $r = r_{\Gamma} (\varphi)$ gives
  \begin{multline*}
    \int_E | p_h v | \leq C \int_{\alpha}^{\beta} (| p_h v | r)_{r =r_{i} (\varphi)} d \varphi = C \int_{\alpha}^{\beta}
    \int_{r_{\Gamma} (\varphi)}^{r_i (\varphi)} \frac{\partial | p_h {vr}
    |}{\partial r} {drd} \varphi \\
    \leq C \left( \int_{\alpha}^{\beta}
    \int_{r_{\Gamma} (\varphi)}^{r_i (\varphi)} \left| \frac{\partial
    p_h}{\partial r} v \right| {rdrd} \varphi + \| p_h \|_{0, \mathcal{F}_T}
    \| \nabla v \|_{0, \mathcal{F}_T} + \frac{1}{r_{\min}} \| p_h \|_{0, \mathcal{F}_T}
    \| v \|_{0, \mathcal{F}_T} \right)\\
  \end{multline*}
  We set $l (\varphi) = r_i (\varphi) - r_{\Gamma} (\varphi)$ and bound the
  first integral above using, for any $\varphi$ fixed, an inverse inequality for
  $p_h$ on the interval $(r_{\Gamma}(\varphi),r_i(\varphi))$ and Poincar\'e inequality for $v$ on the same interval (recall that $v = 0$ at $r =
  r_{\Gamma} (\varphi)$)
  \begin{multline}\label{InverseThin}
     \int_{\alpha}^{\beta} \int_{r_{\Gamma} (\varphi)}^{r_i (\varphi)} \left|
     \frac{\partial p_h}{\partial r} v \right| {rdrd} \varphi \leq
     r_{\max} \int_{\alpha}^{\beta} \left( \int_{r_{\Gamma} (\varphi)}^{r_i
     (\varphi)} \left( \frac{\partial p_h}{\partial r} \right)^2 {dr}
     \right)^{\frac{1}{2}} \left( \int_{r_{\Gamma} (\varphi)}^{r_i (\varphi)}
     v^2 {dr} \right)^{\frac{1}{2}} d \varphi \\
     \leq Cr_{\max}
     \int_{\alpha}^{\beta} \frac{1}{l (\varphi)} \left( \int_{r_{\Gamma}
     (\varphi)}^{r_i (\varphi)} p_h^2 {dr} \right)^{\frac{1}{2}} \times l
     (\varphi) \left( \int_{r_{\Gamma} (\varphi)}^{r_i (\varphi)} \left(
     \frac{\partial v}{\partial r} \right)^2 {dr} \right)^{\frac{1}{2}} d
     \varphi \\
     \leq C \frac{r_{\max}}{r_{\min}} \|p_h \|_{0, \mathcal{F}_T} \| \nabla
     v \|_{0, \mathcal{F}_T}   
  \end{multline}
  Recalling the bound on $\frac{r_{\max}}{r_{\min}}$
  (which implies, in particular, $r_{\min} \geq \frac{h}{\rho}$) we conclude
  \begin{equation}
    \label{A4aux1T} \int_E | p_h v | \leq C \left( \| p_h \|_{0, \mathcal{F}_T} \|
    \nabla v \|_{0, \mathcal{F}_T} + \frac{1}{h} \| p_h \|_{0, \mathcal{F}_T} \| v
    \|_{0, \mathcal{F}_T} \right)
  \end{equation}
  
  On the other hand, if $| \mathcal{F}_T | \geq \theta | T |$, extending $v$ by 0
  outside $\mathcal{F}$, applying Proposition \ref{A0} and an inverse inequality
  (valid on the whole triangle $T$) also yields (\ref{A4aux1T}):
  \begin{multline*}
    \int_E | p_h v | \leq \sqrt{h}\| p_h v \|_{0,\partial T} \leq C (\| p_h v
     \|_{0, T} + h | p_h v |_{1, T}) \\
     \leq C (\| p_h
     \|_{\infty, T} \| v \|_{0, \mathcal{F}_T} + h \| \nabla p_h \|_{\infty, T} \|
     v \|_{0, \mathcal{F}_T} + h \| p_h \|_{\infty, T} \| \nabla v \|_{0,
     \mathcal{F}_T}) \\
     \leq C \left( \| p_h \|_{0, \mathcal{F}_T} \| \nabla v \|_{0,
     \mathcal{F}_T} + \frac{1}{h} \| p_h \|_{0, \mathcal{F}_T} \| v \|_{0, \mathcal{F}_T}
     \right) 
  \end{multline*}
  Summing (\ref{A4aux1T}) over all $T \in \mathcal{T}_h^{\Gamma}$ having a \
  side on $\partial \mathcal{F}_h^i$  yields
  \[ \int_{\partial \mathcal{F}_h^i} | p_h v | \leq C \| p_h \|_{0,
     B_h^{\Gamma}} \left( \| \nabla v \|_{0, B_h^{\Gamma}} + \frac{1}{h}
     \| v \|_{0, B_h^{\Gamma}} \right) \]
  Recall that $v = 0$ on $\Gamma$ and the width of $B_h^{\Gamma}$ is of
  order $h$, so that $\| v \|_{0, B_h^{\Gamma}} \leq Ch \| \nabla v
  \|_{0, B_h^{\Gamma}}$ by a Poincar\'e inequality. We have thus proved (\ref{A4aux1bound}).
 \qed
\end{proof}

\begin{proposition}\label{A4aux2}
There exists a continuous piecewise linear vector-valued
  function $\psi_h$ on mesh $\mathcal{T}_h^e$ such that $\psi_h\cdot n\ge 0$ on $\Gamma$, 
  $\Div \psi_h \ge \delta_0$ on all the triangles of $\mathcal{T}_h^{\Gamma}$, and $\Div
  \psi_h \ge - \delta_1 h$ on all the triangles of $\mathcal{T}_h^i$ with positive constants
  $\delta_0, \delta_1$. Moreover, there is a constant $C>0$ such that for any $p_h \in Q_h$ 
  \begin{equation}
    \label{A4aux2bound}  | p_h \psi_h |_{1, \mathcal{F}} + \frac{1}{\sqrt{h}}
    \|p_h \psi_h \|_{0, \Gamma} \leq C \|p_h \|_{0, \FF}  
  \end{equation}
\end{proposition}

\begin{proof}
Let $\mathcal{B}_\eta = \{ x \in \mathbb{R}^2 / \operatorname{dist} (x, \Gamma) < \eta
  \}$ for  $\eta>0$. Thanks to the smoothness of $\Gamma$, one can introduce orthogonal
  coordinates $(\xi_1, \xi_2)$ on $\mathcal{B}_\eta$ with some mesh-independent $\eta>0$ such that $\xi_2 = 0$ on
  $\Gamma$ and $\xi_2 < 0$ on  $\FF\cap\mathcal{B}_\eta$. Let $e_i$ denote the basis vectors of these coordinates ($e_i = \partial \mathbf{r}/\partial \xi_i$). One can safely assume that $\xi_2$ measures the distance to $\Gamma$ so that $|e_2|=1$ on $B_h$ and, moreover, $|e_1|=1$ on $\Gamma$. Assuming $\eta>h$, let us introduce the vector-valued function $\psi$ given on
  $\mathcal{B}_\eta$ by $\psi = \xi_2e_2$ for $|\xi_2|<h$, $\psi = -h\frac{\eta+\xi_2}{\eta-h}e_2$ for $-\eta<\xi_2<-h$, left undefined for $h<\xi_2<\eta$, and extended by 0 on $\FF\setminus\mathcal{B}_\eta$. This function is thus well defined and continuous on $\FF$. Let $\psi_h = I_h \psi + \delta\psi_h$ where $I_h$ is the standard nodal interpolation operator to continuous $\P_1$ FE space on $\Th^e$ and $\delta\psi_h$ is a small correction of order $h^2$ at each mesh node, which is also a continuous $\P_1$ FE function on $\Th^e$ to be specified below.

Clearly, $\Div \psi = 1$ on $\Gamma$, hence $\Div \psi \geq \frac{1}{2}$ on $\mathcal{B}_h$ by
  continuity for sufficiently small $h$. Since $\mathcal{B}_h
  \supset \mathcal{F}_h^{\Gamma}$, one observes  on all the
  triangles of $\mathcal{T}_h^{\Gamma}$
  \[ \Div \psi_h \geq \frac{1}{2} - \Div (\psi - I_h\psi) + \Div \delta\psi_h \geq
     \frac{1}{2} - {Ch} \| \psi \|_{W^{2, \infty} (\mathcal{B}_h)} - \frac Ch \|\delta\psi_h\|_{\infty,\Th^\Gamma}=
     \delta_0 > 0 \]
  since $h$ is sufficiently small and $\psi$ is sufficiently smooth thanks to
  the hypothesis on $\Gamma$. 
  Turning to the triangles of $\Th^i$ we make the following observation: if $\Gamma$ were a straight line, the coordinate system $(\xi_1,\xi_2)$ would be Cartesian, $\psi\cdot e_1$ would vanish, and $\psi\cdot e_2$ would be piecewise linear function of $\xi_2$ with a positive slope on  $-h<\xi_2<h$ and with the negative slope $-\frac{h}{\eta-h}$ on $-\eta<\xi_2<-h$ so that $\Div I_h\psi\ge -\frac{h}{\eta-h}$ on the triangles of $\Th^i$. The actual geometry of $\Gamma$ and the addition of $\delta\psi_h$ introduces the corrections of order $h^2$ to the nodal values of $\psi_h$ so that one still has  $\Div\psi_h\ge -\delta_1{h}$ on these triangles. 
  We can now adjust the correction $\delta\psi_h$ in order to satisfy the remaining requirement on  $\psi_h$, namely $\psi_h\cdot n\ge 0$ on $\Gamma$. We have $\psi\cdot n=0$ so that $\psi_h\cdot n\ge -c_0h^2$ on $\Gamma$. We now set $\delta\psi_h=c_1h^2e_2$ at all the nodes of $\Th^\Gamma$ outside $\FF$, $\delta\psi_h=\min(c_1h^2,-\psi\cdot e_2)e_2$ at all the nodes of $\Th^\Gamma$ in $\bar\FF$, and $\delta\psi_h=0$ at all the interior nodes of $\Th^i$ with some constant $c_1>0$. This assures $\psi_h\cdot n\ge 0$ on $\Gamma$ with some sufficiently big $c_1$. Moreover, if a node $x$ of mesh $\Th^i$ is too close to $\Gamma$, i.e. the distance between $x$ and $\Gamma$ is smaller than $h^2$ in order of magnitude, the construction above entails $\psi_h(x)=0$. This means that $|\psi_h|$ on the cut portion $\FF_T$ of any triangle $T\in\Th^\Gamma$ is always bounded by the width of $\FF_T$ (times some mesh independent constant) even if $\FF_T$ is narrower than $h^2$. 
  
Let us now take any $p_h \in Q_h$. Using an inverse inequality we deduce on any
triangle $T \in \mathcal{T}_h^i$
\begin{equation}\label{psihphloc}
  | p_h \psi_h |_{1, T} \leq {Ch} | p_h |_{1, T} + C \| p_h \|_{0, T}
   \leq C \| p_h \|_{0, T} 
\end{equation}
since, by construction of $\psi_h$,
  $$
  \| \psi_h \|_{\infty, \mathcal{F}_h^e}\le Ch \text{ and }
  \| \nabla \psi_h \|_{\infty,\mathcal{F}_h^e} \le C
  $$               
A similar bound also holds on any cut triangle $T \in
\mathcal{T}_h^{\Gamma}$. One cannot use a straightforward inverse inequality
in this case, since the width of the cut portion $\mathcal{F}_T$, say $\varepsilon$, can be much
smaller than $h$. However, the construction of $\psi_h$ implies in such a situation $\|\psi_h\|_{\infty,\FF_T}\le C\varepsilon$. Combining this with the inverse inequality $|p_h|_{1,\FF_T} \le \frac{C}{\varepsilon} \|p_h\|_{0,\FF_T}$, as in the proof of Proposition \ref{A4aux1}, one arrives at 
$| p_h \psi_h |_{1, \FF_T} \leq C \| p_h \|_{0, \FF_T}$, similar to (\ref{psihphloc}). Summing this over all the triangles $T \in \mathcal{T}_h^e$ yields $| p_h \psi_h |_{1, \mathcal{F}} \leq C \| p_h \|_{0,\mathcal{F}} $.

Finally, in order to bound $p_h \psi_h$ in $L^2
(\Gamma)$ we recall that the distance between $\Gamma$ and $\partial\FF_h^i$ is of order $h$. Hence,
\[ \| p_h \psi_h \|_{0, \Gamma} \leq \| p_h \psi_h \|_{0, \partial
   \mathcal{F}_h^i} + C \sqrt{h} | p_h \psi_h |_{1, B_h^{\Gamma}} \leq
   {Ch} \| p_h \|_{0, \partial \mathcal{F}_h^i} + C \sqrt{h} \| p_h
   \|_{0, \mathcal{F}} \]
By scaling, $\| p_h \|_{0, E} \leq \frac{C}{\sqrt{h}} \| p_h \|_{0, T}$
for any edge $E \in \partial \mathcal{F}_h$ adjacent to a triangle $T \in
\mathcal{T}_h^i $. The summation over all such edges yields $\| p_h
\|_{0, \partial \mathcal{F}_h^i} \leq \frac{C}{\sqrt{h}} \| p_h \|_{0,
\mathcal{F}_h^i}$ and consequently $\| p_h \psi_h \|_{0, \Gamma} \leq
C \sqrt{h} \| p_h \|_{0, \mathcal{F}}$ so that (\ref{A4aux2bound}) is established.
\qed  
\end{proof}

\begin{proposition}\label{A4c}
  Under Assumption C, for any $p_h \in Q_h$ there exists $v^p_{h} \in V_h$ such that
\begin{equation}
  - \int_{\mathcal{F}} p_h \Div{v}_h^p = \| p_h \|_{0, \mathcal{F}}^2 
  \text{ and }  |v_h^p |_{1, \mathcal{F}} + \frac{1}{\sqrt{h}} \| v_h^p \|_{0,
  \Gamma} \leq C \|p_h \|_{0, \mathcal{F}} 
  \label{vhpTH}
\end{equation}
\end{proposition}

\begin{proof}
  The continuous inf-sup condition implies that for all $p_h \in Q_h$ there
  exists $v_p \in (H^1_0 (\mathcal{F}))^d$ satisfying (\ref{vp}). 
  Recalling the interpolation operator $I_h^0$ from Proposition \ref{A3int}, we observe
  \begin{align}
      - \int_{\mathcal{F}} p_h \Div{I_h^0v_p} &= - \int_{\mathcal{F}_h^i} p_h \Div{I_h^0v_p} 
      =  \|p_h \|_{0, \mathcal{F}_h^i}^2 
    + \int_{\mathcal{F}_h^i} p_h \Div  (v_p-I^0_h v_p) 
    \notag\\
     &=  \|p_h \|_{0, \mathcal{F}^i_h}^2 + \int_{\partial \mathcal{F}_h^i} p_h n \cdot v_p  - \int_{\mathcal{F}_h^i} \nabla p_h \cdot (v_p - I^0_h v_p)     
     \notag\\
     &\geq  \|p_h \|_{0, \mathcal{F}^i_h}^2 -C \| p_h \|_{0, B_h^\Gamma} | v_p |_{1, B_h^\Gamma} 
      - Ch |p_h|_{1, \mathcal{F}^i_h} | v_p |_{1, \mathcal{F}_h^i} 
     \notag\\ 
      &\geq  \|p_h \|_{0, \mathcal{F}^i_h}^2 - C \left(\| p_h \|_{0, B_h^\Gamma}^2 + h^2  |p_h |_{1, \mathcal{F}_h^i}^2 \right)^{\frac 12} |v_p |_{1, \mathcal{F}}
     \notag\\
       &\geq  \frac 12 \|p_h \|_{0, \mathcal{F}^i_h}^2 - C\| p_h \|_{0, B_h^\Gamma}^2 - C h^2  |p_h |_{1, \mathcal{F}_h^i}^2
  \label{vhpM}
  \end{align}
  We have used Proposition \ref{A4aux1}, the interpolation estimate from Proposition \ref{A3int}, Young inequality and
  $| v_p |_{1, B_h^\Gamma}^2 + |v_p |_{1, \mathcal{F}_h^i}^2 = | v_p |_{1, \mathcal{F}}^2 \le C\| p_h \|_{0, \mathcal{F}}^2
    =C(\| p_h \|_{0, \mathcal{F}_h^i}^2+\| p_h \|_{0, B_h^\Gamma}^2)$.
  Moreover, thanks
  to Proposition \ref{A4} and the inverse inequality there exists $v^{p,
  i}_{h} \in V_h^i$ such that
  \begin{eqnarray}
    - \int_{\mathcal{F}} p_h \Div{v}_h^{p, i} = h^2 |p_h |_{1, \mathcal{F}_h^i}^2 &
    \text{ and } & |v_h^{p, i} |_{1, \mathcal{F}} \leq Ch |p_h |_{1,
    \mathcal{F}_h^i} \leq C \|p_h \|_{0, \mathcal{F}^i_h}  \label{vhpI}
  \end{eqnarray}
  In order to control $p_h$ on $B_h^\Gamma$, we introduce $v_h^{p, \Gamma} = -p_h \psi_h$ with $\psi_h$ from Proposition \ref{A4aux2}. Then
  \begin{multline}\label{vhpG} 
   - \int_{\mathcal{F}} p_h \Div  v^{p, \Gamma}_h =
     \int_{\FF} p_h \nabla p_h \cdot \psi_h + \int_{\FF}
     p^2_h \Div  \psi_h \\
     = \frac{1}{2} \int_{\Gamma} p^2_h
     n \cdot \psi_h + \frac{1}{2} \int_{\FF} p^2_h \Div  \psi_h
     \geq \frac{\delta_0}{2} \| p_h \|^2_{0, B_h^{\Gamma}} - {\delta_1}h \| p_h \|^2_{0, \FF_h^i}
  \end{multline}
  thanks to $n \cdot \psi_h \geq 0$ on $\Gamma$ and the bounds on  $\Div\psi_h$. 
  
  Let $v_h^p = I_h^0 v_p + \kappa v_h^{p, i} + \kappa v_h^{p, \Gamma}$. Taking the sum of  (\ref{vhpM}), (\ref{vhpI}), (\ref{vhpG}), and recalling
  $\| p_h \|_{0, \mathcal{F}}^2 = \|p_h \|_{0, \mathcal{F}_h^i}^2 + \| p_h \|_{0, B_h^\Gamma}^2$ yields for sufficiently big $\kappa > 0$ and sufficiently small $h$
  \[
      - \int_{\mathcal{F}} p_h \Div{v}_h^p       \geq  \frac{1}{2}\|p_h \|_{0, \mathcal{F}}^2     
  \]
  Turning to the second estimate in (\ref{vhpTH}), we recall
  \[ |I_h^0 v_p |_{1, \mathcal{F}} + |v_h^{p, i} |_{1, \mathcal{F}} \leq C \|p_h
     \|_{0, \mathcal{F}} \]
  and $I_h^0 v_p = v^{p, i}_h = 0$ on $\Gamma$. Moreover, $v^{p, \Gamma}_h$ is bounded thanks to (\ref{A4aux2bound}) as
  \[ |v_h^{p} |_{1, \mathcal{F}} + \frac{1}{\sqrt{h}}\| v_h^{p} \|_{0, \Gamma}
     \leq C\| p_h \|_{0, \FF}
  \]
  This entails (\ref{vhpTH}).
\qed  
\end{proof}

\begin{lemma} \label{lemmainfsup2}
 Under Assumption C, taking $\gamma_0$ small enough,
  there exists a mesh-independent constant $c > 0$ such that
  \begin{eqnarray*}
    \inf_{(u_h, p_h, \lambda_h) \in {V}_h \times Q_h \times
    {W}_h} \sup_{(v_h, q_h, \mu_h) \in {V}_h \times Q_h
    \times {W}_h}  \frac{\mathcal{A}^{HR-TH} (u_h, p_h, \lambda_h;v_h,
    q_h, \mu_h)}{\triple u_h, p_h, \lambda_h \triple  \hspace{0.25em} \triple  v_h,
    q_h, \mu_h \triple } \geq c
  \end{eqnarray*}
  where the triple norm is defined by
  \begin{eqnarray*}
    \triple u, p, \lambda \triple  = \left( |u|_{1, \mathcal{F}}^2 +\|p\|_{0, \mathcal{F}}^2 +
    h\| \lambda \|_{0, \Gamma}^2 + \frac{1}{h} \|u\|_{0, \Gamma}^2 \right)^{\frac  12}
  \end{eqnarray*}
\end{lemma}

\begin{proof}
  As in the proof of Lemma \ref{lemmainfsup1}, we observe that
  \begin{eqnarray*}
    \mathcal{A}^{HR-TH} (u_h, p_h, \lambda_h;u_h, - p_h, - \lambda_h) & \geq &
    \frac{1}{K} |u_h|_{1, \mathcal{F}}^2 + \gamma_0 h \| \lambda_h - \widehat{p}_h
    n\|_{0, \Gamma}^2
  \end{eqnarray*}
  thanks to Korn inequality (\ref{korn}) and the smallness of $\gamma_0$.
  Moreover, employing $v_h^p$ from Proposition \ref{A4c} and the estimates from Propositions \ref{A1} and \ref{A2bis},
\begin{eqnarray*}
  \mathcal{A}^{HR - TH} (u_h, p_h, \lambda_h ; v_h^p, 0, 0) & = & 2
  \int_{\mathcal{F}} D (u_h) : D (v_h^p) + \|p_h \|_{0, \mathcal{F}}^2 + \int_{\Gamma}
  \lambda \cdot v_h^p\\
  &  & - \gamma_0 h \int_{\Gamma} (D (\widehat{u_h}) n - \widehat{p_h} n +
  \lambda_h) \cdot D (\widehat{v^p_h}) n\\
  & \ge & \frac{1}{2} \|p_h \|_{0, \mathcal{F}}^2 - C |u_h |_{1, \mathcal{F}}^2 - Ch \|
  \lambda_h \|_{0, \Gamma}^2 - \frac{\gamma_0}{2} h \| \lambda_h -
  \widehat{p_h} n\|_{0, \Gamma}^2
\end{eqnarray*}
We proceed as in the proof of Lemma \ref{lemmainfsup1} and arrive at, cf. (\ref{ine10}),
\begin{multline*}
  \mathcal{A}^{HR - TH}  (u_h, p_h, \lambda_h ; u_h + \kappa v_h^p, - p_h, -
  \lambda_h + \frac{\eta}{h} I_h^\lambda u_h)
  \geq \frac{1}{K} |u_h|_{1, \mathcal{F}}^2 + \frac{\kappa}{2} \|p_h \|_{0, \mathcal{F}}^2 \\
  + \frac{\eta}{2 h} \|u_h \|_{0, \Gamma}^2 
  + \left( \frac{\gamma_0}{2} - C \eta
  \right) h \| \lambda_h - \widehat{p_h} n\|_{0, \Gamma}^2 - C (\kappa + \eta)
  |u_h |_{1, \mathcal{F}} - C \kappa h \| \lambda_h \|_{0, \Gamma}^2
\end{multline*}
The rest of the proof follows again that of Lemma \ref{lemmainfsup1}, with the
only modification that $\| \widehat{p_h} n\|_{0, \Gamma}^2$ rather than $\|p_h
n\|_{0, \Gamma}^2$ will appear in the calculation (\ref{plamsplit}). This gives now 
\[
  \| \widehat{p_h} n - \lambda_h \|_{0, \Gamma}^2 \geq \frac{t}{t + 1} \|
  \lambda_h \|_{0, \Gamma}^2 - \frac{Ct}{h}  \|p_h \|_{0, \mathcal{F}}^2
\]
which is established using Proposition \ref{A2bis} rather than Proposition \ref{A2}. Substituting this into the bound above and taking $t,\kappa,\eta$ sufficiently small leads to
\begin{equation*}
  \mathcal{A}^{HR - TH}  (u_h, p_h, \lambda_h ; u_h +
  \kappa v^p_h , - p_h, - \lambda_h + \frac{\eta}{h} I_h^\lambda u_h)
  \geq c \triple u_h, p_h, \lambda_h \triple ^2 
 \end{equation*}
Finally, the test function $(u_h +  \kappa v^p_h , - p_h, - \lambda_h + \frac{\eta}{h} I_h^\lambda u_h)$ can be bounded in the triple norm via $(u_h, p_h, \lambda_h)$. This ends the proof in the same way as as in the case of Lemma \ref{lemmainfsup1}.  
  \qed
\end{proof}

\begin{theorem}\label{ThPkPk1}
The following a priori error estimate hold under Assumption C for method (\ref{methTH}) with $\P_k$ FE for $v$ and $\P_{k-1}$  FE for $p$ and $\lambda$ ($k\ge 2$):
\begin{multline}
\label{ErrupPkPk1}
 |u-u_{h}|_{1,\mathcal{F}}+\|p-p_{h}\|_{0,\mathcal{F}}+\sqrt{h}\|\lambda-\lambda_{h}\|_{0,\Gamma}\\
 \le Ch^k(|u|_{k+1,\mathcal{F}}+|p|_{k,\mathcal{F}}+|\lambda|_{k-1/2,\Gamma})
\end{multline}
and, assuming the usual elliptic regularity (\ref{StokesElip}) for the Stokes problem (\ref{Stokesvq}), 
\begin{equation}\label{ErrlamPkPk1}
\left|\int_{\Gamma}(\lambda-\lambda_{h})\varphi\right|\le Ch^{k+1}(|u|_{k+1,\mathcal{F}}+|p|_{k,\mathcal{F}}+|\lambda|_{k-1/2,\Gamma}) |\varphi|_{3/2,\Gamma}
\end{equation}
for all $\varphi\in H^{3/2}(\Gamma)$.
\end{theorem}

\begin{proof} The proof follows the same lines as that of Theorem \ref{ThP1P1}.
\qed
\end{proof}

\section{Methods \`a la Burman-Hansbo.}
We turn now to alternative methods generalizing that of  \cite{BurmanHansbo10} to the Stokes equations, cf. (\ref{VarOrig}). The meshes and FE spaces follow the same pattern as before, cf. (\ref{fespaces}).  We shall employ either $\P_0$ or $\P_1$ FE for $\lambda$ and several choices for velocity and pressure.  The method reads: 
\begin{equation}\label{methBH}
\begin{array}{l}
 \text{Find }(u_{h},p_{h},\lambda_{h})\in{V}_{h}\times Q_{h}\times{W}_{h}\text{ such that}\\
 \mathcal{A}^{BH-l-var}((u_{h},p_{h},\lambda_{h}; v_{h},q_{h},\mu_{h})=\mathcal{L}(v_{h},\mu_{h}),\hspace{1em}\forall(v_{h},q_{h},\mu_{h})\in{V}_{h}\times Q_{h}\times{W}_{h},
\end{array}
\end{equation}
where
\begin{equation*}
  \mathcal{A}^{BH-l-var} (u, p, \lambda;v, q, \mu) = \mathcal{A} (u, p, \lambda;v, q, \mu) 
    + \mathcal{S}^l_\lambda(\lambda,\mu) + \mathcal{S}^{var}_p(p,q)
\end{equation*}
Here, $\mathcal{S}^l_\lambda(\lambda,\mu)$ with $l\in\{0,1\}$ is the stabilization term for Lagrange multiplier discretized by $\P_l$ FE. We set
\begin{align*}
 \mathcal{S}^0_\lambda(\lambda,\mu) = - \gamma h  \sum_{E \in \mathcal{E}_h^{\Gamma}}
  \int_{E} [\lambda] \cdot [\mu]
   \quad\text{and}\quad
 \mathcal{S}^1_\lambda(\lambda,\mu) =- \gamma h^2   \int_{\mathcal{F}_h^{\Gamma}} \nabla \lambda : \nabla \mu
\end{align*}
Moreover, $\mathcal{S}^{var}_p(p,q)$ with $var\in\{BP,IP,TH\}$  is the stabilization term for pressure chosen for each velocity-pressure FE-pair as in the following table

\begin{tabular}{l|l|l|l}
 \hline
 Velocity FE & Pressure FE & Acronym & Stabilization \\[1mm]
 \hline
 $\P_1$ & $\P_1$ & BP & $\mathcal{S}^{BP}_p(p,q) = - \theta h^2  \int_{\mathcal{F}_h^e} \nabla p \cdot \nabla q $ \\[1mm]
 \hline
 $\P_1$ & $\P_0$ & IP & $\mathcal{S}^{IP}_p(p,q) = - \theta h  \sum_{E \in \mathcal{E}_h^e}  \int_{E} [p] [q]$ \\[1mm]
 \hline
 $\P_2$ & $\P_1$ & TH & $\mathcal{S}^{TH}_p(p,q) = 0 $ \\
 \hline
\end{tabular}
\\[1mm]

\begin{remark}
Several other choices for FE spaces and corresponding stabilization terms could be proposed and investigated at the expense of more complicated proofs which we hope to present elsewhere. For instance,
\begin{itemize}
 \item In the case of $\P_{1}$ space for $\lambda$, one can use stabilization 
 $$ \tilde{\mathcal{S}}^1_\lambda(\lambda,\mu) = - \gamma h^3  \sum_{E \in \mathcal{E}_h^{\Gamma}}
  \int_{E} [\nabla\lambda] : [\nabla\mu]$$
  as an alternative to $ \mathcal{S}^1$. A similar stabilization is proposed in  \cite{BurmanHansbo10IMA} in the context of interface problems on non-conforming meshes without cut triangles.   
 \item Higher order Taylor-Hood spaces ($\P_k$--$\P_{k-1}$ for $k>2$) can be used for velocity-pressure accompanied with the $\P_{k-1}$ space for $\lambda$. One should then apply a stronger stabilization to $\lambda$, in the spirit of \cite{BurmanHansbo14},  which will control its higher order derivatives.
\end{itemize}
\end{remark}

One can show that all the choices above lead to inf-sup stable methods. We provide here a detailed proof for the case $\mathcal{A}^{BH-1-BP}$ (thus employing $\P_1$ FE for all the 3 variables) and comment briefly on other cases below.

\begin{lemma}\label{lemmainfsup3}
Let $V_h,Q_h,W_h$ in (\ref{fespaces}) be $\P_1$ FE spaces on respective meshes. 
Under Assumption B, for any $\gamma,\theta>0$ there exists a mesh-independent constant $c>0$ such that 
\begin{eqnarray*}
\inf_{(u_{h},p_{h},\lambda_{h})\in{V}_{h}\times Q_{h}\times{W}_{h}}\sup_{(v_{h},q_{h},\mu_{h})\in{V}_{h}\times Q_{h}\times{W}_{h}}\frac{\mathcal{A}^{BH-1-BP}(u_{h},p_{h},\lambda_{h}; v_{h},q_{h},\mu_{h})}{\triple u_{h},p_{h},\lambda_{h}\triple \hspace{0.25em}\triple v_{h},q_{h},\mu_{h}\triple }\geq c
\end{eqnarray*}
where the triple norm is defined by 
\begin{equation*}
\triple u, p, \lambda \triple  = \left( |u|_{1, \mathcal{F}}^2 +\|p\|_{0, \mathcal{F}}^2 + h^2
|p|_{1, \mathcal{F}_h^e}^2 + h\| \lambda \|_{0, \Gamma}^2 + h^2| \lambda|_{1, \mathcal{F}_h^{\Gamma}}^2 + \frac{1}{h} \|u\|_{0, \Gamma}^2 \right)^{1 / 2}
\end{equation*}
\end{lemma}
\begin{proof}
Take $\lambda_h\in W_h$ and let $\widetilde\lambda_h$ be the $\P_1$ FE function on  $\mathcal{F}_h^e$ that
vanishes at all the interior nodes of $\mathcal{F}_h^i$ and coincides with $\lambda_h$ on $\mathcal{F}_h^\Gamma$. 
Obviously, $h\widetilde\lambda_h\in V_h$ and 
\[ \int_{\Gamma} \lambda_h \cdot  h\widetilde\lambda_h  = h\| \lambda_h \|^2_{0,\Gamma} 
\]
Moreover, using a scaling argument and the fact that the distance between $\Gamma$ and $\partial\mathcal{F}_h^i$ is of order $h$, we get
\begin{multline}\label{lambdaht2}
  | \widetilde\lambda_h |_{1, \mathcal{F}}^2 \leq | \widetilde\lambda_h |_{1, \mathcal{F}_h^i}^2 + |
  \lambda_h |_{1, \mathcal{F}_h^{\Gamma}}^2 \leq \frac{C}{h} \| \lambda_h \|_{0,
  \partial \mathcal{F}_h^i}^2 + | \lambda_h |_{1, \mathcal{F}_h^{\Gamma}}^2\\
  \leq \frac{C}{h}  (\| \lambda_h \|_{0, \Gamma}^2 + h| \lambda_h |_{1,
  \mathcal{F}_h^{\Gamma}}^2) + | \lambda_h |_{1, \mathcal{F}_h^{\Gamma}}^2 \leq
  \frac{C}{h} \| \lambda_h \|_{0, \Gamma}^2 + C | \lambda_h |^2_{1,
  \mathcal{F}_h^{\Gamma}}
\end{multline}

   To control the pressure $p_h\in Q_h$, we recall the bound (\ref{pdivvDisc}) involving $v_p$ defined by (\ref{vp}) and interpolation operator
$I_h^0$ from Proposition \ref{A3int}. Thus, fixing $u_h,p_h,\lambda_h$ in the corresponding FE spaces, we have for any $\kappa, \rho, \eta > 0$
\begin{multline*}
  \mathcal{A}^{BH-1-BP}  (u_h, p_h, \lambda_h ; u_h + \kappa I_h^0 v_p +
  \rho h\widetilde\lambda_h, - p_h, - \lambda_h + \frac{\eta}{h}
  I_h^{\lambda} u_h) \ge \\ 2 \|D (u_h)\|_{0, \mathcal{F}}^2 - C | u_h |_{1, \mathcal{F}}
  (\kappa | I_h^0 v_p |_{1,^{} \mathcal{F}} + \rho h|\widetilde\lambda_h |_{1,
  \mathcal{F}})\\
  + \kappa \|p_h \|^2_{0, \mathcal{F}} - C \|p_h \|^{}_{0, \mathcal{F}} (\kappa h | p_h
  |_{1, \mathcal{F}^e_h} + \rho h|\widetilde\lambda_h |_{1, \mathcal{F}}) \\ 
  \rho h \| \lambda_h \|^2_{0, \Gamma} 
  + \frac{\eta}{2 h}  \|u_h \|^2_{0, \Gamma} -
  \frac{\eta}{2 h}  \|u_h - I_h^{\lambda} u_h \|^2_{0, \Gamma} \\ + \gamma h^2 |
  \lambda_h |_{1, \mathcal{F}_h^{\Gamma}}^2 - \gamma \eta h^{} | \lambda_h |_{1,
  \mathcal{F}_h^{\Gamma}} |I^{\lambda}_h u_h |_{1, \mathcal{F}_h^{\Gamma}} + \theta h^2
  |p_h |^2_{1, \mathcal{F}_h^e}
\end{multline*}
with a constant $C > 0$ independent from the mesh and from the parameters $\kappa$,
$\rho$, $\eta$, $\gamma$, $\theta$. We now apply Korn inequality (\ref{korn}),
the Young inequality and the bounds similar to those used in the proof of Lemma \ref{lemmainfsup1},
such as $| I_h^0 v_p |_{1,^{} \mathcal{F}} \leq C \|p_h \|^{}_{0, \mathcal{F}}$, $|
h\widetilde\lambda_h |_{1, \mathcal{F}} \leq C | v_{\lambda} |_{1, \mathcal{F}} \leq
C \sqrt{h} \| \lambda_h \|_{0, \Gamma}^{}$, $\|u_h - I_h^{\lambda} u_h
\|^{}_{0, \Gamma} \leq C \sqrt{h} |u_h |_{1, \mathcal{F}}$, $|I^{\lambda}_h u_h
|_{1, \mathcal{F}_h^{\Gamma}} \leq C |u_h |_{1, \mathcal{F}}$, and (\ref{lambdaht2}). 
This yields
\begin{align*}
 & \mathcal{A}^{BH - 1 - BP}  (u_h, p_h, \lambda_h ; v_h + \kappa I_h^0 v_p +
   \rho h\widetilde\lambda_h, - p_h, - \lambda_h + \frac{\eta}{h}
   I_h^{\lambda} u_h) \\ 
   & \qquad
   \ge (K - C \eta) |u_h |_{1, \mathcal{F}} + \left(
   \frac{\kappa}{2} - C \kappa^2 \right) \|p_h \|^2_{0, \mathcal{F}} + C \left(
   \rho - \rho^2 - \frac{\rho^2}{\kappa} \right) h \| \lambda_h \|^2_{0,
   \Gamma} \\ 
   &\qquad\quad
   + (\theta - C \kappa) h^2 |p_h |^2_{1, \mathcal{F}_h^e} + 
   \left(\gamma - C\rho^2 -C\frac{\rho^2}{\kappa}- \gamma^2 \eta\right) h^2 | \lambda_h |_{1, \mathcal{F}_h^{\Gamma}}^2 +
   \frac{\eta}{2 h}  \| u_h \|^2_{0, \Gamma} \\ 
   &\qquad
   \geq c\triple u_h, p_h, \lambda_h \triple ^2 
\end{align*}
if $\kappa, \rho, \eta > 0$ are chosen sufficiently small. In particular,  $\rho$
should be  small with respect to $\kappa$.

On the other hand, the test function $(u_h + \kappa I_h^0 v_p + \rho
h\widetilde\lambda_h, - p_h, - \lambda_h + \frac{\eta}{h} I_h^{\lambda}
u_h)$ can also be bound from above in the triple norm by $\triple u_h, p_h, \lambda_h \triple $
thanks to the bounds listed above. This leads to the announced inf-sup estimate. 
\qed
\end{proof}

Analogous inf-sup lemmas can be proved for all the other variants of method (\ref{methBH}) introduced above. In particular, the adaptation to the case $BH-1-TH$ is very simple: one should just use the velocity-pressure inf-sup Lemma \ref{A4c} (valid under Assumption C). The adaptation to the case $BH-0-IP$ requires some more substantial changes in the proofs as outlined below:
\begin{itemize}
 \item The term $h^2 | \lambda |^2_{1, \mathcal{F}_h^{\Gamma}}$ in the definition of the triple norm in Lemma \ref{lemmainfsup3} should be replaced by $h|\lambda|^2_{\mathcal{E}_h^{\Gamma}}:={h}\sum_{E\in\mathcal{E}_h^{\Gamma}} \|[\lambda]\|^2_{0,E}$. Similar modifications should be applied to the norm of the pressure, cf. the proof of Theorem \ref{ThP1P0}. 
 \item Using Cl{\'e}ment-type interpolation \cite{Ern}, given any $\P_0$ FE function $\lambda_h$ on $\FF_h^\Gamma$ we can construct a continuous $\P_1$ FE function $\widetilde\lambda_h$ on $\FF_h^\Gamma$ such that
 $$
  \|\widetilde\lambda_h-\lambda_h\|_{0,\Gamma} + \sqrt{h} | \widetilde\lambda_h |_{1, \mathcal{F}_h^\Gamma}
    \le C |\lambda_h|_{\mathcal{E}_h^{\Gamma}}
 $$
 We then extend it to $\widetilde\lambda_h\in V_h$ by setting its values at all the interior nodes of $\mathcal{F}_h^i$ to 0 and replace (\ref{lambdaht2}) with $| \widetilde\lambda_h |_{1, \mathcal{F}}^2 \leq 
  \frac{C}{h} \left(\| \lambda_h \|_{0, \Gamma}^2 + |\lambda_h|_{\mathcal{E}_h^{\Gamma}}^2 \right)$. The rest of the proof of Lemma \ref{lemmainfsup3} can be then reused as is. 
\end{itemize}

Having at our disposal the inf-sup Lemmas of the type \ref{lemmainfsup3}, it is easy to establish the convergence theorems completely analogous to Theorems \ref{ThP1P1}, \ref{ThP1P0}, and \ref{ThPkPk1}.

\begin{theorem}\label{ThBH} 
Consider the three variants of method (\ref{methBH}):  
$BH-1-BP$ under Assumption B with $\P_1$ FE for $v$, $p$ and $\lambda$; 
$BH-0-IP$ under Assumption B with $\P_1$ FE for $v$ and $\P_0$ FE for $p$, $\lambda$; 
$BH-1-TH$ under Assumption C with $\P_2$ FE for $v$ and $\P_1$ FE for $p$, $\lambda$. 
The following a priori error estimates hold for these methods with $k$ denoting the degree of FE space $V_h$
\begin{multline*}
 |u-u_{h}|_{1,\mathcal{F}}+\|p-p_{h}\|_{0,\mathcal{F}}+\sqrt{h}\|\lambda-\lambda_{h}\|_{0,\Gamma}\\
 \le Ch^k(|u|_{k+1,\mathcal{F}}+|p|_{k,\mathcal{F}}+|\lambda|_{k-1/2,\Gamma})
\end{multline*}
and
\begin{equation*}
\left|\int_{\Gamma}(\lambda-\lambda_{h})\varphi\right|\le Ch^{k+1}(|u|_{k+1,\mathcal{F}}+|p|_{k,\mathcal{F}}+|\lambda|_{k-1/2,\Gamma}) |\varphi|_{3/2,\Gamma}
\end{equation*}
for all $\varphi\in H^{3/2}(\Gamma)$
\end{theorem}

\begin{proof} The proof follows the same lines as that of Theorem \ref{ThP1P1}. In particular, all the necessary interpolation estimates can be taken from Proposition \ref{InterpLem}. Note that we no longer require Assumption A there since it is only necessary for the estimates involving $\widehat{I_h^uu}$ and $\widehat{I_h^pp}$.
\qed
\end{proof}

\section{Numerical experiments}
\label{sec:numerical}
In this section we present some numerical tests. 
The fluid-structure domain $D$  is set to $(0,1)^2$. The structure $\mathcal{S}$ is chosen as the disk centered in $[0.5,0.5]$ of radius $R=0.21$. We recall that the fluid domain is outside the structure, i.e. $\mathcal{F}=\mathcal{D}\setminus\bar{\mathcal{S}}$ as represented in Fig. \ref{fig:domain}. 
In practice, boundary $\Gamma$ of $\FF$ is defined by a level-set.
For all tests, the threshold ratio $\theta_{min}$ (cf. Definition \ref{RobRec}) for the "robust reconstruction" is fixed to $0.01$ and
the stabilization parameters are set as $\gamma_0= \theta_0= \gamma  = 0.05$.

The exact solution for the velocity and the pressure is chosen as
 \begin{align*}
   u(x,y) &=\left (\cos(\pi x)\sin(\pi y), -\sin(\pi x)\cos(\pi y)\right),\\
   p(x,y) &= (y-0.5)\cos(2 \pi x) + (x-0.5)\sin(2 \pi y)\\
 \end{align*}
and the right-hand side $f$ in (\ref{system1}) as well as the Dirichlet boundary conditions on $\Gamma_{wall}$ and $\Gamma$ in (\ref{system5})--(\ref{system3}) are set accordingly. We shall report the errors for velocity and pressure in the natural $H^1(\FF)$ and $L^2(\FF)$ norms. The accuracy of the Lagrange multiplier $\lambda$ will be attested only the the integral $\int_\Gamma\lambda$, which has the physical meaning of the force exerted by the fluid on the rigid particle inside. 
 
In the following,  $U$, $P$ and $\Lambda$ are the degrees of freedom vectors for $u_h$, $p_h$ and $\lambda_h$ respectively,
i.e. the coefficients in the expansions in the standard bases $\{\phi_{i_u}\}, \{\psi_{i_p}\}, \{\zeta_{i_{\lambda}}\}$ of $V_h$, $Q_h$ and $W_h$. The direct solver MUMPS \cite{MUMPS} is used for the resulting linear systems.
Rates of convergence are computed on regular meshes based on uniform subdivisions by $N$ points ($N=10, 20, 40, 80, 160$) on each side of  $\Gamma_{wall}$. At our fixed threshold, the three finer meshes require ``robust reconstruction", cf. Assumption A and Definition \ref{RobRec}). More precisely, for $N=40, 80$ and $160$ we have $8, 8$ and  $56$ ``bad elements''.

\subsection{Fictitious domain without any stabilization.} 
First, we present numerical tests without any stabilization as in (\ref{VarOrig}).
The linear system to solve is of the form
\begin{equation}
\label{sys:sans_stab}
\left (
\begin{array}{ccc}
K & B^T&C^T\\
B & 0&0\\
C & 0 & 0\\
\end{array}
\right )
\left (
\begin{array}{c}
U\\
P\\
\Lambda\\
\end{array}
\right ) = 
\left (
\begin{array}{c}
F\\
0\\
G \\
\end{array}
\right ) 
\end{equation}
where $K$, $B$, $C$, $F$ and $G$ are
$$\left ( K\right ) _{i_uj_u} = 2 \int_{\mathcal{F}} D(\phi_{i_u}):D(\phi_{j_u}), 
\left (B\right ) _{i_uj_p} = - \int_{\mathcal{F}} \psi_{j_p}  \Div (\phi_{i_u}),
\left (C\right ) _{i_uj_\lambda}= \int_{\Gamma} \zeta_{j_{\lambda}} \phi_{i_u} $$
$$\left (F\right ) _{i_u}= \int_{\mathcal{F}} f \phi_{i_u}, 
\quad \left (G\right ) _{i_\lambda} =  \int_{\Gamma} g \zeta_{i_\lambda}$$

Rates of convergence are presented in Fig. \ref{fig:sans_stab} for the triples of spaces $\P_2-\P_1-\P_1$,
$\P_2-\P_1-\P_0$, $\P_1-\P_1-\P_1$ (velocity-pressure-multiplier).  The choice $\P_1-\P_1$ for velocity-pressure suffers of course from the non-satisfaction of the mesh-independent inf-sup condition. It has to be stressed that in all the experiments without stabilization, and particularly for the $\P_1-\P_1$ case, a singular linear system could be obtained. However, we did not encounter this in our simulations (singular systems did occur in the experiments with $\P_0$  multiplier, not reported here). 

As expected, the solution with $\P_1-\P_1-\P_1$ FE  is not good. 
On the contrary, optimal convergence is observed for all the unknowns when $\P_2-\P_1$ FE spaces are used for velocity-pressure. 
However, some problems could remain when the  intersections of mesh elements with $\mathcal{F}$ are too small.
We refer to \cite{JFS-Fournie-Court} where this aspect is addressed in more detail.
 
\vspace*{-0.5cm}

\begin{center}
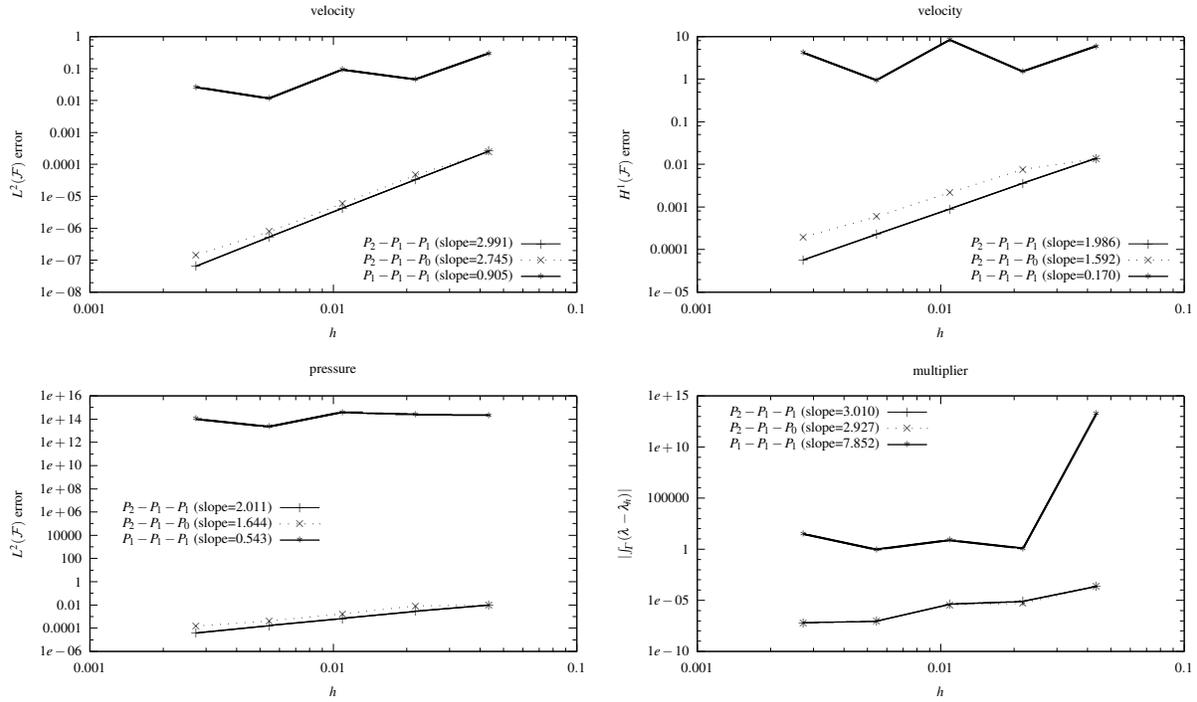
\begin{figure}[h] 
\hspace*{-2cm}
\begin{minipage}{20cm}
 \resizebox{8cm}{!}{
\setlength{\unitlength}{0.240900pt}
\ifx\plotpoint\undefined\newsavebox{\plotpoint}\fi
\sbox{\plotpoint}{\rule[-0.200pt]{0.400pt}{0.400pt}}%
\begin{picture}(1500,900)(0,0)
\sbox{\plotpoint}{\rule[-0.200pt]{0.400pt}{0.400pt}}%
\put(211.0,131.0){\rule[-0.200pt]{4.818pt}{0.400pt}}
\put(191,131){\makebox(0,0)[r]{$1e-08$}}
\put(1419.0,131.0){\rule[-0.200pt]{4.818pt}{0.400pt}}
\put(211.0,155.0){\rule[-0.200pt]{2.409pt}{0.400pt}}
\put(1429.0,155.0){\rule[-0.200pt]{2.409pt}{0.400pt}}
\put(211.0,187.0){\rule[-0.200pt]{2.409pt}{0.400pt}}
\put(1429.0,187.0){\rule[-0.200pt]{2.409pt}{0.400pt}}
\put(211.0,204.0){\rule[-0.200pt]{2.409pt}{0.400pt}}
\put(1429.0,204.0){\rule[-0.200pt]{2.409pt}{0.400pt}}
\put(211.0,212.0){\rule[-0.200pt]{4.818pt}{0.400pt}}
\put(191,212){\makebox(0,0)[r]{$1e-07$}}
\put(1419.0,212.0){\rule[-0.200pt]{4.818pt}{0.400pt}}
\put(211.0,236.0){\rule[-0.200pt]{2.409pt}{0.400pt}}
\put(1429.0,236.0){\rule[-0.200pt]{2.409pt}{0.400pt}}
\put(211.0,268.0){\rule[-0.200pt]{2.409pt}{0.400pt}}
\put(1429.0,268.0){\rule[-0.200pt]{2.409pt}{0.400pt}}
\put(211.0,284.0){\rule[-0.200pt]{2.409pt}{0.400pt}}
\put(1429.0,284.0){\rule[-0.200pt]{2.409pt}{0.400pt}}
\put(211.0,292.0){\rule[-0.200pt]{4.818pt}{0.400pt}}
\put(191,292){\makebox(0,0)[r]{$1e-06$}}
\put(1419.0,292.0){\rule[-0.200pt]{4.818pt}{0.400pt}}
\put(211.0,317.0){\rule[-0.200pt]{2.409pt}{0.400pt}}
\put(1429.0,317.0){\rule[-0.200pt]{2.409pt}{0.400pt}}
\put(211.0,349.0){\rule[-0.200pt]{2.409pt}{0.400pt}}
\put(1429.0,349.0){\rule[-0.200pt]{2.409pt}{0.400pt}}
\put(211.0,365.0){\rule[-0.200pt]{2.409pt}{0.400pt}}
\put(1429.0,365.0){\rule[-0.200pt]{2.409pt}{0.400pt}}
\put(211.0,373.0){\rule[-0.200pt]{4.818pt}{0.400pt}}
\put(191,373){\makebox(0,0)[r]{$1e-05$}}
\put(1419.0,373.0){\rule[-0.200pt]{4.818pt}{0.400pt}}
\put(211.0,397.0){\rule[-0.200pt]{2.409pt}{0.400pt}}
\put(1429.0,397.0){\rule[-0.200pt]{2.409pt}{0.400pt}}
\put(211.0,429.0){\rule[-0.200pt]{2.409pt}{0.400pt}}
\put(1429.0,429.0){\rule[-0.200pt]{2.409pt}{0.400pt}}
\put(211.0,446.0){\rule[-0.200pt]{2.409pt}{0.400pt}}
\put(1429.0,446.0){\rule[-0.200pt]{2.409pt}{0.400pt}}
\put(211.0,454.0){\rule[-0.200pt]{4.818pt}{0.400pt}}
\put(191,454){\makebox(0,0)[r]{$0.0001$}}
\put(1419.0,454.0){\rule[-0.200pt]{4.818pt}{0.400pt}}
\put(211.0,478.0){\rule[-0.200pt]{2.409pt}{0.400pt}}
\put(1429.0,478.0){\rule[-0.200pt]{2.409pt}{0.400pt}}
\put(211.0,510.0){\rule[-0.200pt]{2.409pt}{0.400pt}}
\put(1429.0,510.0){\rule[-0.200pt]{2.409pt}{0.400pt}}
\put(211.0,526.0){\rule[-0.200pt]{2.409pt}{0.400pt}}
\put(1429.0,526.0){\rule[-0.200pt]{2.409pt}{0.400pt}}
\put(211.0,534.0){\rule[-0.200pt]{4.818pt}{0.400pt}}
\put(191,534){\makebox(0,0)[r]{$0.001$}}
\put(1419.0,534.0){\rule[-0.200pt]{4.818pt}{0.400pt}}
\put(211.0,558.0){\rule[-0.200pt]{2.409pt}{0.400pt}}
\put(1429.0,558.0){\rule[-0.200pt]{2.409pt}{0.400pt}}
\put(211.0,590.0){\rule[-0.200pt]{2.409pt}{0.400pt}}
\put(1429.0,590.0){\rule[-0.200pt]{2.409pt}{0.400pt}}
\put(211.0,607.0){\rule[-0.200pt]{2.409pt}{0.400pt}}
\put(1429.0,607.0){\rule[-0.200pt]{2.409pt}{0.400pt}}
\put(211.0,615.0){\rule[-0.200pt]{4.818pt}{0.400pt}}
\put(191,615){\makebox(0,0)[r]{$0.01$}}
\put(1419.0,615.0){\rule[-0.200pt]{4.818pt}{0.400pt}}
\put(211.0,639.0){\rule[-0.200pt]{2.409pt}{0.400pt}}
\put(1429.0,639.0){\rule[-0.200pt]{2.409pt}{0.400pt}}
\put(211.0,671.0){\rule[-0.200pt]{2.409pt}{0.400pt}}
\put(1429.0,671.0){\rule[-0.200pt]{2.409pt}{0.400pt}}
\put(211.0,688.0){\rule[-0.200pt]{2.409pt}{0.400pt}}
\put(1429.0,688.0){\rule[-0.200pt]{2.409pt}{0.400pt}}
\put(211.0,695.0){\rule[-0.200pt]{4.818pt}{0.400pt}}
\put(191,695){\makebox(0,0)[r]{$0.1$}}
\put(1419.0,695.0){\rule[-0.200pt]{4.818pt}{0.400pt}}
\put(211.0,720.0){\rule[-0.200pt]{2.409pt}{0.400pt}}
\put(1429.0,720.0){\rule[-0.200pt]{2.409pt}{0.400pt}}
\put(211.0,752.0){\rule[-0.200pt]{2.409pt}{0.400pt}}
\put(1429.0,752.0){\rule[-0.200pt]{2.409pt}{0.400pt}}
\put(211.0,768.0){\rule[-0.200pt]{2.409pt}{0.400pt}}
\put(1429.0,768.0){\rule[-0.200pt]{2.409pt}{0.400pt}}
\put(211.0,776.0){\rule[-0.200pt]{4.818pt}{0.400pt}}
\put(191,776){\makebox(0,0)[r]{$1$}}
\put(1419.0,776.0){\rule[-0.200pt]{4.818pt}{0.400pt}}
\put(211.0,131.0){\rule[-0.200pt]{0.400pt}{4.818pt}}
\put(211,90){\makebox(0,0){$0.001$}}
\put(211.0,756.0){\rule[-0.200pt]{0.400pt}{4.818pt}}
\put(396.0,131.0){\rule[-0.200pt]{0.400pt}{2.409pt}}
\put(396.0,766.0){\rule[-0.200pt]{0.400pt}{2.409pt}}
\put(504.0,131.0){\rule[-0.200pt]{0.400pt}{2.409pt}}
\put(504.0,766.0){\rule[-0.200pt]{0.400pt}{2.409pt}}
\put(581.0,131.0){\rule[-0.200pt]{0.400pt}{2.409pt}}
\put(581.0,766.0){\rule[-0.200pt]{0.400pt}{2.409pt}}
\put(640.0,131.0){\rule[-0.200pt]{0.400pt}{2.409pt}}
\put(640.0,766.0){\rule[-0.200pt]{0.400pt}{2.409pt}}
\put(689.0,131.0){\rule[-0.200pt]{0.400pt}{2.409pt}}
\put(689.0,766.0){\rule[-0.200pt]{0.400pt}{2.409pt}}
\put(730.0,131.0){\rule[-0.200pt]{0.400pt}{2.409pt}}
\put(730.0,766.0){\rule[-0.200pt]{0.400pt}{2.409pt}}
\put(765.0,131.0){\rule[-0.200pt]{0.400pt}{2.409pt}}
\put(765.0,766.0){\rule[-0.200pt]{0.400pt}{2.409pt}}
\put(797.0,131.0){\rule[-0.200pt]{0.400pt}{2.409pt}}
\put(797.0,766.0){\rule[-0.200pt]{0.400pt}{2.409pt}}
\put(825.0,131.0){\rule[-0.200pt]{0.400pt}{4.818pt}}
\put(825,90){\makebox(0,0){$0.01$}}
\put(825.0,756.0){\rule[-0.200pt]{0.400pt}{4.818pt}}
\put(1010.0,131.0){\rule[-0.200pt]{0.400pt}{2.409pt}}
\put(1010.0,766.0){\rule[-0.200pt]{0.400pt}{2.409pt}}
\put(1118.0,131.0){\rule[-0.200pt]{0.400pt}{2.409pt}}
\put(1118.0,766.0){\rule[-0.200pt]{0.400pt}{2.409pt}}
\put(1195.0,131.0){\rule[-0.200pt]{0.400pt}{2.409pt}}
\put(1195.0,766.0){\rule[-0.200pt]{0.400pt}{2.409pt}}
\put(1254.0,131.0){\rule[-0.200pt]{0.400pt}{2.409pt}}
\put(1254.0,766.0){\rule[-0.200pt]{0.400pt}{2.409pt}}
\put(1303.0,131.0){\rule[-0.200pt]{0.400pt}{2.409pt}}
\put(1303.0,766.0){\rule[-0.200pt]{0.400pt}{2.409pt}}
\put(1344.0,131.0){\rule[-0.200pt]{0.400pt}{2.409pt}}
\put(1344.0,766.0){\rule[-0.200pt]{0.400pt}{2.409pt}}
\put(1379.0,131.0){\rule[-0.200pt]{0.400pt}{2.409pt}}
\put(1379.0,766.0){\rule[-0.200pt]{0.400pt}{2.409pt}}
\put(1411.0,131.0){\rule[-0.200pt]{0.400pt}{2.409pt}}
\put(1411.0,766.0){\rule[-0.200pt]{0.400pt}{2.409pt}}
\put(1439.0,131.0){\rule[-0.200pt]{0.400pt}{4.818pt}}
\put(1439,90){\makebox(0,0){$0.1$}}
\put(1439.0,756.0){\rule[-0.200pt]{0.400pt}{4.818pt}}
\put(211.0,131.0){\rule[-0.200pt]{0.400pt}{155.380pt}}
\put(211.0,131.0){\rule[-0.200pt]{295.825pt}{0.400pt}}
\put(1439.0,131.0){\rule[-0.200pt]{0.400pt}{155.380pt}}
\put(211.0,776.0){\rule[-0.200pt]{295.825pt}{0.400pt}}
\put(30,453){\makebox(0,0){\rotatebox{90}{$L^2(\mathcal{F})$ error}}}
\put(825,29){\makebox(0,0){$h$}}
\put(825,838){\makebox(0,0){velocity}}
\put(1279,253){\makebox(0,0)[r]{$P_2-P_1-P_1$ (slope=2.991)}}
\put(1299.0,253.0){\rule[-0.200pt]{24.090pt}{0.400pt}}
\put(1218,487){\usebox{\plotpoint}}
\multiput(1213.32,485.92)(-1.288,-0.499){141}{\rule{1.128pt}{0.120pt}}
\multiput(1215.66,486.17)(-182.659,-72.000){2}{\rule{0.564pt}{0.400pt}}
\multiput(1028.34,413.92)(-1.281,-0.499){141}{\rule{1.122pt}{0.120pt}}
\multiput(1030.67,414.17)(-181.671,-72.000){2}{\rule{0.561pt}{0.400pt}}
\multiput(844.38,341.92)(-1.270,-0.499){143}{\rule{1.114pt}{0.120pt}}
\multiput(846.69,342.17)(-182.688,-73.000){2}{\rule{0.557pt}{0.400pt}}
\multiput(659.43,268.92)(-1.253,-0.499){145}{\rule{1.100pt}{0.120pt}}
\multiput(661.72,269.17)(-182.717,-74.000){2}{\rule{0.550pt}{0.400pt}}
\put(1218,487){\makebox(0,0){$+$}}
\put(1033,415){\makebox(0,0){$+$}}
\put(849,343){\makebox(0,0){$+$}}
\put(664,270){\makebox(0,0){$+$}}
\put(479,196){\makebox(0,0){$+$}}
\put(1349,253){\makebox(0,0){$+$}}
\put(1279,212){\makebox(0,0)[r]{$P_2-P_1-P_0$ (slope=2.745)}}
\multiput(1299,212)(20.756,0.000){5}{\usebox{\plotpoint}}
\put(1399,212){\usebox{\plotpoint}}
\put(1218,486){\usebox{\plotpoint}}
\multiput(1218,486)(-19.774,-6.306){10}{\usebox{\plotpoint}}
\multiput(1033,427)(-19.328,-7.563){9}{\usebox{\plotpoint}}
\multiput(849,355)(-19.377,-7.437){10}{\usebox{\plotpoint}}
\multiput(664,284)(-19.774,-6.306){9}{\usebox{\plotpoint}}
\put(479,225){\usebox{\plotpoint}}
\put(1218,486){\makebox(0,0){$\times$}}
\put(1033,427){\makebox(0,0){$\times$}}
\put(849,355){\makebox(0,0){$\times$}}
\put(664,284){\makebox(0,0){$\times$}}
\put(479,225){\makebox(0,0){$\times$}}
\put(1349,212){\makebox(0,0){$\times$}}
\sbox{\plotpoint}{\rule[-0.400pt]{0.800pt}{0.800pt}}%
\sbox{\plotpoint}{\rule[-0.200pt]{0.400pt}{0.400pt}}%
\put(1279,171){\makebox(0,0)[r]{$P_1-P_1-P_1$ (slope=0.905)}}
\sbox{\plotpoint}{\rule[-0.400pt]{0.800pt}{0.800pt}}%
\put(1299.0,171.0){\rule[-0.400pt]{24.090pt}{0.800pt}}
\put(1218,734){\usebox{\plotpoint}}
\multiput(1207.86,732.09)(-1.410,-0.501){125}{\rule{2.442pt}{0.121pt}}
\multiput(1212.93,732.34)(-179.931,-66.000){2}{\rule{1.221pt}{0.800pt}}
\multiput(1007.73,669.41)(-3.779,0.504){43}{\rule{6.088pt}{0.121pt}}
\multiput(1020.36,666.34)(-171.364,25.000){2}{\rule{3.044pt}{0.800pt}}
\multiput(839.75,691.09)(-1.274,-0.501){139}{\rule{2.227pt}{0.121pt}}
\multiput(844.38,691.34)(-180.377,-73.000){2}{\rule{1.114pt}{0.800pt}}
\multiput(641.98,621.41)(-3.260,0.504){51}{\rule{5.303pt}{0.121pt}}
\multiput(652.99,618.34)(-173.992,29.000){2}{\rule{2.652pt}{0.800pt}}
\put(1218,734){\makebox(0,0){$\ast$}}
\put(1033,668){\makebox(0,0){$\ast$}}
\put(849,693){\makebox(0,0){$\ast$}}
\put(664,620){\makebox(0,0){$\ast$}}
\put(479,649){\makebox(0,0){$\ast$}}
\put(1349,171){\makebox(0,0){$\ast$}}
\sbox{\plotpoint}{\rule[-0.200pt]{0.400pt}{0.400pt}}%
\put(211.0,131.0){\rule[-0.200pt]{0.400pt}{155.380pt}}
\put(211.0,131.0){\rule[-0.200pt]{295.825pt}{0.400pt}}
\put(1439.0,131.0){\rule[-0.200pt]{0.400pt}{155.380pt}}
\put(211.0,776.0){\rule[-0.200pt]{295.825pt}{0.400pt}}
\end{picture}}
 \resizebox{8cm}{!}{
\setlength{\unitlength}{0.240900pt}
\ifx\plotpoint\undefined\newsavebox{\plotpoint}\fi
\begin{picture}(1500,900)(0,0)
\sbox{\plotpoint}{\rule[-0.200pt]{0.400pt}{0.400pt}}%
\put(211.0,131.0){\rule[-0.200pt]{4.818pt}{0.400pt}}
\put(191,131){\makebox(0,0)[r]{$1e-05$}}
\put(1419.0,131.0){\rule[-0.200pt]{4.818pt}{0.400pt}}
\put(211.0,163.0){\rule[-0.200pt]{2.409pt}{0.400pt}}
\put(1429.0,163.0){\rule[-0.200pt]{2.409pt}{0.400pt}}
\put(211.0,206.0){\rule[-0.200pt]{2.409pt}{0.400pt}}
\put(1429.0,206.0){\rule[-0.200pt]{2.409pt}{0.400pt}}
\put(211.0,228.0){\rule[-0.200pt]{2.409pt}{0.400pt}}
\put(1429.0,228.0){\rule[-0.200pt]{2.409pt}{0.400pt}}
\put(211.0,239.0){\rule[-0.200pt]{4.818pt}{0.400pt}}
\put(191,239){\makebox(0,0)[r]{$0.0001$}}
\put(1419.0,239.0){\rule[-0.200pt]{4.818pt}{0.400pt}}
\put(211.0,271.0){\rule[-0.200pt]{2.409pt}{0.400pt}}
\put(1429.0,271.0){\rule[-0.200pt]{2.409pt}{0.400pt}}
\put(211.0,314.0){\rule[-0.200pt]{2.409pt}{0.400pt}}
\put(1429.0,314.0){\rule[-0.200pt]{2.409pt}{0.400pt}}
\put(211.0,336.0){\rule[-0.200pt]{2.409pt}{0.400pt}}
\put(1429.0,336.0){\rule[-0.200pt]{2.409pt}{0.400pt}}
\put(211.0,346.0){\rule[-0.200pt]{4.818pt}{0.400pt}}
\put(191,346){\makebox(0,0)[r]{$0.001$}}
\put(1419.0,346.0){\rule[-0.200pt]{4.818pt}{0.400pt}}
\put(211.0,378.0){\rule[-0.200pt]{2.409pt}{0.400pt}}
\put(1429.0,378.0){\rule[-0.200pt]{2.409pt}{0.400pt}}
\put(211.0,421.0){\rule[-0.200pt]{2.409pt}{0.400pt}}
\put(1429.0,421.0){\rule[-0.200pt]{2.409pt}{0.400pt}}
\put(211.0,443.0){\rule[-0.200pt]{2.409pt}{0.400pt}}
\put(1429.0,443.0){\rule[-0.200pt]{2.409pt}{0.400pt}}
\put(211.0,454.0){\rule[-0.200pt]{4.818pt}{0.400pt}}
\put(191,454){\makebox(0,0)[r]{$0.01$}}
\put(1419.0,454.0){\rule[-0.200pt]{4.818pt}{0.400pt}}
\put(211.0,486.0){\rule[-0.200pt]{2.409pt}{0.400pt}}
\put(1429.0,486.0){\rule[-0.200pt]{2.409pt}{0.400pt}}
\put(211.0,529.0){\rule[-0.200pt]{2.409pt}{0.400pt}}
\put(1429.0,529.0){\rule[-0.200pt]{2.409pt}{0.400pt}}
\put(211.0,551.0){\rule[-0.200pt]{2.409pt}{0.400pt}}
\put(1429.0,551.0){\rule[-0.200pt]{2.409pt}{0.400pt}}
\put(211.0,561.0){\rule[-0.200pt]{4.818pt}{0.400pt}}
\put(191,561){\makebox(0,0)[r]{$0.1$}}
\put(1419.0,561.0){\rule[-0.200pt]{4.818pt}{0.400pt}}
\put(211.0,593.0){\rule[-0.200pt]{2.409pt}{0.400pt}}
\put(1429.0,593.0){\rule[-0.200pt]{2.409pt}{0.400pt}}
\put(211.0,636.0){\rule[-0.200pt]{2.409pt}{0.400pt}}
\put(1429.0,636.0){\rule[-0.200pt]{2.409pt}{0.400pt}}
\put(211.0,658.0){\rule[-0.200pt]{2.409pt}{0.400pt}}
\put(1429.0,658.0){\rule[-0.200pt]{2.409pt}{0.400pt}}
\put(211.0,669.0){\rule[-0.200pt]{4.818pt}{0.400pt}}
\put(191,669){\makebox(0,0)[r]{$1$}}
\put(1419.0,669.0){\rule[-0.200pt]{4.818pt}{0.400pt}}
\put(211.0,701.0){\rule[-0.200pt]{2.409pt}{0.400pt}}
\put(1429.0,701.0){\rule[-0.200pt]{2.409pt}{0.400pt}}
\put(211.0,744.0){\rule[-0.200pt]{2.409pt}{0.400pt}}
\put(1429.0,744.0){\rule[-0.200pt]{2.409pt}{0.400pt}}
\put(211.0,766.0){\rule[-0.200pt]{2.409pt}{0.400pt}}
\put(1429.0,766.0){\rule[-0.200pt]{2.409pt}{0.400pt}}
\put(211.0,776.0){\rule[-0.200pt]{4.818pt}{0.400pt}}
\put(191,776){\makebox(0,0)[r]{$10$}}
\put(1419.0,776.0){\rule[-0.200pt]{4.818pt}{0.400pt}}
\put(211.0,131.0){\rule[-0.200pt]{0.400pt}{4.818pt}}
\put(211,90){\makebox(0,0){$0.001$}}
\put(211.0,756.0){\rule[-0.200pt]{0.400pt}{4.818pt}}
\put(396.0,131.0){\rule[-0.200pt]{0.400pt}{2.409pt}}
\put(396.0,766.0){\rule[-0.200pt]{0.400pt}{2.409pt}}
\put(504.0,131.0){\rule[-0.200pt]{0.400pt}{2.409pt}}
\put(504.0,766.0){\rule[-0.200pt]{0.400pt}{2.409pt}}
\put(581.0,131.0){\rule[-0.200pt]{0.400pt}{2.409pt}}
\put(581.0,766.0){\rule[-0.200pt]{0.400pt}{2.409pt}}
\put(640.0,131.0){\rule[-0.200pt]{0.400pt}{2.409pt}}
\put(640.0,766.0){\rule[-0.200pt]{0.400pt}{2.409pt}}
\put(689.0,131.0){\rule[-0.200pt]{0.400pt}{2.409pt}}
\put(689.0,766.0){\rule[-0.200pt]{0.400pt}{2.409pt}}
\put(730.0,131.0){\rule[-0.200pt]{0.400pt}{2.409pt}}
\put(730.0,766.0){\rule[-0.200pt]{0.400pt}{2.409pt}}
\put(765.0,131.0){\rule[-0.200pt]{0.400pt}{2.409pt}}
\put(765.0,766.0){\rule[-0.200pt]{0.400pt}{2.409pt}}
\put(797.0,131.0){\rule[-0.200pt]{0.400pt}{2.409pt}}
\put(797.0,766.0){\rule[-0.200pt]{0.400pt}{2.409pt}}
\put(825.0,131.0){\rule[-0.200pt]{0.400pt}{4.818pt}}
\put(825,90){\makebox(0,0){$0.01$}}
\put(825.0,756.0){\rule[-0.200pt]{0.400pt}{4.818pt}}
\put(1010.0,131.0){\rule[-0.200pt]{0.400pt}{2.409pt}}
\put(1010.0,766.0){\rule[-0.200pt]{0.400pt}{2.409pt}}
\put(1118.0,131.0){\rule[-0.200pt]{0.400pt}{2.409pt}}
\put(1118.0,766.0){\rule[-0.200pt]{0.400pt}{2.409pt}}
\put(1195.0,131.0){\rule[-0.200pt]{0.400pt}{2.409pt}}
\put(1195.0,766.0){\rule[-0.200pt]{0.400pt}{2.409pt}}
\put(1254.0,131.0){\rule[-0.200pt]{0.400pt}{2.409pt}}
\put(1254.0,766.0){\rule[-0.200pt]{0.400pt}{2.409pt}}
\put(1303.0,131.0){\rule[-0.200pt]{0.400pt}{2.409pt}}
\put(1303.0,766.0){\rule[-0.200pt]{0.400pt}{2.409pt}}
\put(1344.0,131.0){\rule[-0.200pt]{0.400pt}{2.409pt}}
\put(1344.0,766.0){\rule[-0.200pt]{0.400pt}{2.409pt}}
\put(1379.0,131.0){\rule[-0.200pt]{0.400pt}{2.409pt}}
\put(1379.0,766.0){\rule[-0.200pt]{0.400pt}{2.409pt}}
\put(1411.0,131.0){\rule[-0.200pt]{0.400pt}{2.409pt}}
\put(1411.0,766.0){\rule[-0.200pt]{0.400pt}{2.409pt}}
\put(1439.0,131.0){\rule[-0.200pt]{0.400pt}{4.818pt}}
\put(1439,90){\makebox(0,0){$0.1$}}
\put(1439.0,756.0){\rule[-0.200pt]{0.400pt}{4.818pt}}
\put(211.0,131.0){\rule[-0.200pt]{0.400pt}{155.380pt}}
\put(211.0,131.0){\rule[-0.200pt]{295.825pt}{0.400pt}}
\put(1439.0,131.0){\rule[-0.200pt]{0.400pt}{155.380pt}}
\put(211.0,776.0){\rule[-0.200pt]{295.825pt}{0.400pt}}
\put(30,453){\makebox(0,0){\rotatebox{90}{$H^1(\mathcal{F})$ error}}}
\put(825,29){\makebox(0,0){$h$}}
\put(825,838){\makebox(0,0){velocity}}
\put(1279,253){\makebox(0,0)[r]{$P_2-P_1-P_1$ (slope=1.986)}}
\put(1299.0,253.0){\rule[-0.200pt]{24.090pt}{0.400pt}}
\put(1218,469){\usebox{\plotpoint}}
\multiput(1212.71,467.92)(-1.473,-0.499){123}{\rule{1.275pt}{0.120pt}}
\multiput(1215.35,468.17)(-182.354,-63.000){2}{\rule{0.637pt}{0.400pt}}
\multiput(1027.88,404.92)(-1.420,-0.499){127}{\rule{1.232pt}{0.120pt}}
\multiput(1030.44,405.17)(-181.442,-65.000){2}{\rule{0.616pt}{0.400pt}}
\multiput(843.79,339.92)(-1.450,-0.499){125}{\rule{1.256pt}{0.120pt}}
\multiput(846.39,340.17)(-182.393,-64.000){2}{\rule{0.628pt}{0.400pt}}
\multiput(658.86,275.92)(-1.427,-0.499){127}{\rule{1.238pt}{0.120pt}}
\multiput(661.43,276.17)(-182.430,-65.000){2}{\rule{0.619pt}{0.400pt}}
\put(1218,469){\makebox(0,0){$+$}}
\put(1033,406){\makebox(0,0){$+$}}
\put(849,341){\makebox(0,0){$+$}}
\put(664,277){\makebox(0,0){$+$}}
\put(479,212){\makebox(0,0){$+$}}
\put(1349,253){\makebox(0,0){$+$}}
\put(1279,212){\makebox(0,0)[r]{$P_2-P_1-P_0$ (slope=1.592)}}
\multiput(1299,212)(20.756,0.000){5}{\usebox{\plotpoint}}
\put(1399,212){\usebox{\plotpoint}}
\put(1218,468){\usebox{\plotpoint}}
\multiput(1218,468)(-20.538,-2.997){10}{\usebox{\plotpoint}}
\multiput(1033,441)(-19.795,-6.240){9}{\usebox{\plotpoint}}
\multiput(849,383)(-19.712,-6.500){9}{\usebox{\plotpoint}}
\multiput(664,322)(-19.981,-5.616){9}{\usebox{\plotpoint}}
\put(479,270){\usebox{\plotpoint}}
\put(1218,468){\makebox(0,0){$\times$}}
\put(1033,441){\makebox(0,0){$\times$}}
\put(849,383){\makebox(0,0){$\times$}}
\put(664,322){\makebox(0,0){$\times$}}
\put(479,270){\makebox(0,0){$\times$}}
\put(1349,212){\makebox(0,0){$\times$}}
\sbox{\plotpoint}{\rule[-0.400pt]{0.800pt}{0.800pt}}%
\sbox{\plotpoint}{\rule[-0.200pt]{0.400pt}{0.400pt}}%
\put(1279,171){\makebox(0,0)[r]{$P_1-P_1-P_1$ (slope=0.170)}}
\sbox{\plotpoint}{\rule[-0.400pt]{0.800pt}{0.800pt}}%
\put(1299.0,171.0){\rule[-0.400pt]{24.090pt}{0.800pt}}
\put(1218,752){\usebox{\plotpoint}}
\multiput(1207.57,750.09)(-1.455,-0.501){121}{\rule{2.513pt}{0.121pt}}
\multiput(1212.79,750.34)(-179.785,-64.000){2}{\rule{1.256pt}{0.800pt}}
\multiput(1024.53,689.41)(-1.155,0.501){153}{\rule{2.040pt}{0.121pt}}
\multiput(1028.77,686.34)(-179.766,80.000){2}{\rule{1.020pt}{0.800pt}}
\multiput(842.15,766.09)(-0.909,-0.501){197}{\rule{1.651pt}{0.121pt}}
\multiput(845.57,766.34)(-181.573,-102.000){2}{\rule{0.825pt}{0.800pt}}
\multiput(654.39,667.41)(-1.329,0.501){133}{\rule{2.314pt}{0.121pt}}
\multiput(659.20,664.34)(-180.197,70.000){2}{\rule{1.157pt}{0.800pt}}
\put(1218,752){\makebox(0,0){$\ast$}}
\put(1033,688){\makebox(0,0){$\ast$}}
\put(849,768){\makebox(0,0){$\ast$}}
\put(664,666){\makebox(0,0){$\ast$}}
\put(479,736){\makebox(0,0){$\ast$}}
\put(1349,171){\makebox(0,0){$\ast$}}
\sbox{\plotpoint}{\rule[-0.200pt]{0.400pt}{0.400pt}}%
\put(211.0,131.0){\rule[-0.200pt]{0.400pt}{155.380pt}}
\put(211.0,131.0){\rule[-0.200pt]{295.825pt}{0.400pt}}
\put(1439.0,131.0){\rule[-0.200pt]{0.400pt}{155.380pt}}
\put(211.0,776.0){\rule[-0.200pt]{295.825pt}{0.400pt}}
\end{picture}}\\
\resizebox{8cm}{!}{
\setlength{\unitlength}{0.240900pt}
\ifx\plotpoint\undefined\newsavebox{\plotpoint}\fi
\begin{picture}(1500,900)(0,0)
\sbox{\plotpoint}{\rule[-0.200pt]{0.400pt}{0.400pt}}%
\put(211.0,131.0){\rule[-0.200pt]{4.818pt}{0.400pt}}
\put(191,131){\makebox(0,0)[r]{$1e-06$}}
\put(1419.0,131.0){\rule[-0.200pt]{4.818pt}{0.400pt}}
\put(211.0,160.0){\rule[-0.200pt]{2.409pt}{0.400pt}}
\put(1429.0,160.0){\rule[-0.200pt]{2.409pt}{0.400pt}}
\put(211.0,190.0){\rule[-0.200pt]{4.818pt}{0.400pt}}
\put(191,190){\makebox(0,0)[r]{$0.0001$}}
\put(1419.0,190.0){\rule[-0.200pt]{4.818pt}{0.400pt}}
\put(211.0,219.0){\rule[-0.200pt]{2.409pt}{0.400pt}}
\put(1429.0,219.0){\rule[-0.200pt]{2.409pt}{0.400pt}}
\put(211.0,248.0){\rule[-0.200pt]{4.818pt}{0.400pt}}
\put(191,248){\makebox(0,0)[r]{$0.01$}}
\put(1419.0,248.0){\rule[-0.200pt]{4.818pt}{0.400pt}}
\put(211.0,278.0){\rule[-0.200pt]{2.409pt}{0.400pt}}
\put(1429.0,278.0){\rule[-0.200pt]{2.409pt}{0.400pt}}
\put(211.0,307.0){\rule[-0.200pt]{4.818pt}{0.400pt}}
\put(191,307){\makebox(0,0)[r]{$1$}}
\put(1419.0,307.0){\rule[-0.200pt]{4.818pt}{0.400pt}}
\put(211.0,336.0){\rule[-0.200pt]{2.409pt}{0.400pt}}
\put(1429.0,336.0){\rule[-0.200pt]{2.409pt}{0.400pt}}
\put(211.0,366.0){\rule[-0.200pt]{4.818pt}{0.400pt}}
\put(191,366){\makebox(0,0)[r]{$100$}}
\put(1419.0,366.0){\rule[-0.200pt]{4.818pt}{0.400pt}}
\put(211.0,395.0){\rule[-0.200pt]{2.409pt}{0.400pt}}
\put(1429.0,395.0){\rule[-0.200pt]{2.409pt}{0.400pt}}
\put(211.0,424.0){\rule[-0.200pt]{4.818pt}{0.400pt}}
\put(191,424){\makebox(0,0)[r]{$10000$}}
\put(1419.0,424.0){\rule[-0.200pt]{4.818pt}{0.400pt}}
\put(211.0,454.0){\rule[-0.200pt]{2.409pt}{0.400pt}}
\put(1429.0,454.0){\rule[-0.200pt]{2.409pt}{0.400pt}}
\put(211.0,483.0){\rule[-0.200pt]{4.818pt}{0.400pt}}
\put(191,483){\makebox(0,0)[r]{$1e+06$}}
\put(1419.0,483.0){\rule[-0.200pt]{4.818pt}{0.400pt}}
\put(211.0,512.0){\rule[-0.200pt]{2.409pt}{0.400pt}}
\put(1429.0,512.0){\rule[-0.200pt]{2.409pt}{0.400pt}}
\put(211.0,541.0){\rule[-0.200pt]{4.818pt}{0.400pt}}
\put(191,541){\makebox(0,0)[r]{$1e+08$}}
\put(1419.0,541.0){\rule[-0.200pt]{4.818pt}{0.400pt}}
\put(211.0,571.0){\rule[-0.200pt]{2.409pt}{0.400pt}}
\put(1429.0,571.0){\rule[-0.200pt]{2.409pt}{0.400pt}}
\put(211.0,600.0){\rule[-0.200pt]{4.818pt}{0.400pt}}
\put(191,600){\makebox(0,0)[r]{$1e+10$}}
\put(1419.0,600.0){\rule[-0.200pt]{4.818pt}{0.400pt}}
\put(211.0,629.0){\rule[-0.200pt]{2.409pt}{0.400pt}}
\put(1429.0,629.0){\rule[-0.200pt]{2.409pt}{0.400pt}}
\put(211.0,659.0){\rule[-0.200pt]{4.818pt}{0.400pt}}
\put(191,659){\makebox(0,0)[r]{$1e+12$}}
\put(1419.0,659.0){\rule[-0.200pt]{4.818pt}{0.400pt}}
\put(211.0,688.0){\rule[-0.200pt]{2.409pt}{0.400pt}}
\put(1429.0,688.0){\rule[-0.200pt]{2.409pt}{0.400pt}}
\put(211.0,717.0){\rule[-0.200pt]{4.818pt}{0.400pt}}
\put(191,717){\makebox(0,0)[r]{$1e+14$}}
\put(1419.0,717.0){\rule[-0.200pt]{4.818pt}{0.400pt}}
\put(211.0,747.0){\rule[-0.200pt]{2.409pt}{0.400pt}}
\put(1429.0,747.0){\rule[-0.200pt]{2.409pt}{0.400pt}}
\put(211.0,776.0){\rule[-0.200pt]{4.818pt}{0.400pt}}
\put(191,776){\makebox(0,0)[r]{$1e+16$}}
\put(1419.0,776.0){\rule[-0.200pt]{4.818pt}{0.400pt}}
\put(211.0,131.0){\rule[-0.200pt]{0.400pt}{4.818pt}}
\put(211,90){\makebox(0,0){$0.001$}}
\put(211.0,756.0){\rule[-0.200pt]{0.400pt}{4.818pt}}
\put(396.0,131.0){\rule[-0.200pt]{0.400pt}{2.409pt}}
\put(396.0,766.0){\rule[-0.200pt]{0.400pt}{2.409pt}}
\put(504.0,131.0){\rule[-0.200pt]{0.400pt}{2.409pt}}
\put(504.0,766.0){\rule[-0.200pt]{0.400pt}{2.409pt}}
\put(581.0,131.0){\rule[-0.200pt]{0.400pt}{2.409pt}}
\put(581.0,766.0){\rule[-0.200pt]{0.400pt}{2.409pt}}
\put(640.0,131.0){\rule[-0.200pt]{0.400pt}{2.409pt}}
\put(640.0,766.0){\rule[-0.200pt]{0.400pt}{2.409pt}}
\put(689.0,131.0){\rule[-0.200pt]{0.400pt}{2.409pt}}
\put(689.0,766.0){\rule[-0.200pt]{0.400pt}{2.409pt}}
\put(730.0,131.0){\rule[-0.200pt]{0.400pt}{2.409pt}}
\put(730.0,766.0){\rule[-0.200pt]{0.400pt}{2.409pt}}
\put(765.0,131.0){\rule[-0.200pt]{0.400pt}{2.409pt}}
\put(765.0,766.0){\rule[-0.200pt]{0.400pt}{2.409pt}}
\put(797.0,131.0){\rule[-0.200pt]{0.400pt}{2.409pt}}
\put(797.0,766.0){\rule[-0.200pt]{0.400pt}{2.409pt}}
\put(825.0,131.0){\rule[-0.200pt]{0.400pt}{4.818pt}}
\put(825,90){\makebox(0,0){$0.01$}}
\put(825.0,756.0){\rule[-0.200pt]{0.400pt}{4.818pt}}
\put(1010.0,131.0){\rule[-0.200pt]{0.400pt}{2.409pt}}
\put(1010.0,766.0){\rule[-0.200pt]{0.400pt}{2.409pt}}
\put(1118.0,131.0){\rule[-0.200pt]{0.400pt}{2.409pt}}
\put(1118.0,766.0){\rule[-0.200pt]{0.400pt}{2.409pt}}
\put(1195.0,131.0){\rule[-0.200pt]{0.400pt}{2.409pt}}
\put(1195.0,766.0){\rule[-0.200pt]{0.400pt}{2.409pt}}
\put(1254.0,131.0){\rule[-0.200pt]{0.400pt}{2.409pt}}
\put(1254.0,766.0){\rule[-0.200pt]{0.400pt}{2.409pt}}
\put(1303.0,131.0){\rule[-0.200pt]{0.400pt}{2.409pt}}
\put(1303.0,766.0){\rule[-0.200pt]{0.400pt}{2.409pt}}
\put(1344.0,131.0){\rule[-0.200pt]{0.400pt}{2.409pt}}
\put(1344.0,766.0){\rule[-0.200pt]{0.400pt}{2.409pt}}
\put(1379.0,131.0){\rule[-0.200pt]{0.400pt}{2.409pt}}
\put(1379.0,766.0){\rule[-0.200pt]{0.400pt}{2.409pt}}
\put(1411.0,131.0){\rule[-0.200pt]{0.400pt}{2.409pt}}
\put(1411.0,766.0){\rule[-0.200pt]{0.400pt}{2.409pt}}
\put(1439.0,131.0){\rule[-0.200pt]{0.400pt}{4.818pt}}
\put(1439,90){\makebox(0,0){$0.1$}}
\put(1439.0,756.0){\rule[-0.200pt]{0.400pt}{4.818pt}}
\put(211.0,131.0){\rule[-0.200pt]{0.400pt}{155.380pt}}
\put(211.0,131.0){\rule[-0.200pt]{295.825pt}{0.400pt}}
\put(1439.0,131.0){\rule[-0.200pt]{0.400pt}{155.380pt}}
\put(211.0,776.0){\rule[-0.200pt]{295.825pt}{0.400pt}}
\put(30,453){\makebox(0,0){\rotatebox{90}{$L^2(\mathcal{F})$ error}}}
\put(825,29){\makebox(0,0){$h$}}
\put(825,838){\makebox(0,0){pressure}}
\put(671,494){\makebox(0,0)[r]{$P_2-P_1-P_1$ (slope=2.011)}}
\put(691.0,494.0){\rule[-0.200pt]{24.090pt}{0.400pt}}
\put(1218,248){\usebox{\plotpoint}}
\multiput(1198.39,246.92)(-5.907,-0.494){29}{\rule{4.725pt}{0.119pt}}
\multiput(1208.19,247.17)(-175.193,-16.000){2}{\rule{2.363pt}{0.400pt}}
\multiput(1015.61,230.92)(-5.207,-0.495){33}{\rule{4.189pt}{0.119pt}}
\multiput(1024.31,231.17)(-175.306,-18.000){2}{\rule{2.094pt}{0.400pt}}
\multiput(831.52,212.92)(-5.235,-0.495){33}{\rule{4.211pt}{0.119pt}}
\multiput(840.26,213.17)(-176.260,-18.000){2}{\rule{2.106pt}{0.400pt}}
\multiput(647.42,194.92)(-4.953,-0.495){35}{\rule{3.995pt}{0.119pt}}
\multiput(655.71,195.17)(-176.709,-19.000){2}{\rule{1.997pt}{0.400pt}}
\put(1218,248){\makebox(0,0){$+$}}
\put(1033,232){\makebox(0,0){$+$}}
\put(849,214){\makebox(0,0){$+$}}
\put(664,196){\makebox(0,0){$+$}}
\put(479,177){\makebox(0,0){$+$}}
\put(741,494){\makebox(0,0){$+$}}
\put(671,453){\makebox(0,0)[r]{$P_2-P_1-P_0$ (slope=1.644)}}
\multiput(691,453)(20.756,0.000){5}{\usebox{\plotpoint}}
\put(791,453){\usebox{\plotpoint}}
\put(1218,248){\usebox{\plotpoint}}
\multiput(1218,248)(-20.754,-0.224){9}{\usebox{\plotpoint}}
\multiput(1033,246)(-20.634,-2.243){9}{\usebox{\plotpoint}}
\multiput(849,226)(-20.658,-2.010){9}{\usebox{\plotpoint}}
\multiput(664,208)(-20.704,-1.455){9}{\usebox{\plotpoint}}
\put(479,195){\usebox{\plotpoint}}
\put(1218,248){\makebox(0,0){$\times$}}
\put(1033,246){\makebox(0,0){$\times$}}
\put(849,226){\makebox(0,0){$\times$}}
\put(664,208){\makebox(0,0){$\times$}}
\put(479,195){\makebox(0,0){$\times$}}
\put(741,453){\makebox(0,0){$\times$}}
\sbox{\plotpoint}{\rule[-0.400pt]{0.800pt}{0.800pt}}%
\sbox{\plotpoint}{\rule[-0.200pt]{0.400pt}{0.400pt}}%
\put(671,412){\makebox(0,0)[r]{$P_1-P_1-P_1$ (slope=0.543)}}
\sbox{\plotpoint}{\rule[-0.400pt]{0.800pt}{0.800pt}}%
\put(691.0,412.0){\rule[-0.400pt]{24.090pt}{0.800pt}}
\put(1218,727){\usebox{\plotpoint}}
\put(1033,726.34){\rule{44.567pt}{0.800pt}}
\multiput(1125.50,725.34)(-92.500,2.000){2}{\rule{22.283pt}{0.800pt}}
\multiput(930.33,730.39)(-20.332,0.536){5}{\rule{24.733pt}{0.129pt}}
\multiput(981.66,727.34)(-132.665,6.000){2}{\rule{12.367pt}{0.800pt}}
\multiput(831.57,733.09)(-2.539,-0.503){67}{\rule{4.200pt}{0.121pt}}
\multiput(840.28,733.34)(-176.283,-37.000){2}{\rule{2.100pt}{0.800pt}}
\multiput(632.45,699.41)(-4.795,0.505){33}{\rule{7.600pt}{0.122pt}}
\multiput(648.23,696.34)(-169.226,20.000){2}{\rule{3.800pt}{0.800pt}}
\put(1218,727){\makebox(0,0){$\ast$}}
\put(1033,729){\makebox(0,0){$\ast$}}
\put(849,735){\makebox(0,0){$\ast$}}
\put(664,698){\makebox(0,0){$\ast$}}
\put(479,718){\makebox(0,0){$\ast$}}
\put(741,412){\makebox(0,0){$\ast$}}
\sbox{\plotpoint}{\rule[-0.200pt]{0.400pt}{0.400pt}}%
\put(211.0,131.0){\rule[-0.200pt]{0.400pt}{155.380pt}}
\put(211.0,131.0){\rule[-0.200pt]{295.825pt}{0.400pt}}
\put(1439.0,131.0){\rule[-0.200pt]{0.400pt}{155.380pt}}
\put(211.0,776.0){\rule[-0.200pt]{295.825pt}{0.400pt}}
\end{picture}}
\resizebox{8cm}{!}{
\setlength{\unitlength}{0.240900pt}
\ifx\plotpoint\undefined\newsavebox{\plotpoint}\fi
\begin{picture}(1500,900)(0,0)
\sbox{\plotpoint}{\rule[-0.200pt]{0.400pt}{0.400pt}}%
\put(211.0,131.0){\rule[-0.200pt]{4.818pt}{0.400pt}}
\put(191,131){\makebox(0,0)[r]{$1e-10$}}
\put(1419.0,131.0){\rule[-0.200pt]{4.818pt}{0.400pt}}
\put(211.0,157.0){\rule[-0.200pt]{2.409pt}{0.400pt}}
\put(1429.0,157.0){\rule[-0.200pt]{2.409pt}{0.400pt}}
\put(211.0,183.0){\rule[-0.200pt]{2.409pt}{0.400pt}}
\put(1429.0,183.0){\rule[-0.200pt]{2.409pt}{0.400pt}}
\put(211.0,208.0){\rule[-0.200pt]{2.409pt}{0.400pt}}
\put(1429.0,208.0){\rule[-0.200pt]{2.409pt}{0.400pt}}
\put(211.0,234.0){\rule[-0.200pt]{2.409pt}{0.400pt}}
\put(1429.0,234.0){\rule[-0.200pt]{2.409pt}{0.400pt}}
\put(211.0,260.0){\rule[-0.200pt]{4.818pt}{0.400pt}}
\put(191,260){\makebox(0,0)[r]{$1e-05$}}
\put(1419.0,260.0){\rule[-0.200pt]{4.818pt}{0.400pt}}
\put(211.0,286.0){\rule[-0.200pt]{2.409pt}{0.400pt}}
\put(1429.0,286.0){\rule[-0.200pt]{2.409pt}{0.400pt}}
\put(211.0,312.0){\rule[-0.200pt]{2.409pt}{0.400pt}}
\put(1429.0,312.0){\rule[-0.200pt]{2.409pt}{0.400pt}}
\put(211.0,337.0){\rule[-0.200pt]{2.409pt}{0.400pt}}
\put(1429.0,337.0){\rule[-0.200pt]{2.409pt}{0.400pt}}
\put(211.0,363.0){\rule[-0.200pt]{2.409pt}{0.400pt}}
\put(1429.0,363.0){\rule[-0.200pt]{2.409pt}{0.400pt}}
\put(211.0,389.0){\rule[-0.200pt]{4.818pt}{0.400pt}}
\put(191,389){\makebox(0,0)[r]{$1$}}
\put(1419.0,389.0){\rule[-0.200pt]{4.818pt}{0.400pt}}
\put(211.0,415.0){\rule[-0.200pt]{2.409pt}{0.400pt}}
\put(1429.0,415.0){\rule[-0.200pt]{2.409pt}{0.400pt}}
\put(211.0,441.0){\rule[-0.200pt]{2.409pt}{0.400pt}}
\put(1429.0,441.0){\rule[-0.200pt]{2.409pt}{0.400pt}}
\put(211.0,466.0){\rule[-0.200pt]{2.409pt}{0.400pt}}
\put(1429.0,466.0){\rule[-0.200pt]{2.409pt}{0.400pt}}
\put(211.0,492.0){\rule[-0.200pt]{2.409pt}{0.400pt}}
\put(1429.0,492.0){\rule[-0.200pt]{2.409pt}{0.400pt}}
\put(211.0,518.0){\rule[-0.200pt]{4.818pt}{0.400pt}}
\put(191,518){\makebox(0,0)[r]{$100000$}}
\put(1419.0,518.0){\rule[-0.200pt]{4.818pt}{0.400pt}}
\put(211.0,544.0){\rule[-0.200pt]{2.409pt}{0.400pt}}
\put(1429.0,544.0){\rule[-0.200pt]{2.409pt}{0.400pt}}
\put(211.0,570.0){\rule[-0.200pt]{2.409pt}{0.400pt}}
\put(1429.0,570.0){\rule[-0.200pt]{2.409pt}{0.400pt}}
\put(211.0,595.0){\rule[-0.200pt]{2.409pt}{0.400pt}}
\put(1429.0,595.0){\rule[-0.200pt]{2.409pt}{0.400pt}}
\put(211.0,621.0){\rule[-0.200pt]{2.409pt}{0.400pt}}
\put(1429.0,621.0){\rule[-0.200pt]{2.409pt}{0.400pt}}
\put(211.0,647.0){\rule[-0.200pt]{4.818pt}{0.400pt}}
\put(191,647){\makebox(0,0)[r]{$1e+10$}}
\put(1419.0,647.0){\rule[-0.200pt]{4.818pt}{0.400pt}}
\put(211.0,673.0){\rule[-0.200pt]{2.409pt}{0.400pt}}
\put(1429.0,673.0){\rule[-0.200pt]{2.409pt}{0.400pt}}
\put(211.0,699.0){\rule[-0.200pt]{2.409pt}{0.400pt}}
\put(1429.0,699.0){\rule[-0.200pt]{2.409pt}{0.400pt}}
\put(211.0,724.0){\rule[-0.200pt]{2.409pt}{0.400pt}}
\put(1429.0,724.0){\rule[-0.200pt]{2.409pt}{0.400pt}}
\put(211.0,750.0){\rule[-0.200pt]{2.409pt}{0.400pt}}
\put(1429.0,750.0){\rule[-0.200pt]{2.409pt}{0.400pt}}
\put(211.0,776.0){\rule[-0.200pt]{4.818pt}{0.400pt}}
\put(191,776){\makebox(0,0)[r]{$1e+15$}}
\put(1419.0,776.0){\rule[-0.200pt]{4.818pt}{0.400pt}}
\put(211.0,131.0){\rule[-0.200pt]{0.400pt}{4.818pt}}
\put(211,90){\makebox(0,0){$0.001$}}
\put(211.0,756.0){\rule[-0.200pt]{0.400pt}{4.818pt}}
\put(396.0,131.0){\rule[-0.200pt]{0.400pt}{2.409pt}}
\put(396.0,766.0){\rule[-0.200pt]{0.400pt}{2.409pt}}
\put(504.0,131.0){\rule[-0.200pt]{0.400pt}{2.409pt}}
\put(504.0,766.0){\rule[-0.200pt]{0.400pt}{2.409pt}}
\put(581.0,131.0){\rule[-0.200pt]{0.400pt}{2.409pt}}
\put(581.0,766.0){\rule[-0.200pt]{0.400pt}{2.409pt}}
\put(640.0,131.0){\rule[-0.200pt]{0.400pt}{2.409pt}}
\put(640.0,766.0){\rule[-0.200pt]{0.400pt}{2.409pt}}
\put(689.0,131.0){\rule[-0.200pt]{0.400pt}{2.409pt}}
\put(689.0,766.0){\rule[-0.200pt]{0.400pt}{2.409pt}}
\put(730.0,131.0){\rule[-0.200pt]{0.400pt}{2.409pt}}
\put(730.0,766.0){\rule[-0.200pt]{0.400pt}{2.409pt}}
\put(765.0,131.0){\rule[-0.200pt]{0.400pt}{2.409pt}}
\put(765.0,766.0){\rule[-0.200pt]{0.400pt}{2.409pt}}
\put(797.0,131.0){\rule[-0.200pt]{0.400pt}{2.409pt}}
\put(797.0,766.0){\rule[-0.200pt]{0.400pt}{2.409pt}}
\put(825.0,131.0){\rule[-0.200pt]{0.400pt}{4.818pt}}
\put(825,90){\makebox(0,0){$0.01$}}
\put(825.0,756.0){\rule[-0.200pt]{0.400pt}{4.818pt}}
\put(1010.0,131.0){\rule[-0.200pt]{0.400pt}{2.409pt}}
\put(1010.0,766.0){\rule[-0.200pt]{0.400pt}{2.409pt}}
\put(1118.0,131.0){\rule[-0.200pt]{0.400pt}{2.409pt}}
\put(1118.0,766.0){\rule[-0.200pt]{0.400pt}{2.409pt}}
\put(1195.0,131.0){\rule[-0.200pt]{0.400pt}{2.409pt}}
\put(1195.0,766.0){\rule[-0.200pt]{0.400pt}{2.409pt}}
\put(1254.0,131.0){\rule[-0.200pt]{0.400pt}{2.409pt}}
\put(1254.0,766.0){\rule[-0.200pt]{0.400pt}{2.409pt}}
\put(1303.0,131.0){\rule[-0.200pt]{0.400pt}{2.409pt}}
\put(1303.0,766.0){\rule[-0.200pt]{0.400pt}{2.409pt}}
\put(1344.0,131.0){\rule[-0.200pt]{0.400pt}{2.409pt}}
\put(1344.0,766.0){\rule[-0.200pt]{0.400pt}{2.409pt}}
\put(1379.0,131.0){\rule[-0.200pt]{0.400pt}{2.409pt}}
\put(1379.0,766.0){\rule[-0.200pt]{0.400pt}{2.409pt}}
\put(1411.0,131.0){\rule[-0.200pt]{0.400pt}{2.409pt}}
\put(1411.0,766.0){\rule[-0.200pt]{0.400pt}{2.409pt}}
\put(1439.0,131.0){\rule[-0.200pt]{0.400pt}{4.818pt}}
\put(1439,90){\makebox(0,0){$0.1$}}
\put(1439.0,756.0){\rule[-0.200pt]{0.400pt}{4.818pt}}
\put(211.0,131.0){\rule[-0.200pt]{0.400pt}{155.380pt}}
\put(211.0,131.0){\rule[-0.200pt]{295.825pt}{0.400pt}}
\put(1439.0,131.0){\rule[-0.200pt]{0.400pt}{155.380pt}}
\put(211.0,776.0){\rule[-0.200pt]{295.825pt}{0.400pt}}
\put(30,453){\makebox(0,0){\rotatebox{90}{$\left|\int_{\Gamma}(\lambda-\lambda_{h})\right|$}}}
\put(825,29){\makebox(0,0){$h$}}
\put(825,838){\makebox(0,0){multiplier}}
\put(671,735){\makebox(0,0)[r]{$P_2-P_1-P_1$ (slope=3.010)}}
\put(691.0,735.0){\rule[-0.200pt]{24.090pt}{0.400pt}}
\put(1218,295){\usebox{\plotpoint}}
\multiput(1209.50,293.92)(-2.451,-0.498){73}{\rule{2.047pt}{0.120pt}}
\multiput(1213.75,294.17)(-180.751,-38.000){2}{\rule{1.024pt}{0.400pt}}
\multiput(988.94,255.93)(-13.994,-0.485){11}{\rule{10.614pt}{0.117pt}}
\multiput(1010.97,256.17)(-161.970,-7.000){2}{\rule{5.307pt}{0.400pt}}
\multiput(841.44,248.92)(-2.164,-0.498){83}{\rule{1.821pt}{0.120pt}}
\multiput(845.22,249.17)(-181.221,-43.000){2}{\rule{0.910pt}{0.400pt}}
\multiput(586.79,205.94)(-26.947,-0.468){5}{\rule{18.600pt}{0.113pt}}
\multiput(625.39,206.17)(-146.395,-4.000){2}{\rule{9.300pt}{0.400pt}}
\put(1218,295){\makebox(0,0){$+$}}
\put(1033,257){\makebox(0,0){$+$}}
\put(849,250){\makebox(0,0){$+$}}
\put(664,207){\makebox(0,0){$+$}}
\put(479,203){\makebox(0,0){$+$}}
\put(741,735){\makebox(0,0){$+$}}
\put(671,694){\makebox(0,0)[r]{$P_2-P_1-P_0$ (slope=2.927)}}
\multiput(691,694)(20.756,0.000){5}{\usebox{\plotpoint}}
\put(791,694){\usebox{\plotpoint}}
\put(1218,295){\usebox{\plotpoint}}
\multiput(1218,295)(-20.217,-4.699){10}{\usebox{\plotpoint}}
\multiput(1033,252)(-20.753,-0.338){9}{\usebox{\plotpoint}}
\multiput(849,249)(-20.264,-4.491){9}{\usebox{\plotpoint}}
\multiput(664,208)(-20.748,-0.561){9}{\usebox{\plotpoint}}
\put(479,203){\usebox{\plotpoint}}
\put(1218,295){\makebox(0,0){$\times$}}
\put(1033,252){\makebox(0,0){$\times$}}
\put(849,249){\makebox(0,0){$\times$}}
\put(664,208){\makebox(0,0){$\times$}}
\put(479,203){\makebox(0,0){$\times$}}
\put(741,694){\makebox(0,0){$\times$}}
\sbox{\plotpoint}{\rule[-0.400pt]{0.800pt}{0.800pt}}%
\sbox{\plotpoint}{\rule[-0.200pt]{0.400pt}{0.400pt}}%
\put(671,653){\makebox(0,0)[r]{$P_1-P_1-P_1$ (slope=7.852)}}
\sbox{\plotpoint}{\rule[-0.400pt]{0.800pt}{0.800pt}}%
\put(691.0,653.0){\rule[-0.400pt]{24.090pt}{0.800pt}}
\put(1218,731){\usebox{\plotpoint}}
\multiput(1216.09,724.07)(-0.500,-0.920){363}{\rule{0.121pt}{1.670pt}}
\multiput(1216.34,727.53)(-185.000,-336.533){2}{\rule{0.800pt}{0.835pt}}
\multiput(1003.07,392.41)(-4.532,0.505){35}{\rule{7.210pt}{0.122pt}}
\multiput(1018.04,389.34)(-169.036,21.000){2}{\rule{3.605pt}{0.800pt}}
\multiput(822.57,410.09)(-3.964,-0.504){41}{\rule{6.367pt}{0.122pt}}
\multiput(835.79,410.34)(-171.786,-24.000){2}{\rule{3.183pt}{0.800pt}}
\multiput(647.81,389.41)(-2.345,0.502){73}{\rule{3.900pt}{0.121pt}}
\multiput(655.91,386.34)(-176.905,40.000){2}{\rule{1.950pt}{0.800pt}}
\put(1218,731){\makebox(0,0){$\ast$}}
\put(1033,391){\makebox(0,0){$\ast$}}
\put(849,412){\makebox(0,0){$\ast$}}
\put(664,388){\makebox(0,0){$\ast$}}
\put(479,428){\makebox(0,0){$\ast$}}
\put(741,653){\makebox(0,0){$\ast$}}
\sbox{\plotpoint}{\rule[-0.200pt]{0.400pt}{0.400pt}}%
\put(211.0,131.0){\rule[-0.200pt]{0.400pt}{155.380pt}}
\put(211.0,131.0){\rule[-0.200pt]{295.825pt}{0.400pt}}
\put(1439.0,131.0){\rule[-0.200pt]{0.400pt}{155.380pt}}
\put(211.0,776.0){\rule[-0.200pt]{295.825pt}{0.400pt}}
\end{picture}}
\end{minipage}
\caption{Rates of convergence without stabilization for $\|u-u_h\|_{0,\mathcal{F}}$, 
$\|u-u_h\|_{1,\mathcal{F}}$, $\|p-p_h\|_{0,\mathcal{F}}$ and $\left|\int_{\Gamma}(\lambda-\lambda_{h})\right|$ }
\label{fig:sans_stab}
\end{figure}
\end{center}
\vspace*{-1.5cm}

\subsection{Methods \`a la Barbosa-Hughes.}
\label{sec:Barbosa-Hughes}
We consider now stabilization  \`a la Barbosa-Hughes, i.e. (\ref{methBP}) or (\ref{methTH}) without the distinction between good and bad triangles or pressure stabilization ($\theta_{\min}=\theta=0)$. Stabilization terms multiplied by $\gamma_0h$ are thus added to system (\ref{sys:sans_stab}):
\begin{equation}
\label{sys:Barbosa-Hughes}
\left (
\begin{array}{ccc}
K +S^{\gamma_0}_{uu} & B^T+{S^{\gamma_0}_{up}}^{\hspace*{-0.1cm}T}&C^T + {S^{\gamma_0}_{u\lambda}}^{\hspace*{-0.1cm}T}\\
B + S^{\gamma_0}_{up} & S^{\gamma_0}_{pp}& {S^{\gamma_0}_{p\lambda}}^{\hspace*{-0.1cm}T}\\
C + S^{\gamma_0}_{u \lambda} & S^{\gamma_0}_{p \lambda} & S^{\gamma_0}_{\lambda \lambda}\\
\end{array}
\right )
\left (
\begin{array}{c}
U\\
P\\
\Lambda\\
\end{array}
\right ) = 
\left (
\begin{array}{c}
F\\
0\\
G \\
\end{array}
\right ) 
\end{equation}
where
$$
\left (S^{\gamma_0}_{uu}\right )_{i_uj_u} = -4 \gamma_0h \int_{\Gamma} D(\phi_{i_u}) n\cdot D(\phi_{j_u})n, \
\left (S^{\gamma_0}_{up}\right )_{i_uj_p}   = 2\gamma_0h \int_{\Gamma} D(\phi_{i_u})n\cdot  \psi_{j_p}n, \
\left (S^{\gamma_0}_{u \lambda}\right )_{i_uj_\lambda}   = -2\gamma_0h \int_{\Gamma} D(\phi_{i_u})n\cdot  \zeta_{j_\lambda}$$
$$
\left (S^{\gamma_0}_{pp}\right )_{i_pj_p}   = -\gamma_0h \int_{\Gamma} \psi_{i_p} \psi_{j_p}, \
\left (S^{\gamma_0}_{p \lambda}\right )_{i_pj_\lambda}  = \gamma_0 h \int_{\Gamma}\psi_{i_p} n\cdot \zeta_{j_\lambda}, \
\left (S^{\gamma_0}_{\lambda \lambda}\right )_{i_\lambda j_\lambda}  = -\gamma_0 h \int_{\Gamma}\zeta_{i_\lambda} .\zeta_{j_\lambda}
$$
Notice that no "robust reconstruction" is applied (cf. Definition \ref{RobRec}) although
small intersections with the domain do occur.
We report in Fig. \ref{fig:Barbosa-Hughes}  the rates of convergence (cf.  \cite{Court14} as well).
The spaces considered are the same as in the previous tests without stabilization. 
Results for velocity and pressure are similar with optimal rates of convergence.
The improvement is clear in the
$\P_1-\P_1-\P_1$ case where the force on $\Gamma$ is well computed with optimal error.

\vspace*{-0.5cm} 

\begin{center}
\begin{figure}
\hspace*{-2cm}
\begin{minipage}{20cm}
 \resizebox{8cm}{!}{
\setlength{\unitlength}{0.240900pt}
\ifx\plotpoint\undefined\newsavebox{\plotpoint}\fi
\begin{picture}(1500,900)(0,0)
\sbox{\plotpoint}{\rule[-0.200pt]{0.400pt}{0.400pt}}%
\put(211.0,131.0){\rule[-0.200pt]{4.818pt}{0.400pt}}
\put(191,131){\makebox(0,0)[r]{$1e-08$}}
\put(1419.0,131.0){\rule[-0.200pt]{4.818pt}{0.400pt}}
\put(211.0,159.0){\rule[-0.200pt]{2.409pt}{0.400pt}}
\put(1429.0,159.0){\rule[-0.200pt]{2.409pt}{0.400pt}}
\put(211.0,195.0){\rule[-0.200pt]{2.409pt}{0.400pt}}
\put(1429.0,195.0){\rule[-0.200pt]{2.409pt}{0.400pt}}
\put(211.0,214.0){\rule[-0.200pt]{2.409pt}{0.400pt}}
\put(1429.0,214.0){\rule[-0.200pt]{2.409pt}{0.400pt}}
\put(211.0,223.0){\rule[-0.200pt]{4.818pt}{0.400pt}}
\put(191,223){\makebox(0,0)[r]{$1e-07$}}
\put(1419.0,223.0){\rule[-0.200pt]{4.818pt}{0.400pt}}
\put(211.0,251.0){\rule[-0.200pt]{2.409pt}{0.400pt}}
\put(1429.0,251.0){\rule[-0.200pt]{2.409pt}{0.400pt}}
\put(211.0,288.0){\rule[-0.200pt]{2.409pt}{0.400pt}}
\put(1429.0,288.0){\rule[-0.200pt]{2.409pt}{0.400pt}}
\put(211.0,306.0){\rule[-0.200pt]{2.409pt}{0.400pt}}
\put(1429.0,306.0){\rule[-0.200pt]{2.409pt}{0.400pt}}
\put(211.0,315.0){\rule[-0.200pt]{4.818pt}{0.400pt}}
\put(191,315){\makebox(0,0)[r]{$1e-06$}}
\put(1419.0,315.0){\rule[-0.200pt]{4.818pt}{0.400pt}}
\put(211.0,343.0){\rule[-0.200pt]{2.409pt}{0.400pt}}
\put(1429.0,343.0){\rule[-0.200pt]{2.409pt}{0.400pt}}
\put(211.0,380.0){\rule[-0.200pt]{2.409pt}{0.400pt}}
\put(1429.0,380.0){\rule[-0.200pt]{2.409pt}{0.400pt}}
\put(211.0,398.0){\rule[-0.200pt]{2.409pt}{0.400pt}}
\put(1429.0,398.0){\rule[-0.200pt]{2.409pt}{0.400pt}}
\put(211.0,407.0){\rule[-0.200pt]{4.818pt}{0.400pt}}
\put(191,407){\makebox(0,0)[r]{$1e-05$}}
\put(1419.0,407.0){\rule[-0.200pt]{4.818pt}{0.400pt}}
\put(211.0,435.0){\rule[-0.200pt]{2.409pt}{0.400pt}}
\put(1429.0,435.0){\rule[-0.200pt]{2.409pt}{0.400pt}}
\put(211.0,472.0){\rule[-0.200pt]{2.409pt}{0.400pt}}
\put(1429.0,472.0){\rule[-0.200pt]{2.409pt}{0.400pt}}
\put(211.0,491.0){\rule[-0.200pt]{2.409pt}{0.400pt}}
\put(1429.0,491.0){\rule[-0.200pt]{2.409pt}{0.400pt}}
\put(211.0,500.0){\rule[-0.200pt]{4.818pt}{0.400pt}}
\put(191,500){\makebox(0,0)[r]{$0.0001$}}
\put(1419.0,500.0){\rule[-0.200pt]{4.818pt}{0.400pt}}
\put(211.0,527.0){\rule[-0.200pt]{2.409pt}{0.400pt}}
\put(1429.0,527.0){\rule[-0.200pt]{2.409pt}{0.400pt}}
\put(211.0,564.0){\rule[-0.200pt]{2.409pt}{0.400pt}}
\put(1429.0,564.0){\rule[-0.200pt]{2.409pt}{0.400pt}}
\put(211.0,583.0){\rule[-0.200pt]{2.409pt}{0.400pt}}
\put(1429.0,583.0){\rule[-0.200pt]{2.409pt}{0.400pt}}
\put(211.0,592.0){\rule[-0.200pt]{4.818pt}{0.400pt}}
\put(191,592){\makebox(0,0)[r]{$0.001$}}
\put(1419.0,592.0){\rule[-0.200pt]{4.818pt}{0.400pt}}
\put(211.0,619.0){\rule[-0.200pt]{2.409pt}{0.400pt}}
\put(1429.0,619.0){\rule[-0.200pt]{2.409pt}{0.400pt}}
\put(211.0,656.0){\rule[-0.200pt]{2.409pt}{0.400pt}}
\put(1429.0,656.0){\rule[-0.200pt]{2.409pt}{0.400pt}}
\put(211.0,675.0){\rule[-0.200pt]{2.409pt}{0.400pt}}
\put(1429.0,675.0){\rule[-0.200pt]{2.409pt}{0.400pt}}
\put(211.0,684.0){\rule[-0.200pt]{4.818pt}{0.400pt}}
\put(191,684){\makebox(0,0)[r]{$0.01$}}
\put(1419.0,684.0){\rule[-0.200pt]{4.818pt}{0.400pt}}
\put(211.0,712.0){\rule[-0.200pt]{2.409pt}{0.400pt}}
\put(1429.0,712.0){\rule[-0.200pt]{2.409pt}{0.400pt}}
\put(211.0,748.0){\rule[-0.200pt]{2.409pt}{0.400pt}}
\put(1429.0,748.0){\rule[-0.200pt]{2.409pt}{0.400pt}}
\put(211.0,767.0){\rule[-0.200pt]{2.409pt}{0.400pt}}
\put(1429.0,767.0){\rule[-0.200pt]{2.409pt}{0.400pt}}
\put(211.0,776.0){\rule[-0.200pt]{4.818pt}{0.400pt}}
\put(191,776){\makebox(0,0)[r]{$0.1$}}
\put(1419.0,776.0){\rule[-0.200pt]{4.818pt}{0.400pt}}
\put(211.0,131.0){\rule[-0.200pt]{0.400pt}{4.818pt}}
\put(211,90){\makebox(0,0){$0.001$}}
\put(211.0,756.0){\rule[-0.200pt]{0.400pt}{4.818pt}}
\put(396.0,131.0){\rule[-0.200pt]{0.400pt}{2.409pt}}
\put(396.0,766.0){\rule[-0.200pt]{0.400pt}{2.409pt}}
\put(504.0,131.0){\rule[-0.200pt]{0.400pt}{2.409pt}}
\put(504.0,766.0){\rule[-0.200pt]{0.400pt}{2.409pt}}
\put(581.0,131.0){\rule[-0.200pt]{0.400pt}{2.409pt}}
\put(581.0,766.0){\rule[-0.200pt]{0.400pt}{2.409pt}}
\put(640.0,131.0){\rule[-0.200pt]{0.400pt}{2.409pt}}
\put(640.0,766.0){\rule[-0.200pt]{0.400pt}{2.409pt}}
\put(689.0,131.0){\rule[-0.200pt]{0.400pt}{2.409pt}}
\put(689.0,766.0){\rule[-0.200pt]{0.400pt}{2.409pt}}
\put(730.0,131.0){\rule[-0.200pt]{0.400pt}{2.409pt}}
\put(730.0,766.0){\rule[-0.200pt]{0.400pt}{2.409pt}}
\put(765.0,131.0){\rule[-0.200pt]{0.400pt}{2.409pt}}
\put(765.0,766.0){\rule[-0.200pt]{0.400pt}{2.409pt}}
\put(797.0,131.0){\rule[-0.200pt]{0.400pt}{2.409pt}}
\put(797.0,766.0){\rule[-0.200pt]{0.400pt}{2.409pt}}
\put(825.0,131.0){\rule[-0.200pt]{0.400pt}{4.818pt}}
\put(825,90){\makebox(0,0){$0.01$}}
\put(825.0,756.0){\rule[-0.200pt]{0.400pt}{4.818pt}}
\put(1010.0,131.0){\rule[-0.200pt]{0.400pt}{2.409pt}}
\put(1010.0,766.0){\rule[-0.200pt]{0.400pt}{2.409pt}}
\put(1118.0,131.0){\rule[-0.200pt]{0.400pt}{2.409pt}}
\put(1118.0,766.0){\rule[-0.200pt]{0.400pt}{2.409pt}}
\put(1195.0,131.0){\rule[-0.200pt]{0.400pt}{2.409pt}}
\put(1195.0,766.0){\rule[-0.200pt]{0.400pt}{2.409pt}}
\put(1254.0,131.0){\rule[-0.200pt]{0.400pt}{2.409pt}}
\put(1254.0,766.0){\rule[-0.200pt]{0.400pt}{2.409pt}}
\put(1303.0,131.0){\rule[-0.200pt]{0.400pt}{2.409pt}}
\put(1303.0,766.0){\rule[-0.200pt]{0.400pt}{2.409pt}}
\put(1344.0,131.0){\rule[-0.200pt]{0.400pt}{2.409pt}}
\put(1344.0,766.0){\rule[-0.200pt]{0.400pt}{2.409pt}}
\put(1379.0,131.0){\rule[-0.200pt]{0.400pt}{2.409pt}}
\put(1379.0,766.0){\rule[-0.200pt]{0.400pt}{2.409pt}}
\put(1411.0,131.0){\rule[-0.200pt]{0.400pt}{2.409pt}}
\put(1411.0,766.0){\rule[-0.200pt]{0.400pt}{2.409pt}}
\put(1439.0,131.0){\rule[-0.200pt]{0.400pt}{4.818pt}}
\put(1439,90){\makebox(0,0){$0.1$}}
\put(1439.0,756.0){\rule[-0.200pt]{0.400pt}{4.818pt}}
\put(211.0,131.0){\rule[-0.200pt]{0.400pt}{155.380pt}}
\put(211.0,131.0){\rule[-0.200pt]{295.825pt}{0.400pt}}
\put(1439.0,131.0){\rule[-0.200pt]{0.400pt}{155.380pt}}
\put(211.0,776.0){\rule[-0.200pt]{295.825pt}{0.400pt}}
\put(30,453){\makebox(0,0){\rotatebox{90}{$L^2(\mathcal{F})$ error}}}
\put(825,29){\makebox(0,0){$h$}}
\put(825,838){\makebox(0,0){velocity}}
\put(1279,253){\makebox(0,0)[r]{$P_2-P_1-P_1$ (slope=2.991)}}
\put(1299.0,253.0){\rule[-0.200pt]{24.090pt}{0.400pt}}
\put(1218,538){\usebox{\plotpoint}}
\multiput(1213.84,536.92)(-1.130,-0.499){161}{\rule{1.002pt}{0.120pt}}
\multiput(1215.92,537.17)(-182.919,-82.000){2}{\rule{0.501pt}{0.400pt}}
\multiput(1028.90,454.92)(-1.111,-0.499){163}{\rule{0.987pt}{0.120pt}}
\multiput(1030.95,455.17)(-181.952,-83.000){2}{\rule{0.493pt}{0.400pt}}
\multiput(844.88,371.92)(-1.117,-0.499){163}{\rule{0.992pt}{0.120pt}}
\multiput(846.94,372.17)(-182.942,-83.000){2}{\rule{0.496pt}{0.400pt}}
\multiput(659.93,288.92)(-1.103,-0.499){165}{\rule{0.981pt}{0.120pt}}
\multiput(661.96,289.17)(-182.964,-84.000){2}{\rule{0.490pt}{0.400pt}}
\put(1218,538){\makebox(0,0){$+$}}
\put(1033,456){\makebox(0,0){$+$}}
\put(849,373){\makebox(0,0){$+$}}
\put(664,290){\makebox(0,0){$+$}}
\put(479,206){\makebox(0,0){$+$}}
\put(1349,253){\makebox(0,0){$+$}}
\put(1279,212){\makebox(0,0)[r]{$P_2-P_1-P_0$ (slope=2.699)}}
\multiput(1299,212)(20.756,0.000){5}{\usebox{\plotpoint}}
\put(1399,212){\usebox{\plotpoint}}
\put(1218,540){\usebox{\plotpoint}}
\multiput(1218,540)(-20.142,-5.008){10}{\usebox{\plotpoint}}
\multiput(1033,494)(-18.111,-10.138){10}{\usebox{\plotpoint}}
\multiput(849,391)(-19.377,-7.437){9}{\usebox{\plotpoint}}
\multiput(664,320)(-19.515,-7.068){10}{\usebox{\plotpoint}}
\put(479,253){\usebox{\plotpoint}}
\put(1218,540){\makebox(0,0){$\times$}}
\put(1033,494){\makebox(0,0){$\times$}}
\put(849,391){\makebox(0,0){$\times$}}
\put(664,320){\makebox(0,0){$\times$}}
\put(479,253){\makebox(0,0){$\times$}}
\put(1349,212){\makebox(0,0){$\times$}}
\sbox{\plotpoint}{\rule[-0.400pt]{0.800pt}{0.800pt}}%
\sbox{\plotpoint}{\rule[-0.200pt]{0.400pt}{0.400pt}}%
\put(1279,171){\makebox(0,0)[r]{$P_1-P_1-P_1$ (slope=2.117)}}
\sbox{\plotpoint}{\rule[-0.400pt]{0.800pt}{0.800pt}}%
\put(1299.0,171.0){\rule[-0.400pt]{24.090pt}{0.800pt}}
\put(1218,722){\usebox{\plotpoint}}
\multiput(1207.72,720.09)(-1.432,-0.501){123}{\rule{2.477pt}{0.121pt}}
\multiput(1212.86,720.34)(-179.859,-65.000){2}{\rule{1.238pt}{0.800pt}}
\multiput(1021.45,655.09)(-1.627,-0.502){107}{\rule{2.782pt}{0.121pt}}
\multiput(1027.22,655.34)(-178.225,-57.000){2}{\rule{1.391pt}{0.800pt}}
\multiput(837.58,598.09)(-1.607,-0.502){109}{\rule{2.752pt}{0.121pt}}
\multiput(843.29,598.34)(-179.289,-58.000){2}{\rule{1.376pt}{0.800pt}}
\multiput(652.20,540.09)(-1.665,-0.502){105}{\rule{2.843pt}{0.121pt}}
\multiput(658.10,540.34)(-179.100,-56.000){2}{\rule{1.421pt}{0.800pt}}
\put(1218,722){\makebox(0,0){$\ast$}}
\put(1033,657){\makebox(0,0){$\ast$}}
\put(849,600){\makebox(0,0){$\ast$}}
\put(664,542){\makebox(0,0){$\ast$}}
\put(479,486){\makebox(0,0){$\ast$}}
\put(1349,171){\makebox(0,0){$\ast$}}
\sbox{\plotpoint}{\rule[-0.200pt]{0.400pt}{0.400pt}}%
\put(211.0,131.0){\rule[-0.200pt]{0.400pt}{155.380pt}}
\put(211.0,131.0){\rule[-0.200pt]{295.825pt}{0.400pt}}
\put(1439.0,131.0){\rule[-0.200pt]{0.400pt}{155.380pt}}
\put(211.0,776.0){\rule[-0.200pt]{295.825pt}{0.400pt}}
\end{picture}}
 \resizebox{8cm}{!}{
\setlength{\unitlength}{0.240900pt}
\ifx\plotpoint\undefined\newsavebox{\plotpoint}\fi
\begin{picture}(1500,900)(0,0)
\sbox{\plotpoint}{\rule[-0.200pt]{0.400pt}{0.400pt}}%
\put(211.0,131.0){\rule[-0.200pt]{4.818pt}{0.400pt}}
\put(191,131){\makebox(0,0)[r]{$1e-05$}}
\put(1419.0,131.0){\rule[-0.200pt]{4.818pt}{0.400pt}}
\put(211.0,170.0){\rule[-0.200pt]{2.409pt}{0.400pt}}
\put(1429.0,170.0){\rule[-0.200pt]{2.409pt}{0.400pt}}
\put(211.0,221.0){\rule[-0.200pt]{2.409pt}{0.400pt}}
\put(1429.0,221.0){\rule[-0.200pt]{2.409pt}{0.400pt}}
\put(211.0,247.0){\rule[-0.200pt]{2.409pt}{0.400pt}}
\put(1429.0,247.0){\rule[-0.200pt]{2.409pt}{0.400pt}}
\put(211.0,260.0){\rule[-0.200pt]{4.818pt}{0.400pt}}
\put(191,260){\makebox(0,0)[r]{$0.0001$}}
\put(1419.0,260.0){\rule[-0.200pt]{4.818pt}{0.400pt}}
\put(211.0,299.0){\rule[-0.200pt]{2.409pt}{0.400pt}}
\put(1429.0,299.0){\rule[-0.200pt]{2.409pt}{0.400pt}}
\put(211.0,350.0){\rule[-0.200pt]{2.409pt}{0.400pt}}
\put(1429.0,350.0){\rule[-0.200pt]{2.409pt}{0.400pt}}
\put(211.0,376.0){\rule[-0.200pt]{2.409pt}{0.400pt}}
\put(1429.0,376.0){\rule[-0.200pt]{2.409pt}{0.400pt}}
\put(211.0,389.0){\rule[-0.200pt]{4.818pt}{0.400pt}}
\put(191,389){\makebox(0,0)[r]{$0.001$}}
\put(1419.0,389.0){\rule[-0.200pt]{4.818pt}{0.400pt}}
\put(211.0,428.0){\rule[-0.200pt]{2.409pt}{0.400pt}}
\put(1429.0,428.0){\rule[-0.200pt]{2.409pt}{0.400pt}}
\put(211.0,479.0){\rule[-0.200pt]{2.409pt}{0.400pt}}
\put(1429.0,479.0){\rule[-0.200pt]{2.409pt}{0.400pt}}
\put(211.0,505.0){\rule[-0.200pt]{2.409pt}{0.400pt}}
\put(1429.0,505.0){\rule[-0.200pt]{2.409pt}{0.400pt}}
\put(211.0,518.0){\rule[-0.200pt]{4.818pt}{0.400pt}}
\put(191,518){\makebox(0,0)[r]{$0.01$}}
\put(1419.0,518.0){\rule[-0.200pt]{4.818pt}{0.400pt}}
\put(211.0,557.0){\rule[-0.200pt]{2.409pt}{0.400pt}}
\put(1429.0,557.0){\rule[-0.200pt]{2.409pt}{0.400pt}}
\put(211.0,608.0){\rule[-0.200pt]{2.409pt}{0.400pt}}
\put(1429.0,608.0){\rule[-0.200pt]{2.409pt}{0.400pt}}
\put(211.0,634.0){\rule[-0.200pt]{2.409pt}{0.400pt}}
\put(1429.0,634.0){\rule[-0.200pt]{2.409pt}{0.400pt}}
\put(211.0,647.0){\rule[-0.200pt]{4.818pt}{0.400pt}}
\put(191,647){\makebox(0,0)[r]{$0.1$}}
\put(1419.0,647.0){\rule[-0.200pt]{4.818pt}{0.400pt}}
\put(211.0,686.0){\rule[-0.200pt]{2.409pt}{0.400pt}}
\put(1429.0,686.0){\rule[-0.200pt]{2.409pt}{0.400pt}}
\put(211.0,737.0){\rule[-0.200pt]{2.409pt}{0.400pt}}
\put(1429.0,737.0){\rule[-0.200pt]{2.409pt}{0.400pt}}
\put(211.0,763.0){\rule[-0.200pt]{2.409pt}{0.400pt}}
\put(1429.0,763.0){\rule[-0.200pt]{2.409pt}{0.400pt}}
\put(211.0,776.0){\rule[-0.200pt]{4.818pt}{0.400pt}}
\put(191,776){\makebox(0,0)[r]{$1$}}
\put(1419.0,776.0){\rule[-0.200pt]{4.818pt}{0.400pt}}
\put(211.0,131.0){\rule[-0.200pt]{0.400pt}{4.818pt}}
\put(211,90){\makebox(0,0){$0.001$}}
\put(211.0,756.0){\rule[-0.200pt]{0.400pt}{4.818pt}}
\put(396.0,131.0){\rule[-0.200pt]{0.400pt}{2.409pt}}
\put(396.0,766.0){\rule[-0.200pt]{0.400pt}{2.409pt}}
\put(504.0,131.0){\rule[-0.200pt]{0.400pt}{2.409pt}}
\put(504.0,766.0){\rule[-0.200pt]{0.400pt}{2.409pt}}
\put(581.0,131.0){\rule[-0.200pt]{0.400pt}{2.409pt}}
\put(581.0,766.0){\rule[-0.200pt]{0.400pt}{2.409pt}}
\put(640.0,131.0){\rule[-0.200pt]{0.400pt}{2.409pt}}
\put(640.0,766.0){\rule[-0.200pt]{0.400pt}{2.409pt}}
\put(689.0,131.0){\rule[-0.200pt]{0.400pt}{2.409pt}}
\put(689.0,766.0){\rule[-0.200pt]{0.400pt}{2.409pt}}
\put(730.0,131.0){\rule[-0.200pt]{0.400pt}{2.409pt}}
\put(730.0,766.0){\rule[-0.200pt]{0.400pt}{2.409pt}}
\put(765.0,131.0){\rule[-0.200pt]{0.400pt}{2.409pt}}
\put(765.0,766.0){\rule[-0.200pt]{0.400pt}{2.409pt}}
\put(797.0,131.0){\rule[-0.200pt]{0.400pt}{2.409pt}}
\put(797.0,766.0){\rule[-0.200pt]{0.400pt}{2.409pt}}
\put(825.0,131.0){\rule[-0.200pt]{0.400pt}{4.818pt}}
\put(825,90){\makebox(0,0){$0.01$}}
\put(825.0,756.0){\rule[-0.200pt]{0.400pt}{4.818pt}}
\put(1010.0,131.0){\rule[-0.200pt]{0.400pt}{2.409pt}}
\put(1010.0,766.0){\rule[-0.200pt]{0.400pt}{2.409pt}}
\put(1118.0,131.0){\rule[-0.200pt]{0.400pt}{2.409pt}}
\put(1118.0,766.0){\rule[-0.200pt]{0.400pt}{2.409pt}}
\put(1195.0,131.0){\rule[-0.200pt]{0.400pt}{2.409pt}}
\put(1195.0,766.0){\rule[-0.200pt]{0.400pt}{2.409pt}}
\put(1254.0,131.0){\rule[-0.200pt]{0.400pt}{2.409pt}}
\put(1254.0,766.0){\rule[-0.200pt]{0.400pt}{2.409pt}}
\put(1303.0,131.0){\rule[-0.200pt]{0.400pt}{2.409pt}}
\put(1303.0,766.0){\rule[-0.200pt]{0.400pt}{2.409pt}}
\put(1344.0,131.0){\rule[-0.200pt]{0.400pt}{2.409pt}}
\put(1344.0,766.0){\rule[-0.200pt]{0.400pt}{2.409pt}}
\put(1379.0,131.0){\rule[-0.200pt]{0.400pt}{2.409pt}}
\put(1379.0,766.0){\rule[-0.200pt]{0.400pt}{2.409pt}}
\put(1411.0,131.0){\rule[-0.200pt]{0.400pt}{2.409pt}}
\put(1411.0,766.0){\rule[-0.200pt]{0.400pt}{2.409pt}}
\put(1439.0,131.0){\rule[-0.200pt]{0.400pt}{4.818pt}}
\put(1439,90){\makebox(0,0){$0.1$}}
\put(1439.0,756.0){\rule[-0.200pt]{0.400pt}{4.818pt}}
\put(211.0,131.0){\rule[-0.200pt]{0.400pt}{155.380pt}}
\put(211.0,131.0){\rule[-0.200pt]{295.825pt}{0.400pt}}
\put(1439.0,131.0){\rule[-0.200pt]{0.400pt}{155.380pt}}
\put(211.0,776.0){\rule[-0.200pt]{295.825pt}{0.400pt}}
\put(30,453){\makebox(0,0){\rotatebox{90}{$H^1(\mathcal{F})$ error}}}
\put(825,29){\makebox(0,0){$h$}}
\put(825,838){\makebox(0,0){velocity}}
\put(1279,253){\makebox(0,0)[r]{$P_2-P_1-P_1$ (slope=1.948)}}
\put(1299.0,253.0){\rule[-0.200pt]{24.090pt}{0.400pt}}
\put(1218,537){\usebox{\plotpoint}}
\multiput(1213.49,535.92)(-1.236,-0.499){147}{\rule{1.087pt}{0.120pt}}
\multiput(1215.74,536.17)(-182.745,-75.000){2}{\rule{0.543pt}{0.400pt}}
\multiput(1028.67,460.92)(-1.182,-0.499){153}{\rule{1.044pt}{0.120pt}}
\multiput(1030.83,461.17)(-181.834,-78.000){2}{\rule{0.522pt}{0.400pt}}
\multiput(843.55,382.92)(-1.522,-0.499){119}{\rule{1.313pt}{0.120pt}}
\multiput(846.27,383.17)(-182.275,-61.000){2}{\rule{0.657pt}{0.400pt}}
\multiput(660.35,321.92)(-0.975,-0.499){187}{\rule{0.879pt}{0.120pt}}
\multiput(662.18,322.17)(-183.176,-95.000){2}{\rule{0.439pt}{0.400pt}}
\put(1218,537){\makebox(0,0){$+$}}
\put(1033,462){\makebox(0,0){$+$}}
\put(849,384){\makebox(0,0){$+$}}
\put(664,323){\makebox(0,0){$+$}}
\put(479,228){\makebox(0,0){$+$}}
\put(1349,253){\makebox(0,0){$+$}}
\put(1279,212){\makebox(0,0)[r]{$P_2-P_1-P_0$ (slope=1.590)}}
\multiput(1299,212)(20.756,0.000){5}{\usebox{\plotpoint}}
\put(1399,212){\usebox{\plotpoint}}
\put(1218,541){\usebox{\plotpoint}}
\multiput(1218,541)(-20.009,5.516){10}{\usebox{\plotpoint}}
\multiput(1033,592)(-15.704,-13.571){11}{\usebox{\plotpoint}}
\multiput(849,433)(-20.741,-0.785){9}{\usebox{\plotpoint}}
\multiput(664,426)(-17.798,-10.679){11}{\usebox{\plotpoint}}
\put(479,315){\usebox{\plotpoint}}
\put(1218,541){\makebox(0,0){$\times$}}
\put(1033,592){\makebox(0,0){$\times$}}
\put(849,433){\makebox(0,0){$\times$}}
\put(664,426){\makebox(0,0){$\times$}}
\put(479,315){\makebox(0,0){$\times$}}
\put(1349,212){\makebox(0,0){$\times$}}
\sbox{\plotpoint}{\rule[-0.400pt]{0.800pt}{0.800pt}}%
\sbox{\plotpoint}{\rule[-0.200pt]{0.400pt}{0.400pt}}%
\put(1279,171){\makebox(0,0)[r]{$P_1-P_1-P_1$ (slope=1.097)}}
\sbox{\plotpoint}{\rule[-0.400pt]{0.800pt}{0.800pt}}%
\put(1299.0,171.0){\rule[-0.400pt]{24.090pt}{0.800pt}}
\put(1218,772){\usebox{\plotpoint}}
\multiput(1205.58,770.09)(-1.761,-0.502){99}{\rule{2.992pt}{0.121pt}}
\multiput(1211.79,770.34)(-178.789,-53.000){2}{\rule{1.496pt}{0.800pt}}
\multiput(1017.27,717.09)(-2.274,-0.502){75}{\rule{3.790pt}{0.121pt}}
\multiput(1025.13,717.34)(-176.133,-41.000){2}{\rule{1.895pt}{0.800pt}}
\multiput(832.81,676.09)(-2.345,-0.502){73}{\rule{3.900pt}{0.121pt}}
\multiput(840.91,676.34)(-176.905,-40.000){2}{\rule{1.950pt}{0.800pt}}
\multiput(647.00,636.09)(-2.471,-0.503){69}{\rule{4.095pt}{0.121pt}}
\multiput(655.50,636.34)(-176.501,-38.000){2}{\rule{2.047pt}{0.800pt}}
\put(1218,772){\makebox(0,0){$\ast$}}
\put(1033,719){\makebox(0,0){$\ast$}}
\put(849,678){\makebox(0,0){$\ast$}}
\put(664,638){\makebox(0,0){$\ast$}}
\put(479,600){\makebox(0,0){$\ast$}}
\put(1349,171){\makebox(0,0){$\ast$}}
\sbox{\plotpoint}{\rule[-0.200pt]{0.400pt}{0.400pt}}%
\put(211.0,131.0){\rule[-0.200pt]{0.400pt}{155.380pt}}
\put(211.0,131.0){\rule[-0.200pt]{295.825pt}{0.400pt}}
\put(1439.0,131.0){\rule[-0.200pt]{0.400pt}{155.380pt}}
\put(211.0,776.0){\rule[-0.200pt]{295.825pt}{0.400pt}}
\end{picture}}\\
\resizebox{8cm}{!}{\input{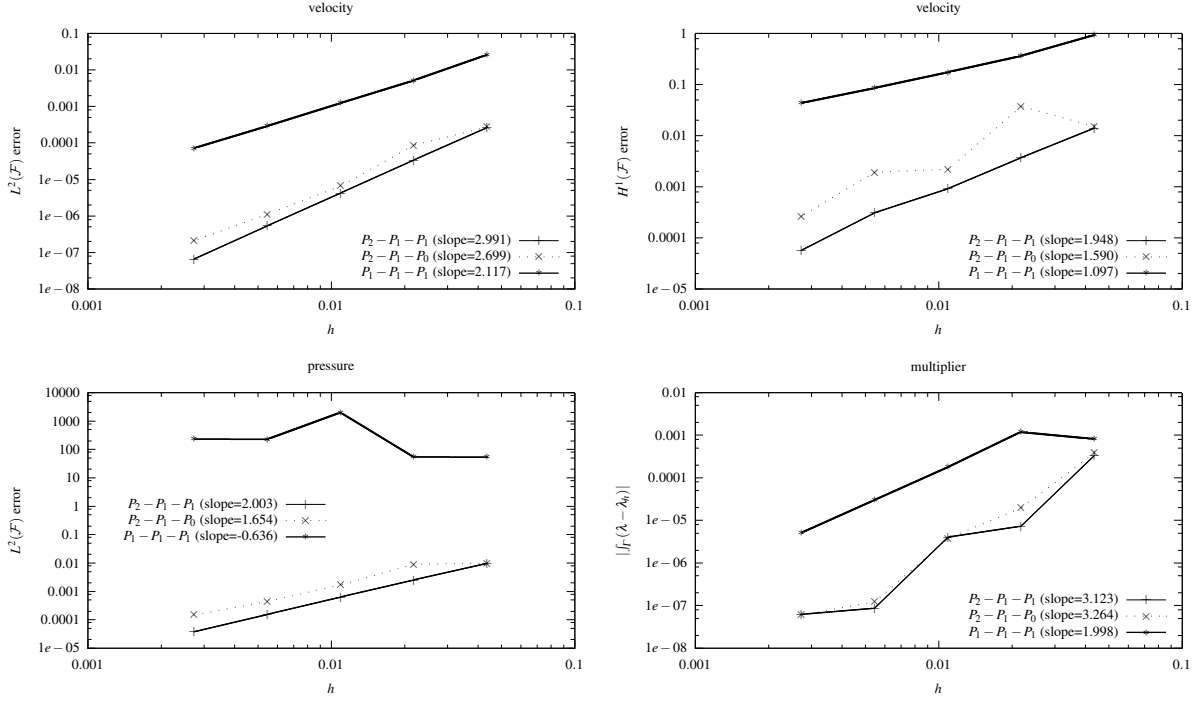}}
\resizebox{8cm}{!}{
\setlength{\unitlength}{0.240900pt}
\ifx\plotpoint\undefined\newsavebox{\plotpoint}\fi
\begin{picture}(1500,900)(0,0)
\sbox{\plotpoint}{\rule[-0.200pt]{0.400pt}{0.400pt}}%
\put(211.0,131.0){\rule[-0.200pt]{4.818pt}{0.400pt}}
\put(191,131){\makebox(0,0)[r]{$1e-08$}}
\put(1419.0,131.0){\rule[-0.200pt]{4.818pt}{0.400pt}}
\put(211.0,163.0){\rule[-0.200pt]{2.409pt}{0.400pt}}
\put(1429.0,163.0){\rule[-0.200pt]{2.409pt}{0.400pt}}
\put(211.0,206.0){\rule[-0.200pt]{2.409pt}{0.400pt}}
\put(1429.0,206.0){\rule[-0.200pt]{2.409pt}{0.400pt}}
\put(211.0,228.0){\rule[-0.200pt]{2.409pt}{0.400pt}}
\put(1429.0,228.0){\rule[-0.200pt]{2.409pt}{0.400pt}}
\put(211.0,239.0){\rule[-0.200pt]{4.818pt}{0.400pt}}
\put(191,239){\makebox(0,0)[r]{$1e-07$}}
\put(1419.0,239.0){\rule[-0.200pt]{4.818pt}{0.400pt}}
\put(211.0,271.0){\rule[-0.200pt]{2.409pt}{0.400pt}}
\put(1429.0,271.0){\rule[-0.200pt]{2.409pt}{0.400pt}}
\put(211.0,314.0){\rule[-0.200pt]{2.409pt}{0.400pt}}
\put(1429.0,314.0){\rule[-0.200pt]{2.409pt}{0.400pt}}
\put(211.0,336.0){\rule[-0.200pt]{2.409pt}{0.400pt}}
\put(1429.0,336.0){\rule[-0.200pt]{2.409pt}{0.400pt}}
\put(211.0,346.0){\rule[-0.200pt]{4.818pt}{0.400pt}}
\put(191,346){\makebox(0,0)[r]{$1e-06$}}
\put(1419.0,346.0){\rule[-0.200pt]{4.818pt}{0.400pt}}
\put(211.0,378.0){\rule[-0.200pt]{2.409pt}{0.400pt}}
\put(1429.0,378.0){\rule[-0.200pt]{2.409pt}{0.400pt}}
\put(211.0,421.0){\rule[-0.200pt]{2.409pt}{0.400pt}}
\put(1429.0,421.0){\rule[-0.200pt]{2.409pt}{0.400pt}}
\put(211.0,443.0){\rule[-0.200pt]{2.409pt}{0.400pt}}
\put(1429.0,443.0){\rule[-0.200pt]{2.409pt}{0.400pt}}
\put(211.0,454.0){\rule[-0.200pt]{4.818pt}{0.400pt}}
\put(191,454){\makebox(0,0)[r]{$1e-05$}}
\put(1419.0,454.0){\rule[-0.200pt]{4.818pt}{0.400pt}}
\put(211.0,486.0){\rule[-0.200pt]{2.409pt}{0.400pt}}
\put(1429.0,486.0){\rule[-0.200pt]{2.409pt}{0.400pt}}
\put(211.0,529.0){\rule[-0.200pt]{2.409pt}{0.400pt}}
\put(1429.0,529.0){\rule[-0.200pt]{2.409pt}{0.400pt}}
\put(211.0,551.0){\rule[-0.200pt]{2.409pt}{0.400pt}}
\put(1429.0,551.0){\rule[-0.200pt]{2.409pt}{0.400pt}}
\put(211.0,561.0){\rule[-0.200pt]{4.818pt}{0.400pt}}
\put(191,561){\makebox(0,0)[r]{$0.0001$}}
\put(1419.0,561.0){\rule[-0.200pt]{4.818pt}{0.400pt}}
\put(211.0,593.0){\rule[-0.200pt]{2.409pt}{0.400pt}}
\put(1429.0,593.0){\rule[-0.200pt]{2.409pt}{0.400pt}}
\put(211.0,636.0){\rule[-0.200pt]{2.409pt}{0.400pt}}
\put(1429.0,636.0){\rule[-0.200pt]{2.409pt}{0.400pt}}
\put(211.0,658.0){\rule[-0.200pt]{2.409pt}{0.400pt}}
\put(1429.0,658.0){\rule[-0.200pt]{2.409pt}{0.400pt}}
\put(211.0,669.0){\rule[-0.200pt]{4.818pt}{0.400pt}}
\put(191,669){\makebox(0,0)[r]{$0.001$}}
\put(1419.0,669.0){\rule[-0.200pt]{4.818pt}{0.400pt}}
\put(211.0,701.0){\rule[-0.200pt]{2.409pt}{0.400pt}}
\put(1429.0,701.0){\rule[-0.200pt]{2.409pt}{0.400pt}}
\put(211.0,744.0){\rule[-0.200pt]{2.409pt}{0.400pt}}
\put(1429.0,744.0){\rule[-0.200pt]{2.409pt}{0.400pt}}
\put(211.0,766.0){\rule[-0.200pt]{2.409pt}{0.400pt}}
\put(1429.0,766.0){\rule[-0.200pt]{2.409pt}{0.400pt}}
\put(211.0,776.0){\rule[-0.200pt]{4.818pt}{0.400pt}}
\put(191,776){\makebox(0,0)[r]{$0.01$}}
\put(1419.0,776.0){\rule[-0.200pt]{4.818pt}{0.400pt}}
\put(211.0,131.0){\rule[-0.200pt]{0.400pt}{4.818pt}}
\put(211,90){\makebox(0,0){$0.001$}}
\put(211.0,756.0){\rule[-0.200pt]{0.400pt}{4.818pt}}
\put(396.0,131.0){\rule[-0.200pt]{0.400pt}{2.409pt}}
\put(396.0,766.0){\rule[-0.200pt]{0.400pt}{2.409pt}}
\put(504.0,131.0){\rule[-0.200pt]{0.400pt}{2.409pt}}
\put(504.0,766.0){\rule[-0.200pt]{0.400pt}{2.409pt}}
\put(581.0,131.0){\rule[-0.200pt]{0.400pt}{2.409pt}}
\put(581.0,766.0){\rule[-0.200pt]{0.400pt}{2.409pt}}
\put(640.0,131.0){\rule[-0.200pt]{0.400pt}{2.409pt}}
\put(640.0,766.0){\rule[-0.200pt]{0.400pt}{2.409pt}}
\put(689.0,131.0){\rule[-0.200pt]{0.400pt}{2.409pt}}
\put(689.0,766.0){\rule[-0.200pt]{0.400pt}{2.409pt}}
\put(730.0,131.0){\rule[-0.200pt]{0.400pt}{2.409pt}}
\put(730.0,766.0){\rule[-0.200pt]{0.400pt}{2.409pt}}
\put(765.0,131.0){\rule[-0.200pt]{0.400pt}{2.409pt}}
\put(765.0,766.0){\rule[-0.200pt]{0.400pt}{2.409pt}}
\put(797.0,131.0){\rule[-0.200pt]{0.400pt}{2.409pt}}
\put(797.0,766.0){\rule[-0.200pt]{0.400pt}{2.409pt}}
\put(825.0,131.0){\rule[-0.200pt]{0.400pt}{4.818pt}}
\put(825,90){\makebox(0,0){$0.01$}}
\put(825.0,756.0){\rule[-0.200pt]{0.400pt}{4.818pt}}
\put(1010.0,131.0){\rule[-0.200pt]{0.400pt}{2.409pt}}
\put(1010.0,766.0){\rule[-0.200pt]{0.400pt}{2.409pt}}
\put(1118.0,131.0){\rule[-0.200pt]{0.400pt}{2.409pt}}
\put(1118.0,766.0){\rule[-0.200pt]{0.400pt}{2.409pt}}
\put(1195.0,131.0){\rule[-0.200pt]{0.400pt}{2.409pt}}
\put(1195.0,766.0){\rule[-0.200pt]{0.400pt}{2.409pt}}
\put(1254.0,131.0){\rule[-0.200pt]{0.400pt}{2.409pt}}
\put(1254.0,766.0){\rule[-0.200pt]{0.400pt}{2.409pt}}
\put(1303.0,131.0){\rule[-0.200pt]{0.400pt}{2.409pt}}
\put(1303.0,766.0){\rule[-0.200pt]{0.400pt}{2.409pt}}
\put(1344.0,131.0){\rule[-0.200pt]{0.400pt}{2.409pt}}
\put(1344.0,766.0){\rule[-0.200pt]{0.400pt}{2.409pt}}
\put(1379.0,131.0){\rule[-0.200pt]{0.400pt}{2.409pt}}
\put(1379.0,766.0){\rule[-0.200pt]{0.400pt}{2.409pt}}
\put(1411.0,131.0){\rule[-0.200pt]{0.400pt}{2.409pt}}
\put(1411.0,766.0){\rule[-0.200pt]{0.400pt}{2.409pt}}
\put(1439.0,131.0){\rule[-0.200pt]{0.400pt}{4.818pt}}
\put(1439,90){\makebox(0,0){$0.1$}}
\put(1439.0,756.0){\rule[-0.200pt]{0.400pt}{4.818pt}}
\put(211.0,131.0){\rule[-0.200pt]{0.400pt}{155.380pt}}
\put(211.0,131.0){\rule[-0.200pt]{295.825pt}{0.400pt}}
\put(1439.0,131.0){\rule[-0.200pt]{0.400pt}{155.380pt}}
\put(211.0,776.0){\rule[-0.200pt]{295.825pt}{0.400pt}}
\put(30,453){\makebox(0,0){\rotatebox{90}{$\left|\int_{\Gamma}(\lambda-\lambda_{h})\right|$}}}
\put(825,29){\makebox(0,0){$h$}}
\put(825,838){\makebox(0,0){multiplier}}
\put(1279,253){\makebox(0,0)[r]{$P_2-P_1-P_1$ (slope=3.123)}}
\put(1299.0,253.0){\rule[-0.200pt]{24.090pt}{0.400pt}}
\put(1218,617){\usebox{\plotpoint}}
\multiput(1215.86,615.92)(-0.519,-0.500){353}{\rule{0.516pt}{0.120pt}}
\multiput(1216.93,616.17)(-183.930,-178.000){2}{\rule{0.258pt}{0.400pt}}
\multiput(1021.67,437.92)(-3.321,-0.497){53}{\rule{2.729pt}{0.120pt}}
\multiput(1027.34,438.17)(-178.337,-28.000){2}{\rule{1.364pt}{0.400pt}}
\multiput(846.87,409.92)(-0.517,-0.500){355}{\rule{0.513pt}{0.120pt}}
\multiput(847.93,410.17)(-183.934,-179.000){2}{\rule{0.257pt}{0.400pt}}
\multiput(644.39,230.92)(-5.907,-0.494){29}{\rule{4.725pt}{0.119pt}}
\multiput(654.19,231.17)(-175.193,-16.000){2}{\rule{2.363pt}{0.400pt}}
\put(1218,617){\makebox(0,0){$+$}}
\put(1033,439){\makebox(0,0){$+$}}
\put(849,411){\makebox(0,0){$+$}}
\put(664,232){\makebox(0,0){$+$}}
\put(479,216){\makebox(0,0){$+$}}
\put(1349,253){\makebox(0,0){$+$}}
\put(1279,212){\makebox(0,0)[r]{$P_2-P_1-P_0$ (slope=3.264)}}
\multiput(1299,212)(20.756,0.000){5}{\usebox{\plotpoint}}
\put(1399,212){\usebox{\plotpoint}}
\put(1218,624){\usebox{\plotpoint}}
\multiput(1218,624)(-16.594,-12.468){12}{\usebox{\plotpoint}}
\multiput(1033,485)(-19.184,-7.924){9}{\usebox{\plotpoint}}
\multiput(849,409)(-15.657,-13.626){12}{\usebox{\plotpoint}}
\multiput(664,248)(-20.433,-3.645){9}{\usebox{\plotpoint}}
\put(479,215){\usebox{\plotpoint}}
\put(1218,624){\makebox(0,0){$\times$}}
\put(1033,485){\makebox(0,0){$\times$}}
\put(849,409){\makebox(0,0){$\times$}}
\put(664,248){\makebox(0,0){$\times$}}
\put(479,215){\makebox(0,0){$\times$}}
\put(1349,212){\makebox(0,0){$\times$}}
\sbox{\plotpoint}{\rule[-0.400pt]{0.800pt}{0.800pt}}%
\sbox{\plotpoint}{\rule[-0.200pt]{0.400pt}{0.400pt}}%
\put(1279,171){\makebox(0,0)[r]{$P_1-P_1-P_1$ (slope=1.998)}}
\sbox{\plotpoint}{\rule[-0.400pt]{0.800pt}{0.800pt}}%
\put(1299.0,171.0){\rule[-0.400pt]{24.090pt}{0.800pt}}
\put(1218,659){\usebox{\plotpoint}}
\multiput(1183.04,660.41)(-5.359,0.506){29}{\rule{8.422pt}{0.122pt}}
\multiput(1200.52,657.34)(-167.519,18.000){2}{\rule{4.211pt}{0.800pt}}
\multiput(1025.30,675.09)(-1.037,-0.501){171}{\rule{1.854pt}{0.121pt}}
\multiput(1029.15,675.34)(-180.152,-89.000){2}{\rule{0.927pt}{0.800pt}}
\multiput(840.77,586.09)(-1.119,-0.501){159}{\rule{1.983pt}{0.121pt}}
\multiput(844.88,586.34)(-180.884,-83.000){2}{\rule{0.992pt}{0.800pt}}
\multiput(655.77,503.09)(-1.119,-0.501){159}{\rule{1.983pt}{0.121pt}}
\multiput(659.88,503.34)(-180.884,-83.000){2}{\rule{0.992pt}{0.800pt}}
\put(1218,659){\makebox(0,0){$\ast$}}
\put(1033,677){\makebox(0,0){$\ast$}}
\put(849,588){\makebox(0,0){$\ast$}}
\put(664,505){\makebox(0,0){$\ast$}}
\put(479,422){\makebox(0,0){$\ast$}}
\put(1349,171){\makebox(0,0){$\ast$}}
\sbox{\plotpoint}{\rule[-0.200pt]{0.400pt}{0.400pt}}%
\put(211.0,131.0){\rule[-0.200pt]{0.400pt}{155.380pt}}
\put(211.0,131.0){\rule[-0.200pt]{295.825pt}{0.400pt}}
\put(1439.0,131.0){\rule[-0.200pt]{0.400pt}{155.380pt}}
\put(211.0,776.0){\rule[-0.200pt]{295.825pt}{0.400pt}}
\end{picture}}
\end{minipage}
\caption{Rates of convergence with Barbosa-Hughes stabilization for $\|u-u_h\|_{0,\mathcal{F}}$, 
$\|u-u_h\|_{1,\mathcal{F}}$, $\|p-p_h\|_{0,\mathcal{F}}$ and $\left|\int_{\Gamma}(\lambda-\lambda_{h})\right|$}
\label{fig:Barbosa-Hughes}
\end{figure}
\end{center}

\subsection{Methods \`a la Haslinger-Renard.}
\subsubsection{$\P_1-\P_1$ velocity-pressure spaces with Brezzi-Pitk{\"a}ranta stabilization.}

Here, the system (\ref{sys:Barbosa-Hughes}) is modified using Haslinger-Renard strategy of robust reconstruction (Definition \ref{RobRec}) for $u$ only and  adding the term $S^{\theta_0}_{pp}$ defined by
$$\left ( S^{\theta_0}_{pp} \right )_{i_pj_p} = -\theta_0 h^2 \int_{\mathcal{F}_h^e} \nabla \psi_{i_p} . \nabla \psi_{j_p}$$
The system to solve is thus
\begin{equation}
\label{sys:Brezzi-Pitkaranta}
\left (
\begin{array}{ccc}
K +S^{\gamma_0}_{\hat{u}\hat{u}} & B^T+{S^{\gamma_0}_{\hat{u}p}}^{\hspace*{-0.1cm}T}&C^T + {S^{\gamma_0}_{\hat{u} \lambda}}^{\hspace*{-0.1cm}T}\\
B + S^{\gamma_0}_{\hat{u}p} & S^{\gamma_0}_{pp} + S^{\theta_0}_{pp}& {S^{\gamma_0}_{p \lambda}}^{\hspace*{-0.1cm}T}\\
C + S^{\gamma_0}_{\hat{u} \lambda} & S^{\gamma_0}_{p \lambda} & S^{\gamma_0}_{\lambda \lambda}\\
\end{array}
\right )
\left (
\begin{array}{c}
U\\
P\\
\Lambda\\
\end{array}
\right ) = 
\left (
\begin{array}{c}
F\\
0\\
G \\
\end{array}
\right ) 
\end{equation}
where $S^{\gamma_0}_{\hat{u}\hat{u}}$, $S^{\gamma_0}_{\hat{u}p}$, $S^{\gamma_0}_{\hat{u} \lambda}$ are modified from $S^{\gamma_0}_{{u}{u}}$, $S^{\gamma_0}_{{u}p}$, $S^{\gamma_0}_{{u} \lambda}$ by incorporating the extensions of polynomials from ``good'' to ``bad'' triangles. For example,
\begin{equation}\label{SMatHat}
\left(S^{\gamma_0}_{uu}\right )_{i_uj_u} = -4 \gamma_0h \int_{\Gamma} D(\widehat{\phi_{i_u}}) n\cdot D(\widehat{\phi_{j_u}})n
\end{equation}
with $\widehat{\cdot}$ from Definition \ref{RobRec}. 

The results are reported in Fig. \ref{fig:Brezzi-Pitkaranta}.
The method is indeed robust. The optimal rates of convergence are clearly observed.
As expected, much better results are observed for the pressure in comparison with Fig. \ref{fig:Barbosa-Hughes}.
The difference between $\P_1-\P_1-\P_1$ and $\P_1-\P_1-\P_0$ variants is very small.

\vspace*{-1cm} 

\begin{center}
\begin{figure}
\hspace*{-2cm}
\begin{minipage}{20cm}
 \resizebox{8cm}{!}{
\setlength{\unitlength}{0.240900pt}
\ifx\plotpoint\undefined\newsavebox{\plotpoint}\fi
\begin{picture}(1500,900)(0,0)
\sbox{\plotpoint}{\rule[-0.200pt]{0.400pt}{0.400pt}}%
\put(211.0,131.0){\rule[-0.200pt]{4.818pt}{0.400pt}}
\put(191,131){\makebox(0,0)[r]{$1e-05$}}
\put(1419.0,131.0){\rule[-0.200pt]{4.818pt}{0.400pt}}
\put(211.0,196.0){\rule[-0.200pt]{2.409pt}{0.400pt}}
\put(1429.0,196.0){\rule[-0.200pt]{2.409pt}{0.400pt}}
\put(211.0,234.0){\rule[-0.200pt]{2.409pt}{0.400pt}}
\put(1429.0,234.0){\rule[-0.200pt]{2.409pt}{0.400pt}}
\put(211.0,260.0){\rule[-0.200pt]{2.409pt}{0.400pt}}
\put(1429.0,260.0){\rule[-0.200pt]{2.409pt}{0.400pt}}
\put(211.0,281.0){\rule[-0.200pt]{2.409pt}{0.400pt}}
\put(1429.0,281.0){\rule[-0.200pt]{2.409pt}{0.400pt}}
\put(211.0,298.0){\rule[-0.200pt]{2.409pt}{0.400pt}}
\put(1429.0,298.0){\rule[-0.200pt]{2.409pt}{0.400pt}}
\put(211.0,313.0){\rule[-0.200pt]{2.409pt}{0.400pt}}
\put(1429.0,313.0){\rule[-0.200pt]{2.409pt}{0.400pt}}
\put(211.0,325.0){\rule[-0.200pt]{2.409pt}{0.400pt}}
\put(1429.0,325.0){\rule[-0.200pt]{2.409pt}{0.400pt}}
\put(211.0,336.0){\rule[-0.200pt]{2.409pt}{0.400pt}}
\put(1429.0,336.0){\rule[-0.200pt]{2.409pt}{0.400pt}}
\put(211.0,346.0){\rule[-0.200pt]{4.818pt}{0.400pt}}
\put(191,346){\makebox(0,0)[r]{$0.0001$}}
\put(1419.0,346.0){\rule[-0.200pt]{4.818pt}{0.400pt}}
\put(211.0,411.0){\rule[-0.200pt]{2.409pt}{0.400pt}}
\put(1429.0,411.0){\rule[-0.200pt]{2.409pt}{0.400pt}}
\put(211.0,449.0){\rule[-0.200pt]{2.409pt}{0.400pt}}
\put(1429.0,449.0){\rule[-0.200pt]{2.409pt}{0.400pt}}
\put(211.0,475.0){\rule[-0.200pt]{2.409pt}{0.400pt}}
\put(1429.0,475.0){\rule[-0.200pt]{2.409pt}{0.400pt}}
\put(211.0,496.0){\rule[-0.200pt]{2.409pt}{0.400pt}}
\put(1429.0,496.0){\rule[-0.200pt]{2.409pt}{0.400pt}}
\put(211.0,513.0){\rule[-0.200pt]{2.409pt}{0.400pt}}
\put(1429.0,513.0){\rule[-0.200pt]{2.409pt}{0.400pt}}
\put(211.0,528.0){\rule[-0.200pt]{2.409pt}{0.400pt}}
\put(1429.0,528.0){\rule[-0.200pt]{2.409pt}{0.400pt}}
\put(211.0,540.0){\rule[-0.200pt]{2.409pt}{0.400pt}}
\put(1429.0,540.0){\rule[-0.200pt]{2.409pt}{0.400pt}}
\put(211.0,551.0){\rule[-0.200pt]{2.409pt}{0.400pt}}
\put(1429.0,551.0){\rule[-0.200pt]{2.409pt}{0.400pt}}
\put(211.0,561.0){\rule[-0.200pt]{4.818pt}{0.400pt}}
\put(191,561){\makebox(0,0)[r]{$0.001$}}
\put(1419.0,561.0){\rule[-0.200pt]{4.818pt}{0.400pt}}
\put(211.0,626.0){\rule[-0.200pt]{2.409pt}{0.400pt}}
\put(1429.0,626.0){\rule[-0.200pt]{2.409pt}{0.400pt}}
\put(211.0,664.0){\rule[-0.200pt]{2.409pt}{0.400pt}}
\put(1429.0,664.0){\rule[-0.200pt]{2.409pt}{0.400pt}}
\put(211.0,690.0){\rule[-0.200pt]{2.409pt}{0.400pt}}
\put(1429.0,690.0){\rule[-0.200pt]{2.409pt}{0.400pt}}
\put(211.0,711.0){\rule[-0.200pt]{2.409pt}{0.400pt}}
\put(1429.0,711.0){\rule[-0.200pt]{2.409pt}{0.400pt}}
\put(211.0,728.0){\rule[-0.200pt]{2.409pt}{0.400pt}}
\put(1429.0,728.0){\rule[-0.200pt]{2.409pt}{0.400pt}}
\put(211.0,743.0){\rule[-0.200pt]{2.409pt}{0.400pt}}
\put(1429.0,743.0){\rule[-0.200pt]{2.409pt}{0.400pt}}
\put(211.0,755.0){\rule[-0.200pt]{2.409pt}{0.400pt}}
\put(1429.0,755.0){\rule[-0.200pt]{2.409pt}{0.400pt}}
\put(211.0,766.0){\rule[-0.200pt]{2.409pt}{0.400pt}}
\put(1429.0,766.0){\rule[-0.200pt]{2.409pt}{0.400pt}}
\put(211.0,776.0){\rule[-0.200pt]{4.818pt}{0.400pt}}
\put(191,776){\makebox(0,0)[r]{$0.01$}}
\put(1419.0,776.0){\rule[-0.200pt]{4.818pt}{0.400pt}}
\put(211.0,131.0){\rule[-0.200pt]{0.400pt}{4.818pt}}
\put(211,90){\makebox(0,0){$0.001$}}
\put(211.0,756.0){\rule[-0.200pt]{0.400pt}{4.818pt}}
\put(396.0,131.0){\rule[-0.200pt]{0.400pt}{2.409pt}}
\put(396.0,766.0){\rule[-0.200pt]{0.400pt}{2.409pt}}
\put(504.0,131.0){\rule[-0.200pt]{0.400pt}{2.409pt}}
\put(504.0,766.0){\rule[-0.200pt]{0.400pt}{2.409pt}}
\put(581.0,131.0){\rule[-0.200pt]{0.400pt}{2.409pt}}
\put(581.0,766.0){\rule[-0.200pt]{0.400pt}{2.409pt}}
\put(640.0,131.0){\rule[-0.200pt]{0.400pt}{2.409pt}}
\put(640.0,766.0){\rule[-0.200pt]{0.400pt}{2.409pt}}
\put(689.0,131.0){\rule[-0.200pt]{0.400pt}{2.409pt}}
\put(689.0,766.0){\rule[-0.200pt]{0.400pt}{2.409pt}}
\put(730.0,131.0){\rule[-0.200pt]{0.400pt}{2.409pt}}
\put(730.0,766.0){\rule[-0.200pt]{0.400pt}{2.409pt}}
\put(765.0,131.0){\rule[-0.200pt]{0.400pt}{2.409pt}}
\put(765.0,766.0){\rule[-0.200pt]{0.400pt}{2.409pt}}
\put(797.0,131.0){\rule[-0.200pt]{0.400pt}{2.409pt}}
\put(797.0,766.0){\rule[-0.200pt]{0.400pt}{2.409pt}}
\put(825.0,131.0){\rule[-0.200pt]{0.400pt}{4.818pt}}
\put(825,90){\makebox(0,0){$0.01$}}
\put(825.0,756.0){\rule[-0.200pt]{0.400pt}{4.818pt}}
\put(1010.0,131.0){\rule[-0.200pt]{0.400pt}{2.409pt}}
\put(1010.0,766.0){\rule[-0.200pt]{0.400pt}{2.409pt}}
\put(1118.0,131.0){\rule[-0.200pt]{0.400pt}{2.409pt}}
\put(1118.0,766.0){\rule[-0.200pt]{0.400pt}{2.409pt}}
\put(1195.0,131.0){\rule[-0.200pt]{0.400pt}{2.409pt}}
\put(1195.0,766.0){\rule[-0.200pt]{0.400pt}{2.409pt}}
\put(1254.0,131.0){\rule[-0.200pt]{0.400pt}{2.409pt}}
\put(1254.0,766.0){\rule[-0.200pt]{0.400pt}{2.409pt}}
\put(1303.0,131.0){\rule[-0.200pt]{0.400pt}{2.409pt}}
\put(1303.0,766.0){\rule[-0.200pt]{0.400pt}{2.409pt}}
\put(1344.0,131.0){\rule[-0.200pt]{0.400pt}{2.409pt}}
\put(1344.0,766.0){\rule[-0.200pt]{0.400pt}{2.409pt}}
\put(1379.0,131.0){\rule[-0.200pt]{0.400pt}{2.409pt}}
\put(1379.0,766.0){\rule[-0.200pt]{0.400pt}{2.409pt}}
\put(1411.0,131.0){\rule[-0.200pt]{0.400pt}{2.409pt}}
\put(1411.0,766.0){\rule[-0.200pt]{0.400pt}{2.409pt}}
\put(1439.0,131.0){\rule[-0.200pt]{0.400pt}{4.818pt}}
\put(1439,90){\makebox(0,0){$0.1$}}
\put(1439.0,756.0){\rule[-0.200pt]{0.400pt}{4.818pt}}
\put(211.0,131.0){\rule[-0.200pt]{0.400pt}{155.380pt}}
\put(211.0,131.0){\rule[-0.200pt]{295.825pt}{0.400pt}}
\put(1439.0,131.0){\rule[-0.200pt]{0.400pt}{155.380pt}}
\put(211.0,776.0){\rule[-0.200pt]{295.825pt}{0.400pt}}
\put(30,453){\makebox(0,0){\rotatebox{90}{$L^2(\mathcal{F})$ error}}}
\put(825,29){\makebox(0,0){$h$}}
\put(825,838){\makebox(0,0){velocity}}
\put(1279,212){\makebox(0,0)[r]{$P_1-P_1-P_1$ (slope=1.981)}}
\put(1299.0,212.0){\rule[-0.200pt]{24.090pt}{0.400pt}}
\put(1218,752){\usebox{\plotpoint}}
\multiput(1215.15,750.92)(-0.734,-0.499){249}{\rule{0.687pt}{0.120pt}}
\multiput(1216.57,751.17)(-183.573,-126.000){2}{\rule{0.344pt}{0.400pt}}
\multiput(1030.22,624.92)(-0.713,-0.499){255}{\rule{0.671pt}{0.120pt}}
\multiput(1031.61,625.17)(-182.608,-129.000){2}{\rule{0.335pt}{0.400pt}}
\multiput(846.20,495.92)(-0.717,-0.499){255}{\rule{0.674pt}{0.120pt}}
\multiput(847.60,496.17)(-183.602,-129.000){2}{\rule{0.337pt}{0.400pt}}
\multiput(661.19,366.92)(-0.723,-0.499){253}{\rule{0.678pt}{0.120pt}}
\multiput(662.59,367.17)(-183.593,-128.000){2}{\rule{0.339pt}{0.400pt}}
\put(1218,752){\makebox(0,0){$+$}}
\put(1033,626){\makebox(0,0){$+$}}
\put(849,497){\makebox(0,0){$+$}}
\put(664,368){\makebox(0,0){$+$}}
\put(479,240){\makebox(0,0){$+$}}
\put(1349,212){\makebox(0,0){$+$}}
\put(1279,171){\makebox(0,0)[r]{$P_1-P_1-P_0$ (slope=1.982)}}
\multiput(1299,171)(20.756,0.000){5}{\usebox{\plotpoint}}
\put(1399,171){\usebox{\plotpoint}}
\put(1218,752){\usebox{\plotpoint}}
\multiput(1218,752)(-17.155,-11.684){11}{\usebox{\plotpoint}}
\multiput(1033,626)(-16.995,-11.915){11}{\usebox{\plotpoint}}
\multiput(849,497)(-17.025,-11.872){11}{\usebox{\plotpoint}}
\multiput(664,368)(-17.068,-11.809){11}{\usebox{\plotpoint}}
\put(479,240){\usebox{\plotpoint}}
\put(1218,752){\makebox(0,0){$\times$}}
\put(1033,626){\makebox(0,0){$\times$}}
\put(849,497){\makebox(0,0){$\times$}}
\put(664,368){\makebox(0,0){$\times$}}
\put(479,240){\makebox(0,0){$\times$}}
\put(1349,171){\makebox(0,0){$\times$}}
\put(211.0,131.0){\rule[-0.200pt]{0.400pt}{155.380pt}}
\put(211.0,131.0){\rule[-0.200pt]{295.825pt}{0.400pt}}
\put(1439.0,131.0){\rule[-0.200pt]{0.400pt}{155.380pt}}
\put(211.0,776.0){\rule[-0.200pt]{295.825pt}{0.400pt}}
\end{picture}}
 \resizebox{8cm}{!}{
\setlength{\unitlength}{0.240900pt}
\ifx\plotpoint\undefined\newsavebox{\plotpoint}\fi
\begin{picture}(1500,900)(0,0)
\sbox{\plotpoint}{\rule[-0.200pt]{0.400pt}{0.400pt}}%
\put(171.0,131.0){\rule[-0.200pt]{4.818pt}{0.400pt}}
\put(151,131){\makebox(0,0)[r]{$0.01$}}
\put(1419.0,131.0){\rule[-0.200pt]{4.818pt}{0.400pt}}
\put(171.0,228.0){\rule[-0.200pt]{2.409pt}{0.400pt}}
\put(1429.0,228.0){\rule[-0.200pt]{2.409pt}{0.400pt}}
\put(171.0,285.0){\rule[-0.200pt]{2.409pt}{0.400pt}}
\put(1429.0,285.0){\rule[-0.200pt]{2.409pt}{0.400pt}}
\put(171.0,325.0){\rule[-0.200pt]{2.409pt}{0.400pt}}
\put(1429.0,325.0){\rule[-0.200pt]{2.409pt}{0.400pt}}
\put(171.0,356.0){\rule[-0.200pt]{2.409pt}{0.400pt}}
\put(1429.0,356.0){\rule[-0.200pt]{2.409pt}{0.400pt}}
\put(171.0,382.0){\rule[-0.200pt]{2.409pt}{0.400pt}}
\put(1429.0,382.0){\rule[-0.200pt]{2.409pt}{0.400pt}}
\put(171.0,404.0){\rule[-0.200pt]{2.409pt}{0.400pt}}
\put(1429.0,404.0){\rule[-0.200pt]{2.409pt}{0.400pt}}
\put(171.0,422.0){\rule[-0.200pt]{2.409pt}{0.400pt}}
\put(1429.0,422.0){\rule[-0.200pt]{2.409pt}{0.400pt}}
\put(171.0,439.0){\rule[-0.200pt]{2.409pt}{0.400pt}}
\put(1429.0,439.0){\rule[-0.200pt]{2.409pt}{0.400pt}}
\put(171.0,454.0){\rule[-0.200pt]{4.818pt}{0.400pt}}
\put(151,454){\makebox(0,0)[r]{$0.1$}}
\put(1419.0,454.0){\rule[-0.200pt]{4.818pt}{0.400pt}}
\put(171.0,551.0){\rule[-0.200pt]{2.409pt}{0.400pt}}
\put(1429.0,551.0){\rule[-0.200pt]{2.409pt}{0.400pt}}
\put(171.0,607.0){\rule[-0.200pt]{2.409pt}{0.400pt}}
\put(1429.0,607.0){\rule[-0.200pt]{2.409pt}{0.400pt}}
\put(171.0,648.0){\rule[-0.200pt]{2.409pt}{0.400pt}}
\put(1429.0,648.0){\rule[-0.200pt]{2.409pt}{0.400pt}}
\put(171.0,679.0){\rule[-0.200pt]{2.409pt}{0.400pt}}
\put(1429.0,679.0){\rule[-0.200pt]{2.409pt}{0.400pt}}
\put(171.0,704.0){\rule[-0.200pt]{2.409pt}{0.400pt}}
\put(1429.0,704.0){\rule[-0.200pt]{2.409pt}{0.400pt}}
\put(171.0,726.0){\rule[-0.200pt]{2.409pt}{0.400pt}}
\put(1429.0,726.0){\rule[-0.200pt]{2.409pt}{0.400pt}}
\put(171.0,745.0){\rule[-0.200pt]{2.409pt}{0.400pt}}
\put(1429.0,745.0){\rule[-0.200pt]{2.409pt}{0.400pt}}
\put(171.0,761.0){\rule[-0.200pt]{2.409pt}{0.400pt}}
\put(1429.0,761.0){\rule[-0.200pt]{2.409pt}{0.400pt}}
\put(171.0,776.0){\rule[-0.200pt]{4.818pt}{0.400pt}}
\put(151,776){\makebox(0,0)[r]{$1$}}
\put(1419.0,776.0){\rule[-0.200pt]{4.818pt}{0.400pt}}
\put(171.0,131.0){\rule[-0.200pt]{0.400pt}{4.818pt}}
\put(171,90){\makebox(0,0){$0.001$}}
\put(171.0,756.0){\rule[-0.200pt]{0.400pt}{4.818pt}}
\put(362.0,131.0){\rule[-0.200pt]{0.400pt}{2.409pt}}
\put(362.0,766.0){\rule[-0.200pt]{0.400pt}{2.409pt}}
\put(473.0,131.0){\rule[-0.200pt]{0.400pt}{2.409pt}}
\put(473.0,766.0){\rule[-0.200pt]{0.400pt}{2.409pt}}
\put(553.0,131.0){\rule[-0.200pt]{0.400pt}{2.409pt}}
\put(553.0,766.0){\rule[-0.200pt]{0.400pt}{2.409pt}}
\put(614.0,131.0){\rule[-0.200pt]{0.400pt}{2.409pt}}
\put(614.0,766.0){\rule[-0.200pt]{0.400pt}{2.409pt}}
\put(664.0,131.0){\rule[-0.200pt]{0.400pt}{2.409pt}}
\put(664.0,766.0){\rule[-0.200pt]{0.400pt}{2.409pt}}
\put(707.0,131.0){\rule[-0.200pt]{0.400pt}{2.409pt}}
\put(707.0,766.0){\rule[-0.200pt]{0.400pt}{2.409pt}}
\put(744.0,131.0){\rule[-0.200pt]{0.400pt}{2.409pt}}
\put(744.0,766.0){\rule[-0.200pt]{0.400pt}{2.409pt}}
\put(776.0,131.0){\rule[-0.200pt]{0.400pt}{2.409pt}}
\put(776.0,766.0){\rule[-0.200pt]{0.400pt}{2.409pt}}
\put(805.0,131.0){\rule[-0.200pt]{0.400pt}{4.818pt}}
\put(805,90){\makebox(0,0){$0.01$}}
\put(805.0,756.0){\rule[-0.200pt]{0.400pt}{4.818pt}}
\put(996.0,131.0){\rule[-0.200pt]{0.400pt}{2.409pt}}
\put(996.0,766.0){\rule[-0.200pt]{0.400pt}{2.409pt}}
\put(1107.0,131.0){\rule[-0.200pt]{0.400pt}{2.409pt}}
\put(1107.0,766.0){\rule[-0.200pt]{0.400pt}{2.409pt}}
\put(1187.0,131.0){\rule[-0.200pt]{0.400pt}{2.409pt}}
\put(1187.0,766.0){\rule[-0.200pt]{0.400pt}{2.409pt}}
\put(1248.0,131.0){\rule[-0.200pt]{0.400pt}{2.409pt}}
\put(1248.0,766.0){\rule[-0.200pt]{0.400pt}{2.409pt}}
\put(1298.0,131.0){\rule[-0.200pt]{0.400pt}{2.409pt}}
\put(1298.0,766.0){\rule[-0.200pt]{0.400pt}{2.409pt}}
\put(1341.0,131.0){\rule[-0.200pt]{0.400pt}{2.409pt}}
\put(1341.0,766.0){\rule[-0.200pt]{0.400pt}{2.409pt}}
\put(1378.0,131.0){\rule[-0.200pt]{0.400pt}{2.409pt}}
\put(1378.0,766.0){\rule[-0.200pt]{0.400pt}{2.409pt}}
\put(1410.0,131.0){\rule[-0.200pt]{0.400pt}{2.409pt}}
\put(1410.0,766.0){\rule[-0.200pt]{0.400pt}{2.409pt}}
\put(1439.0,131.0){\rule[-0.200pt]{0.400pt}{4.818pt}}
\put(1439,90){\makebox(0,0){$0.1$}}
\put(1439.0,756.0){\rule[-0.200pt]{0.400pt}{4.818pt}}
\put(171.0,131.0){\rule[-0.200pt]{0.400pt}{155.380pt}}
\put(171.0,131.0){\rule[-0.200pt]{305.461pt}{0.400pt}}
\put(1439.0,131.0){\rule[-0.200pt]{0.400pt}{155.380pt}}
\put(171.0,776.0){\rule[-0.200pt]{305.461pt}{0.400pt}}
\put(30,453){\makebox(0,0){\rotatebox{90}{$H^1(\mathcal{F})$ error}}}
\put(805,29){\makebox(0,0){$h$}}
\put(805,838){\makebox(0,0){velocity}}
\put(1279,212){\makebox(0,0)[r]{$P_1-P_1-P_1$ (slope=1.025)}}
\put(1299.0,212.0){\rule[-0.200pt]{24.090pt}{0.400pt}}
\put(1211,601){\usebox{\plotpoint}}
\multiput(1207.32,599.92)(-0.986,-0.499){191}{\rule{0.888pt}{0.120pt}}
\multiput(1209.16,600.17)(-189.158,-97.000){2}{\rule{0.444pt}{0.400pt}}
\multiput(1016.51,502.92)(-0.928,-0.499){203}{\rule{0.842pt}{0.120pt}}
\multiput(1018.25,503.17)(-189.253,-103.000){2}{\rule{0.421pt}{0.400pt}}
\multiput(825.40,399.92)(-0.961,-0.499){195}{\rule{0.868pt}{0.120pt}}
\multiput(827.20,400.17)(-188.199,-99.000){2}{\rule{0.434pt}{0.400pt}}
\multiput(635.35,300.92)(-0.976,-0.499){193}{\rule{0.880pt}{0.120pt}}
\multiput(637.17,301.17)(-189.174,-98.000){2}{\rule{0.440pt}{0.400pt}}
\put(1211,601){\makebox(0,0){$+$}}
\put(1020,504){\makebox(0,0){$+$}}
\put(829,401){\makebox(0,0){$+$}}
\put(639,302){\makebox(0,0){$+$}}
\put(448,204){\makebox(0,0){$+$}}
\put(1349,212){\makebox(0,0){$+$}}
\put(1279,171){\makebox(0,0)[r]{$P_1-P_1-P_0$ (slope=1.027)}}
\multiput(1299,171)(20.756,0.000){5}{\usebox{\plotpoint}}
\put(1399,171){\usebox{\plotpoint}}
\put(1211,602){\usebox{\plotpoint}}
\multiput(1211,602)(-18.467,-9.475){11}{\usebox{\plotpoint}}
\multiput(1020,504)(-18.308,-9.777){10}{\usebox{\plotpoint}}
\multiput(829,402)(-18.367,-9.667){11}{\usebox{\plotpoint}}
\multiput(639,302)(-18.467,-9.475){10}{\usebox{\plotpoint}}
\put(448,204){\usebox{\plotpoint}}
\put(1211,602){\makebox(0,0){$\times$}}
\put(1020,504){\makebox(0,0){$\times$}}
\put(829,402){\makebox(0,0){$\times$}}
\put(639,302){\makebox(0,0){$\times$}}
\put(448,204){\makebox(0,0){$\times$}}
\put(1349,171){\makebox(0,0){$\times$}}
\put(171.0,131.0){\rule[-0.200pt]{0.400pt}{155.380pt}}
\put(171.0,131.0){\rule[-0.200pt]{305.461pt}{0.400pt}}
\put(1439.0,131.0){\rule[-0.200pt]{0.400pt}{155.380pt}}
\put(171.0,776.0){\rule[-0.200pt]{305.461pt}{0.400pt}}
\end{picture}}\\
\resizebox{8cm}{!}{
\setlength{\unitlength}{0.240900pt}
\ifx\plotpoint\undefined\newsavebox{\plotpoint}\fi
\begin{picture}(1500,900)(0,0)
\sbox{\plotpoint}{\rule[-0.200pt]{0.400pt}{0.400pt}}%
\put(191.0,131.0){\rule[-0.200pt]{4.818pt}{0.400pt}}
\put(171,131){\makebox(0,0)[r]{$0.001$}}
\put(1419.0,131.0){\rule[-0.200pt]{4.818pt}{0.400pt}}
\put(191.0,196.0){\rule[-0.200pt]{2.409pt}{0.400pt}}
\put(1429.0,196.0){\rule[-0.200pt]{2.409pt}{0.400pt}}
\put(191.0,234.0){\rule[-0.200pt]{2.409pt}{0.400pt}}
\put(1429.0,234.0){\rule[-0.200pt]{2.409pt}{0.400pt}}
\put(191.0,260.0){\rule[-0.200pt]{2.409pt}{0.400pt}}
\put(1429.0,260.0){\rule[-0.200pt]{2.409pt}{0.400pt}}
\put(191.0,281.0){\rule[-0.200pt]{2.409pt}{0.400pt}}
\put(1429.0,281.0){\rule[-0.200pt]{2.409pt}{0.400pt}}
\put(191.0,298.0){\rule[-0.200pt]{2.409pt}{0.400pt}}
\put(1429.0,298.0){\rule[-0.200pt]{2.409pt}{0.400pt}}
\put(191.0,313.0){\rule[-0.200pt]{2.409pt}{0.400pt}}
\put(1429.0,313.0){\rule[-0.200pt]{2.409pt}{0.400pt}}
\put(191.0,325.0){\rule[-0.200pt]{2.409pt}{0.400pt}}
\put(1429.0,325.0){\rule[-0.200pt]{2.409pt}{0.400pt}}
\put(191.0,336.0){\rule[-0.200pt]{2.409pt}{0.400pt}}
\put(1429.0,336.0){\rule[-0.200pt]{2.409pt}{0.400pt}}
\put(191.0,346.0){\rule[-0.200pt]{4.818pt}{0.400pt}}
\put(171,346){\makebox(0,0)[r]{$0.01$}}
\put(1419.0,346.0){\rule[-0.200pt]{4.818pt}{0.400pt}}
\put(191.0,411.0){\rule[-0.200pt]{2.409pt}{0.400pt}}
\put(1429.0,411.0){\rule[-0.200pt]{2.409pt}{0.400pt}}
\put(191.0,449.0){\rule[-0.200pt]{2.409pt}{0.400pt}}
\put(1429.0,449.0){\rule[-0.200pt]{2.409pt}{0.400pt}}
\put(191.0,475.0){\rule[-0.200pt]{2.409pt}{0.400pt}}
\put(1429.0,475.0){\rule[-0.200pt]{2.409pt}{0.400pt}}
\put(191.0,496.0){\rule[-0.200pt]{2.409pt}{0.400pt}}
\put(1429.0,496.0){\rule[-0.200pt]{2.409pt}{0.400pt}}
\put(191.0,513.0){\rule[-0.200pt]{2.409pt}{0.400pt}}
\put(1429.0,513.0){\rule[-0.200pt]{2.409pt}{0.400pt}}
\put(191.0,528.0){\rule[-0.200pt]{2.409pt}{0.400pt}}
\put(1429.0,528.0){\rule[-0.200pt]{2.409pt}{0.400pt}}
\put(191.0,540.0){\rule[-0.200pt]{2.409pt}{0.400pt}}
\put(1429.0,540.0){\rule[-0.200pt]{2.409pt}{0.400pt}}
\put(191.0,551.0){\rule[-0.200pt]{2.409pt}{0.400pt}}
\put(1429.0,551.0){\rule[-0.200pt]{2.409pt}{0.400pt}}
\put(191.0,561.0){\rule[-0.200pt]{4.818pt}{0.400pt}}
\put(171,561){\makebox(0,0)[r]{$0.1$}}
\put(1419.0,561.0){\rule[-0.200pt]{4.818pt}{0.400pt}}
\put(191.0,626.0){\rule[-0.200pt]{2.409pt}{0.400pt}}
\put(1429.0,626.0){\rule[-0.200pt]{2.409pt}{0.400pt}}
\put(191.0,664.0){\rule[-0.200pt]{2.409pt}{0.400pt}}
\put(1429.0,664.0){\rule[-0.200pt]{2.409pt}{0.400pt}}
\put(191.0,690.0){\rule[-0.200pt]{2.409pt}{0.400pt}}
\put(1429.0,690.0){\rule[-0.200pt]{2.409pt}{0.400pt}}
\put(191.0,711.0){\rule[-0.200pt]{2.409pt}{0.400pt}}
\put(1429.0,711.0){\rule[-0.200pt]{2.409pt}{0.400pt}}
\put(191.0,728.0){\rule[-0.200pt]{2.409pt}{0.400pt}}
\put(1429.0,728.0){\rule[-0.200pt]{2.409pt}{0.400pt}}
\put(191.0,743.0){\rule[-0.200pt]{2.409pt}{0.400pt}}
\put(1429.0,743.0){\rule[-0.200pt]{2.409pt}{0.400pt}}
\put(191.0,755.0){\rule[-0.200pt]{2.409pt}{0.400pt}}
\put(1429.0,755.0){\rule[-0.200pt]{2.409pt}{0.400pt}}
\put(191.0,766.0){\rule[-0.200pt]{2.409pt}{0.400pt}}
\put(1429.0,766.0){\rule[-0.200pt]{2.409pt}{0.400pt}}
\put(191.0,776.0){\rule[-0.200pt]{4.818pt}{0.400pt}}
\put(171,776){\makebox(0,0)[r]{$1$}}
\put(1419.0,776.0){\rule[-0.200pt]{4.818pt}{0.400pt}}
\put(191.0,131.0){\rule[-0.200pt]{0.400pt}{4.818pt}}
\put(191,90){\makebox(0,0){$0.001$}}
\put(191.0,756.0){\rule[-0.200pt]{0.400pt}{4.818pt}}
\put(379.0,131.0){\rule[-0.200pt]{0.400pt}{2.409pt}}
\put(379.0,766.0){\rule[-0.200pt]{0.400pt}{2.409pt}}
\put(489.0,131.0){\rule[-0.200pt]{0.400pt}{2.409pt}}
\put(489.0,766.0){\rule[-0.200pt]{0.400pt}{2.409pt}}
\put(567.0,131.0){\rule[-0.200pt]{0.400pt}{2.409pt}}
\put(567.0,766.0){\rule[-0.200pt]{0.400pt}{2.409pt}}
\put(627.0,131.0){\rule[-0.200pt]{0.400pt}{2.409pt}}
\put(627.0,766.0){\rule[-0.200pt]{0.400pt}{2.409pt}}
\put(677.0,131.0){\rule[-0.200pt]{0.400pt}{2.409pt}}
\put(677.0,766.0){\rule[-0.200pt]{0.400pt}{2.409pt}}
\put(718.0,131.0){\rule[-0.200pt]{0.400pt}{2.409pt}}
\put(718.0,766.0){\rule[-0.200pt]{0.400pt}{2.409pt}}
\put(755.0,131.0){\rule[-0.200pt]{0.400pt}{2.409pt}}
\put(755.0,766.0){\rule[-0.200pt]{0.400pt}{2.409pt}}
\put(786.0,131.0){\rule[-0.200pt]{0.400pt}{2.409pt}}
\put(786.0,766.0){\rule[-0.200pt]{0.400pt}{2.409pt}}
\put(815.0,131.0){\rule[-0.200pt]{0.400pt}{4.818pt}}
\put(815,90){\makebox(0,0){$0.01$}}
\put(815.0,756.0){\rule[-0.200pt]{0.400pt}{4.818pt}}
\put(1003.0,131.0){\rule[-0.200pt]{0.400pt}{2.409pt}}
\put(1003.0,766.0){\rule[-0.200pt]{0.400pt}{2.409pt}}
\put(1113.0,131.0){\rule[-0.200pt]{0.400pt}{2.409pt}}
\put(1113.0,766.0){\rule[-0.200pt]{0.400pt}{2.409pt}}
\put(1191.0,131.0){\rule[-0.200pt]{0.400pt}{2.409pt}}
\put(1191.0,766.0){\rule[-0.200pt]{0.400pt}{2.409pt}}
\put(1251.0,131.0){\rule[-0.200pt]{0.400pt}{2.409pt}}
\put(1251.0,766.0){\rule[-0.200pt]{0.400pt}{2.409pt}}
\put(1301.0,131.0){\rule[-0.200pt]{0.400pt}{2.409pt}}
\put(1301.0,766.0){\rule[-0.200pt]{0.400pt}{2.409pt}}
\put(1342.0,131.0){\rule[-0.200pt]{0.400pt}{2.409pt}}
\put(1342.0,766.0){\rule[-0.200pt]{0.400pt}{2.409pt}}
\put(1379.0,131.0){\rule[-0.200pt]{0.400pt}{2.409pt}}
\put(1379.0,766.0){\rule[-0.200pt]{0.400pt}{2.409pt}}
\put(1410.0,131.0){\rule[-0.200pt]{0.400pt}{2.409pt}}
\put(1410.0,766.0){\rule[-0.200pt]{0.400pt}{2.409pt}}
\put(1439.0,131.0){\rule[-0.200pt]{0.400pt}{4.818pt}}
\put(1439,90){\makebox(0,0){$0.1$}}
\put(1439.0,756.0){\rule[-0.200pt]{0.400pt}{4.818pt}}
\put(191.0,131.0){\rule[-0.200pt]{0.400pt}{155.380pt}}
\put(191.0,131.0){\rule[-0.200pt]{300.643pt}{0.400pt}}
\put(1439.0,131.0){\rule[-0.200pt]{0.400pt}{155.380pt}}
\put(191.0,776.0){\rule[-0.200pt]{300.643pt}{0.400pt}}
\put(30,453){\makebox(0,0){\rotatebox{90}{$L^2(\mathcal{F})$ error}}}
\put(815,29){\makebox(0,0){$h$}}
\put(815,838){\makebox(0,0){pressure}}
\put(1279,212){\makebox(0,0)[r]{$P_1-P_1-P_1$ (slope=1.564)}}
\put(1299.0,212.0){\rule[-0.200pt]{24.090pt}{0.400pt}}
\put(1215,729){\usebox{\plotpoint}}
\multiput(1211.85,727.92)(-0.825,-0.499){225}{\rule{0.760pt}{0.120pt}}
\multiput(1213.42,728.17)(-186.423,-114.000){2}{\rule{0.380pt}{0.400pt}}
\multiput(1023.43,613.92)(-0.951,-0.499){195}{\rule{0.860pt}{0.120pt}}
\multiput(1025.22,614.17)(-186.216,-99.000){2}{\rule{0.430pt}{0.400pt}}
\multiput(835.37,514.92)(-0.970,-0.499){191}{\rule{0.875pt}{0.120pt}}
\multiput(837.18,515.17)(-186.183,-97.000){2}{\rule{0.438pt}{0.400pt}}
\multiput(647.40,417.92)(-0.960,-0.499){193}{\rule{0.867pt}{0.120pt}}
\multiput(649.20,418.17)(-186.200,-98.000){2}{\rule{0.434pt}{0.400pt}}
\put(1215,729){\makebox(0,0){$+$}}
\put(1027,615){\makebox(0,0){$+$}}
\put(839,516){\makebox(0,0){$+$}}
\put(651,419){\makebox(0,0){$+$}}
\put(463,321){\makebox(0,0){$+$}}
\put(1349,212){\makebox(0,0){$+$}}
\put(1279,171){\makebox(0,0)[r]{$P_1-P_1-P_0$ (slope=1.561)}}
\multiput(1299,171)(20.756,0.000){5}{\usebox{\plotpoint}}
\put(1399,171){\usebox{\plotpoint}}
\put(1215,728){\usebox{\plotpoint}}
\multiput(1215,728)(-17.789,-10.693){11}{\usebox{\plotpoint}}
\multiput(1027,615)(-18.365,-9.671){10}{\usebox{\plotpoint}}
\multiput(839,516)(-18.405,-9.594){11}{\usebox{\plotpoint}}
\multiput(651,418)(-18.445,-9.517){10}{\usebox{\plotpoint}}
\put(463,321){\usebox{\plotpoint}}
\put(1215,728){\makebox(0,0){$\times$}}
\put(1027,615){\makebox(0,0){$\times$}}
\put(839,516){\makebox(0,0){$\times$}}
\put(651,418){\makebox(0,0){$\times$}}
\put(463,321){\makebox(0,0){$\times$}}
\put(1349,171){\makebox(0,0){$\times$}}
\put(191.0,131.0){\rule[-0.200pt]{0.400pt}{155.380pt}}
\put(191.0,131.0){\rule[-0.200pt]{300.643pt}{0.400pt}}
\put(1439.0,131.0){\rule[-0.200pt]{0.400pt}{155.380pt}}
\put(191.0,776.0){\rule[-0.200pt]{300.643pt}{0.400pt}}
\end{picture}}
\resizebox{8cm}{!}{
\setlength{\unitlength}{0.240900pt}
\ifx\plotpoint\undefined\newsavebox{\plotpoint}\fi
\begin{picture}(1500,900)(0,0)
\sbox{\plotpoint}{\rule[-0.200pt]{0.400pt}{0.400pt}}%
\put(211.0,131.0){\rule[-0.200pt]{4.818pt}{0.400pt}}
\put(191,131){\makebox(0,0)[r]{$1e-05$}}
\put(1419.0,131.0){\rule[-0.200pt]{4.818pt}{0.400pt}}
\put(211.0,196.0){\rule[-0.200pt]{2.409pt}{0.400pt}}
\put(1429.0,196.0){\rule[-0.200pt]{2.409pt}{0.400pt}}
\put(211.0,234.0){\rule[-0.200pt]{2.409pt}{0.400pt}}
\put(1429.0,234.0){\rule[-0.200pt]{2.409pt}{0.400pt}}
\put(211.0,260.0){\rule[-0.200pt]{2.409pt}{0.400pt}}
\put(1429.0,260.0){\rule[-0.200pt]{2.409pt}{0.400pt}}
\put(211.0,281.0){\rule[-0.200pt]{2.409pt}{0.400pt}}
\put(1429.0,281.0){\rule[-0.200pt]{2.409pt}{0.400pt}}
\put(211.0,298.0){\rule[-0.200pt]{2.409pt}{0.400pt}}
\put(1429.0,298.0){\rule[-0.200pt]{2.409pt}{0.400pt}}
\put(211.0,313.0){\rule[-0.200pt]{2.409pt}{0.400pt}}
\put(1429.0,313.0){\rule[-0.200pt]{2.409pt}{0.400pt}}
\put(211.0,325.0){\rule[-0.200pt]{2.409pt}{0.400pt}}
\put(1429.0,325.0){\rule[-0.200pt]{2.409pt}{0.400pt}}
\put(211.0,336.0){\rule[-0.200pt]{2.409pt}{0.400pt}}
\put(1429.0,336.0){\rule[-0.200pt]{2.409pt}{0.400pt}}
\put(211.0,346.0){\rule[-0.200pt]{4.818pt}{0.400pt}}
\put(191,346){\makebox(0,0)[r]{$0.0001$}}
\put(1419.0,346.0){\rule[-0.200pt]{4.818pt}{0.400pt}}
\put(211.0,411.0){\rule[-0.200pt]{2.409pt}{0.400pt}}
\put(1429.0,411.0){\rule[-0.200pt]{2.409pt}{0.400pt}}
\put(211.0,449.0){\rule[-0.200pt]{2.409pt}{0.400pt}}
\put(1429.0,449.0){\rule[-0.200pt]{2.409pt}{0.400pt}}
\put(211.0,475.0){\rule[-0.200pt]{2.409pt}{0.400pt}}
\put(1429.0,475.0){\rule[-0.200pt]{2.409pt}{0.400pt}}
\put(211.0,496.0){\rule[-0.200pt]{2.409pt}{0.400pt}}
\put(1429.0,496.0){\rule[-0.200pt]{2.409pt}{0.400pt}}
\put(211.0,513.0){\rule[-0.200pt]{2.409pt}{0.400pt}}
\put(1429.0,513.0){\rule[-0.200pt]{2.409pt}{0.400pt}}
\put(211.0,528.0){\rule[-0.200pt]{2.409pt}{0.400pt}}
\put(1429.0,528.0){\rule[-0.200pt]{2.409pt}{0.400pt}}
\put(211.0,540.0){\rule[-0.200pt]{2.409pt}{0.400pt}}
\put(1429.0,540.0){\rule[-0.200pt]{2.409pt}{0.400pt}}
\put(211.0,551.0){\rule[-0.200pt]{2.409pt}{0.400pt}}
\put(1429.0,551.0){\rule[-0.200pt]{2.409pt}{0.400pt}}
\put(211.0,561.0){\rule[-0.200pt]{4.818pt}{0.400pt}}
\put(191,561){\makebox(0,0)[r]{$0.001$}}
\put(1419.0,561.0){\rule[-0.200pt]{4.818pt}{0.400pt}}
\put(211.0,626.0){\rule[-0.200pt]{2.409pt}{0.400pt}}
\put(1429.0,626.0){\rule[-0.200pt]{2.409pt}{0.400pt}}
\put(211.0,664.0){\rule[-0.200pt]{2.409pt}{0.400pt}}
\put(1429.0,664.0){\rule[-0.200pt]{2.409pt}{0.400pt}}
\put(211.0,690.0){\rule[-0.200pt]{2.409pt}{0.400pt}}
\put(1429.0,690.0){\rule[-0.200pt]{2.409pt}{0.400pt}}
\put(211.0,711.0){\rule[-0.200pt]{2.409pt}{0.400pt}}
\put(1429.0,711.0){\rule[-0.200pt]{2.409pt}{0.400pt}}
\put(211.0,728.0){\rule[-0.200pt]{2.409pt}{0.400pt}}
\put(1429.0,728.0){\rule[-0.200pt]{2.409pt}{0.400pt}}
\put(211.0,743.0){\rule[-0.200pt]{2.409pt}{0.400pt}}
\put(1429.0,743.0){\rule[-0.200pt]{2.409pt}{0.400pt}}
\put(211.0,755.0){\rule[-0.200pt]{2.409pt}{0.400pt}}
\put(1429.0,755.0){\rule[-0.200pt]{2.409pt}{0.400pt}}
\put(211.0,766.0){\rule[-0.200pt]{2.409pt}{0.400pt}}
\put(1429.0,766.0){\rule[-0.200pt]{2.409pt}{0.400pt}}
\put(211.0,776.0){\rule[-0.200pt]{4.818pt}{0.400pt}}
\put(191,776){\makebox(0,0)[r]{$0.01$}}
\put(1419.0,776.0){\rule[-0.200pt]{4.818pt}{0.400pt}}
\put(211.0,131.0){\rule[-0.200pt]{0.400pt}{4.818pt}}
\put(211,90){\makebox(0,0){$0.001$}}
\put(211.0,756.0){\rule[-0.200pt]{0.400pt}{4.818pt}}
\put(396.0,131.0){\rule[-0.200pt]{0.400pt}{2.409pt}}
\put(396.0,766.0){\rule[-0.200pt]{0.400pt}{2.409pt}}
\put(504.0,131.0){\rule[-0.200pt]{0.400pt}{2.409pt}}
\put(504.0,766.0){\rule[-0.200pt]{0.400pt}{2.409pt}}
\put(581.0,131.0){\rule[-0.200pt]{0.400pt}{2.409pt}}
\put(581.0,766.0){\rule[-0.200pt]{0.400pt}{2.409pt}}
\put(640.0,131.0){\rule[-0.200pt]{0.400pt}{2.409pt}}
\put(640.0,766.0){\rule[-0.200pt]{0.400pt}{2.409pt}}
\put(689.0,131.0){\rule[-0.200pt]{0.400pt}{2.409pt}}
\put(689.0,766.0){\rule[-0.200pt]{0.400pt}{2.409pt}}
\put(730.0,131.0){\rule[-0.200pt]{0.400pt}{2.409pt}}
\put(730.0,766.0){\rule[-0.200pt]{0.400pt}{2.409pt}}
\put(765.0,131.0){\rule[-0.200pt]{0.400pt}{2.409pt}}
\put(765.0,766.0){\rule[-0.200pt]{0.400pt}{2.409pt}}
\put(797.0,131.0){\rule[-0.200pt]{0.400pt}{2.409pt}}
\put(797.0,766.0){\rule[-0.200pt]{0.400pt}{2.409pt}}
\put(825.0,131.0){\rule[-0.200pt]{0.400pt}{4.818pt}}
\put(825,90){\makebox(0,0){$0.01$}}
\put(825.0,756.0){\rule[-0.200pt]{0.400pt}{4.818pt}}
\put(1010.0,131.0){\rule[-0.200pt]{0.400pt}{2.409pt}}
\put(1010.0,766.0){\rule[-0.200pt]{0.400pt}{2.409pt}}
\put(1118.0,131.0){\rule[-0.200pt]{0.400pt}{2.409pt}}
\put(1118.0,766.0){\rule[-0.200pt]{0.400pt}{2.409pt}}
\put(1195.0,131.0){\rule[-0.200pt]{0.400pt}{2.409pt}}
\put(1195.0,766.0){\rule[-0.200pt]{0.400pt}{2.409pt}}
\put(1254.0,131.0){\rule[-0.200pt]{0.400pt}{2.409pt}}
\put(1254.0,766.0){\rule[-0.200pt]{0.400pt}{2.409pt}}
\put(1303.0,131.0){\rule[-0.200pt]{0.400pt}{2.409pt}}
\put(1303.0,766.0){\rule[-0.200pt]{0.400pt}{2.409pt}}
\put(1344.0,131.0){\rule[-0.200pt]{0.400pt}{2.409pt}}
\put(1344.0,766.0){\rule[-0.200pt]{0.400pt}{2.409pt}}
\put(1379.0,131.0){\rule[-0.200pt]{0.400pt}{2.409pt}}
\put(1379.0,766.0){\rule[-0.200pt]{0.400pt}{2.409pt}}
\put(1411.0,131.0){\rule[-0.200pt]{0.400pt}{2.409pt}}
\put(1411.0,766.0){\rule[-0.200pt]{0.400pt}{2.409pt}}
\put(1439.0,131.0){\rule[-0.200pt]{0.400pt}{4.818pt}}
\put(1439,90){\makebox(0,0){$0.1$}}
\put(1439.0,756.0){\rule[-0.200pt]{0.400pt}{4.818pt}}
\put(211.0,131.0){\rule[-0.200pt]{0.400pt}{155.380pt}}
\put(211.0,131.0){\rule[-0.200pt]{295.825pt}{0.400pt}}
\put(1439.0,131.0){\rule[-0.200pt]{0.400pt}{155.380pt}}
\put(211.0,776.0){\rule[-0.200pt]{295.825pt}{0.400pt}}
\put(30,453){\makebox(0,0){\rotatebox{90}{$\left|\int_{\Gamma}(\lambda-\lambda_{h})\right|$}}}
\put(825,29){\makebox(0,0){$h$}}
\put(825,838){\makebox(0,0){multiplier}}
\put(1279,212){\makebox(0,0)[r]{$P_1-P_1-P_1$ (slope=2.131)}}
\put(1299.0,212.0){\rule[-0.200pt]{24.090pt}{0.400pt}}
\put(1218,680){\usebox{\plotpoint}}
\multiput(1212.79,678.92)(-1.450,-0.499){125}{\rule{1.256pt}{0.120pt}}
\multiput(1215.39,679.17)(-182.393,-64.000){2}{\rule{0.628pt}{0.400pt}}
\multiput(1030.72,614.92)(-0.561,-0.500){325}{\rule{0.549pt}{0.120pt}}
\multiput(1031.86,615.17)(-182.861,-164.000){2}{\rule{0.274pt}{0.400pt}}
\multiput(846.60,450.92)(-0.597,-0.499){307}{\rule{0.577pt}{0.120pt}}
\multiput(847.80,451.17)(-183.802,-155.000){2}{\rule{0.289pt}{0.400pt}}
\multiput(661.51,295.92)(-0.625,-0.499){293}{\rule{0.600pt}{0.120pt}}
\multiput(662.75,296.17)(-183.755,-148.000){2}{\rule{0.300pt}{0.400pt}}
\put(1218,680){\makebox(0,0){$+$}}
\put(1033,616){\makebox(0,0){$+$}}
\put(849,452){\makebox(0,0){$+$}}
\put(664,297){\makebox(0,0){$+$}}
\put(479,149){\makebox(0,0){$+$}}
\put(1349,212){\makebox(0,0){$+$}}
\put(1279,171){\makebox(0,0)[r]{$P_1-P_1-P_0$ (slope=2.140)}}
\multiput(1299,171)(20.756,0.000){5}{\usebox{\plotpoint}}
\put(1399,171){\usebox{\plotpoint}}
\put(1218,681){\usebox{\plotpoint}}
\multiput(1218,681)(-19.712,-6.500){10}{\usebox{\plotpoint}}
\multiput(1033,620)(-15.328,-13.995){12}{\usebox{\plotpoint}}
\multiput(849,452)(-15.910,-13.330){12}{\usebox{\plotpoint}}
\multiput(664,297)(-16.250,-12.912){11}{\usebox{\plotpoint}}
\put(479,150){\usebox{\plotpoint}}
\put(1218,681){\makebox(0,0){$\times$}}
\put(1033,620){\makebox(0,0){$\times$}}
\put(849,452){\makebox(0,0){$\times$}}
\put(664,297){\makebox(0,0){$\times$}}
\put(479,150){\makebox(0,0){$\times$}}
\put(1349,171){\makebox(0,0){$\times$}}
\put(211.0,131.0){\rule[-0.200pt]{0.400pt}{155.380pt}}
\put(211.0,131.0){\rule[-0.200pt]{295.825pt}{0.400pt}}
\put(1439.0,131.0){\rule[-0.200pt]{0.400pt}{155.380pt}}
\put(211.0,776.0){\rule[-0.200pt]{295.825pt}{0.400pt}}
\end{picture}}
\end{minipage}
\caption{Rates of convergence with Brezzi-Pitkaranta stabilization for $\|u-u_h\|_{0,\mathcal{F}}$, 
$\|u-u_h\|_{1,\mathcal{F}}$, $\|p-p_h\|_{0,\mathcal{F}}$ and $\left|\int_{\Gamma}(\lambda-\lambda_{h})\right|$ }
\label{fig:Brezzi-Pitkaranta}
\end{figure}
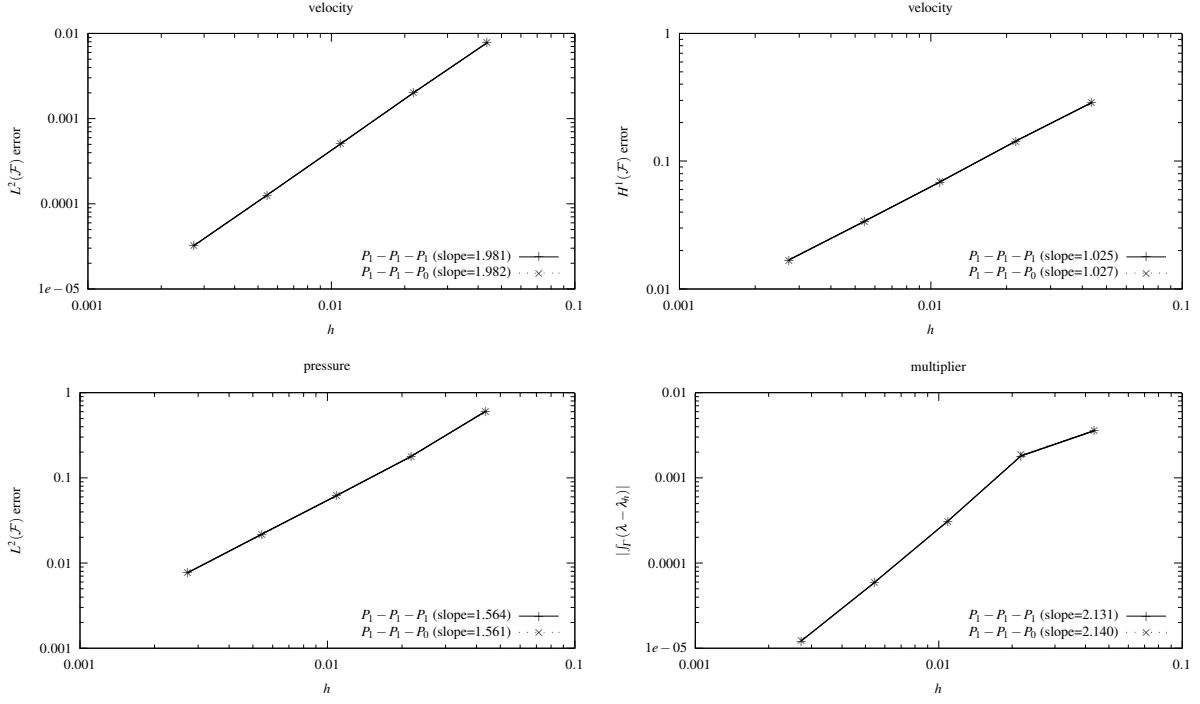
\end{center}

\subsubsection{$\P_1-\P_0$ velocity-pressure spaces with interior penalty stabilization.}
The system to solve is the same as (\ref{sys:Brezzi-Pitkaranta}) but $S^{\theta_0}_{pp}$ is replaced by $S^{\theta_0}_{[p][p]}$ with 
$$ \left (S^{\theta_0}_{[p][p]} \right )_{i_pj_p} =  - \theta_0 h  \sum_{E \in \mathcal{E}_h^e} \int_{E} [\psi_{i_p}] [\psi_{j_p}]$$
The system is thus given by
\begin{equation}
\label{sys:Interior-Penalty}
\left (
\begin{array}{ccc}
K +S^{\gamma_0}_{\hat{u}\hat{u}} & B^T+{S^{\gamma_0}_{\hat{u}p}}^{\hspace*{-0.1cm}T}&C^T + {S^{\gamma_0}_{\hat{u} \lambda}}^{\hspace*{-0.1cm}T}\\
B + S^{\gamma_0}_{\hat{u}p} & S^{\gamma_0}_{pp} +S^{\theta_0}_{[p][p]} & {S^{\gamma_0}_{p \lambda}}^{\hspace*{-0.1cm}T}\\
C + S^{\gamma_0}_{\hat{u} \lambda} & S^{\gamma_0}_{p \lambda} & S^{\gamma_0}_{\lambda \lambda}\\
\end{array}
\right )
\left (
\begin{array}{c}
U\\
P\\
\Lambda\\
\end{array}
\right ) = 
\left (
\begin{array}{c}
F\\
0\\
G \\
\end{array}
\right ) 
\end{equation}
The results are reported in Fig. \ref{fig:Interior-Penalty} and are close to those in Fig. \ref{fig:Brezzi-Pitkaranta} except for the pressure which is less accurate.
Here again, the difference between $\P_1-\P_0-\P_1$ and $\P_1-\P_0-\P_0$ is very small.
\begin{center}
\begin{figure}
\hspace*{-2cm}
\begin{minipage}{20cm}
 \resizebox{8cm}{!}{
\setlength{\unitlength}{0.240900pt}
\ifx\plotpoint\undefined\newsavebox{\plotpoint}\fi
\begin{picture}(1500,900)(0,0)
\sbox{\plotpoint}{\rule[-0.200pt]{0.400pt}{0.400pt}}%
\put(211.0,131.0){\rule[-0.200pt]{4.818pt}{0.400pt}}
\put(191,131){\makebox(0,0)[r]{$1e-05$}}
\put(1419.0,131.0){\rule[-0.200pt]{4.818pt}{0.400pt}}
\put(211.0,196.0){\rule[-0.200pt]{2.409pt}{0.400pt}}
\put(1429.0,196.0){\rule[-0.200pt]{2.409pt}{0.400pt}}
\put(211.0,234.0){\rule[-0.200pt]{2.409pt}{0.400pt}}
\put(1429.0,234.0){\rule[-0.200pt]{2.409pt}{0.400pt}}
\put(211.0,260.0){\rule[-0.200pt]{2.409pt}{0.400pt}}
\put(1429.0,260.0){\rule[-0.200pt]{2.409pt}{0.400pt}}
\put(211.0,281.0){\rule[-0.200pt]{2.409pt}{0.400pt}}
\put(1429.0,281.0){\rule[-0.200pt]{2.409pt}{0.400pt}}
\put(211.0,298.0){\rule[-0.200pt]{2.409pt}{0.400pt}}
\put(1429.0,298.0){\rule[-0.200pt]{2.409pt}{0.400pt}}
\put(211.0,313.0){\rule[-0.200pt]{2.409pt}{0.400pt}}
\put(1429.0,313.0){\rule[-0.200pt]{2.409pt}{0.400pt}}
\put(211.0,325.0){\rule[-0.200pt]{2.409pt}{0.400pt}}
\put(1429.0,325.0){\rule[-0.200pt]{2.409pt}{0.400pt}}
\put(211.0,336.0){\rule[-0.200pt]{2.409pt}{0.400pt}}
\put(1429.0,336.0){\rule[-0.200pt]{2.409pt}{0.400pt}}
\put(211.0,346.0){\rule[-0.200pt]{4.818pt}{0.400pt}}
\put(191,346){\makebox(0,0)[r]{$0.0001$}}
\put(1419.0,346.0){\rule[-0.200pt]{4.818pt}{0.400pt}}
\put(211.0,411.0){\rule[-0.200pt]{2.409pt}{0.400pt}}
\put(1429.0,411.0){\rule[-0.200pt]{2.409pt}{0.400pt}}
\put(211.0,449.0){\rule[-0.200pt]{2.409pt}{0.400pt}}
\put(1429.0,449.0){\rule[-0.200pt]{2.409pt}{0.400pt}}
\put(211.0,475.0){\rule[-0.200pt]{2.409pt}{0.400pt}}
\put(1429.0,475.0){\rule[-0.200pt]{2.409pt}{0.400pt}}
\put(211.0,496.0){\rule[-0.200pt]{2.409pt}{0.400pt}}
\put(1429.0,496.0){\rule[-0.200pt]{2.409pt}{0.400pt}}
\put(211.0,513.0){\rule[-0.200pt]{2.409pt}{0.400pt}}
\put(1429.0,513.0){\rule[-0.200pt]{2.409pt}{0.400pt}}
\put(211.0,528.0){\rule[-0.200pt]{2.409pt}{0.400pt}}
\put(1429.0,528.0){\rule[-0.200pt]{2.409pt}{0.400pt}}
\put(211.0,540.0){\rule[-0.200pt]{2.409pt}{0.400pt}}
\put(1429.0,540.0){\rule[-0.200pt]{2.409pt}{0.400pt}}
\put(211.0,551.0){\rule[-0.200pt]{2.409pt}{0.400pt}}
\put(1429.0,551.0){\rule[-0.200pt]{2.409pt}{0.400pt}}
\put(211.0,561.0){\rule[-0.200pt]{4.818pt}{0.400pt}}
\put(191,561){\makebox(0,0)[r]{$0.001$}}
\put(1419.0,561.0){\rule[-0.200pt]{4.818pt}{0.400pt}}
\put(211.0,626.0){\rule[-0.200pt]{2.409pt}{0.400pt}}
\put(1429.0,626.0){\rule[-0.200pt]{2.409pt}{0.400pt}}
\put(211.0,664.0){\rule[-0.200pt]{2.409pt}{0.400pt}}
\put(1429.0,664.0){\rule[-0.200pt]{2.409pt}{0.400pt}}
\put(211.0,690.0){\rule[-0.200pt]{2.409pt}{0.400pt}}
\put(1429.0,690.0){\rule[-0.200pt]{2.409pt}{0.400pt}}
\put(211.0,711.0){\rule[-0.200pt]{2.409pt}{0.400pt}}
\put(1429.0,711.0){\rule[-0.200pt]{2.409pt}{0.400pt}}
\put(211.0,728.0){\rule[-0.200pt]{2.409pt}{0.400pt}}
\put(1429.0,728.0){\rule[-0.200pt]{2.409pt}{0.400pt}}
\put(211.0,743.0){\rule[-0.200pt]{2.409pt}{0.400pt}}
\put(1429.0,743.0){\rule[-0.200pt]{2.409pt}{0.400pt}}
\put(211.0,755.0){\rule[-0.200pt]{2.409pt}{0.400pt}}
\put(1429.0,755.0){\rule[-0.200pt]{2.409pt}{0.400pt}}
\put(211.0,766.0){\rule[-0.200pt]{2.409pt}{0.400pt}}
\put(1429.0,766.0){\rule[-0.200pt]{2.409pt}{0.400pt}}
\put(211.0,776.0){\rule[-0.200pt]{4.818pt}{0.400pt}}
\put(191,776){\makebox(0,0)[r]{$0.01$}}
\put(1419.0,776.0){\rule[-0.200pt]{4.818pt}{0.400pt}}
\put(211.0,131.0){\rule[-0.200pt]{0.400pt}{4.818pt}}
\put(211,90){\makebox(0,0){$0.001$}}
\put(211.0,756.0){\rule[-0.200pt]{0.400pt}{4.818pt}}
\put(396.0,131.0){\rule[-0.200pt]{0.400pt}{2.409pt}}
\put(396.0,766.0){\rule[-0.200pt]{0.400pt}{2.409pt}}
\put(504.0,131.0){\rule[-0.200pt]{0.400pt}{2.409pt}}
\put(504.0,766.0){\rule[-0.200pt]{0.400pt}{2.409pt}}
\put(581.0,131.0){\rule[-0.200pt]{0.400pt}{2.409pt}}
\put(581.0,766.0){\rule[-0.200pt]{0.400pt}{2.409pt}}
\put(640.0,131.0){\rule[-0.200pt]{0.400pt}{2.409pt}}
\put(640.0,766.0){\rule[-0.200pt]{0.400pt}{2.409pt}}
\put(689.0,131.0){\rule[-0.200pt]{0.400pt}{2.409pt}}
\put(689.0,766.0){\rule[-0.200pt]{0.400pt}{2.409pt}}
\put(730.0,131.0){\rule[-0.200pt]{0.400pt}{2.409pt}}
\put(730.0,766.0){\rule[-0.200pt]{0.400pt}{2.409pt}}
\put(765.0,131.0){\rule[-0.200pt]{0.400pt}{2.409pt}}
\put(765.0,766.0){\rule[-0.200pt]{0.400pt}{2.409pt}}
\put(797.0,131.0){\rule[-0.200pt]{0.400pt}{2.409pt}}
\put(797.0,766.0){\rule[-0.200pt]{0.400pt}{2.409pt}}
\put(825.0,131.0){\rule[-0.200pt]{0.400pt}{4.818pt}}
\put(825,90){\makebox(0,0){$0.01$}}
\put(825.0,756.0){\rule[-0.200pt]{0.400pt}{4.818pt}}
\put(1010.0,131.0){\rule[-0.200pt]{0.400pt}{2.409pt}}
\put(1010.0,766.0){\rule[-0.200pt]{0.400pt}{2.409pt}}
\put(1118.0,131.0){\rule[-0.200pt]{0.400pt}{2.409pt}}
\put(1118.0,766.0){\rule[-0.200pt]{0.400pt}{2.409pt}}
\put(1195.0,131.0){\rule[-0.200pt]{0.400pt}{2.409pt}}
\put(1195.0,766.0){\rule[-0.200pt]{0.400pt}{2.409pt}}
\put(1254.0,131.0){\rule[-0.200pt]{0.400pt}{2.409pt}}
\put(1254.0,766.0){\rule[-0.200pt]{0.400pt}{2.409pt}}
\put(1303.0,131.0){\rule[-0.200pt]{0.400pt}{2.409pt}}
\put(1303.0,766.0){\rule[-0.200pt]{0.400pt}{2.409pt}}
\put(1344.0,131.0){\rule[-0.200pt]{0.400pt}{2.409pt}}
\put(1344.0,766.0){\rule[-0.200pt]{0.400pt}{2.409pt}}
\put(1379.0,131.0){\rule[-0.200pt]{0.400pt}{2.409pt}}
\put(1379.0,766.0){\rule[-0.200pt]{0.400pt}{2.409pt}}
\put(1411.0,131.0){\rule[-0.200pt]{0.400pt}{2.409pt}}
\put(1411.0,766.0){\rule[-0.200pt]{0.400pt}{2.409pt}}
\put(1439.0,131.0){\rule[-0.200pt]{0.400pt}{4.818pt}}
\put(1439,90){\makebox(0,0){$0.1$}}
\put(1439.0,756.0){\rule[-0.200pt]{0.400pt}{4.818pt}}
\put(211.0,131.0){\rule[-0.200pt]{0.400pt}{155.380pt}}
\put(211.0,131.0){\rule[-0.200pt]{295.825pt}{0.400pt}}
\put(1439.0,131.0){\rule[-0.200pt]{0.400pt}{155.380pt}}
\put(211.0,776.0){\rule[-0.200pt]{295.825pt}{0.400pt}}
\put(30,453){\makebox(0,0){\rotatebox{90}{$L^2(\mathcal{F})$ error}}}
\put(825,29){\makebox(0,0){$h$}}
\put(825,838){\makebox(0,0){velocity}}
\put(1279,212){\makebox(0,0)[r]{$P_1-P_0-P_1$ (slope=1.989)}}
\put(1299.0,212.0){\rule[-0.200pt]{24.090pt}{0.400pt}}
\put(1218,771){\usebox{\plotpoint}}
\multiput(1215.17,769.92)(-0.729,-0.499){251}{\rule{0.683pt}{0.120pt}}
\multiput(1216.58,770.17)(-183.583,-127.000){2}{\rule{0.341pt}{0.400pt}}
\multiput(1030.23,642.92)(-0.708,-0.499){257}{\rule{0.666pt}{0.120pt}}
\multiput(1031.62,643.17)(-182.617,-130.000){2}{\rule{0.333pt}{0.400pt}}
\multiput(846.19,512.92)(-0.723,-0.499){253}{\rule{0.678pt}{0.120pt}}
\multiput(847.59,513.17)(-183.593,-128.000){2}{\rule{0.339pt}{0.400pt}}
\multiput(661.20,384.92)(-0.717,-0.499){255}{\rule{0.674pt}{0.120pt}}
\multiput(662.60,385.17)(-183.602,-129.000){2}{\rule{0.337pt}{0.400pt}}
\put(1218,771){\makebox(0,0){$+$}}
\put(1033,644){\makebox(0,0){$+$}}
\put(849,514){\makebox(0,0){$+$}}
\put(664,386){\makebox(0,0){$+$}}
\put(479,257){\makebox(0,0){$+$}}
\put(1349,212){\makebox(0,0){$+$}}
\put(1279,171){\makebox(0,0)[r]{$P_1-P_0-P_0$ (slope=1.990)}}
\multiput(1299,171)(20.756,0.000){5}{\usebox{\plotpoint}}
\put(1399,171){\usebox{\plotpoint}}
\put(1218,772){\usebox{\plotpoint}}
\multiput(1218,772)(-17.068,-11.809){11}{\usebox{\plotpoint}}
\multiput(1033,644)(-16.995,-11.915){11}{\usebox{\plotpoint}}
\multiput(849,515)(-17.025,-11.872){11}{\usebox{\plotpoint}}
\multiput(664,386)(-17.025,-11.872){11}{\usebox{\plotpoint}}
\put(479,257){\usebox{\plotpoint}}
\put(1218,772){\makebox(0,0){$\times$}}
\put(1033,644){\makebox(0,0){$\times$}}
\put(849,515){\makebox(0,0){$\times$}}
\put(664,386){\makebox(0,0){$\times$}}
\put(479,257){\makebox(0,0){$\times$}}
\put(1349,171){\makebox(0,0){$\times$}}
\put(211.0,131.0){\rule[-0.200pt]{0.400pt}{155.380pt}}
\put(211.0,131.0){\rule[-0.200pt]{295.825pt}{0.400pt}}
\put(1439.0,131.0){\rule[-0.200pt]{0.400pt}{155.380pt}}
\put(211.0,776.0){\rule[-0.200pt]{295.825pt}{0.400pt}}
\end{picture}}
 \resizebox{8cm}{!}{
\setlength{\unitlength}{0.240900pt}
\ifx\plotpoint\undefined\newsavebox{\plotpoint}\fi
\begin{picture}(1500,900)(0,0)
\sbox{\plotpoint}{\rule[-0.200pt]{0.400pt}{0.400pt}}%
\put(171.0,131.0){\rule[-0.200pt]{4.818pt}{0.400pt}}
\put(151,131){\makebox(0,0)[r]{$0.01$}}
\put(1419.0,131.0){\rule[-0.200pt]{4.818pt}{0.400pt}}
\put(171.0,228.0){\rule[-0.200pt]{2.409pt}{0.400pt}}
\put(1429.0,228.0){\rule[-0.200pt]{2.409pt}{0.400pt}}
\put(171.0,285.0){\rule[-0.200pt]{2.409pt}{0.400pt}}
\put(1429.0,285.0){\rule[-0.200pt]{2.409pt}{0.400pt}}
\put(171.0,325.0){\rule[-0.200pt]{2.409pt}{0.400pt}}
\put(1429.0,325.0){\rule[-0.200pt]{2.409pt}{0.400pt}}
\put(171.0,356.0){\rule[-0.200pt]{2.409pt}{0.400pt}}
\put(1429.0,356.0){\rule[-0.200pt]{2.409pt}{0.400pt}}
\put(171.0,382.0){\rule[-0.200pt]{2.409pt}{0.400pt}}
\put(1429.0,382.0){\rule[-0.200pt]{2.409pt}{0.400pt}}
\put(171.0,404.0){\rule[-0.200pt]{2.409pt}{0.400pt}}
\put(1429.0,404.0){\rule[-0.200pt]{2.409pt}{0.400pt}}
\put(171.0,422.0){\rule[-0.200pt]{2.409pt}{0.400pt}}
\put(1429.0,422.0){\rule[-0.200pt]{2.409pt}{0.400pt}}
\put(171.0,439.0){\rule[-0.200pt]{2.409pt}{0.400pt}}
\put(1429.0,439.0){\rule[-0.200pt]{2.409pt}{0.400pt}}
\put(171.0,454.0){\rule[-0.200pt]{4.818pt}{0.400pt}}
\put(151,454){\makebox(0,0)[r]{$0.1$}}
\put(1419.0,454.0){\rule[-0.200pt]{4.818pt}{0.400pt}}
\put(171.0,551.0){\rule[-0.200pt]{2.409pt}{0.400pt}}
\put(1429.0,551.0){\rule[-0.200pt]{2.409pt}{0.400pt}}
\put(171.0,607.0){\rule[-0.200pt]{2.409pt}{0.400pt}}
\put(1429.0,607.0){\rule[-0.200pt]{2.409pt}{0.400pt}}
\put(171.0,648.0){\rule[-0.200pt]{2.409pt}{0.400pt}}
\put(1429.0,648.0){\rule[-0.200pt]{2.409pt}{0.400pt}}
\put(171.0,679.0){\rule[-0.200pt]{2.409pt}{0.400pt}}
\put(1429.0,679.0){\rule[-0.200pt]{2.409pt}{0.400pt}}
\put(171.0,704.0){\rule[-0.200pt]{2.409pt}{0.400pt}}
\put(1429.0,704.0){\rule[-0.200pt]{2.409pt}{0.400pt}}
\put(171.0,726.0){\rule[-0.200pt]{2.409pt}{0.400pt}}
\put(1429.0,726.0){\rule[-0.200pt]{2.409pt}{0.400pt}}
\put(171.0,745.0){\rule[-0.200pt]{2.409pt}{0.400pt}}
\put(1429.0,745.0){\rule[-0.200pt]{2.409pt}{0.400pt}}
\put(171.0,761.0){\rule[-0.200pt]{2.409pt}{0.400pt}}
\put(1429.0,761.0){\rule[-0.200pt]{2.409pt}{0.400pt}}
\put(171.0,776.0){\rule[-0.200pt]{4.818pt}{0.400pt}}
\put(151,776){\makebox(0,0)[r]{$1$}}
\put(1419.0,776.0){\rule[-0.200pt]{4.818pt}{0.400pt}}
\put(171.0,131.0){\rule[-0.200pt]{0.400pt}{4.818pt}}
\put(171,90){\makebox(0,0){$0.001$}}
\put(171.0,756.0){\rule[-0.200pt]{0.400pt}{4.818pt}}
\put(362.0,131.0){\rule[-0.200pt]{0.400pt}{2.409pt}}
\put(362.0,766.0){\rule[-0.200pt]{0.400pt}{2.409pt}}
\put(473.0,131.0){\rule[-0.200pt]{0.400pt}{2.409pt}}
\put(473.0,766.0){\rule[-0.200pt]{0.400pt}{2.409pt}}
\put(553.0,131.0){\rule[-0.200pt]{0.400pt}{2.409pt}}
\put(553.0,766.0){\rule[-0.200pt]{0.400pt}{2.409pt}}
\put(614.0,131.0){\rule[-0.200pt]{0.400pt}{2.409pt}}
\put(614.0,766.0){\rule[-0.200pt]{0.400pt}{2.409pt}}
\put(664.0,131.0){\rule[-0.200pt]{0.400pt}{2.409pt}}
\put(664.0,766.0){\rule[-0.200pt]{0.400pt}{2.409pt}}
\put(707.0,131.0){\rule[-0.200pt]{0.400pt}{2.409pt}}
\put(707.0,766.0){\rule[-0.200pt]{0.400pt}{2.409pt}}
\put(744.0,131.0){\rule[-0.200pt]{0.400pt}{2.409pt}}
\put(744.0,766.0){\rule[-0.200pt]{0.400pt}{2.409pt}}
\put(776.0,131.0){\rule[-0.200pt]{0.400pt}{2.409pt}}
\put(776.0,766.0){\rule[-0.200pt]{0.400pt}{2.409pt}}
\put(805.0,131.0){\rule[-0.200pt]{0.400pt}{4.818pt}}
\put(805,90){\makebox(0,0){$0.01$}}
\put(805.0,756.0){\rule[-0.200pt]{0.400pt}{4.818pt}}
\put(996.0,131.0){\rule[-0.200pt]{0.400pt}{2.409pt}}
\put(996.0,766.0){\rule[-0.200pt]{0.400pt}{2.409pt}}
\put(1107.0,131.0){\rule[-0.200pt]{0.400pt}{2.409pt}}
\put(1107.0,766.0){\rule[-0.200pt]{0.400pt}{2.409pt}}
\put(1187.0,131.0){\rule[-0.200pt]{0.400pt}{2.409pt}}
\put(1187.0,766.0){\rule[-0.200pt]{0.400pt}{2.409pt}}
\put(1248.0,131.0){\rule[-0.200pt]{0.400pt}{2.409pt}}
\put(1248.0,766.0){\rule[-0.200pt]{0.400pt}{2.409pt}}
\put(1298.0,131.0){\rule[-0.200pt]{0.400pt}{2.409pt}}
\put(1298.0,766.0){\rule[-0.200pt]{0.400pt}{2.409pt}}
\put(1341.0,131.0){\rule[-0.200pt]{0.400pt}{2.409pt}}
\put(1341.0,766.0){\rule[-0.200pt]{0.400pt}{2.409pt}}
\put(1378.0,131.0){\rule[-0.200pt]{0.400pt}{2.409pt}}
\put(1378.0,766.0){\rule[-0.200pt]{0.400pt}{2.409pt}}
\put(1410.0,131.0){\rule[-0.200pt]{0.400pt}{2.409pt}}
\put(1410.0,766.0){\rule[-0.200pt]{0.400pt}{2.409pt}}
\put(1439.0,131.0){\rule[-0.200pt]{0.400pt}{4.818pt}}
\put(1439,90){\makebox(0,0){$0.1$}}
\put(1439.0,756.0){\rule[-0.200pt]{0.400pt}{4.818pt}}
\put(171.0,131.0){\rule[-0.200pt]{0.400pt}{155.380pt}}
\put(171.0,131.0){\rule[-0.200pt]{305.461pt}{0.400pt}}
\put(1439.0,131.0){\rule[-0.200pt]{0.400pt}{155.380pt}}
\put(171.0,776.0){\rule[-0.200pt]{305.461pt}{0.400pt}}
\put(30,453){\makebox(0,0){\rotatebox{90}{$H^1(\mathcal{F})$ error}}}
\put(805,29){\makebox(0,0){$h$}}
\put(805,838){\makebox(0,0){velocity}}
\put(1279,212){\makebox(0,0)[r]{$P_1-P_0-P_1$ (slope=1.003)}}
\put(1299.0,212.0){\rule[-0.200pt]{24.090pt}{0.400pt}}
\put(1211,647){\usebox{\plotpoint}}
\multiput(1207.21,645.92)(-1.017,-0.499){185}{\rule{0.913pt}{0.120pt}}
\multiput(1209.11,646.17)(-189.106,-94.000){2}{\rule{0.456pt}{0.400pt}}
\multiput(1016.38,551.92)(-0.966,-0.499){195}{\rule{0.872pt}{0.120pt}}
\multiput(1018.19,552.17)(-189.191,-99.000){2}{\rule{0.436pt}{0.400pt}}
\multiput(825.37,452.92)(-0.971,-0.499){193}{\rule{0.876pt}{0.120pt}}
\multiput(827.18,453.17)(-188.183,-98.000){2}{\rule{0.438pt}{0.400pt}}
\multiput(635.32,354.92)(-0.986,-0.499){191}{\rule{0.888pt}{0.120pt}}
\multiput(637.16,355.17)(-189.158,-97.000){2}{\rule{0.444pt}{0.400pt}}
\put(1211,647){\makebox(0,0){$+$}}
\put(1020,553){\makebox(0,0){$+$}}
\put(829,454){\makebox(0,0){$+$}}
\put(639,356){\makebox(0,0){$+$}}
\put(448,259){\makebox(0,0){$+$}}
\put(1349,212){\makebox(0,0){$+$}}
\put(1279,171){\makebox(0,0)[r]{$P_1-P_0-P_0$ (slope=1.004)}}
\multiput(1299,171)(20.756,0.000){5}{\usebox{\plotpoint}}
\put(1399,171){\usebox{\plotpoint}}
\put(1211,648){\usebox{\plotpoint}}
\multiput(1211,648)(-18.584,-9.243){11}{\usebox{\plotpoint}}
\multiput(1020,553)(-18.427,-9.551){10}{\usebox{\plotpoint}}
\multiput(829,454)(-18.446,-9.514){10}{\usebox{\plotpoint}}
\multiput(639,356)(-18.506,-9.398){11}{\usebox{\plotpoint}}
\put(448,259){\usebox{\plotpoint}}
\put(1211,648){\makebox(0,0){$\times$}}
\put(1020,553){\makebox(0,0){$\times$}}
\put(829,454){\makebox(0,0){$\times$}}
\put(639,356){\makebox(0,0){$\times$}}
\put(448,259){\makebox(0,0){$\times$}}
\put(1349,171){\makebox(0,0){$\times$}}
\put(171.0,131.0){\rule[-0.200pt]{0.400pt}{155.380pt}}
\put(171.0,131.0){\rule[-0.200pt]{305.461pt}{0.400pt}}
\put(1439.0,131.0){\rule[-0.200pt]{0.400pt}{155.380pt}}
\put(171.0,776.0){\rule[-0.200pt]{305.461pt}{0.400pt}}
\end{picture}}\\
\resizebox{8cm}{!}{
\setlength{\unitlength}{0.240900pt}
\ifx\plotpoint\undefined\newsavebox{\plotpoint}\fi
\begin{picture}(1500,900)(0,0)
\sbox{\plotpoint}{\rule[-0.200pt]{0.400pt}{0.400pt}}%
\put(171.0,131.0){\rule[-0.200pt]{4.818pt}{0.400pt}}
\put(151,131){\makebox(0,0)[r]{$0.01$}}
\put(1419.0,131.0){\rule[-0.200pt]{4.818pt}{0.400pt}}
\put(171.0,228.0){\rule[-0.200pt]{2.409pt}{0.400pt}}
\put(1429.0,228.0){\rule[-0.200pt]{2.409pt}{0.400pt}}
\put(171.0,285.0){\rule[-0.200pt]{2.409pt}{0.400pt}}
\put(1429.0,285.0){\rule[-0.200pt]{2.409pt}{0.400pt}}
\put(171.0,325.0){\rule[-0.200pt]{2.409pt}{0.400pt}}
\put(1429.0,325.0){\rule[-0.200pt]{2.409pt}{0.400pt}}
\put(171.0,356.0){\rule[-0.200pt]{2.409pt}{0.400pt}}
\put(1429.0,356.0){\rule[-0.200pt]{2.409pt}{0.400pt}}
\put(171.0,382.0){\rule[-0.200pt]{2.409pt}{0.400pt}}
\put(1429.0,382.0){\rule[-0.200pt]{2.409pt}{0.400pt}}
\put(171.0,404.0){\rule[-0.200pt]{2.409pt}{0.400pt}}
\put(1429.0,404.0){\rule[-0.200pt]{2.409pt}{0.400pt}}
\put(171.0,422.0){\rule[-0.200pt]{2.409pt}{0.400pt}}
\put(1429.0,422.0){\rule[-0.200pt]{2.409pt}{0.400pt}}
\put(171.0,439.0){\rule[-0.200pt]{2.409pt}{0.400pt}}
\put(1429.0,439.0){\rule[-0.200pt]{2.409pt}{0.400pt}}
\put(171.0,454.0){\rule[-0.200pt]{4.818pt}{0.400pt}}
\put(151,454){\makebox(0,0)[r]{$0.1$}}
\put(1419.0,454.0){\rule[-0.200pt]{4.818pt}{0.400pt}}
\put(171.0,551.0){\rule[-0.200pt]{2.409pt}{0.400pt}}
\put(1429.0,551.0){\rule[-0.200pt]{2.409pt}{0.400pt}}
\put(171.0,607.0){\rule[-0.200pt]{2.409pt}{0.400pt}}
\put(1429.0,607.0){\rule[-0.200pt]{2.409pt}{0.400pt}}
\put(171.0,648.0){\rule[-0.200pt]{2.409pt}{0.400pt}}
\put(1429.0,648.0){\rule[-0.200pt]{2.409pt}{0.400pt}}
\put(171.0,679.0){\rule[-0.200pt]{2.409pt}{0.400pt}}
\put(1429.0,679.0){\rule[-0.200pt]{2.409pt}{0.400pt}}
\put(171.0,704.0){\rule[-0.200pt]{2.409pt}{0.400pt}}
\put(1429.0,704.0){\rule[-0.200pt]{2.409pt}{0.400pt}}
\put(171.0,726.0){\rule[-0.200pt]{2.409pt}{0.400pt}}
\put(1429.0,726.0){\rule[-0.200pt]{2.409pt}{0.400pt}}
\put(171.0,745.0){\rule[-0.200pt]{2.409pt}{0.400pt}}
\put(1429.0,745.0){\rule[-0.200pt]{2.409pt}{0.400pt}}
\put(171.0,761.0){\rule[-0.200pt]{2.409pt}{0.400pt}}
\put(1429.0,761.0){\rule[-0.200pt]{2.409pt}{0.400pt}}
\put(171.0,776.0){\rule[-0.200pt]{4.818pt}{0.400pt}}
\put(151,776){\makebox(0,0)[r]{$1$}}
\put(1419.0,776.0){\rule[-0.200pt]{4.818pt}{0.400pt}}
\put(171.0,131.0){\rule[-0.200pt]{0.400pt}{4.818pt}}
\put(171,90){\makebox(0,0){$0.001$}}
\put(171.0,756.0){\rule[-0.200pt]{0.400pt}{4.818pt}}
\put(362.0,131.0){\rule[-0.200pt]{0.400pt}{2.409pt}}
\put(362.0,766.0){\rule[-0.200pt]{0.400pt}{2.409pt}}
\put(473.0,131.0){\rule[-0.200pt]{0.400pt}{2.409pt}}
\put(473.0,766.0){\rule[-0.200pt]{0.400pt}{2.409pt}}
\put(553.0,131.0){\rule[-0.200pt]{0.400pt}{2.409pt}}
\put(553.0,766.0){\rule[-0.200pt]{0.400pt}{2.409pt}}
\put(614.0,131.0){\rule[-0.200pt]{0.400pt}{2.409pt}}
\put(614.0,766.0){\rule[-0.200pt]{0.400pt}{2.409pt}}
\put(664.0,131.0){\rule[-0.200pt]{0.400pt}{2.409pt}}
\put(664.0,766.0){\rule[-0.200pt]{0.400pt}{2.409pt}}
\put(707.0,131.0){\rule[-0.200pt]{0.400pt}{2.409pt}}
\put(707.0,766.0){\rule[-0.200pt]{0.400pt}{2.409pt}}
\put(744.0,131.0){\rule[-0.200pt]{0.400pt}{2.409pt}}
\put(744.0,766.0){\rule[-0.200pt]{0.400pt}{2.409pt}}
\put(776.0,131.0){\rule[-0.200pt]{0.400pt}{2.409pt}}
\put(776.0,766.0){\rule[-0.200pt]{0.400pt}{2.409pt}}
\put(805.0,131.0){\rule[-0.200pt]{0.400pt}{4.818pt}}
\put(805,90){\makebox(0,0){$0.01$}}
\put(805.0,756.0){\rule[-0.200pt]{0.400pt}{4.818pt}}
\put(996.0,131.0){\rule[-0.200pt]{0.400pt}{2.409pt}}
\put(996.0,766.0){\rule[-0.200pt]{0.400pt}{2.409pt}}
\put(1107.0,131.0){\rule[-0.200pt]{0.400pt}{2.409pt}}
\put(1107.0,766.0){\rule[-0.200pt]{0.400pt}{2.409pt}}
\put(1187.0,131.0){\rule[-0.200pt]{0.400pt}{2.409pt}}
\put(1187.0,766.0){\rule[-0.200pt]{0.400pt}{2.409pt}}
\put(1248.0,131.0){\rule[-0.200pt]{0.400pt}{2.409pt}}
\put(1248.0,766.0){\rule[-0.200pt]{0.400pt}{2.409pt}}
\put(1298.0,131.0){\rule[-0.200pt]{0.400pt}{2.409pt}}
\put(1298.0,766.0){\rule[-0.200pt]{0.400pt}{2.409pt}}
\put(1341.0,131.0){\rule[-0.200pt]{0.400pt}{2.409pt}}
\put(1341.0,766.0){\rule[-0.200pt]{0.400pt}{2.409pt}}
\put(1378.0,131.0){\rule[-0.200pt]{0.400pt}{2.409pt}}
\put(1378.0,766.0){\rule[-0.200pt]{0.400pt}{2.409pt}}
\put(1410.0,131.0){\rule[-0.200pt]{0.400pt}{2.409pt}}
\put(1410.0,766.0){\rule[-0.200pt]{0.400pt}{2.409pt}}
\put(1439.0,131.0){\rule[-0.200pt]{0.400pt}{4.818pt}}
\put(1439,90){\makebox(0,0){$0.1$}}
\put(1439.0,756.0){\rule[-0.200pt]{0.400pt}{4.818pt}}
\put(171.0,131.0){\rule[-0.200pt]{0.400pt}{155.380pt}}
\put(171.0,131.0){\rule[-0.200pt]{305.461pt}{0.400pt}}
\put(1439.0,131.0){\rule[-0.200pt]{0.400pt}{155.380pt}}
\put(171.0,776.0){\rule[-0.200pt]{305.461pt}{0.400pt}}
\put(30,453){\makebox(0,0){\rotatebox{90}{$L^2(\mathcal{F})$ error}}}
\put(805,29){\makebox(0,0){$h$}}
\put(805,838){\makebox(0,0){pressure}}
\put(1279,212){\makebox(0,0)[r]{$P_1-P_0-P_1$ (slope=1.302)}}
\put(1299.0,212.0){\rule[-0.200pt]{24.090pt}{0.400pt}}
\put(1211,733){\usebox{\plotpoint}}
\multiput(1208.56,731.92)(-0.608,-0.499){311}{\rule{0.587pt}{0.120pt}}
\multiput(1209.78,732.17)(-189.782,-157.000){2}{\rule{0.293pt}{0.400pt}}
\multiput(1017.07,574.92)(-0.758,-0.499){249}{\rule{0.706pt}{0.120pt}}
\multiput(1018.53,575.17)(-189.534,-126.000){2}{\rule{0.353pt}{0.400pt}}
\multiput(825.91,448.92)(-0.806,-0.499){233}{\rule{0.744pt}{0.120pt}}
\multiput(827.46,449.17)(-188.456,-118.000){2}{\rule{0.372pt}{0.400pt}}
\multiput(635.68,330.92)(-0.877,-0.499){215}{\rule{0.801pt}{0.120pt}}
\multiput(637.34,331.17)(-189.338,-109.000){2}{\rule{0.400pt}{0.400pt}}
\put(1211,733){\makebox(0,0){$+$}}
\put(1020,576){\makebox(0,0){$+$}}
\put(829,450){\makebox(0,0){$+$}}
\put(639,332){\makebox(0,0){$+$}}
\put(448,223){\makebox(0,0){$+$}}
\put(1349,212){\makebox(0,0){$+$}}
\put(1279,171){\makebox(0,0)[r]{$P_1-P_0-P_0$ (slope=1.298)}}
\multiput(1299,171)(20.756,0.000){5}{\usebox{\plotpoint}}
\put(1399,171){\usebox{\plotpoint}}
\put(1211,731){\usebox{\plotpoint}}
\multiput(1211,731)(-16.116,-13.079){12}{\usebox{\plotpoint}}
\multiput(1020,576)(-17.284,-11.492){11}{\usebox{\plotpoint}}
\multiput(829,449)(-17.673,-10.883){11}{\usebox{\plotpoint}}
\multiput(639,332)(-18.027,-10.287){11}{\usebox{\plotpoint}}
\put(448,223){\usebox{\plotpoint}}
\put(1211,731){\makebox(0,0){$\times$}}
\put(1020,576){\makebox(0,0){$\times$}}
\put(829,449){\makebox(0,0){$\times$}}
\put(639,332){\makebox(0,0){$\times$}}
\put(448,223){\makebox(0,0){$\times$}}
\put(1349,171){\makebox(0,0){$\times$}}
\put(171.0,131.0){\rule[-0.200pt]{0.400pt}{155.380pt}}
\put(171.0,131.0){\rule[-0.200pt]{305.461pt}{0.400pt}}
\put(1439.0,131.0){\rule[-0.200pt]{0.400pt}{155.380pt}}
\put(171.0,776.0){\rule[-0.200pt]{305.461pt}{0.400pt}}
\end{picture}}
\resizebox{8cm}{!}{
\setlength{\unitlength}{0.240900pt}
\ifx\plotpoint\undefined\newsavebox{\plotpoint}\fi
\begin{picture}(1500,900)(0,0)
\sbox{\plotpoint}{\rule[-0.200pt]{0.400pt}{0.400pt}}%
\put(211.0,131.0){\rule[-0.200pt]{4.818pt}{0.400pt}}
\put(191,131){\makebox(0,0)[r]{$1e-05$}}
\put(1419.0,131.0){\rule[-0.200pt]{4.818pt}{0.400pt}}
\put(211.0,196.0){\rule[-0.200pt]{2.409pt}{0.400pt}}
\put(1429.0,196.0){\rule[-0.200pt]{2.409pt}{0.400pt}}
\put(211.0,234.0){\rule[-0.200pt]{2.409pt}{0.400pt}}
\put(1429.0,234.0){\rule[-0.200pt]{2.409pt}{0.400pt}}
\put(211.0,260.0){\rule[-0.200pt]{2.409pt}{0.400pt}}
\put(1429.0,260.0){\rule[-0.200pt]{2.409pt}{0.400pt}}
\put(211.0,281.0){\rule[-0.200pt]{2.409pt}{0.400pt}}
\put(1429.0,281.0){\rule[-0.200pt]{2.409pt}{0.400pt}}
\put(211.0,298.0){\rule[-0.200pt]{2.409pt}{0.400pt}}
\put(1429.0,298.0){\rule[-0.200pt]{2.409pt}{0.400pt}}
\put(211.0,313.0){\rule[-0.200pt]{2.409pt}{0.400pt}}
\put(1429.0,313.0){\rule[-0.200pt]{2.409pt}{0.400pt}}
\put(211.0,325.0){\rule[-0.200pt]{2.409pt}{0.400pt}}
\put(1429.0,325.0){\rule[-0.200pt]{2.409pt}{0.400pt}}
\put(211.0,336.0){\rule[-0.200pt]{2.409pt}{0.400pt}}
\put(1429.0,336.0){\rule[-0.200pt]{2.409pt}{0.400pt}}
\put(211.0,346.0){\rule[-0.200pt]{4.818pt}{0.400pt}}
\put(191,346){\makebox(0,0)[r]{$0.0001$}}
\put(1419.0,346.0){\rule[-0.200pt]{4.818pt}{0.400pt}}
\put(211.0,411.0){\rule[-0.200pt]{2.409pt}{0.400pt}}
\put(1429.0,411.0){\rule[-0.200pt]{2.409pt}{0.400pt}}
\put(211.0,449.0){\rule[-0.200pt]{2.409pt}{0.400pt}}
\put(1429.0,449.0){\rule[-0.200pt]{2.409pt}{0.400pt}}
\put(211.0,475.0){\rule[-0.200pt]{2.409pt}{0.400pt}}
\put(1429.0,475.0){\rule[-0.200pt]{2.409pt}{0.400pt}}
\put(211.0,496.0){\rule[-0.200pt]{2.409pt}{0.400pt}}
\put(1429.0,496.0){\rule[-0.200pt]{2.409pt}{0.400pt}}
\put(211.0,513.0){\rule[-0.200pt]{2.409pt}{0.400pt}}
\put(1429.0,513.0){\rule[-0.200pt]{2.409pt}{0.400pt}}
\put(211.0,528.0){\rule[-0.200pt]{2.409pt}{0.400pt}}
\put(1429.0,528.0){\rule[-0.200pt]{2.409pt}{0.400pt}}
\put(211.0,540.0){\rule[-0.200pt]{2.409pt}{0.400pt}}
\put(1429.0,540.0){\rule[-0.200pt]{2.409pt}{0.400pt}}
\put(211.0,551.0){\rule[-0.200pt]{2.409pt}{0.400pt}}
\put(1429.0,551.0){\rule[-0.200pt]{2.409pt}{0.400pt}}
\put(211.0,561.0){\rule[-0.200pt]{4.818pt}{0.400pt}}
\put(191,561){\makebox(0,0)[r]{$0.001$}}
\put(1419.0,561.0){\rule[-0.200pt]{4.818pt}{0.400pt}}
\put(211.0,626.0){\rule[-0.200pt]{2.409pt}{0.400pt}}
\put(1429.0,626.0){\rule[-0.200pt]{2.409pt}{0.400pt}}
\put(211.0,664.0){\rule[-0.200pt]{2.409pt}{0.400pt}}
\put(1429.0,664.0){\rule[-0.200pt]{2.409pt}{0.400pt}}
\put(211.0,690.0){\rule[-0.200pt]{2.409pt}{0.400pt}}
\put(1429.0,690.0){\rule[-0.200pt]{2.409pt}{0.400pt}}
\put(211.0,711.0){\rule[-0.200pt]{2.409pt}{0.400pt}}
\put(1429.0,711.0){\rule[-0.200pt]{2.409pt}{0.400pt}}
\put(211.0,728.0){\rule[-0.200pt]{2.409pt}{0.400pt}}
\put(1429.0,728.0){\rule[-0.200pt]{2.409pt}{0.400pt}}
\put(211.0,743.0){\rule[-0.200pt]{2.409pt}{0.400pt}}
\put(1429.0,743.0){\rule[-0.200pt]{2.409pt}{0.400pt}}
\put(211.0,755.0){\rule[-0.200pt]{2.409pt}{0.400pt}}
\put(1429.0,755.0){\rule[-0.200pt]{2.409pt}{0.400pt}}
\put(211.0,766.0){\rule[-0.200pt]{2.409pt}{0.400pt}}
\put(1429.0,766.0){\rule[-0.200pt]{2.409pt}{0.400pt}}
\put(211.0,776.0){\rule[-0.200pt]{4.818pt}{0.400pt}}
\put(191,776){\makebox(0,0)[r]{$0.01$}}
\put(1419.0,776.0){\rule[-0.200pt]{4.818pt}{0.400pt}}
\put(211.0,131.0){\rule[-0.200pt]{0.400pt}{4.818pt}}
\put(211,90){\makebox(0,0){$0.001$}}
\put(211.0,756.0){\rule[-0.200pt]{0.400pt}{4.818pt}}
\put(396.0,131.0){\rule[-0.200pt]{0.400pt}{2.409pt}}
\put(396.0,766.0){\rule[-0.200pt]{0.400pt}{2.409pt}}
\put(504.0,131.0){\rule[-0.200pt]{0.400pt}{2.409pt}}
\put(504.0,766.0){\rule[-0.200pt]{0.400pt}{2.409pt}}
\put(581.0,131.0){\rule[-0.200pt]{0.400pt}{2.409pt}}
\put(581.0,766.0){\rule[-0.200pt]{0.400pt}{2.409pt}}
\put(640.0,131.0){\rule[-0.200pt]{0.400pt}{2.409pt}}
\put(640.0,766.0){\rule[-0.200pt]{0.400pt}{2.409pt}}
\put(689.0,131.0){\rule[-0.200pt]{0.400pt}{2.409pt}}
\put(689.0,766.0){\rule[-0.200pt]{0.400pt}{2.409pt}}
\put(730.0,131.0){\rule[-0.200pt]{0.400pt}{2.409pt}}
\put(730.0,766.0){\rule[-0.200pt]{0.400pt}{2.409pt}}
\put(765.0,131.0){\rule[-0.200pt]{0.400pt}{2.409pt}}
\put(765.0,766.0){\rule[-0.200pt]{0.400pt}{2.409pt}}
\put(797.0,131.0){\rule[-0.200pt]{0.400pt}{2.409pt}}
\put(797.0,766.0){\rule[-0.200pt]{0.400pt}{2.409pt}}
\put(825.0,131.0){\rule[-0.200pt]{0.400pt}{4.818pt}}
\put(825,90){\makebox(0,0){$0.01$}}
\put(825.0,756.0){\rule[-0.200pt]{0.400pt}{4.818pt}}
\put(1010.0,131.0){\rule[-0.200pt]{0.400pt}{2.409pt}}
\put(1010.0,766.0){\rule[-0.200pt]{0.400pt}{2.409pt}}
\put(1118.0,131.0){\rule[-0.200pt]{0.400pt}{2.409pt}}
\put(1118.0,766.0){\rule[-0.200pt]{0.400pt}{2.409pt}}
\put(1195.0,131.0){\rule[-0.200pt]{0.400pt}{2.409pt}}
\put(1195.0,766.0){\rule[-0.200pt]{0.400pt}{2.409pt}}
\put(1254.0,131.0){\rule[-0.200pt]{0.400pt}{2.409pt}}
\put(1254.0,766.0){\rule[-0.200pt]{0.400pt}{2.409pt}}
\put(1303.0,131.0){\rule[-0.200pt]{0.400pt}{2.409pt}}
\put(1303.0,766.0){\rule[-0.200pt]{0.400pt}{2.409pt}}
\put(1344.0,131.0){\rule[-0.200pt]{0.400pt}{2.409pt}}
\put(1344.0,766.0){\rule[-0.200pt]{0.400pt}{2.409pt}}
\put(1379.0,131.0){\rule[-0.200pt]{0.400pt}{2.409pt}}
\put(1379.0,766.0){\rule[-0.200pt]{0.400pt}{2.409pt}}
\put(1411.0,131.0){\rule[-0.200pt]{0.400pt}{2.409pt}}
\put(1411.0,766.0){\rule[-0.200pt]{0.400pt}{2.409pt}}
\put(1439.0,131.0){\rule[-0.200pt]{0.400pt}{4.818pt}}
\put(1439,90){\makebox(0,0){$0.1$}}
\put(1439.0,756.0){\rule[-0.200pt]{0.400pt}{4.818pt}}
\put(211.0,131.0){\rule[-0.200pt]{0.400pt}{155.380pt}}
\put(211.0,131.0){\rule[-0.200pt]{295.825pt}{0.400pt}}
\put(1439.0,131.0){\rule[-0.200pt]{0.400pt}{155.380pt}}
\put(211.0,776.0){\rule[-0.200pt]{295.825pt}{0.400pt}}
\put(30,453){\makebox(0,0){\rotatebox{90}{$\left|\int_{\Gamma}(\lambda-\lambda_{h})\right|$}}}
\put(825,29){\makebox(0,0){$h$}}
\put(825,838){\makebox(0,0){multiplier}}
\put(1279,212){\makebox(0,0)[r]{$P_1-P_0-P_1$ (slope=2.166)}}
\put(1299.0,212.0){\rule[-0.200pt]{24.090pt}{0.400pt}}
\put(1218,682){\usebox{\plotpoint}}
\multiput(1213.07,680.92)(-1.364,-0.499){133}{\rule{1.188pt}{0.120pt}}
\multiput(1215.53,681.17)(-182.534,-68.000){2}{\rule{0.594pt}{0.400pt}}
\multiput(1030.74,612.92)(-0.554,-0.500){329}{\rule{0.543pt}{0.120pt}}
\multiput(1031.87,613.17)(-182.872,-166.000){2}{\rule{0.272pt}{0.400pt}}
\multiput(846.62,446.92)(-0.593,-0.499){309}{\rule{0.574pt}{0.120pt}}
\multiput(847.81,447.17)(-183.808,-156.000){2}{\rule{0.287pt}{0.400pt}}
\multiput(661.54,290.92)(-0.617,-0.499){297}{\rule{0.593pt}{0.120pt}}
\multiput(662.77,291.17)(-183.769,-150.000){2}{\rule{0.297pt}{0.400pt}}
\put(1218,682){\makebox(0,0){$+$}}
\put(1033,614){\makebox(0,0){$+$}}
\put(849,448){\makebox(0,0){$+$}}
\put(664,292){\makebox(0,0){$+$}}
\put(479,142){\makebox(0,0){$+$}}
\put(1349,212){\makebox(0,0){$+$}}
\put(1279,171){\makebox(0,0)[r]{$P_1-P_0-P_0$ (slope=2.164)}}
\multiput(1299,171)(20.756,0.000){5}{\usebox{\plotpoint}}
\put(1399,171){\usebox{\plotpoint}}
\put(1218,680){\usebox{\plotpoint}}
\multiput(1218,680)(-19.680,-6.595){10}{\usebox{\plotpoint}}
\multiput(1033,618)(-15.245,-14.085){12}{\usebox{\plotpoint}}
\multiput(849,448)(-15.825,-13.430){12}{\usebox{\plotpoint}}
\multiput(664,291)(-16.207,-12.966){11}{\usebox{\plotpoint}}
\put(479,143){\usebox{\plotpoint}}
\put(1218,680){\makebox(0,0){$\times$}}
\put(1033,618){\makebox(0,0){$\times$}}
\put(849,448){\makebox(0,0){$\times$}}
\put(664,291){\makebox(0,0){$\times$}}
\put(479,143){\makebox(0,0){$\times$}}
\put(1349,171){\makebox(0,0){$\times$}}
\put(211.0,131.0){\rule[-0.200pt]{0.400pt}{155.380pt}}
\put(211.0,131.0){\rule[-0.200pt]{295.825pt}{0.400pt}}
\put(1439.0,131.0){\rule[-0.200pt]{0.400pt}{155.380pt}}
\put(211.0,776.0){\rule[-0.200pt]{295.825pt}{0.400pt}}
\end{picture}}
\end{minipage}
\caption{Rates of convergence with Interior Penalty stabilization for $\|u-u_h\|_{0,\mathcal{F}}$, 
$\|u-u_h\|_{1,\mathcal{F}}$,
$\|p-p_h\|_{0,\mathcal{F}}$ and $\left|\int_{\Gamma}(\lambda-\lambda_{h})\right|$}
\label{fig:Interior-Penalty}
\end{figure}
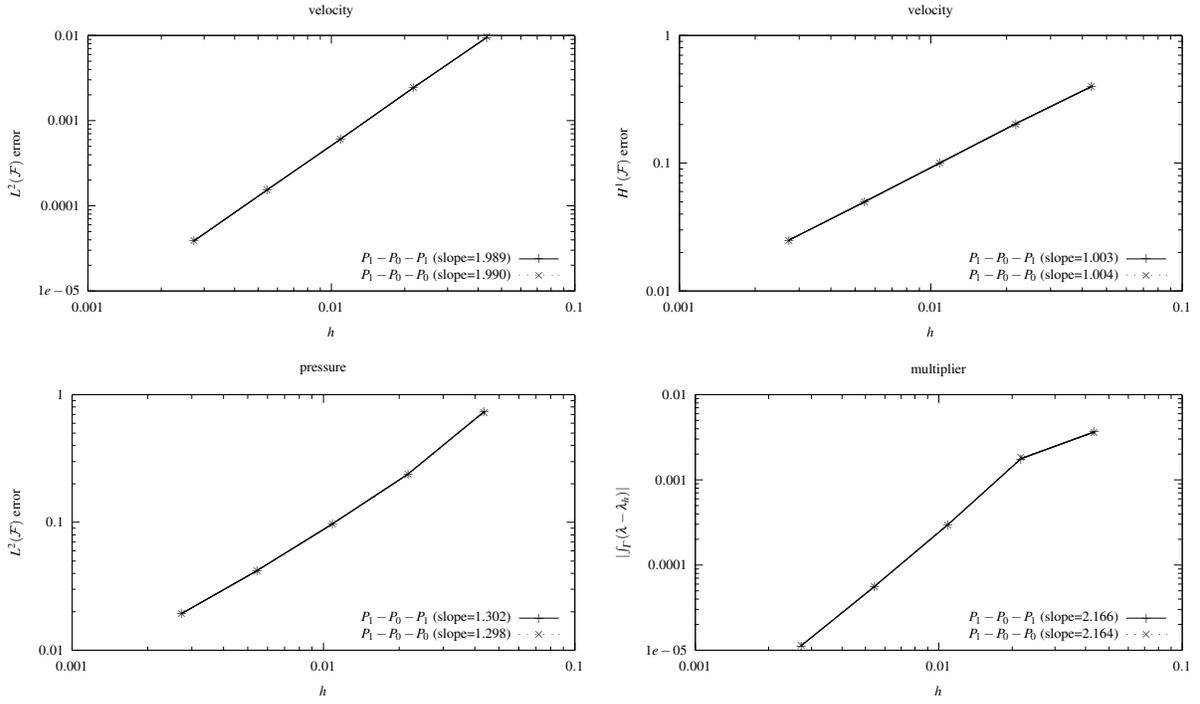
\end{center}

\subsubsection{Taylor-Hood spaces.}
Here, system (\ref{sys:Barbosa-Hughes}) is modified using the robust reconstruction from Definition \ref{RobRec} for both $u$ and $p$. 
This gives
\begin{equation}
\label{sys:Taylor-Hood}
\left (
\begin{array}{ccc}
K +S^{\gamma_0}_{\hat{u}\hat{u}} & B^T+{S^{\gamma_0}_{\hat{u}\hat{p}}}^{\hspace*{-0.1cm}T}&C^T + {S^{\gamma_0}_{\hat{u}l}}^{\hspace*{-0.03cm}T}\\
B + S^{\gamma_0}_{\hat{u}\hat{p}} & S^{\gamma_0}_{\hat{p}\hat{p}}& {S^{\gamma_0}_{\hat{p} \lambda}}^{\hspace*{-0.1cm}T}\\
C + S^{\gamma_0}_{\hat{u} \lambda} & S^{\gamma_0}_{\hat{p} \lambda} & S^{\gamma_0}_{\lambda \lambda}\\
\end{array}
\right )
\left (
\begin{array}{c}
U\\
P\\
\Lambda\\
\end{array}
\right ) = 
\left (
\begin{array}{c}
F\\
0\\
G \\
\end{array}
\right ) 
\end{equation}
where matrices $S^{\gamma_0}_{\hat{u}\hat{p}}, S^{\gamma_0}_{\hat{p}\hat{p}},\ldots$ are constructed from $S^{\gamma_0}_{\hat{u}{p}}, S^{\gamma_0}_{{p}{p}},\ldots$ by adding the ``robust reconstruction'' of $p$ similarly to that of $u$ in (\ref{SMatHat}).

The results are presented in Fig. \ref{fig:Taylor-Hood}. Comparing them to those in Fig. \ref{fig:Barbosa-Hughes} (Barbosa-Hughes stabilization without the ``robust reconstruction'') we observe that they are very close to each other.
This is due to the simple configurations considered in the present study. 
 We refer to \cite{JFS-Fournie-Court} for more considerations.
 
\begin{center}
\begin{figure}
\hspace*{-2cm}
\begin{minipage}{20cm}
 \resizebox{8cm}{!}{
\setlength{\unitlength}{0.240900pt}
\ifx\plotpoint\undefined\newsavebox{\plotpoint}\fi
\begin{picture}(1500,900)(0,0)
\sbox{\plotpoint}{\rule[-0.200pt]{0.400pt}{0.400pt}}%
\put(211.0,131.0){\rule[-0.200pt]{4.818pt}{0.400pt}}
\put(191,131){\makebox(0,0)[r]{$1e-08$}}
\put(1419.0,131.0){\rule[-0.200pt]{4.818pt}{0.400pt}}
\put(211.0,170.0){\rule[-0.200pt]{2.409pt}{0.400pt}}
\put(1429.0,170.0){\rule[-0.200pt]{2.409pt}{0.400pt}}
\put(211.0,221.0){\rule[-0.200pt]{2.409pt}{0.400pt}}
\put(1429.0,221.0){\rule[-0.200pt]{2.409pt}{0.400pt}}
\put(211.0,247.0){\rule[-0.200pt]{2.409pt}{0.400pt}}
\put(1429.0,247.0){\rule[-0.200pt]{2.409pt}{0.400pt}}
\put(211.0,260.0){\rule[-0.200pt]{4.818pt}{0.400pt}}
\put(191,260){\makebox(0,0)[r]{$1e-07$}}
\put(1419.0,260.0){\rule[-0.200pt]{4.818pt}{0.400pt}}
\put(211.0,299.0){\rule[-0.200pt]{2.409pt}{0.400pt}}
\put(1429.0,299.0){\rule[-0.200pt]{2.409pt}{0.400pt}}
\put(211.0,350.0){\rule[-0.200pt]{2.409pt}{0.400pt}}
\put(1429.0,350.0){\rule[-0.200pt]{2.409pt}{0.400pt}}
\put(211.0,376.0){\rule[-0.200pt]{2.409pt}{0.400pt}}
\put(1429.0,376.0){\rule[-0.200pt]{2.409pt}{0.400pt}}
\put(211.0,389.0){\rule[-0.200pt]{4.818pt}{0.400pt}}
\put(191,389){\makebox(0,0)[r]{$1e-06$}}
\put(1419.0,389.0){\rule[-0.200pt]{4.818pt}{0.400pt}}
\put(211.0,428.0){\rule[-0.200pt]{2.409pt}{0.400pt}}
\put(1429.0,428.0){\rule[-0.200pt]{2.409pt}{0.400pt}}
\put(211.0,479.0){\rule[-0.200pt]{2.409pt}{0.400pt}}
\put(1429.0,479.0){\rule[-0.200pt]{2.409pt}{0.400pt}}
\put(211.0,505.0){\rule[-0.200pt]{2.409pt}{0.400pt}}
\put(1429.0,505.0){\rule[-0.200pt]{2.409pt}{0.400pt}}
\put(211.0,518.0){\rule[-0.200pt]{4.818pt}{0.400pt}}
\put(191,518){\makebox(0,0)[r]{$1e-05$}}
\put(1419.0,518.0){\rule[-0.200pt]{4.818pt}{0.400pt}}
\put(211.0,557.0){\rule[-0.200pt]{2.409pt}{0.400pt}}
\put(1429.0,557.0){\rule[-0.200pt]{2.409pt}{0.400pt}}
\put(211.0,608.0){\rule[-0.200pt]{2.409pt}{0.400pt}}
\put(1429.0,608.0){\rule[-0.200pt]{2.409pt}{0.400pt}}
\put(211.0,634.0){\rule[-0.200pt]{2.409pt}{0.400pt}}
\put(1429.0,634.0){\rule[-0.200pt]{2.409pt}{0.400pt}}
\put(211.0,647.0){\rule[-0.200pt]{4.818pt}{0.400pt}}
\put(191,647){\makebox(0,0)[r]{$0.0001$}}
\put(1419.0,647.0){\rule[-0.200pt]{4.818pt}{0.400pt}}
\put(211.0,686.0){\rule[-0.200pt]{2.409pt}{0.400pt}}
\put(1429.0,686.0){\rule[-0.200pt]{2.409pt}{0.400pt}}
\put(211.0,737.0){\rule[-0.200pt]{2.409pt}{0.400pt}}
\put(1429.0,737.0){\rule[-0.200pt]{2.409pt}{0.400pt}}
\put(211.0,763.0){\rule[-0.200pt]{2.409pt}{0.400pt}}
\put(1429.0,763.0){\rule[-0.200pt]{2.409pt}{0.400pt}}
\put(211.0,776.0){\rule[-0.200pt]{4.818pt}{0.400pt}}
\put(191,776){\makebox(0,0)[r]{$0.001$}}
\put(1419.0,776.0){\rule[-0.200pt]{4.818pt}{0.400pt}}
\put(211.0,131.0){\rule[-0.200pt]{0.400pt}{4.818pt}}
\put(211,90){\makebox(0,0){$0.001$}}
\put(211.0,756.0){\rule[-0.200pt]{0.400pt}{4.818pt}}
\put(396.0,131.0){\rule[-0.200pt]{0.400pt}{2.409pt}}
\put(396.0,766.0){\rule[-0.200pt]{0.400pt}{2.409pt}}
\put(504.0,131.0){\rule[-0.200pt]{0.400pt}{2.409pt}}
\put(504.0,766.0){\rule[-0.200pt]{0.400pt}{2.409pt}}
\put(581.0,131.0){\rule[-0.200pt]{0.400pt}{2.409pt}}
\put(581.0,766.0){\rule[-0.200pt]{0.400pt}{2.409pt}}
\put(640.0,131.0){\rule[-0.200pt]{0.400pt}{2.409pt}}
\put(640.0,766.0){\rule[-0.200pt]{0.400pt}{2.409pt}}
\put(689.0,131.0){\rule[-0.200pt]{0.400pt}{2.409pt}}
\put(689.0,766.0){\rule[-0.200pt]{0.400pt}{2.409pt}}
\put(730.0,131.0){\rule[-0.200pt]{0.400pt}{2.409pt}}
\put(730.0,766.0){\rule[-0.200pt]{0.400pt}{2.409pt}}
\put(765.0,131.0){\rule[-0.200pt]{0.400pt}{2.409pt}}
\put(765.0,766.0){\rule[-0.200pt]{0.400pt}{2.409pt}}
\put(797.0,131.0){\rule[-0.200pt]{0.400pt}{2.409pt}}
\put(797.0,766.0){\rule[-0.200pt]{0.400pt}{2.409pt}}
\put(825.0,131.0){\rule[-0.200pt]{0.400pt}{4.818pt}}
\put(825,90){\makebox(0,0){$0.01$}}
\put(825.0,756.0){\rule[-0.200pt]{0.400pt}{4.818pt}}
\put(1010.0,131.0){\rule[-0.200pt]{0.400pt}{2.409pt}}
\put(1010.0,766.0){\rule[-0.200pt]{0.400pt}{2.409pt}}
\put(1118.0,131.0){\rule[-0.200pt]{0.400pt}{2.409pt}}
\put(1118.0,766.0){\rule[-0.200pt]{0.400pt}{2.409pt}}
\put(1195.0,131.0){\rule[-0.200pt]{0.400pt}{2.409pt}}
\put(1195.0,766.0){\rule[-0.200pt]{0.400pt}{2.409pt}}
\put(1254.0,131.0){\rule[-0.200pt]{0.400pt}{2.409pt}}
\put(1254.0,766.0){\rule[-0.200pt]{0.400pt}{2.409pt}}
\put(1303.0,131.0){\rule[-0.200pt]{0.400pt}{2.409pt}}
\put(1303.0,766.0){\rule[-0.200pt]{0.400pt}{2.409pt}}
\put(1344.0,131.0){\rule[-0.200pt]{0.400pt}{2.409pt}}
\put(1344.0,766.0){\rule[-0.200pt]{0.400pt}{2.409pt}}
\put(1379.0,131.0){\rule[-0.200pt]{0.400pt}{2.409pt}}
\put(1379.0,766.0){\rule[-0.200pt]{0.400pt}{2.409pt}}
\put(1411.0,131.0){\rule[-0.200pt]{0.400pt}{2.409pt}}
\put(1411.0,766.0){\rule[-0.200pt]{0.400pt}{2.409pt}}
\put(1439.0,131.0){\rule[-0.200pt]{0.400pt}{4.818pt}}
\put(1439,90){\makebox(0,0){$0.1$}}
\put(1439.0,756.0){\rule[-0.200pt]{0.400pt}{4.818pt}}
\put(211.0,131.0){\rule[-0.200pt]{0.400pt}{155.380pt}}
\put(211.0,131.0){\rule[-0.200pt]{295.825pt}{0.400pt}}
\put(1439.0,131.0){\rule[-0.200pt]{0.400pt}{155.380pt}}
\put(211.0,776.0){\rule[-0.200pt]{295.825pt}{0.400pt}}
\put(30,453){\makebox(0,0){\rotatebox{90}{$L^2(\mathcal{F})$ error}}}
\put(825,29){\makebox(0,0){$h$}}
\put(825,838){\makebox(0,0){velocity}}
\put(1279,212){\makebox(0,0)[r]{$P_2-P_1-P_1$ (slope=2.991)}}
\put(1299.0,212.0){\rule[-0.200pt]{24.090pt}{0.400pt}}
\put(1218,701){\usebox{\plotpoint}}
\multiput(1214.94,699.92)(-0.798,-0.499){229}{\rule{0.738pt}{0.120pt}}
\multiput(1216.47,700.17)(-183.468,-116.000){2}{\rule{0.369pt}{0.400pt}}
\multiput(1029.95,583.92)(-0.794,-0.499){229}{\rule{0.734pt}{0.120pt}}
\multiput(1031.48,584.17)(-182.476,-116.000){2}{\rule{0.367pt}{0.400pt}}
\multiput(845.94,467.92)(-0.798,-0.499){229}{\rule{0.738pt}{0.120pt}}
\multiput(847.47,468.17)(-183.468,-116.000){2}{\rule{0.369pt}{0.400pt}}
\multiput(660.96,351.92)(-0.791,-0.499){231}{\rule{0.732pt}{0.120pt}}
\multiput(662.48,352.17)(-183.480,-117.000){2}{\rule{0.366pt}{0.400pt}}
\put(1218,701){\makebox(0,0){$+$}}
\put(1033,585){\makebox(0,0){$+$}}
\put(849,469){\makebox(0,0){$+$}}
\put(664,353){\makebox(0,0){$+$}}
\put(479,236){\makebox(0,0){$+$}}
\put(1349,212){\makebox(0,0){$+$}}
\put(1279,171){\makebox(0,0)[r]{$P_2-P_1-P_0$ (slope=2.650)}}
\multiput(1299,171)(20.756,0.000){5}{\usebox{\plotpoint}}
\put(1399,171){\usebox{\plotpoint}}
\put(1218,703){\usebox{\plotpoint}}
\multiput(1218,703)(-18.937,-8.496){10}{\usebox{\plotpoint}}
\multiput(1033,620)(-17.212,-11.599){11}{\usebox{\plotpoint}}
\multiput(849,496)(-18.217,-9.946){10}{\usebox{\plotpoint}}
\multiput(664,395)(-18.504,-9.402){10}{\usebox{\plotpoint}}
\put(479,301){\usebox{\plotpoint}}
\put(1218,703){\makebox(0,0){$\times$}}
\put(1033,620){\makebox(0,0){$\times$}}
\put(849,496){\makebox(0,0){$\times$}}
\put(664,395){\makebox(0,0){$\times$}}
\put(479,301){\makebox(0,0){$\times$}}
\put(1349,171){\makebox(0,0){$\times$}}
\put(211.0,131.0){\rule[-0.200pt]{0.400pt}{155.380pt}}
\put(211.0,131.0){\rule[-0.200pt]{295.825pt}{0.400pt}}
\put(1439.0,131.0){\rule[-0.200pt]{0.400pt}{155.380pt}}
\put(211.0,776.0){\rule[-0.200pt]{295.825pt}{0.400pt}}
\end{picture}}
 \resizebox{8cm}{!}{
\setlength{\unitlength}{0.240900pt}
\ifx\plotpoint\undefined\newsavebox{\plotpoint}\fi
\begin{picture}(1500,900)(0,0)
\sbox{\plotpoint}{\rule[-0.200pt]{0.400pt}{0.400pt}}%
\put(211.0,131.0){\rule[-0.200pt]{4.818pt}{0.400pt}}
\put(191,131){\makebox(0,0)[r]{$1e-05$}}
\put(1419.0,131.0){\rule[-0.200pt]{4.818pt}{0.400pt}}
\put(211.0,180.0){\rule[-0.200pt]{2.409pt}{0.400pt}}
\put(1429.0,180.0){\rule[-0.200pt]{2.409pt}{0.400pt}}
\put(211.0,208.0){\rule[-0.200pt]{2.409pt}{0.400pt}}
\put(1429.0,208.0){\rule[-0.200pt]{2.409pt}{0.400pt}}
\put(211.0,228.0){\rule[-0.200pt]{2.409pt}{0.400pt}}
\put(1429.0,228.0){\rule[-0.200pt]{2.409pt}{0.400pt}}
\put(211.0,244.0){\rule[-0.200pt]{2.409pt}{0.400pt}}
\put(1429.0,244.0){\rule[-0.200pt]{2.409pt}{0.400pt}}
\put(211.0,256.0){\rule[-0.200pt]{2.409pt}{0.400pt}}
\put(1429.0,256.0){\rule[-0.200pt]{2.409pt}{0.400pt}}
\put(211.0,267.0){\rule[-0.200pt]{2.409pt}{0.400pt}}
\put(1429.0,267.0){\rule[-0.200pt]{2.409pt}{0.400pt}}
\put(211.0,277.0){\rule[-0.200pt]{2.409pt}{0.400pt}}
\put(1429.0,277.0){\rule[-0.200pt]{2.409pt}{0.400pt}}
\put(211.0,285.0){\rule[-0.200pt]{2.409pt}{0.400pt}}
\put(1429.0,285.0){\rule[-0.200pt]{2.409pt}{0.400pt}}
\put(211.0,292.0){\rule[-0.200pt]{4.818pt}{0.400pt}}
\put(191,292){\makebox(0,0)[r]{$0.0001$}}
\put(1419.0,292.0){\rule[-0.200pt]{4.818pt}{0.400pt}}
\put(211.0,341.0){\rule[-0.200pt]{2.409pt}{0.400pt}}
\put(1429.0,341.0){\rule[-0.200pt]{2.409pt}{0.400pt}}
\put(211.0,369.0){\rule[-0.200pt]{2.409pt}{0.400pt}}
\put(1429.0,369.0){\rule[-0.200pt]{2.409pt}{0.400pt}}
\put(211.0,389.0){\rule[-0.200pt]{2.409pt}{0.400pt}}
\put(1429.0,389.0){\rule[-0.200pt]{2.409pt}{0.400pt}}
\put(211.0,405.0){\rule[-0.200pt]{2.409pt}{0.400pt}}
\put(1429.0,405.0){\rule[-0.200pt]{2.409pt}{0.400pt}}
\put(211.0,418.0){\rule[-0.200pt]{2.409pt}{0.400pt}}
\put(1429.0,418.0){\rule[-0.200pt]{2.409pt}{0.400pt}}
\put(211.0,429.0){\rule[-0.200pt]{2.409pt}{0.400pt}}
\put(1429.0,429.0){\rule[-0.200pt]{2.409pt}{0.400pt}}
\put(211.0,438.0){\rule[-0.200pt]{2.409pt}{0.400pt}}
\put(1429.0,438.0){\rule[-0.200pt]{2.409pt}{0.400pt}}
\put(211.0,446.0){\rule[-0.200pt]{2.409pt}{0.400pt}}
\put(1429.0,446.0){\rule[-0.200pt]{2.409pt}{0.400pt}}
\put(211.0,454.0){\rule[-0.200pt]{4.818pt}{0.400pt}}
\put(191,454){\makebox(0,0)[r]{$0.001$}}
\put(1419.0,454.0){\rule[-0.200pt]{4.818pt}{0.400pt}}
\put(211.0,502.0){\rule[-0.200pt]{2.409pt}{0.400pt}}
\put(1429.0,502.0){\rule[-0.200pt]{2.409pt}{0.400pt}}
\put(211.0,530.0){\rule[-0.200pt]{2.409pt}{0.400pt}}
\put(1429.0,530.0){\rule[-0.200pt]{2.409pt}{0.400pt}}
\put(211.0,551.0){\rule[-0.200pt]{2.409pt}{0.400pt}}
\put(1429.0,551.0){\rule[-0.200pt]{2.409pt}{0.400pt}}
\put(211.0,566.0){\rule[-0.200pt]{2.409pt}{0.400pt}}
\put(1429.0,566.0){\rule[-0.200pt]{2.409pt}{0.400pt}}
\put(211.0,579.0){\rule[-0.200pt]{2.409pt}{0.400pt}}
\put(1429.0,579.0){\rule[-0.200pt]{2.409pt}{0.400pt}}
\put(211.0,590.0){\rule[-0.200pt]{2.409pt}{0.400pt}}
\put(1429.0,590.0){\rule[-0.200pt]{2.409pt}{0.400pt}}
\put(211.0,599.0){\rule[-0.200pt]{2.409pt}{0.400pt}}
\put(1429.0,599.0){\rule[-0.200pt]{2.409pt}{0.400pt}}
\put(211.0,607.0){\rule[-0.200pt]{2.409pt}{0.400pt}}
\put(1429.0,607.0){\rule[-0.200pt]{2.409pt}{0.400pt}}
\put(211.0,615.0){\rule[-0.200pt]{4.818pt}{0.400pt}}
\put(191,615){\makebox(0,0)[r]{$0.01$}}
\put(1419.0,615.0){\rule[-0.200pt]{4.818pt}{0.400pt}}
\put(211.0,663.0){\rule[-0.200pt]{2.409pt}{0.400pt}}
\put(1429.0,663.0){\rule[-0.200pt]{2.409pt}{0.400pt}}
\put(211.0,692.0){\rule[-0.200pt]{2.409pt}{0.400pt}}
\put(1429.0,692.0){\rule[-0.200pt]{2.409pt}{0.400pt}}
\put(211.0,712.0){\rule[-0.200pt]{2.409pt}{0.400pt}}
\put(1429.0,712.0){\rule[-0.200pt]{2.409pt}{0.400pt}}
\put(211.0,727.0){\rule[-0.200pt]{2.409pt}{0.400pt}}
\put(1429.0,727.0){\rule[-0.200pt]{2.409pt}{0.400pt}}
\put(211.0,740.0){\rule[-0.200pt]{2.409pt}{0.400pt}}
\put(1429.0,740.0){\rule[-0.200pt]{2.409pt}{0.400pt}}
\put(211.0,751.0){\rule[-0.200pt]{2.409pt}{0.400pt}}
\put(1429.0,751.0){\rule[-0.200pt]{2.409pt}{0.400pt}}
\put(211.0,760.0){\rule[-0.200pt]{2.409pt}{0.400pt}}
\put(1429.0,760.0){\rule[-0.200pt]{2.409pt}{0.400pt}}
\put(211.0,769.0){\rule[-0.200pt]{2.409pt}{0.400pt}}
\put(1429.0,769.0){\rule[-0.200pt]{2.409pt}{0.400pt}}
\put(211.0,776.0){\rule[-0.200pt]{4.818pt}{0.400pt}}
\put(191,776){\makebox(0,0)[r]{$0.1$}}
\put(1419.0,776.0){\rule[-0.200pt]{4.818pt}{0.400pt}}
\put(211.0,131.0){\rule[-0.200pt]{0.400pt}{4.818pt}}
\put(211,90){\makebox(0,0){$0.001$}}
\put(211.0,756.0){\rule[-0.200pt]{0.400pt}{4.818pt}}
\put(396.0,131.0){\rule[-0.200pt]{0.400pt}{2.409pt}}
\put(396.0,766.0){\rule[-0.200pt]{0.400pt}{2.409pt}}
\put(504.0,131.0){\rule[-0.200pt]{0.400pt}{2.409pt}}
\put(504.0,766.0){\rule[-0.200pt]{0.400pt}{2.409pt}}
\put(581.0,131.0){\rule[-0.200pt]{0.400pt}{2.409pt}}
\put(581.0,766.0){\rule[-0.200pt]{0.400pt}{2.409pt}}
\put(640.0,131.0){\rule[-0.200pt]{0.400pt}{2.409pt}}
\put(640.0,766.0){\rule[-0.200pt]{0.400pt}{2.409pt}}
\put(689.0,131.0){\rule[-0.200pt]{0.400pt}{2.409pt}}
\put(689.0,766.0){\rule[-0.200pt]{0.400pt}{2.409pt}}
\put(730.0,131.0){\rule[-0.200pt]{0.400pt}{2.409pt}}
\put(730.0,766.0){\rule[-0.200pt]{0.400pt}{2.409pt}}
\put(765.0,131.0){\rule[-0.200pt]{0.400pt}{2.409pt}}
\put(765.0,766.0){\rule[-0.200pt]{0.400pt}{2.409pt}}
\put(797.0,131.0){\rule[-0.200pt]{0.400pt}{2.409pt}}
\put(797.0,766.0){\rule[-0.200pt]{0.400pt}{2.409pt}}
\put(825.0,131.0){\rule[-0.200pt]{0.400pt}{4.818pt}}
\put(825,90){\makebox(0,0){$0.01$}}
\put(825.0,756.0){\rule[-0.200pt]{0.400pt}{4.818pt}}
\put(1010.0,131.0){\rule[-0.200pt]{0.400pt}{2.409pt}}
\put(1010.0,766.0){\rule[-0.200pt]{0.400pt}{2.409pt}}
\put(1118.0,131.0){\rule[-0.200pt]{0.400pt}{2.409pt}}
\put(1118.0,766.0){\rule[-0.200pt]{0.400pt}{2.409pt}}
\put(1195.0,131.0){\rule[-0.200pt]{0.400pt}{2.409pt}}
\put(1195.0,766.0){\rule[-0.200pt]{0.400pt}{2.409pt}}
\put(1254.0,131.0){\rule[-0.200pt]{0.400pt}{2.409pt}}
\put(1254.0,766.0){\rule[-0.200pt]{0.400pt}{2.409pt}}
\put(1303.0,131.0){\rule[-0.200pt]{0.400pt}{2.409pt}}
\put(1303.0,766.0){\rule[-0.200pt]{0.400pt}{2.409pt}}
\put(1344.0,131.0){\rule[-0.200pt]{0.400pt}{2.409pt}}
\put(1344.0,766.0){\rule[-0.200pt]{0.400pt}{2.409pt}}
\put(1379.0,131.0){\rule[-0.200pt]{0.400pt}{2.409pt}}
\put(1379.0,766.0){\rule[-0.200pt]{0.400pt}{2.409pt}}
\put(1411.0,131.0){\rule[-0.200pt]{0.400pt}{2.409pt}}
\put(1411.0,766.0){\rule[-0.200pt]{0.400pt}{2.409pt}}
\put(1439.0,131.0){\rule[-0.200pt]{0.400pt}{4.818pt}}
\put(1439,90){\makebox(0,0){$0.1$}}
\put(1439.0,756.0){\rule[-0.200pt]{0.400pt}{4.818pt}}
\put(211.0,131.0){\rule[-0.200pt]{0.400pt}{155.380pt}}
\put(211.0,131.0){\rule[-0.200pt]{295.825pt}{0.400pt}}
\put(1439.0,131.0){\rule[-0.200pt]{0.400pt}{155.380pt}}
\put(211.0,776.0){\rule[-0.200pt]{295.825pt}{0.400pt}}
\put(30,453){\makebox(0,0){\rotatebox{90}{$H^1(\mathcal{F})$ error}}}
\put(825,29){\makebox(0,0){$h$}}
\put(825,838){\makebox(0,0){velocity}}
\put(1279,212){\makebox(0,0)[r]{$P_2-P_1-P_1$ (slope=1.947)}}
\put(1299.0,212.0){\rule[-0.200pt]{24.090pt}{0.400pt}}
\put(1218,638){\usebox{\plotpoint}}
\multiput(1214.28,636.92)(-0.996,-0.499){183}{\rule{0.896pt}{0.120pt}}
\multiput(1216.14,637.17)(-183.141,-93.000){2}{\rule{0.448pt}{0.400pt}}
\multiput(1029.47,543.92)(-0.940,-0.499){193}{\rule{0.851pt}{0.120pt}}
\multiput(1031.23,544.17)(-182.234,-98.000){2}{\rule{0.426pt}{0.400pt}}
\multiput(844.54,445.92)(-1.220,-0.499){149}{\rule{1.074pt}{0.120pt}}
\multiput(846.77,446.17)(-182.772,-76.000){2}{\rule{0.537pt}{0.400pt}}
\multiput(660.98,369.92)(-0.784,-0.499){233}{\rule{0.727pt}{0.120pt}}
\multiput(662.49,370.17)(-183.491,-118.000){2}{\rule{0.364pt}{0.400pt}}
\put(1218,638){\makebox(0,0){$+$}}
\put(1033,545){\makebox(0,0){$+$}}
\put(849,447){\makebox(0,0){$+$}}
\put(664,371){\makebox(0,0){$+$}}
\put(479,253){\makebox(0,0){$+$}}
\put(1349,212){\makebox(0,0){$+$}}
\put(1279,171){\makebox(0,0)[r]{$P_2-P_1-P_0$ (slope=1.485)}}
\multiput(1299,171)(20.756,0.000){5}{\usebox{\plotpoint}}
\put(1399,171){\usebox{\plotpoint}}
\put(1218,643){\usebox{\plotpoint}}
\multiput(1218,643)(-20.696,1.566){9}{\usebox{\plotpoint}}
\multiput(1033,657)(-16.044,-13.167){12}{\usebox{\plotpoint}}
\multiput(849,506)(-20.745,-0.673){9}{\usebox{\plotpoint}}
\multiput(664,500)(-16.594,-12.468){11}{\usebox{\plotpoint}}
\put(479,361){\usebox{\plotpoint}}
\put(1218,643){\makebox(0,0){$\times$}}
\put(1033,657){\makebox(0,0){$\times$}}
\put(849,506){\makebox(0,0){$\times$}}
\put(664,500){\makebox(0,0){$\times$}}
\put(479,361){\makebox(0,0){$\times$}}
\put(1349,171){\makebox(0,0){$\times$}}
\put(211.0,131.0){\rule[-0.200pt]{0.400pt}{155.380pt}}
\put(211.0,131.0){\rule[-0.200pt]{295.825pt}{0.400pt}}
\put(1439.0,131.0){\rule[-0.200pt]{0.400pt}{155.380pt}}
\put(211.0,776.0){\rule[-0.200pt]{295.825pt}{0.400pt}}
\end{picture}}\\
\resizebox{8cm}{!}{
\setlength{\unitlength}{0.240900pt}
\ifx\plotpoint\undefined\newsavebox{\plotpoint}\fi
\begin{picture}(1500,900)(0,0)
\sbox{\plotpoint}{\rule[-0.200pt]{0.400pt}{0.400pt}}%
\put(211.0,131.0){\rule[-0.200pt]{4.818pt}{0.400pt}}
\put(191,131){\makebox(0,0)[r]{$1e-05$}}
\put(1419.0,131.0){\rule[-0.200pt]{4.818pt}{0.400pt}}
\put(211.0,180.0){\rule[-0.200pt]{2.409pt}{0.400pt}}
\put(1429.0,180.0){\rule[-0.200pt]{2.409pt}{0.400pt}}
\put(211.0,208.0){\rule[-0.200pt]{2.409pt}{0.400pt}}
\put(1429.0,208.0){\rule[-0.200pt]{2.409pt}{0.400pt}}
\put(211.0,228.0){\rule[-0.200pt]{2.409pt}{0.400pt}}
\put(1429.0,228.0){\rule[-0.200pt]{2.409pt}{0.400pt}}
\put(211.0,244.0){\rule[-0.200pt]{2.409pt}{0.400pt}}
\put(1429.0,244.0){\rule[-0.200pt]{2.409pt}{0.400pt}}
\put(211.0,256.0){\rule[-0.200pt]{2.409pt}{0.400pt}}
\put(1429.0,256.0){\rule[-0.200pt]{2.409pt}{0.400pt}}
\put(211.0,267.0){\rule[-0.200pt]{2.409pt}{0.400pt}}
\put(1429.0,267.0){\rule[-0.200pt]{2.409pt}{0.400pt}}
\put(211.0,277.0){\rule[-0.200pt]{2.409pt}{0.400pt}}
\put(1429.0,277.0){\rule[-0.200pt]{2.409pt}{0.400pt}}
\put(211.0,285.0){\rule[-0.200pt]{2.409pt}{0.400pt}}
\put(1429.0,285.0){\rule[-0.200pt]{2.409pt}{0.400pt}}
\put(211.0,292.0){\rule[-0.200pt]{4.818pt}{0.400pt}}
\put(191,292){\makebox(0,0)[r]{$0.0001$}}
\put(1419.0,292.0){\rule[-0.200pt]{4.818pt}{0.400pt}}
\put(211.0,341.0){\rule[-0.200pt]{2.409pt}{0.400pt}}
\put(1429.0,341.0){\rule[-0.200pt]{2.409pt}{0.400pt}}
\put(211.0,369.0){\rule[-0.200pt]{2.409pt}{0.400pt}}
\put(1429.0,369.0){\rule[-0.200pt]{2.409pt}{0.400pt}}
\put(211.0,389.0){\rule[-0.200pt]{2.409pt}{0.400pt}}
\put(1429.0,389.0){\rule[-0.200pt]{2.409pt}{0.400pt}}
\put(211.0,405.0){\rule[-0.200pt]{2.409pt}{0.400pt}}
\put(1429.0,405.0){\rule[-0.200pt]{2.409pt}{0.400pt}}
\put(211.0,418.0){\rule[-0.200pt]{2.409pt}{0.400pt}}
\put(1429.0,418.0){\rule[-0.200pt]{2.409pt}{0.400pt}}
\put(211.0,429.0){\rule[-0.200pt]{2.409pt}{0.400pt}}
\put(1429.0,429.0){\rule[-0.200pt]{2.409pt}{0.400pt}}
\put(211.0,438.0){\rule[-0.200pt]{2.409pt}{0.400pt}}
\put(1429.0,438.0){\rule[-0.200pt]{2.409pt}{0.400pt}}
\put(211.0,446.0){\rule[-0.200pt]{2.409pt}{0.400pt}}
\put(1429.0,446.0){\rule[-0.200pt]{2.409pt}{0.400pt}}
\put(211.0,454.0){\rule[-0.200pt]{4.818pt}{0.400pt}}
\put(191,454){\makebox(0,0)[r]{$0.001$}}
\put(1419.0,454.0){\rule[-0.200pt]{4.818pt}{0.400pt}}
\put(211.0,502.0){\rule[-0.200pt]{2.409pt}{0.400pt}}
\put(1429.0,502.0){\rule[-0.200pt]{2.409pt}{0.400pt}}
\put(211.0,530.0){\rule[-0.200pt]{2.409pt}{0.400pt}}
\put(1429.0,530.0){\rule[-0.200pt]{2.409pt}{0.400pt}}
\put(211.0,551.0){\rule[-0.200pt]{2.409pt}{0.400pt}}
\put(1429.0,551.0){\rule[-0.200pt]{2.409pt}{0.400pt}}
\put(211.0,566.0){\rule[-0.200pt]{2.409pt}{0.400pt}}
\put(1429.0,566.0){\rule[-0.200pt]{2.409pt}{0.400pt}}
\put(211.0,579.0){\rule[-0.200pt]{2.409pt}{0.400pt}}
\put(1429.0,579.0){\rule[-0.200pt]{2.409pt}{0.400pt}}
\put(211.0,590.0){\rule[-0.200pt]{2.409pt}{0.400pt}}
\put(1429.0,590.0){\rule[-0.200pt]{2.409pt}{0.400pt}}
\put(211.0,599.0){\rule[-0.200pt]{2.409pt}{0.400pt}}
\put(1429.0,599.0){\rule[-0.200pt]{2.409pt}{0.400pt}}
\put(211.0,607.0){\rule[-0.200pt]{2.409pt}{0.400pt}}
\put(1429.0,607.0){\rule[-0.200pt]{2.409pt}{0.400pt}}
\put(211.0,615.0){\rule[-0.200pt]{4.818pt}{0.400pt}}
\put(191,615){\makebox(0,0)[r]{$0.01$}}
\put(1419.0,615.0){\rule[-0.200pt]{4.818pt}{0.400pt}}
\put(211.0,663.0){\rule[-0.200pt]{2.409pt}{0.400pt}}
\put(1429.0,663.0){\rule[-0.200pt]{2.409pt}{0.400pt}}
\put(211.0,692.0){\rule[-0.200pt]{2.409pt}{0.400pt}}
\put(1429.0,692.0){\rule[-0.200pt]{2.409pt}{0.400pt}}
\put(211.0,712.0){\rule[-0.200pt]{2.409pt}{0.400pt}}
\put(1429.0,712.0){\rule[-0.200pt]{2.409pt}{0.400pt}}
\put(211.0,727.0){\rule[-0.200pt]{2.409pt}{0.400pt}}
\put(1429.0,727.0){\rule[-0.200pt]{2.409pt}{0.400pt}}
\put(211.0,740.0){\rule[-0.200pt]{2.409pt}{0.400pt}}
\put(1429.0,740.0){\rule[-0.200pt]{2.409pt}{0.400pt}}
\put(211.0,751.0){\rule[-0.200pt]{2.409pt}{0.400pt}}
\put(1429.0,751.0){\rule[-0.200pt]{2.409pt}{0.400pt}}
\put(211.0,760.0){\rule[-0.200pt]{2.409pt}{0.400pt}}
\put(1429.0,760.0){\rule[-0.200pt]{2.409pt}{0.400pt}}
\put(211.0,769.0){\rule[-0.200pt]{2.409pt}{0.400pt}}
\put(1429.0,769.0){\rule[-0.200pt]{2.409pt}{0.400pt}}
\put(211.0,776.0){\rule[-0.200pt]{4.818pt}{0.400pt}}
\put(191,776){\makebox(0,0)[r]{$0.1$}}
\put(1419.0,776.0){\rule[-0.200pt]{4.818pt}{0.400pt}}
\put(211.0,131.0){\rule[-0.200pt]{0.400pt}{4.818pt}}
\put(211,90){\makebox(0,0){$0.001$}}
\put(211.0,756.0){\rule[-0.200pt]{0.400pt}{4.818pt}}
\put(396.0,131.0){\rule[-0.200pt]{0.400pt}{2.409pt}}
\put(396.0,766.0){\rule[-0.200pt]{0.400pt}{2.409pt}}
\put(504.0,131.0){\rule[-0.200pt]{0.400pt}{2.409pt}}
\put(504.0,766.0){\rule[-0.200pt]{0.400pt}{2.409pt}}
\put(581.0,131.0){\rule[-0.200pt]{0.400pt}{2.409pt}}
\put(581.0,766.0){\rule[-0.200pt]{0.400pt}{2.409pt}}
\put(640.0,131.0){\rule[-0.200pt]{0.400pt}{2.409pt}}
\put(640.0,766.0){\rule[-0.200pt]{0.400pt}{2.409pt}}
\put(689.0,131.0){\rule[-0.200pt]{0.400pt}{2.409pt}}
\put(689.0,766.0){\rule[-0.200pt]{0.400pt}{2.409pt}}
\put(730.0,131.0){\rule[-0.200pt]{0.400pt}{2.409pt}}
\put(730.0,766.0){\rule[-0.200pt]{0.400pt}{2.409pt}}
\put(765.0,131.0){\rule[-0.200pt]{0.400pt}{2.409pt}}
\put(765.0,766.0){\rule[-0.200pt]{0.400pt}{2.409pt}}
\put(797.0,131.0){\rule[-0.200pt]{0.400pt}{2.409pt}}
\put(797.0,766.0){\rule[-0.200pt]{0.400pt}{2.409pt}}
\put(825.0,131.0){\rule[-0.200pt]{0.400pt}{4.818pt}}
\put(825,90){\makebox(0,0){$0.01$}}
\put(825.0,756.0){\rule[-0.200pt]{0.400pt}{4.818pt}}
\put(1010.0,131.0){\rule[-0.200pt]{0.400pt}{2.409pt}}
\put(1010.0,766.0){\rule[-0.200pt]{0.400pt}{2.409pt}}
\put(1118.0,131.0){\rule[-0.200pt]{0.400pt}{2.409pt}}
\put(1118.0,766.0){\rule[-0.200pt]{0.400pt}{2.409pt}}
\put(1195.0,131.0){\rule[-0.200pt]{0.400pt}{2.409pt}}
\put(1195.0,766.0){\rule[-0.200pt]{0.400pt}{2.409pt}}
\put(1254.0,131.0){\rule[-0.200pt]{0.400pt}{2.409pt}}
\put(1254.0,766.0){\rule[-0.200pt]{0.400pt}{2.409pt}}
\put(1303.0,131.0){\rule[-0.200pt]{0.400pt}{2.409pt}}
\put(1303.0,766.0){\rule[-0.200pt]{0.400pt}{2.409pt}}
\put(1344.0,131.0){\rule[-0.200pt]{0.400pt}{2.409pt}}
\put(1344.0,766.0){\rule[-0.200pt]{0.400pt}{2.409pt}}
\put(1379.0,131.0){\rule[-0.200pt]{0.400pt}{2.409pt}}
\put(1379.0,766.0){\rule[-0.200pt]{0.400pt}{2.409pt}}
\put(1411.0,131.0){\rule[-0.200pt]{0.400pt}{2.409pt}}
\put(1411.0,766.0){\rule[-0.200pt]{0.400pt}{2.409pt}}
\put(1439.0,131.0){\rule[-0.200pt]{0.400pt}{4.818pt}}
\put(1439,90){\makebox(0,0){$0.1$}}
\put(1439.0,756.0){\rule[-0.200pt]{0.400pt}{4.818pt}}
\put(211.0,131.0){\rule[-0.200pt]{0.400pt}{155.380pt}}
\put(211.0,131.0){\rule[-0.200pt]{295.825pt}{0.400pt}}
\put(1439.0,131.0){\rule[-0.200pt]{0.400pt}{155.380pt}}
\put(211.0,776.0){\rule[-0.200pt]{295.825pt}{0.400pt}}
\put(30,453){\makebox(0,0){\rotatebox{90}{$L^2(\mathcal{F})$ error}}}
\put(825,29){\makebox(0,0){$h$}}
\put(825,838){\makebox(0,0){pressure}}
\put(1279,212){\makebox(0,0)[r]{$P_2-P_1-P_1$ (slope=2.004)}}
\put(1299.0,212.0){\rule[-0.200pt]{24.090pt}{0.400pt}}
\put(1218,613){\usebox{\plotpoint}}
\multiput(1214.39,611.92)(-0.965,-0.499){189}{\rule{0.871pt}{0.120pt}}
\multiput(1216.19,612.17)(-183.193,-96.000){2}{\rule{0.435pt}{0.400pt}}
\multiput(1029.40,515.92)(-0.960,-0.499){189}{\rule{0.867pt}{0.120pt}}
\multiput(1031.20,516.17)(-182.201,-96.000){2}{\rule{0.433pt}{0.400pt}}
\multiput(845.45,419.92)(-0.945,-0.499){193}{\rule{0.855pt}{0.120pt}}
\multiput(847.23,420.17)(-183.225,-98.000){2}{\rule{0.428pt}{0.400pt}}
\multiput(660.48,321.92)(-0.935,-0.499){195}{\rule{0.847pt}{0.120pt}}
\multiput(662.24,322.17)(-183.241,-99.000){2}{\rule{0.424pt}{0.400pt}}
\put(1218,613){\makebox(0,0){$+$}}
\put(1033,517){\makebox(0,0){$+$}}
\put(849,421){\makebox(0,0){$+$}}
\put(664,323){\makebox(0,0){$+$}}
\put(479,224){\makebox(0,0){$+$}}
\put(1349,212){\makebox(0,0){$+$}}
\put(1279,171){\makebox(0,0)[r]{$P_2-P_1-P_0$ (slope=1.637)}}
\multiput(1299,171)(20.756,0.000){5}{\usebox{\plotpoint}}
\put(1399,171){\usebox{\plotpoint}}
\put(1218,616){\usebox{\plotpoint}}
\multiput(1218,616)(-20.688,-1.677){9}{\usebox{\plotpoint}}
\multiput(1033,601)(-17.644,-10.931){11}{\usebox{\plotpoint}}
\multiput(849,487)(-18.664,-9.080){10}{\usebox{\plotpoint}}
\multiput(664,397)(-19.199,-7.887){9}{\usebox{\plotpoint}}
\put(479,321){\usebox{\plotpoint}}
\put(1218,616){\makebox(0,0){$\times$}}
\put(1033,601){\makebox(0,0){$\times$}}
\put(849,487){\makebox(0,0){$\times$}}
\put(664,397){\makebox(0,0){$\times$}}
\put(479,321){\makebox(0,0){$\times$}}
\put(1349,171){\makebox(0,0){$\times$}}
\put(211.0,131.0){\rule[-0.200pt]{0.400pt}{155.380pt}}
\put(211.0,131.0){\rule[-0.200pt]{295.825pt}{0.400pt}}
\put(1439.0,131.0){\rule[-0.200pt]{0.400pt}{155.380pt}}
\put(211.0,776.0){\rule[-0.200pt]{295.825pt}{0.400pt}}
\end{picture}}
\resizebox{8cm}{!}{
\setlength{\unitlength}{0.240900pt}
\ifx\plotpoint\undefined\newsavebox{\plotpoint}\fi
\begin{picture}(1500,900)(0,0)
\sbox{\plotpoint}{\rule[-0.200pt]{0.400pt}{0.400pt}}%
\put(211.0,131.0){\rule[-0.200pt]{4.818pt}{0.400pt}}
\put(191,131){\makebox(0,0)[r]{$1e-08$}}
\put(1419.0,131.0){\rule[-0.200pt]{4.818pt}{0.400pt}}
\put(211.0,170.0){\rule[-0.200pt]{2.409pt}{0.400pt}}
\put(1429.0,170.0){\rule[-0.200pt]{2.409pt}{0.400pt}}
\put(211.0,221.0){\rule[-0.200pt]{2.409pt}{0.400pt}}
\put(1429.0,221.0){\rule[-0.200pt]{2.409pt}{0.400pt}}
\put(211.0,247.0){\rule[-0.200pt]{2.409pt}{0.400pt}}
\put(1429.0,247.0){\rule[-0.200pt]{2.409pt}{0.400pt}}
\put(211.0,260.0){\rule[-0.200pt]{4.818pt}{0.400pt}}
\put(191,260){\makebox(0,0)[r]{$1e-07$}}
\put(1419.0,260.0){\rule[-0.200pt]{4.818pt}{0.400pt}}
\put(211.0,299.0){\rule[-0.200pt]{2.409pt}{0.400pt}}
\put(1429.0,299.0){\rule[-0.200pt]{2.409pt}{0.400pt}}
\put(211.0,350.0){\rule[-0.200pt]{2.409pt}{0.400pt}}
\put(1429.0,350.0){\rule[-0.200pt]{2.409pt}{0.400pt}}
\put(211.0,376.0){\rule[-0.200pt]{2.409pt}{0.400pt}}
\put(1429.0,376.0){\rule[-0.200pt]{2.409pt}{0.400pt}}
\put(211.0,389.0){\rule[-0.200pt]{4.818pt}{0.400pt}}
\put(191,389){\makebox(0,0)[r]{$1e-06$}}
\put(1419.0,389.0){\rule[-0.200pt]{4.818pt}{0.400pt}}
\put(211.0,428.0){\rule[-0.200pt]{2.409pt}{0.400pt}}
\put(1429.0,428.0){\rule[-0.200pt]{2.409pt}{0.400pt}}
\put(211.0,479.0){\rule[-0.200pt]{2.409pt}{0.400pt}}
\put(1429.0,479.0){\rule[-0.200pt]{2.409pt}{0.400pt}}
\put(211.0,505.0){\rule[-0.200pt]{2.409pt}{0.400pt}}
\put(1429.0,505.0){\rule[-0.200pt]{2.409pt}{0.400pt}}
\put(211.0,518.0){\rule[-0.200pt]{4.818pt}{0.400pt}}
\put(191,518){\makebox(0,0)[r]{$1e-05$}}
\put(1419.0,518.0){\rule[-0.200pt]{4.818pt}{0.400pt}}
\put(211.0,557.0){\rule[-0.200pt]{2.409pt}{0.400pt}}
\put(1429.0,557.0){\rule[-0.200pt]{2.409pt}{0.400pt}}
\put(211.0,608.0){\rule[-0.200pt]{2.409pt}{0.400pt}}
\put(1429.0,608.0){\rule[-0.200pt]{2.409pt}{0.400pt}}
\put(211.0,634.0){\rule[-0.200pt]{2.409pt}{0.400pt}}
\put(1429.0,634.0){\rule[-0.200pt]{2.409pt}{0.400pt}}
\put(211.0,647.0){\rule[-0.200pt]{4.818pt}{0.400pt}}
\put(191,647){\makebox(0,0)[r]{$0.0001$}}
\put(1419.0,647.0){\rule[-0.200pt]{4.818pt}{0.400pt}}
\put(211.0,686.0){\rule[-0.200pt]{2.409pt}{0.400pt}}
\put(1429.0,686.0){\rule[-0.200pt]{2.409pt}{0.400pt}}
\put(211.0,737.0){\rule[-0.200pt]{2.409pt}{0.400pt}}
\put(1429.0,737.0){\rule[-0.200pt]{2.409pt}{0.400pt}}
\put(211.0,763.0){\rule[-0.200pt]{2.409pt}{0.400pt}}
\put(1429.0,763.0){\rule[-0.200pt]{2.409pt}{0.400pt}}
\put(211.0,776.0){\rule[-0.200pt]{4.818pt}{0.400pt}}
\put(191,776){\makebox(0,0)[r]{$0.001$}}
\put(1419.0,776.0){\rule[-0.200pt]{4.818pt}{0.400pt}}
\put(211.0,131.0){\rule[-0.200pt]{0.400pt}{4.818pt}}
\put(211,90){\makebox(0,0){$0.001$}}
\put(211.0,756.0){\rule[-0.200pt]{0.400pt}{4.818pt}}
\put(396.0,131.0){\rule[-0.200pt]{0.400pt}{2.409pt}}
\put(396.0,766.0){\rule[-0.200pt]{0.400pt}{2.409pt}}
\put(504.0,131.0){\rule[-0.200pt]{0.400pt}{2.409pt}}
\put(504.0,766.0){\rule[-0.200pt]{0.400pt}{2.409pt}}
\put(581.0,131.0){\rule[-0.200pt]{0.400pt}{2.409pt}}
\put(581.0,766.0){\rule[-0.200pt]{0.400pt}{2.409pt}}
\put(640.0,131.0){\rule[-0.200pt]{0.400pt}{2.409pt}}
\put(640.0,766.0){\rule[-0.200pt]{0.400pt}{2.409pt}}
\put(689.0,131.0){\rule[-0.200pt]{0.400pt}{2.409pt}}
\put(689.0,766.0){\rule[-0.200pt]{0.400pt}{2.409pt}}
\put(730.0,131.0){\rule[-0.200pt]{0.400pt}{2.409pt}}
\put(730.0,766.0){\rule[-0.200pt]{0.400pt}{2.409pt}}
\put(765.0,131.0){\rule[-0.200pt]{0.400pt}{2.409pt}}
\put(765.0,766.0){\rule[-0.200pt]{0.400pt}{2.409pt}}
\put(797.0,131.0){\rule[-0.200pt]{0.400pt}{2.409pt}}
\put(797.0,766.0){\rule[-0.200pt]{0.400pt}{2.409pt}}
\put(825.0,131.0){\rule[-0.200pt]{0.400pt}{4.818pt}}
\put(825,90){\makebox(0,0){$0.01$}}
\put(825.0,756.0){\rule[-0.200pt]{0.400pt}{4.818pt}}
\put(1010.0,131.0){\rule[-0.200pt]{0.400pt}{2.409pt}}
\put(1010.0,766.0){\rule[-0.200pt]{0.400pt}{2.409pt}}
\put(1118.0,131.0){\rule[-0.200pt]{0.400pt}{2.409pt}}
\put(1118.0,766.0){\rule[-0.200pt]{0.400pt}{2.409pt}}
\put(1195.0,131.0){\rule[-0.200pt]{0.400pt}{2.409pt}}
\put(1195.0,766.0){\rule[-0.200pt]{0.400pt}{2.409pt}}
\put(1254.0,131.0){\rule[-0.200pt]{0.400pt}{2.409pt}}
\put(1254.0,766.0){\rule[-0.200pt]{0.400pt}{2.409pt}}
\put(1303.0,131.0){\rule[-0.200pt]{0.400pt}{2.409pt}}
\put(1303.0,766.0){\rule[-0.200pt]{0.400pt}{2.409pt}}
\put(1344.0,131.0){\rule[-0.200pt]{0.400pt}{2.409pt}}
\put(1344.0,766.0){\rule[-0.200pt]{0.400pt}{2.409pt}}
\put(1379.0,131.0){\rule[-0.200pt]{0.400pt}{2.409pt}}
\put(1379.0,766.0){\rule[-0.200pt]{0.400pt}{2.409pt}}
\put(1411.0,131.0){\rule[-0.200pt]{0.400pt}{2.409pt}}
\put(1411.0,766.0){\rule[-0.200pt]{0.400pt}{2.409pt}}
\put(1439.0,131.0){\rule[-0.200pt]{0.400pt}{4.818pt}}
\put(1439,90){\makebox(0,0){$0.1$}}
\put(1439.0,756.0){\rule[-0.200pt]{0.400pt}{4.818pt}}
\put(211.0,131.0){\rule[-0.200pt]{0.400pt}{155.380pt}}
\put(211.0,131.0){\rule[-0.200pt]{295.825pt}{0.400pt}}
\put(1439.0,131.0){\rule[-0.200pt]{0.400pt}{155.380pt}}
\put(211.0,776.0){\rule[-0.200pt]{295.825pt}{0.400pt}}
\put(30,453){\makebox(0,0){\rotatebox{90}{$\left|\int_{\Gamma}(\lambda-\lambda_{h})\right|$}}}
\put(825,29){\makebox(0,0){$h$}}
\put(825,838){\makebox(0,0){multiplier}}
\put(1279,212){\makebox(0,0)[r]{$P_2-P_1-P_1$ (slope=3.116)}}
\put(1299.0,212.0){\rule[-0.200pt]{24.090pt}{0.400pt}}
\put(1218,714){\usebox{\plotpoint}}
\multiput(1216.92,711.65)(-0.500,-0.584){367}{\rule{0.120pt}{0.567pt}}
\multiput(1217.17,712.82)(-185.000,-214.823){2}{\rule{0.400pt}{0.284pt}}
\multiput(1022.73,496.92)(-2.995,-0.497){59}{\rule{2.474pt}{0.120pt}}
\multiput(1027.86,497.17)(-178.865,-31.000){2}{\rule{1.237pt}{0.400pt}}
\multiput(847.92,464.66)(-0.500,-0.581){367}{\rule{0.120pt}{0.565pt}}
\multiput(848.17,465.83)(-185.000,-213.828){2}{\rule{0.400pt}{0.282pt}}
\multiput(648.23,250.92)(-4.701,-0.496){37}{\rule{3.800pt}{0.119pt}}
\multiput(656.11,251.17)(-177.113,-20.000){2}{\rule{1.900pt}{0.400pt}}
\put(1218,714){\makebox(0,0){$+$}}
\put(1033,498){\makebox(0,0){$+$}}
\put(849,467){\makebox(0,0){$+$}}
\put(664,252){\makebox(0,0){$+$}}
\put(479,232){\makebox(0,0){$+$}}
\put(1349,212){\makebox(0,0){$+$}}
\put(1279,171){\makebox(0,0)[r]{$P_2-P_1-P_0$ (slope=3.301)}}
\multiput(1299,171)(20.756,0.000){5}{\usebox{\plotpoint}}
\put(1399,171){\usebox{\plotpoint}}
\put(1218,723){\usebox{\plotpoint}}
\multiput(1218,723)(-15.994,-13.228){12}{\usebox{\plotpoint}}
\multiput(1033,570)(-18.027,-10.287){10}{\usebox{\plotpoint}}
\multiput(849,465)(-14.324,-15.021){13}{\usebox{\plotpoint}}
\multiput(664,271)(-20.309,-4.281){9}{\usebox{\plotpoint}}
\put(479,232){\usebox{\plotpoint}}
\put(1218,723){\makebox(0,0){$\times$}}
\put(1033,570){\makebox(0,0){$\times$}}
\put(849,465){\makebox(0,0){$\times$}}
\put(664,271){\makebox(0,0){$\times$}}
\put(479,232){\makebox(0,0){$\times$}}
\put(1349,171){\makebox(0,0){$\times$}}
\put(211.0,131.0){\rule[-0.200pt]{0.400pt}{155.380pt}}
\put(211.0,131.0){\rule[-0.200pt]{295.825pt}{0.400pt}}
\put(1439.0,131.0){\rule[-0.200pt]{0.400pt}{155.380pt}}
\put(211.0,776.0){\rule[-0.200pt]{295.825pt}{0.400pt}}
\end{picture}}
\end{minipage}
\caption{Rates of convergence for Taylor-Hood elements with Haslinger-Renard stabilization for $\|u-u_h\|_{0,\mathcal{F}}$, 
$\|u-u_h\|_{1,\mathcal{F}}$, $\|p-p_h\|_{0,\mathcal{F}}$ and $\left|\int_{\Gamma}(\lambda-\lambda_{h})\right|$ }
\label{fig:Taylor-Hood}
\end{figure}
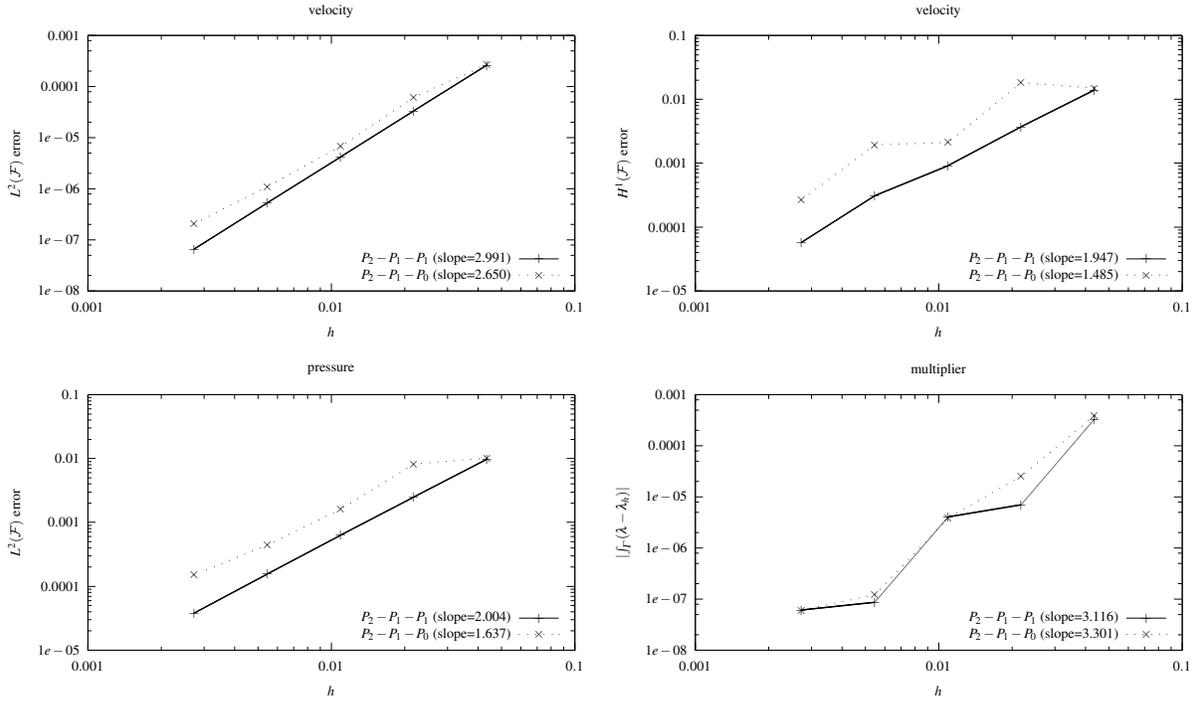
\end{center}

\subsection{Methods \`a la Burman-Hansbo.}
For the methods \`a la Burman-Hansbo, some stabilization terms (multiplied by $\gamma$ and, eventually, $\theta$) are added to the system (\ref{sys:sans_stab}).
This yields
\begin{equation}
\label{sys:burman-hansbo}
\left (
\begin{array}{ccc}
K & B^T&C^T\\
B & S_{l_{p}}^{\theta_0}& 0\\
C & 0& S^{\gamma}_{l_{\lambda}}\\
\end{array}
\right )
\left (
\begin{array}{c}
U\\
P\\
\Lambda\\
\end{array}
\right ) = 
\left (
\begin{array}{c}
F\\
0\\
G \\
\end{array}
\right ) 
\end{equation}
for $l_{\lambda}=0,1$ and $l_{p}=0,1,2$ with \\
$
\begin{array}{l}
 \left ( S^{\gamma}_{0} \right )_{i_\lambda j_\lambda} =  \left ( S^{\gamma,0}_{[\lambda][\lambda]} \right )_{i_\lambda j_\lambda}= - \gamma h  \sum_{E \in \mathcal{E}_h^{\Gamma}}
  \int_{E} [\zeta_{i_\lambda}] \cdot [\zeta_{j_\lambda}],
 \quad \left ( S^{\gamma}_{1} \right )_{i_\lambda j_\lambda}= \left ( S^{\gamma,1}_{\lambda \lambda} \right )_{i_\lambda j_\lambda}=- \gamma h^2   \int_{\mathcal{F}_h^{\Gamma}} \nabla \zeta_{i_\lambda} . \nabla \zeta_{j_\lambda}\\
 S^{\theta}_{0} =  S_{pp}^{\theta}, \quad   S^{\theta}_{1}=S_{[p][p]}^{\theta},  \quad  S^{\theta}_{2}= 0
\end{array}
$

\begin{center}
\begin{figure}
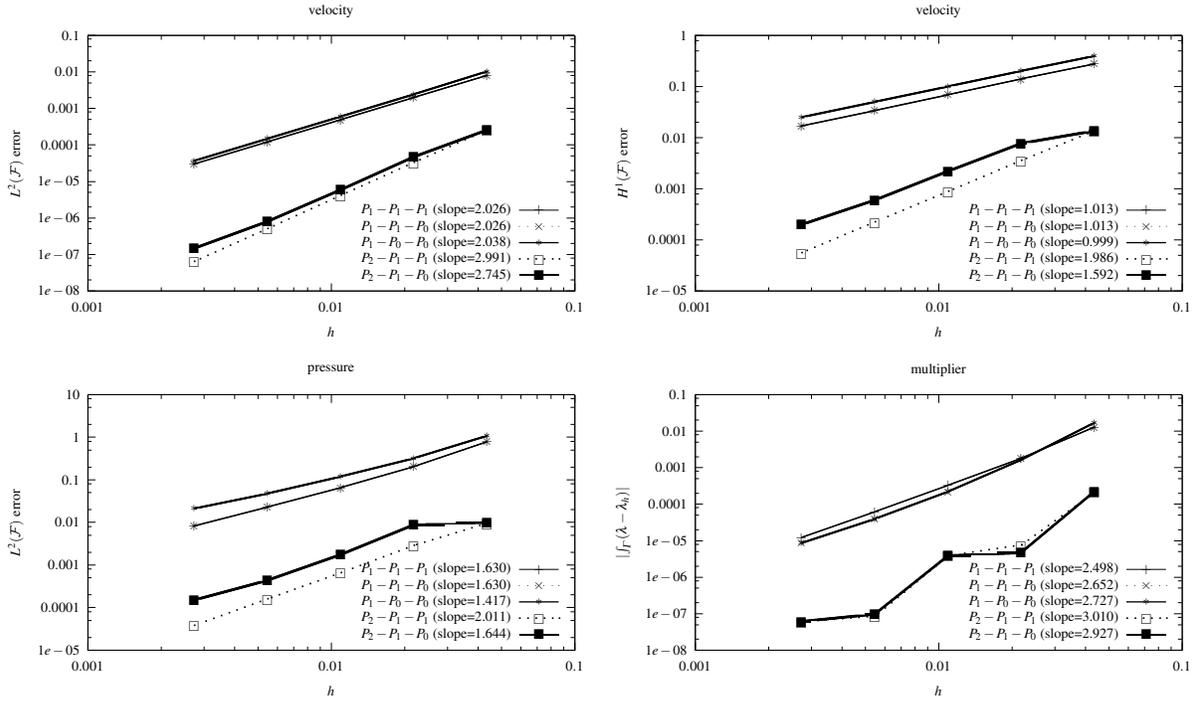

\hspace*{-2cm}
\begin{minipage}{20cm}
 \resizebox{8cm}{!}{\input{FIGURES/fig6-u}}
  \resizebox{8cm}{!}{\input{FIGURES/fig6-u-h1}}\\
\resizebox{8cm}{!}{\input{FIGURES/fig6-p}}
\resizebox{8cm}{!}{\input{FIGURES/fig6-l-force}}
\end{minipage}
\caption{Rates of convergence with Burman-Hansbo stabilization for $\|u-u_h\|_{0,\mathcal{F}}$, 
$\|u-u_h\|_{1,\mathcal{F}}$, $\|p-p_h\|_{0,\mathcal{F}}$ and $\left|\int_{\Gamma}(\lambda-\lambda_{h})\right|$ }
\label{fig:Burman-Hansbo}
\end{figure}
\end{center}

No "robust reconstruction" is applied here. The choice of the stabilization matrix for the multiplier $\lambda$ is determined by its FE space: we use $S_0^{\gamma}$ or  $S_1^{\gamma}$  for $\P_0$ or $\P_1$ space $W_h$ respectively.
The stabilization matrices for the pressure are added as in the preceding variants depending on the velocity-pressure FE couple ($S^{\theta}_{0}$ for $\P_1-\P_1$, $S^{\theta}_{1}$ for $\P_1-\P_0$, $S^{\theta}_{2}$ for $\P_2-\P_1$). 

The results are reported in Fig. \ref{fig:Burman-Hansbo}.  The optimal rates of convergence are recovered for all the variants. The accuracy of the method is close to that of the methods \`a la Haslinger-Renard, considered in the preceding subsections. For example, the results with $\P_2-\P_1-\P_1$ FE are comparable with those reported in Fig. \ref{fig:Taylor-Hood}.

\section{Conclusion}

In this paper, we have proposed fictitious domain methods for the Stokes problem that can be used in the context of fluid-structure interaction with complex interface.
We combine the Barbosa-Hughes approach with several stabilization strategies involving a "robust reconstruction" (Haslinger-Renard)  when small intersections of the mesh elements with the domain are present.  The optimal error estimates proven theoretically under non-restrictive assumptions are also confirmed numerically.
Alternative methods \`a la Burman-Hansbo are considered theoretically and numerically for Stokes problem and allow to recover similar results. 

\begin{acknowledgement}
We wish to thank Prof. Erik Burman for giving us the occasion to participate in the ``Unfitted FEM'' workshop and to contribute to this volume.  
We are indebted to Prof. Yves Renard -- the main developer of GetFEM++ library, used for all our numerical experiments -- for adapting this library for our needs and for useful advice. 
\end{acknowledgement}

\bibliographystyle{spmpsci}
\bibliography{xfem}
\end{document}